\documentclass[english]{smfbook} 

\usepackage{amssymb,bm,graphicx,smfenum,smfthm}

\usepackage{etoolbox}
\patchcmd{\thebibliography}{*}{}{}{}
\pretocmd\thebibliography{\csname c@secnumdepth\endcsname=-2 }{}{}
\usepackage{hyperref}
\pdfoutput=1 
\numberwithin{equation}{chapter}

\renewcommand\d{\partial}
\renewcommand\a{\alpha}
\renewcommand\b{\beta}
\renewcommand\o{\omega}
\newcommand\s{\sigma}
\renewcommand\t{\tau}
\newcommand\R{\mathbb R}\newcommand\N{\mathbb N}\newcommand\Z{\mathbb Z}
\newcommand\C{\mathbb C}
\newcommand\un{\underline}
\def\g{\gamma}

\def\G{\Gamma}
\def\t{\tau}
\def\O{\Omega}
\def\th{\theta}
\def\k{\kappa}

\def\l{\lambda}

\def\e{\varepsilon}

\def\S{\mathbb S}
\def\L{\Lambda}

\def\dsp{\displaystyle}

\def\inn{\mbox{in}}
\def\out{\mbox{out}}

\newcommand\br{\begin{rema}\phantomsection}
\newcommand\er{\end{rema}}
\newcommand\bp{\begin{pmatrix}}
\newcommand\ep{\end{pmatrix}}
\newcommand\be{\begin{equation}}
\newcommand\ee{\end{equation}}
\newcommand\ba{\begin{equation}\begin{aligned}}
\newcommand\ea{\end{aligned}\end{equation}}

\newcommand{\mA}{{\mathbb A}}

\newcommand{\T}{{\mathbb T}}

\newcommand{\Id}{{\rm Id }}

\newcommand{\tr}{{\rm tr }}
\newcommand{\supp}{{\rm supp }}

\newcommand{\B}{{\mathcal B}}

\newtheorem{assump}{Assumption}
\newtheorem{notation}{Notation}
\def\op{{\rm op} }
\def\ope{{\rm op}_\e^\psi }

\makeindex

\begin{document}

\frontmatter 

\title{A stability criterion for high-frequency oscillations}

\author{Yong Lu}
\address{Mathematical Institute,
Charles University, Prague}
\email{luyong@karlin.mff.cuni.cz}

\author{Benjamin Texier}
\address{Institut de Math\'ematiques de Jussieu-Paris Rive Gauche UMR CNRS 7586, Universit\'e Paris-Diderot}
\email{benjamin.texier@imj-prg.fr}

\thanks{Research of both authors was partially supported by the Project ``Instabilities in Hydrodynamics'' funded by the Mairie de Paris (under the ``Emergences'' program) and the Fondation Sciences Math\'ematiques de Paris. B.T. thanks Kevin Zumbrun for interesting discussions. Both authors thank Jeffrey Rauch for his comments on an earlier version of the manuscript.}

\date{\today}

\begin{abstract} We show that a simple Levi compatibility condition determines stability of WKB solutions to semilinear hyperbolic initial-value problems issued from highly-oscillating initial data with large amplitudes. The compatibility condition involves the hyperbolic operator, the fundamental phase associated with the initial oscillation, and the semilinear source term; it states roughly that hyperbolicity is preserved around resonances.

 If the compatibility condition is satisfied, the solutions are defined over time intervals independent of the wavelength, and the associated WKB solutions are stable under a large class of initial perturbations. If the compatibility condition is not satisfied, resonances are exponentially amplified, and arbitrarily small initial perturbations can destabilize the WKB solutions in small time.

 The amplification mechanism is based on the observation that in frequency space, resonances correspond to points of weak hyperbolicity. At such points, the behavior of the system depends on the lower order terms through the compatibility condition.

 The analysis relies, in the unstable case, on a short-time Duhamel representation formula for solutions of zeroth-order pseudo-differential equations.

 Our examples include coupled Klein-Gordon systems, and systems describing Raman and Brillouin instabilities.

  \end{abstract}


\maketitle

\tableofcontents


\mainmatter



\renewcommand{\refname}{References}


\chapter{Introduction}

We study highly-oscillating solutions to semi-linear systems of the form
 \begin{equation} \label{0}
   \d_t u + \frac{1}{\e} A_0 u + \sum_{1 \leq j \leq d} A_j \d_{x_j} u  = \frac{1}{\sqrt \e} B(u,u),
  \end{equation}
   in the small wavelength limit $\e \to 0.$ The unknown $u$ depends on time $t \in \R_+$ and space $x \in \R^d;$ it takes values in $\R^N.$
   The first-order differential operator is symmetric hyperbolic, in the sense that $A_0 \in \R^{N \times N}$ is skew-symmetric, and the $A_j \in \R^{N \times N}$ are symmetric. The source term is $B(u,u) \in \R^N,$ where $B: \R^N \times \R^N \to \R^N$ is bilinear; it has a large prefactor $1/\sqrt\e$ which blows up in the limit $\e \to 0.$

\smallskip

  Thus in \eqref{0} we are considering {\it large perturbations} of symmetric hyperbolic systems. In other words, the regime in \eqref{0} is supercritical: we are considering the propagation, over times $O(1),$ of solutions with amplitude $O(1)$ to systems \eqref{0} with characteristic frequencies $O(1/\e)$ and large $O(1/\sqrt \e)$ source terms.

\smallskip

 The underlying physical problems concern the propagation of light, and relevant data are highly oscillating, of the form
  \be \label{generic-data}
   u(\e,0,x) = \Re e \, a(x) e^{i k \cdot x/\e} + \e^K \phi(\e,x),
  \ee
  where $a$ has a high Sobolev regularity, $k$ is a given wavenumber in $\R^d,$ and $\e^K \phi$ is a small, real perturbation that is smooth in $x$ and may depend singularly on $\e.$

\smallskip

 In this setting, the existence and uniqueness of local-in-time solutions to \eqref{0}-\eqref{generic-data} for fixed $\e > 0$ in smooth Sobolev spaces $H^s,$ with $s > d/2,$ is classical. The a priori existence time is only $O(\e^{1/2}).$ Indeed, the symmetric hyperbolic operator conserves Sobolev norms, so that an energy estimate leads to a differential inequality of the form $(\star)\,\, y' \leq \e^{-1/2} y^2,$ where $y(0)$ is an appropriate Sobolev norm of the datum, typically a semiclassical norm in which derivatives appear as $\e \d_x,$ so that the fast oscillations are bounded: $y(0) = O(1);$ from $(\star)$ we deduce an existence time $O(\e^{1/2}).$

\smallskip

 {\it We consider the situation in which \eqref{0} admits a family of WKB approximate solutions which are defined over time intervals independent of $\e,$ and examine their stability with respect to small initial perturbations.}

\smallskip

 That is, given a WKB approximate solution $u_a$ issued from $u_a(0,x) = \Re e \, a(x) e^{i k \cdot x/\e},$ with an existence time $T_a$ that is uniformly bounded from below as $\e \to 0,$ given an initial perturbation $\e^K \phi,$ possibly with a very small amplitude $\e^K,$ we examine the question whether the exact solution to \eqref{0} issued from \eqref{generic-data} is defined over time intervals independent of $\e$ and stays close to $u_a.$

\smallskip

 Our answer to the above question takes the form of a scalar stability index, which involves the initial wavenumber $k,$ the initial amplitude $a,$ the source term $B,$ and the hyperbolic operator. The associated stability condition is a Levi condition, after E. E. Levi \cite{Levi}, in the sense that it involves both the principal and subprincipal symbols ($A$ and $B,$ respectively).

\smallskip

 Our examples include systems describing the Raman and Brillouin instabilities, and coupled Klein-Gordon systems.

\section{Background} \label{sec:background}

The class of problems \eqref{0} originates in Joly, M\'etivier and Rauch's article on the Maxwell-Bloch equations (see \cite{JMR-TMB}, and paragraph 6.3 of Dumas' survey \cite{Du}). There these authors considered Maxwell-Bloch systems in the critical regime of geometric optics, that is
\be \label{critical} \d_t u + \frac{1}{\e} A_0 u + \sum_{1 \leq j \leq d} A_j \d_{x_j} u = B(u,u), \qquad u(\e,0,x) = \Re e \,  a(x) e^{i k \cdot x/\e}.\ee
By {\it critical}, we mean here that \eqref{critical} is a regime in which nonlinear effects ought to be detected in the small wavelength limit $\e \to 0$ in time $O(1).$ Indeed, the nonlinear source has prefactor $O(1),$ and the amplitude of the datum is $O(1).$

Joly, M\'etivier and Rauch observed that for Maxwell-Bloch systems in the scaling \eqref{critical}, the limiting equations are {\it linear} transport equations. They called {\it transparency} this phenomenon, and explained how it originates in compability conditions involving the hyperbolic operator, the oscillations in the datum, and the source $B.$

Following Joly, M\'etivier and Rauch, it was verified by the second author that the Euler-Maxwell equations satisfy a form of transparency \cite{T1,T3}, and by the first author that the Maxwell-Landau-Lifschitz equations also are transparent in one spatial dimension \cite{Lu}. Cheverry, Gu\`es and M\'etivier showed in \cite{CGM} that for systems of conservation laws, linear degeneracy of a field implies transparency. Jeanne showed in \cite{Je} that the Yang-Mills equations provide another example of a physical system exhibiting transparency properties.

These results imply in particular that for the aforementioned physical systems, relevant regimes are supercritical, meaning that the appropriate scalings (of the observation time or the amplitudes) lead to systems with large nonlinear source terms, as in \eqref{0}.

Being a compatibility condition bearing on a nonlinear term, transparency is analogous to the {\it null form} conditions which imply global existence for nonlinear wave equations, as in the classical work of Klainerman \cite{Kla}. The link between transparency and null forms is one of the topics covered by Lannes in his Bourbaki review \cite{LBo}.

As formulated in \cite{JMR-TMB}, the two main questions in the analysis of the high-frequency limit in supercritical regimes are: (a) does there exist WKB approximate solutions? (b) are WKB solutions stable with respect to initial perturbations? If the answer to question (a) is positive, then typically the leading terms of WKB solutions satisfy limiting equations that are much simpler than the original system. If the answer to question (b) is positive, then the limiting equations can be used to describe the original system, in particular in numerical simulations.

The article \cite{JMR-TMB} shows existence and stability of WKB solutions to Maxwell-Bloch equations in a supercritical regime (different from \eqref{0}; we briefly comment on the difference in Remarks \ref{rem:jmr} and \ref{rem:last} in the Appendix). Later on existence and stability of some supercritical WKB solutions was shown for Yang-Mills in \cite{Je}, for Euler-Maxwell in \cite{T3}, for systems of conservation laws in \cite{CGM}, for Maxwell-Landau-Lifshitz in \cite{Lu}.

The present work all but completes the analysis of systems in the scaling \eqref{0}, as we exhibit a scalar index, which when positive implies instability and when negative implies stability.

\section{Resonances, transparency, and WKB solutions} \label{sec:restransp}

We introduce here the notions of resonance and transparency, which play a preeminent role in Joly, M\'etivier and Rauch's article \cite{JMR-TMB} and the present work.

Consider the initial-value problem \eqref{0}-\eqref{generic-data} with $\phi = 0.$ Under an appropriate \emph{polarization} condition bearing on the initial amplitude $a,$ the spatial oscillations in the datum are propagated in time by the hyperbolic operator in \eqref{0}, at some temporal frequency $\o = \o(k).$ Thus we posit the ansatz
 \be \label{ansatz}
  u(\e, t, x) = u_{0,-1}(t,x) e^{-i(k \cdot x - \o t)/\e} + u_{0,1}(t,x) e^{i (k \cdot x - \o t)/\e} + O(\sqrt \e),
 \ee
 for an approximate solution $u$ to \eqref{0}-\eqref{generic-data}. The bilinear term $B(u,u)$ in \eqref{0} will create {\it harmonics} of the fundamental phases $\pm (\o,k),$ so that the $O(\sqrt \e)$ term in \eqref{ansatz} will likely include oscillations $e^{i q (k \cdot x - \o t)/\e},$ with $q \in \{-2, 0 ,2\},$ in addition to the fundamental harmonics $\{-1,1\}.$ A refinement of \eqref{ansatz} is then
 \be \label{ansatz2}
 u(\e,t,x) = \sum_{p \in \{-1,1\}} u_{0,p}(t,x)  e^{i p (k \cdot x - \o t)/\e}   + \sqrt \e \sum_{|q| \leq 2}u_{1,q} (t,x)  e^{i q (k \cdot x - \o t )/\e}  + O(\e).
 \ee
 We inject \eqref{ansatz2} into \eqref{0} and sort out oscillating frequencies and powers of $\e.$ Thus conditions
 \be \label{wkb1}
 \Big( - i p \o + A_0 + A(i p k) \Big) u_{0,p} = 0, \quad p \in \{-1, 1 \},
 \ee
 and
 \be \label{wkb2}
 \Big( - i q \o + A_0 + A(i q k) \Big) u_{1,q}  + \Big( \d_t + A(\d_x) \Big) u_{0,q} = \sum_{q_1 + q_2 = q} B(u_{0,q_1}, u_{0,q_2}), \quad |q| \leq 2.
 \ee
 with notation
 $$A(\vec e\,) := \sum_j A_j e_j, \qquad \mbox{for any $\vec e = (e_1, \dots, e_d) \in \R^d,$}$$
 imply that \eqref{ansatz2} is an approximate solution to \eqref{0}, with a remainder of size $O(\e).$ In the case $A_0 \neq 0,$ the family of matrices $- i p \o + A_0 + A(i p k),$ for $p \in \Z,$ is not $1$-homogeneous in $(\o,k).$ As a result, only a finite number of these matrices is singular, for instance only those corresponding to $p \in \{ -1, 0, 1\}.$ Then, equation \eqref{wkb1} holds with non-trivial $u_{0,p}$ only if for all $(t,x),$ $u_{0,p}(t,x)$ is pointing in the direction of a element of the kernel of $-i p \o + A_0 + A(ip k):$
 \be \label{pola0}
  u_{0,p}(t,x) \equiv \langle u_{0,p}(t,x), \vec e_{p} \rangle \vec e_{p}, \qquad \vec e_{p} \in \ker \, \big( -i p \o + A_0 + A(ip k) \big), \,\, p \in \{-1,0,1\}.
 \ee
 Condition \eqref{pola0} is the {\it polarization} condition\footnote{For Maxwell's equations, with $u = (B,E),$ condition \eqref{pola0} takes the explicit form $\o B_{0,\pm 1} = k \times E_{0,\pm 1},$ corresponding to polarization of light.}. The mean mode for the initial datum \eqref{generic-data} vanishes identically. In our context, no mean mode is created by the nonlinearity\footnote{We show in Appendix \ref{app:onwkb} that this is a consequence of the bilinearity of $B,$ assumption \eqref{harmonics} on the set of characteristic harmonics, and transparency in the form \eqref{weak:transp}. The creation of a mean mode is called rectification; it was studied in depth in \cite{Lt,CL}.}: $u_{0,0}(t,x) \equiv 0.$ At this stage \eqref{wkb1} is solved and we turn to \eqref{wkb2}. For $q = 0,$ denoting ${\bf \Pi}(0)$ the orthogonal projector onto the kernel of the skew-symmetric matrix $A_0,$ we find
$$
  {\bf \Pi}(0) \Big( B(u_{0,1}, u_{0,-1}) + B(u_{0,-1}, u_{0,1}) \Big) = 0.
$$
With the polarization \eqref{pola0}, the above condition takes the form
\be \label{compa}
{\bf \Pi}(0) \Big( B(\vec e_1, \vec e_{-1}) + B(\vec e_{-1}, \vec e_1) \Big) = 0.
 \ee
The {\it compatibility} condition \eqref{compa} was called {\it transparency} by Joly, M\'etivier and Rauch (Assumption 2.1 in \cite{JMR-TMB}).  This condition is a necessary condition for the existence of WKB solutions for general data \eqref{generic-data}\footnote{The case in which \eqref{compa} does not hold is briefly discussed in Remark \ref{rem:ghosts} on page \pageref{rem:ghosts}.}. Under \eqref{compa}, a WKB approximate solution can be constructed, and the leading amplitudes $u_{0,\pm 1}$ are seen to satisfy \emph{nonlinear transport equations}. This is explained in detail in Appendix \ref{app:onwkb}.

The central question of the present work is whether such WKB solutions are stable under small initial perturbations. This is a perturbative analysis: the question is whether small data generate solutions to
\be \label{perturb}
 \d_t v + \frac{1}{\e} A_0 v + A(\d_x) v = \frac{1}{\sqrt \e} \big( B(u_a, v) + B(v, u_a)\big) + \e^{K_a} r_a
\ee
which grow in time. Here $\e^{K_a} r_a$ is the consistency error of the WKB solution. In this discussion we assume an infinite order of approximation $K_a = \infty,$ so that $\e^{K_a} r_a \equiv 0.$ We denote $B(u_a) v = B(u_a, v) + B(v,u_a)$ in the following. In $u_a,$ the important term is the leading term ${\bf u}_0,$ so that in \eqref{perturb} we may simplify $B(u_a)$ into $B\big( e^{- i \theta} u_{0,-1} + e^{i\theta} u_{0,1}\big),$ with $\theta := (k \cdot x - \o t)/\e.$ The solution to \eqref{perturb} is then given by
\be \label{sol:perturb}
 v(t) = {\mathcal L}_\e(t) v(0) + \frac{1}{\sqrt \e} \int_0^t {\mathcal L}_\e(t - t') B\big( e^{- i \theta} u_{0,-1}(t') + e^{i \theta} u_{0,1}(t')\big) v(t') \, dt',
 \ee
 where
$\dsp{{\mathcal L}_\e(t) := \exp\Big( - \frac{t}{\e} \big( A_0 + A(\e \d_x) \big)\Big).}$
Assuming a smooth spectral decomposition
\be \label{smooth:spec} A_0 + A(i \xi) = \sum_{j} i \l_j(\xi) \Pi_j(\xi),\ee
where $\l_j$ are real eigenvalues and $\Pi_j$ orthogonal projectors, the solution \eqref{sol:perturb} then appears as the sum, over $p,i,j,$ of
\be \label{rep00} e^{- i t \l_i(\e D)/\e} \Pi_i(\e D) v(0) + \frac{1}{\sqrt \e} \int_0^t e^{- i (t - t') \l_i(\e D)/\e} \Pi_i(\e D) B\big(e^{i p \theta} u_{0,p}\big) \Pi_j(\e D) v(t') \, dt',\ee
in which the first term is the free evolution under the solution operator of the initial perturbation. The goal is to bound the second term in \eqref{rep00}, that is the Duhamel term encoding the accumulated response of $B(u_a),$ considered as a linear source. In \eqref{rep00}, the operators $\l_j(\e D)$ and $\Pi_j(\e D)$ are Fourier multipliers in semi-classical quantization\footnote{Notations pertaining to symbols and pseudo-differential operators are set up in Appendix \ref{app:symbols}, where also classical results on action and composition of such operators are recalled.}. There holds, by linearity of $B,$
 $$ \Pi_i(\e D) B(e^{i p \theta} u_{0,p}) = e^{i p \theta} \Pi_i(\e D + p k) B(u_{0,p}),$$ so that the Duhamel term in \eqref{rep00} takes the form
$$
 \frac{1}{\sqrt \e}  e^{i (p k \cdot x - t \l_i(\e D + p k ))/\e} \int_0^t e^{i t' (- p \o + \l_i(\e D + p k ))/\e} \Pi_i(\e D + p k) B(u_{0,p}(t')) \Pi_j(\e D) v(t') \, dt'.
$$
 In the following we overlook the unitary prefactor $e^{i (\dots)}$ in front of the integral. For short times $t \ll \sqrt \e,$ it makes sense to approximate $u_{0,p}$ by its datum $a$ or $a^*,$ and $v$ by the free evolution term in \eqref{sol:perturb}\footnote{In other words, we are considering the first Picard iterate for \eqref{perturb}.}. Thus we are looking at
 $$ \frac{1}{\sqrt \e} \int_0^t e^{i t' (- \o + \l_i(\e D + k))/\e} \Pi_i(\e D + k) B(a) e^{- i t' \l_j(\e D)/\e} \Pi_j(\e D) v(0) \, dt',$$
 where we let $p = 1$ for definiteness. For $t' = O(\sqrt \e),$ up to operators which are $O(\sqrt \e)$ and regularizing, the function $B(a)$ and the Fourier multiplier $e^{- i t' \l_j(\e D)/\e}$ commute\footnote{For a precise statement, we refer to estimate \eqref{est:fourier-mult} in Appendix \ref{app:symbols}.}, and we arrive at
 \be \label{to:est} \frac{1}{\sqrt \e} \int_0^t e^{i t' (- \o + \l_i(\e D + k) - \l_j(\e D))/\e}  \Pi_i(\e D + k) B(a) \Pi_j(\e D) v(0) \, dt'.\ee
 The question is whether we can bound \eqref{to:est} uniformly in $\e.$ This would provide short-time uniform bounds for the solution $v$ to \eqref{perturb}, and thus would represent a first step in a proof of stability of the WKB solution.

 The key frequencies are $\xi$ such that the phase in \eqref{to:est} is stationary. These are the resonances, defined as the solutions $\xi \in \R^d$ to
 \be \label{res0}
  - \o + \l_i(\xi + k) - \l_j(\xi) = 0.
 \ee
 Far from these resonant frequencies, we can integrate by parts in time in the Fourier formulation of \eqref{to:est} and gain a factor $\e.$ Near resonant frequencies, unless the {\it interaction coefficient} $\Pi_j(\xi + k) B(a) \Pi_j(\xi)$ is small, the integral is $\sim (1/\sqrt \e),$ which could lead to an amplification by $e^{c /\sqrt \e}$ of $v.$

 For systems in $u = (u_1, u_2) \in \R^{N_1 \times (N - N_1)}$ and triangular source terms
 $$ \dsp{B(u,u) = \left(\begin{array}{cc} 0 \\ B_2(u_1,u_1) \end{array}\right)},$$
 this sketch of analysis was made rigorous in \cite{em2}, following \cite{JMR-TMB}\footnote{The scaling in \cite{em2} is actually slightly different from \eqref{0}, and yet another scaling was considered in \cite{JMR-TMB}. Remark \ref{rem:jmr} expands on this point.}. That is, smallness of the interaction coefficients at the resonances was seen as a sufficient condition for stability of the WKB solutions.

\section{A criterion for stability} \label{sec:criterion}

As discussed just above, previous analyses
\cite{JMR-TMB,C,em2,T3,Lu} gave only sufficient conditions for
stability of WKB solutions in supercritical regimes. We give here a
condition that is almost necessary and sufficient\footnote{The
degenerate case ${\bf \G} = 0$ (with notation introduced in
\eqref{def:trace}) is not covered by our analysis, hence an
``almost" necessary and sufficient condition.}. The first step in
our analysis is a {\it reduction to $2 \times 2$ interacting
systems}. Then, depending on {\it symbolic spectrum of the
propagator}, we either {\it symmetrize} the interacting system and
prove stability, or use a {\it Duhamel representation} in order to
prove instability.

\medskip

{\it Reduction to $2 \times 2$ interacting systems.} The resonance relation \eqref{res0} appears only implicitly in the sketch of analysis given in Section \ref{sec:restransp} above. We make it play an explicit role by introduction of the variables
 $$ v_i = \op_\e(\chi_{ij}) \big( e^{- i (k \cdot x - \o t)/\e} \op_\e(\Pi_i) \dot u\big) , \quad v_j = \op_\e(\chi_{ij} \Pi_j) \dot u,$$
 where $\dot u$ is the perturbation variable, defined by $u =: u_a + \dot u,$ the $\Pi_j$ are the spectral projectors introduced in \eqref{smooth:spec}, and $\chi_{ij}$ is a frequency cut-off that is supported in a neigborhood of the resonant set $\{ \l_{i}(\cdot + k) = \o + \l_j(\cdot)\},$ which we assume to be bounded.

  The question of the stability of $u_a$ reduces to the question of the growth in time of $(v_i,v_j),$ for all relevant couples of indices $(i,j).$ We denote $\tilde V$ the total variable, that is the collection of relevant couples $(v_i,v_j).$

  Our first key observation is that under a mild {\it partial transparency condition} for the resonances (formulated as Assumption \ref{ass:several-res}(ii), page \pageref{ass:several-res}), the {\it normal form} of the time-evolution system in $(v_i,v_j)$ has the following features:
  \begin{itemize}
  \item it is {\it decoupled} from the system in the other components of the solution (corresponding to resonances $(i',j'),$ with $(i',j') \neq (i,j)$),
  \item it has {\it non-oscillating sources}, and
  \item resonances $\{ \l_i(\cdot + k) = \o + \l_j(\cdot)\}$ appear as the locus of weak hyperbolicity.
 \end{itemize}
 This normal form of the system in $(v_i,v_j)$ is
\be \label{int:systems}
 \d_t \check U_{ij} + \frac{1}{\e} \op_\e\left( \begin{array}{cc} i (\l_i(\cdot + k) - \o) & - \sqrt \e b_{ij} \\ -\sqrt \e b_{ji} & i \l_j\end{array}\right) \check U_{ij} = f,
 \ee
 where $\check U_{ij}$ is the $(i,j)$-component of the total solution $\check U$ after the change of variable to normal form:
 $$ \check U = (\Id + \sqrt \e \op_\e(Q))^{-1} \tilde V(t, x), \quad \mbox{for some appropriate symbol $Q,$}$$
 and the interaction coefficients are
 $$ b_{ij} := \Pi_i(\xi + k) B(u_{0,1}(t,x)) \Pi_j(\xi), \qquad b_{ji} := \Pi_j(\xi) B(u_{0,-1}(t,x)) \Pi_i(\xi + k).$$
 Here we are using notation $u_{0,\pm 1}, \l_i, \l_j$ from Section \ref{sec:restransp} and $\op_\e(\cdot)$ from Appendix \ref{app:symbols}. In \eqref{int:systems}, the source $f = f(u)$ is bounded in $u.$

 System \eqref{int:systems} is nominally $2 N \times 2N.$ However, if the projectors $\Pi_j, \Pi_j$ have rank equal to one, then the matrix of the propagator has rank two, essentially making \eqref{int:systems} a $2 \times 2$ system.

\medskip

{\it Spectrum of the symbol of the propagator.} The eigenvalues of the symbol of the propagator in \eqref{int:systems}, a $2 \times 2$ complex matrix, is
 \be \label{sp:symbol} \frac{i}{2} \big(\l_i(\xi + k) - \o - \l_j(\xi) \big) \pm
\frac{1}{2}
  \Big( 4 \e \tr\,b_{ij} b_{ji} - (\l_i(\xi + k) - \o - \l_j(\xi))^2
  \Big)^{1/2}.\ee
  Thus it appears that the crucial quantity is the sign of the trace of the product of the interaction coefficients at the resonance:
$$
\mbox{sign}\, \mbox{tr}\, b_{ij} b_{ij} \qquad \mbox{at $\xi$ such that $\l_i(\xi + k) - \o - \l_j(\xi) = 0.$}
$$
If the sign of positive, then real eigenvalues occur in \eqref{sp:symbol}, meaning a loss of hyperbolicity around the resonance. Otherwise, eigenvalues are purely imaginary. In the latter case, $\mbox{sign}\,\mbox{tr}\,b_{ij} b_{ji} < 0,$ the propagator in \eqref{int:systems} can be symmetrized. For scalar $b_{ij}$ and $b_{ji},$ a symmetrizer is indeed $\bp 1&0\\0& - b_{ij}^*/b_{ji} \ep.$ Uniform estimates, hence stability, follow.

\medskip

{\it Duhamel representation and instability.} In the case of real eigenvalues in \eqref{sp:symbol}, indicating instability, the task ahead is to convert a spectral information at the level of symbols into bounds for the corresponding system of pseudo-differential equations \eqref{int:systems}.

This is achieved with the Duhamel representation formula introduced by the second author in \cite{T4}. This representation extends the Fourier analysis of the above Section \ref{sec:restransp} (of which a good example is \eqref{rep00}) by incorporating the zeroth-order source terms $b_{ij}$ and $b_{ji}$ into the propagator. Since resonances are points of weak hyperbolicity, and since at such points the stability analysis must include lower-order terms, the source terms $b_{ij}$ and $b_{ji}$ indeed belong in the propagator.

The instability occurs in time $O(\sqrt\e |\ln \e|).$ Indeed, the source term in \eqref{0} or \eqref{int:systems} has a $O(1/\sqrt \e)$ prefactor. Hence a potential growth $\sim e^{t B/\sqrt \e}.$ If we start from a small $\sim \e^K$ initial perturbation, then the instability is recorded only when the time exponential $e^{t B/\sqrt \e}$ reaches a fraction of the size of the initial perturbation $\e^K,$ meaning an instability time of order $\sqrt \e |\ln \e|.$

For this reason in the unstable case we rescale in time
$$ U_{ij}(t, x) := \check U_{ij}(\sqrt \e t, x),$$ so that $U_{ij}$ solves
$$
 \d_t U_{ij} + \frac{1}{\sqrt \e} \op_\e\left( \begin{array}{cc} i (\l_i(\cdot + k) - \o) & - \sqrt \e b_{ij} \\ -\sqrt \e b_{ji} & i \l_j\end{array}\right) U_{ij} = \sqrt \e f,
$$
 where $b_{ij}, b_{ji}$ and $f$ are evaluated at $(\sqrt \e t, x).$

We then localize around resonant frequencies. Since the resonant set is assumed to be bounded, this means multiplying the equation to the left by $\op_\e(\chi),$ where $\chi$ is a smooth, compactly supported frequency cut-off that is identically equal to one in a neighborhood of the resonances. Then $V := \op_\e(U_{ij})$ solves
\be \label{int:systems:truncated}
 \d_t V + \frac{1}{\sqrt \e} \op_\e\left( \chi \left(\begin{array}{cc} i (\l_i(\cdot + k) - \o) & - \sqrt \e b_{ij} \\ -\sqrt \e b_{ji} & i \l_j\end{array}\right)\right) V = \sqrt \e f_V,
 \ee
 where $b_{ij}, b_{ji}$ and $f_V$ are evaluated at $(\sqrt \e t, x),$ and $f_V$ enjoys the same bounds as $f.$

 The representation formula of \cite{T4} states that the solution operator to \eqref{int:systems:truncated} is well approximated, in time $O(|\ln \e|),$ by the para-differential operator $\ope(S_0),$ where $S_0$ is the finite-dimensional solution operator, defined for all $(x,\xi)$ by
 \be \label{ode} \left\{\begin{aligned} \d_t S_0 + \frac{1}{\sqrt \e } \chi(\xi) \left( \begin{array}{cc} i (\l_i(\xi + k) - \o) & - \sqrt \e b_{ij}(\sqrt \e t,x,\xi) \\ -\sqrt \e b_{ji}(\sqrt \e t, x, \xi) & i \l_j(\xi)\end{array}\right)   S_0 & = 0, \\ S_0(\t;\t,x,\xi) \equiv \Id.
 \end{aligned}\right.\ee
That is, the solution to \eqref{int:systems:truncated} admits the representation
\be \label{rep:intro}
 V = \ope(S_0(0;t)) V(0) + \sqrt \e \int_0^t \ope(S_0(t';t)) \tilde f(t') \, dt',
 \ee
 where $\tilde f \simeq f_V.$ Appendix \ref{app:duh} is devoted to a proof of \eqref{rep:intro}.

 A key consequence is that in time $O(|\ln \e|)$ bounds for \eqref{int:systems:truncated} can be deduced from bounds on $S_0:$ the approximation result of \cite{T4} simplifies the analysis of an ordinary differential equation in infinite dimensions (namely, \eqref{int:systems:truncated}\footnote{The propagator in \eqref{int:systems:truncated} is indeed bounded $L^2 \to L^2;$ this is a consequence of the Calder\'on-Vailancourt theorem \cite{CV,CM}, of which a very simple proof is given in \cite{Hw}. A precise statement is given in Appendix \ref{app:symbols}.}) into the analysis of a family of ordinary differential equations in finite dimensions (namely, \eqref{ode}).

Bounds for $S_0$ do not derive trivially from consideration of the spectrum \eqref{sp:symbol}, since the resonant locus is at a distance $O(\sqrt \e)$ from the singular locus
 $$ \Big\{ \xi \in \R^d, \quad 2 \big( \e \tr\, b_{ij} b_{ji}\big)^{1/2} = \l_i(\xi + k) - \o - \l_j(\xi) \Big\},$$
 where eigenvalues \eqref{sp:symbol} coalesce. In particular, the eigenprojectors are {\it not} uniformly bounded in $\e$ near the resonances, and bounds for \eqref{ode} cannot be derived by simply diagonalizing the system. Appendix \ref{app:symb-bound} is devoted to a precise derivation of these bounds in the unstable case of a positive trace.

From \eqref{rep:intro}, armed with optimal bounds for $S_0,$ meaning a lower rate of exponential growth that is arbitrarily close to the upper rate of growth, the task ahead is to derive {\it lower bounds} for the free component of the solution $\ope(S_0(0;t)) \check U_{ij}(0),$ and {\it upper bounds} for the time-integral term in \eqref{rep:intro}.

\smallskip

Lower bounds for $\ope(S_0(0;t)) U_{ij}(0)$ with a maximal rate of growth are achieved by a careful choice of the initial perturbation $U_{ij}(0).$ Essentially, we choose to initially excite frequencies that grow at the highest rate. This is the purpose of Section \ref{sec:lower} in the proof of Theorem \ref{theorem1}. Upper bounds for $\dsp{\int_0^t  \ope(S_0(t';t)) \tilde f(t')}$ derive from bounds for $\ope(S_0),$ which are deduced from bounds on $S_0$ via Calder\'on-Vaillancourt type theorems. Details are given in Sections \ref{est-upper-V0} and \ref{est-upper-V0-a} in the main proof. The comparison of lower bounds with upper bounds in Section \ref{sec:end-insta}
concludes the proof.

 \medskip

 There is a specific difficulty associated with the large prefactor $(1/\sqrt \e)$ in \eqref{ode}. This prefactor implies indeed that $S_0$ has large variations in $\xi:$ $\d_\xi S_0 \sim S_0/\sqrt\e.$ This is problematic in view of Calder\'on-Vaillancourt type theorems, which typically assert boundedness of pseudo-differential operators given boundedness of the symbols {\it and} their $(x,\xi)$-derivatives. We overcome this issue by using a result from H\"ormander \cite{Hom3} (formulated as Proposition \ref{prop:actionH} in Appendix \ref{app:symbols}) which gives operator bounds involving spatial $L^1$ norms of the symbols, and no $\xi$-derivatives. This requires a spatial localization step, since the symbols that we handle are a priori not $L^1$ in space.

\section{On the class of initial perturbations} \label{sec:init:pert}
A salient feature of our analysis in the stable case is that we allow for initial perturbations $\phi(\e,x),$ which do not necessarily depend on $(\e,x)$ periodically through $k \cdot x/\e.$ In particular, we give a geometric optics result for a class of perturbations which is much larger than the class of perturbations allowed in a number of results of the JMR school  \cite{JMR0,DJMR,JMR1,JMR-TMB,Lt,CL,C,Du0,em2}.

In these references, WKB solutions $u_a$ and initial perturbations $\phi$ which are $2\pi$-periodic in the fast variable $(k \cdot x - \o t)/\e$ allow for a representation of the solution in the form of a profile, that is a map ${\bf u}$ of $(t,x,\theta)$ with a $2\pi$-periodic dependence in $\theta,$ the trace of which over $\theta = (k \cdot x - \o t)/\e$ is equal to the original solution:
$$ u(t,x) = {\bf u}\Big(t,x, \frac{k \cdot x - \o t}{\e}\Big).$$
This representation de-singularizes the initial datum, which for profiles appears as
$$ {\bf u}(0,x,\theta) = \Re e \, \big( a(x) e^{i \theta} \big) + \e^K \phi(x, \theta),$$
where $\phi$ is $2\pi$-periodic in $\theta,$ by assumption.
In particular, the leading term $\Re e \, a e^{i \theta}$ is bounded in $H^s(\R^d_x \times \T_\theta).$ The drawback is that the equation in ${\bf u}$ is more singular than the original system \eqref{0}, since it features the singular differential operator $\dsp{\frac{1}{\e} (- \o \d_\theta + \sum_j k_j A_j \d_\theta).}$ This operator, however, contributes zero to $L^2$ estimates in $(x,\theta),$ by symmetry.

\medskip

By contrast, in the present work we do not insist on a periodic dependence in the fast variable $k \cdot x/\e$ for the initial perturbation $\phi$ in \eqref{generic-data}. In particular, $\phi(\e,x)$ may take the form $\phi_0(x/\e),$ where $\phi_0$ is only assumed a high Sobolev regularity. In this context, $\e$-uniform Sobolev estimates may be derived only for $\e$-weighted norms, defined as
 \be \label{weighted:norm} \| u \|_{\e,s} := \left( \int_{\R^d} (1 + |\e \xi|^2)^s |\hat u(\xi)|^2 \, d\xi\right)^{1/2},\ee
 and an important tool is the Sobolev product estimate
\begin{equation} \label{product2:intro}
 \| u v \|_{\e,s} \leq C \big( |u |_{L^\infty} \| v \|_{\e,s} + |(\e \d_x)^s u|_{L^\infty} |v|_{L^2}\big),
 \end{equation}
which can be proved by approximating the product $uv$ by the para-product of $v$ by $u.$ Details are given in Appendix \ref{app:products}. In our use of \eqref{product2:intro}, $u$ is the approximate solution $u_a,$ with a periodic dependence in $(k \cdot x - \o t)/\e,$ and $v$ is the solution to \eqref{0}, with a priori a singular dependence in $x$ via $x/\e,$ just like the initial perturbation. In particular, $|u_a|_{L^\infty}$ and $|(\e \d_x)^s u_a|_{L^\infty}$ are both bounded uniformly in $\e,$ implying the bound
 $$ \| u_a v \|_{\e,s} \leq C(u_a) \|v\|_{\e,s}, \quad \mbox{for $s \geq 0.$}$$
By comparison, the Moser-type estimate
\be \label{moser:intro} \| u v \|_{\e,s} \leq C \big( |u |_{L^\infty} \| v \|_{\e,s} + |v |_{L^\infty} \| u \|_{\e,s}\big)
\ee
would here give only
$$ \| u_a v \|_{\e,s} \leq C(u_a) \big( \| v \|_{\e,s} +  |v|_{L^\infty} \big) \leq \e^{-d/2} C(u_a) \| v \|_{\e,s}, \quad \mbox{for $s > d/2,$}$$
since the Sobolev embedding $H^s \hookrightarrow L^\infty,$ for $s > d/2,$ has a large norm when $H^s$ is equipped with \eqref{weighted:norm}:
 \be \label{embed:intro} |u|_{L^\infty} \leq C_{d,s} \e^{-d/2} \| u \|_{\e,s}, \qquad s > d/2, \,\, C_{d,s} > 0.\ee
However, for {\it semi-linear} terms of the type $v^2$ (or $B(v,v)$) where $v$ is the solution, both \eqref{product2:intro} and \eqref{moser:intro} lead to $\e^{-d/2}$ losses, via \eqref{embed:intro}. This is the main drawback of our approach: while it allows for quite general perturbations, it requires {\it smallness} of these, typically in the form of the lower bound $K > (1 + d)/2,$ in order to prove stability.

\section{Overview of the results} \label{sec:overview}

We give five theorems:

\smallskip

$\bullet$ the first, Theorem \ref{theorem1} (page \pageref{theorem1}), states that stability of WKB solutions is determined by a scalar index, which when positive indicates instability, and when negative indicates stability. The degenerate case of a vanishing stability index covers different possible situations with regard to stability, one of them treated in \cite{T3}. For a discussion of the precise meaning of stability/instability in our context, see in particular Section \ref{sec:comments}. Theorem \ref{theorem1} is formulated under the strong assumption that there be only one non-transparent resonance (Assumption \ref{ass:transp}). The reason for this assumption is that it simplifies the exposition of our main ideas by allowing for relatively simple notation.

\smallskip

$\bullet$ In Theorem \ref{th-two} (page \pageref{th-two}), we allow for several non-transparent resonances, with the same conclusions as in Theorem \ref{theorem1}. This is the framework that is encountered in many examples, in particular the coupled Klein-Gordon systems described in Sections \ref{sec:KGintro} and \ref{sec:KG}, \ref{sec:KG2}. The proof (Section \ref{sec:proof-th-extension}) relies on the same ideas as the proof of Theorem \ref{theorem1}, the extra difficulty being only notational.

\smallskip

The last three results are variations on Theorem \ref{th-two} and its proof:

\smallskip

$\bullet$ We first remark, in Theorem \ref{th-three}, that all non-transparent resonances are amplified. That is, we can observe an instability even though we initially turn on a resonance that is not associated with the maximal rate of growth. Here our assumptions are weakest, in particular are essentially only local in frequency, but the amplification is accordingly weaker.

\smallskip

$\bullet$ Next we remark that instability occurs in asymptotically vanishing balls, provided that we give up a little on the amplification rate. This is Theorem \ref{th:4}.

\smallskip

$\bullet$ Finally, in Theorem \ref{th:5} we prove that if arbitrarily small perturbations of the WKB datum generate exact solution that persist and are bounded uniformly in $(\e,t,x)$ over time intervals $T \sqrt \e |\ln \e|,$ with $T$ large enough, then the amplification goes from $O(\e^K)$ to $O(\e^{K'}),$ with $K$ arbitrarily large and $K'$ arbitrarily small, in time $O(\sqrt\e |\ln\e|).$ Of course, if small perturbations do not generate solutions over such asymptotically small time intervals, or if these solutions are unbounded, then this means instability, in another form, for the WKB solution.

\section{On related instability results} \label{rem:jmr17april} The article \cite{JMR-TMB}, cited in Section \ref{sec:background} as the main inspiration of the current work, contains limited instability results. In Section 10 and 11 of \cite{JMR-TMB}, Joly, M\'etivier and Rauch show that under a condition (Assumption 10.3 in \cite{JMR-TMB}) that is very similar to our instability condition ${\bf \G} > 0$ below, WKB solutions are unstable. They do so for linear equations, and, most importantly, for constant amplitudes, that is, for WKB solutions of the form $\vec a\, e^{i k \cdot x/\e},$ where $\vec a \in \R^N$ is fixed.  This allows an analysis by Fourier transform. For the solution $u,$ there holds $|u|_{L^2} \geq |\hat u|_{L^2(B_\e)},$ with $B_\e = \{ \xi \in \R^d, \, |\e \xi - \xi_0| \leq h \sqrt \e\},$ where $\xi_0$ is a distinguished resonant frequency and $h > 0$ is small. This reduces the analysis to our Lemma \ref{lem:S-resonance}.

For systems of conservation laws, under the strong assumption of a constant eigenvalue, Cheverry, Gu\`es and M\'etivier prove an instability result for high-frequency WKB solutions. This assumption is (6.5) and Hypoth\`ese 6.1 in \cite{CGM}; constancy is in $u,$ in the context of \cite{CGM} eigenvalues are $1$-homogeneous in $\xi.$ Then Cheverry studied in \cite{CG} the viscous relaxation of these instabilities.

 We note that our approach to instabilities in nonlinear equations differs fundamentally from the approach of Grenier in his classical work \cite{Gre}, in which Grenier formulated spectral assumptions {\it bearing on linear differential operators}. By contrast, our spectral assumptions are formulated for {\it symbols of linear pseudo-differential operators}. In particular, our spectral assumptions are, at least theoretically, readily verifiable, since they bear on spectra of matrices. The key Lemma that allows us to transpose the spectral information at the symbolic level into estimates for corresponding systems of partial differential equations is the Duhamel representation Theorem \ref{th:duh}, drawn from \cite{T4}.

The article \cite{LNT} also uses the Duhamel representation of \cite{T4}, to prove a strong Lax-Mizohata result for weakly non-hyperbolic quasi-linear systems.

\section{Examples} \label{sec:KGintro}

Our first class of examples (Section \ref{sec:BK}) are three-wave interaction systems, of the form
 \be\label{raman:intro}\left\{ \begin{aligned} \d_t u_1+ c_1 \d_x u_1 & = \frac{b_1}{\sqrt\e} \bar u_2 u_3,\\
\d_t u_2 + c_2 \d_x u_2 & =\frac{b_2}{\sqrt \e} \bar u_1 u_3,\\ \d_t u_3+ c_3 \d_x u_3 & = \frac{b_3}{\sqrt \e} u_1 u_2,
                          \end{aligned} \right.\ee
                          and
    \be\label{brillouin:intro}\left\{ \begin{aligned} \d_t u_1+ \frac{c_1}{\e} \d_x u_1 & = \frac{b_1}{\e} \bar u_2 u_3,\\
\d_t u_2 + \frac{c_2}{\e}\d_x u_2 & =\frac{b_2}{\e}\bar u_1 u_3,\\ \d_t u_3+ c_3 \d_x u_3 & = b_3 u_1 u_2.
                          \end{aligned} \right.\ee
 We show in Section \ref{sec:der3EM} how these systems are derived in the high-frequency limit from the Euler-Maxwell equations describing laser-plasma interactions. Systems \eqref{raman:intro} and \eqref{brillouin:intro} can be used to describe Raman and Brillouin scattering, respectively.

 In the case $b_2 b_3  >0,$ for any $c_1, c_2, c_3,$ our instability results apply to the reference solutions
 $$ \big( a( x - c_1 t), 0,0\big) \quad \mbox{for \eqref{raman:intro}}, \quad \mbox{and} \quad \big( a(\e x - c_1 t), 0,0 \big) \quad \mbox{for \eqref{brillouin:intro}},$$
  and give a description of the growth of the Raman and Brillouin waves $u_3.$

\medskip

Our second class of examples (Sections \ref{sec:KG} and \ref{sec:KG2}) comprises coupled Klein-Gordon systems of the form
\begin{equation} \label{coupled-KG-explicit}
    \left\{ \begin{aligned}
    \d_t u + \left(\begin{array}{ccc} 0 & \d_x & 0 \\ \d_x \cdot & 0 & 0 \\ 0 & 0 & 0 \end{array}\right) u + \frac{1}{\e} \left(\begin{array}{ccc} 0 &0  &0 \\0 & 0 & \a_0 \o_0 \\ 0 & - \a_0 \o_0 & 0  \end{array}\right) u & = \frac{1}{\sqrt \e} B^1(u,v), \\
 \d_t v +  \left(\begin{array}{ccc} 0 & \theta_0 \d_x & 0 \\ \theta_0 \d_x \cdot & 0 & 0 \\ 0 & 0 & 0 \end{array}\right) v + \frac{1}{\e} \left(\begin{array}{ccc} 0 &0  &0 \\0 & 0 & \o_0 \\ 0 & -  \o_0 & 0  \end{array}\right) v & = \frac{1}{\sqrt \e} B^2(u,v),
     \end{aligned}
    \right.
    \end{equation}
  where $(u,v) = (u_1,u_2,u_3,v_1,v_2,v_3)\in\R^{d+2}\times \R^{d+2}$, $x \in \R^d,$ $0 < \o_0,$ $0 < \a_0,$ $0 < \theta_0 < 1.$ The eigenvalues (as in \eqref{smooth:spec}) are
   $$\Big\{ 0, \, \pm \sqrt{\a_0^2 \o_0^2 + |\xi|^2}, \, \pm \sqrt{\o_0^2 + \theta_0^2 |\xi|^2}\Big\}.$$
 The characteristic varieties for $\a_0 = 1$ and $\a_0 \neq 1,$ depicting the branches of eigenvalues as functions of $\xi,$ are depicted on Figure \ref{fig1} page \pageref{fig1} and Figure \ref{fig2} page \pageref{fig2}, respectively.

   If $\a_0 = 1,$ the masses (corresponding to threshold frequency $\o_0$) are equal. In the context of laser-plasma interactions, the masses are both equal to the plasma frequency, and systems \eqref{coupled-KG-explicit} are simplified Euler-Maxwell systems. This case is covered in Section \ref{sec:KG}. The case of different masses is covered in Section \ref{sec:KG2}.

    In both cases, we give examples of bilinear terms $B^1$ and $B^2$ to which our results, stability or instability of WKB solutions, apply.

\section{Open problems} \label{sec:open}

We conclude this introduction with a list of open problems, listed in what we perceive as an increasing level of difficulty:

\smallskip

{\it Allow for rank-two interaction coefficients.} It would be interesting, especially in view of the extension of our results to the Euler-Maxwell equations (see Section \ref{sec:der3EM}), to handle rank-two interaction coefficients. This would mean extending the bounds of Appendix \ref{app:symb-bound} to symbolic flows defined by interaction matrices of the form
 $$ M = \left(\begin{array}{ccc} i \mu_1 & 0 & -\sqrt \e b_{ijx}^+ \\ 0 & i \mu_1 & - \sqrt \e b_{ijy}^- \\ - \sqrt \e b_{jix}^- & -\sqrt \e b_{jiy}^- & i \mu_2 \end{array}\right).$$

\smallskip

 {\it Weaken the partial transparency condition \eqref{new:1605}.} In our first class of Klein-Gordon examples (Section \ref{sec:KG}), condition \eqref{new:1605} is satisfied only in one space dimension. We note that condition \eqref{new:1605} is used only in the proof of Proposition \ref{prop:several-res-new} (normal form reduction) which decouples the components of the solution associated with non-transparent resonances. Given a specific set of non-transparent resonances, we could probably make appropriate coordinatization choices so as to forgo, or at least weaken, condition \eqref{new:1605}.

\smallskip

 {\it Consider larger initial perturbations.} We take into account the presence of high frequencies $\sim 1/\e$ by measuring $L^2$ norms of $\e$-derivatives. The main drawback of this approach is a very poor control of sup norms, as already seen in \eqref{embed:intro}. By using \eqref{embed:intro}, we are essentially uniformly bounding $a(x) \sin(x/\e)$ by $C \e^{-d/2},$ even if $a$ is smooth and compactly supported.

  This raises the question: {\it Does there exist a Banach algebra of distributions in which high-frequency families $\{ \varphi(x,x/\e) \}_{0 < \e < 1},$ with $\varphi \in C^\infty_c,$ are uniformly bounded, and in which good pseudo-differential bounds are available?}

  The minimal requirements for pseudo-differential bounds would be inclusion of the space of pseudo-differential operators of order zero into the space of linear bounded operators from the Banach algebra to itself, and stability under composition.

   The space ${\mathcal F}L^1$ of distributions with $L^1$ Fourier transform satisfies the first two conditions (algebra, uniform bounds for oscillating families), but not the third (pseudo-differential bounds). The Sobolev space $H^s$ equipped with the semi-classical norm $\| \cdot \|_{\e,s}$ satisfies the last two conditions, but is not a Banach algebra.

   In the absence of a positive answer to the above question, we perform in Section \ref{est-upper-V0-a} estimates in both ${\mathcal F}L^1$ and $H^s,$ so as to combine the advantages of both functional settings. This gives an existence time $T_0$ \eqref{def:T0-K0}, that approaches the ``optimal" existence time $T_\infty = K/(\g |a|_{L^\infty}),$ using notation introduced in Section \ref{sec:assres}. The optimal character of $T_\infty$ is seen on Theorem \ref{th:5}: this existence time allows for the amplification exponent $K'$ to be arbitrarily small, hence for the instability to be almost Lyapunov.

\smallskip

 {\it Allow for more singular scalings.} Laser pulses typically propagate in one spatial dimension $x$ and have large transverse variations in transverse directions $y:$ they have the form $$a(x,y) \sin((k x - \o t)/\e) \sin(y/\sqrt \e),$$ where $a$ is a slowly varying amplitude. A corresponding scaling would be, instead of \eqref{0}, the more singular
 \be \label{0fast} \d_t u + \Big(\frac{1}{\e} A_0 + \frac{1}{\sqrt \e} A(\d_y) + A \d_x\Big) u = \frac{1}{\sqrt \e} B(u,u),\ee
 with data oscillating in $x$ at frequencies $\sim 1/\e.$ In this scaling, the Zakharov equations with non-zero group velocity were formally derived from the Euler-Maxwell equations in \cite{T2}. A stability analysis of WKB solutions to \eqref{0fast} would lead to consideration of symbolic flows of interaction matrices as in \eqref{ode}. The important difference would an $\sqrt\e$-semiclassical scaling of the relevant pseudo-differential operators, with the catastrophic consequence
  $$ \frac{1}{\sqrt \e} \Big( \op_{\sqrt \e}(M S_0) - \op_{\sqrt \e}(M) \op_{\sqrt \e}(S_0)\Big) = \op_{\sqrt \e}(i \d_\xi M \d_y S_0) + \sqrt \e \big( \dots \big),$$
 meaning an error $O(1),$ instead of $O(\sqrt \e),$ in the first step of the construction of a solution operator. Then, $\op_{\sqrt \e}(S_0)$ would not appear as an approximation of the solution operator. What would constitute a good approximation of the solution operator, then ?

\chapter{Assumptions and results} \label{sec:assres}

For the family of systems \eqref{0}, reproduced here:
 $$  \d_t u + \frac{1}{\e} A_0 u + \sum_{1 \leq j \leq d} A_j \d_{x_j} u  = \frac{1}{\sqrt \e} B(u,u),$$
 we make the following assumptions.
  
\begin{assump}[Smooth spectral decomposition]\phantomsection \label{ass:spectral} We assume that the matrices $A_j,$ for $1 \leq j \leq d,$  are real symmetric, that the matrix $A_0$ is real skew-symmetric, and that the family of hermitian matrices $\big\{ A_0/i + \sum_{1 \leq j \leq d} \xi_j A_j\big\}_{\xi \in \R^d}$ has the spectral decomposition
\be\label{sp dec}
A_0/i + \sum_{1 \leq j \leq d} \xi_j A_j :=\sum_{1 \leq j \leq J} \l_j(\xi)\Pi_j(\xi), \ee
for some fixed $J,$ where $\l_j$ are smooth
eigenvalues and $\Pi_j$ are smooth eigenprojectors, satisfying bounds, for all $\b \in \N^d,$ for some $C_\b > 0:$
 \be \label{bd:spectral} |\d_\xi^\b \l_j(\xi)| \leq C_\b (1 + |\xi|^2)^{(1 - |\b|)/2}, \quad |\d_\xi^\b \Pi_j(\xi)| \leq C_\b (1 + |\xi|^2)^{ - |\b|/2}.\ee
\end{assump}

We do {\it not} assume that eigenvalues do not cross; indeed, for physical systems, crossing does typically occur at least for $\xi =0:$ examples are given in Sections \ref{sec:BK}, \ref{sec:KG}, \ref{sec:KG2} of Chapter \ref{chap:ex} and Appendix \ref{app:structure-resonant-set}. The smoothness condition in Assumption \ref{ass:spectral} means that at coalescence points, there is an ordering of the eigenvalues so that regularity is preserved. Bounds \eqref{bd:spectral} mean that $\l_j \in S^1$ and $\Pi_j \in S^0.$ The classical classes $S^m$ of pseudo-differential operators of order $m$ are introduced in Appendix \ref{app:symbols}.

Assumption \ref{ass:spectral} is discussed in Appendix \ref{app:regdec}. There we give, in particular, a sufficient condition for bounds \eqref{bd:spectral} to hold.

\begin{assump}[WKB approximate
solution]\phantomsection\label{ass-u-a}
For some $K_a\in \N$, some $T_a > 0,$ all $\e > 0,$ there exists $u_a$ an approximate solution to \eqref{0}, in the sense that there holds in $[0,T_a]:$
\be \label{0a} \d_t u_a + \frac{1}{\e}A_0u_a + \sum_{1 \leq j \leq d} A_j \d_{x_j} u_a = \frac{1}{\sqrt \e} B(u_a,u_a) + \e^{K_a} r_a^\e.\ee
The approximate solution has the form of a WKB expansion
\begin{equation} \label{def-ua}
 u_a(\e,t,x) = e^{-i(k \cdot x - \o t)/\e} u_{0,-1}(t,x) + e^{i (k\cdot x - \o t )/\e} u_{0,1}(t,x) + \sqrt \e v_a(\e,t,x) \,\, \in \R^N,
\end{equation}
where
\begin{itemize}
\item
 the phases $(\o,k)$ and $(-\o,-k) \in \R^{1 + d}$ are \emph{characteristic} for the hyperbolic operator, in the sense that
 \be \label{e+-} \big(\d_t + \frac{1}{\e} A_0 + \sum_{1 \leq j \leq d} A_j \d_{x_j} \big) \big( e^{\pm i (k \cdot x - \o t)/\e} \vec e_{\pm  1} \big) = 0, \qquad \vec e_{-1} = (\vec e_1)^*,\ee
where $\vec e_{-1}$ and $\vec e_1$ are fixed, unit vectors in $\C^N$ and $(\vec e_1)^*$ denotes component-by-component complex conjugation, and the leading amplitudes $u_{0,\pm1}$ are \emph{polarized} along $\vec e_{\pm 1},$ in the sense that
\be\label{pol-con}u_{0,1}(t,x)= g(t,x) \vec e_{1}, \quad u_{0,-1} = g(t,x)^* \vec e_{-1}, \qquad g \in C^1([0,T_a], H^{s_a}(\R^d)),
\ee
where $g^*$ denotes complex conjugation of the amplitude $g \in \C,$
\item there holds $v_a, r^\e_a \in C^0([0,T_a], H^{s_a}(\R^d)),$ with
\begin{equation} \label{est:va-ra} \sup_{\e > 0} \Big( \sup_{|\a| \leq s_a}  | (\e \d_x)^{\a} (v_a, r_a^\e) |_{L^\infty([0,T],L^2)} + | {\mathcal F}(v_a, r_a^\e)|_{L^\infty(0,T_a], L^1)} \Big) < \infty.
\ee
\end{itemize}
\end{assump}

We give in Appendix \ref{app:onwkb} sufficient conditions for Assumption \ref{ass-u-a} to hold true. An example of non-oscillating data, corresponding to $(\o,k) = (0,0),$ is described in Section \ref{sec:BK}.

We note that it suffices for $v_a$ and $r_a^\e$ to be trigonometric polynomials in $\theta,$ as is typically the case in WKB expansions, in order for the above uniform bound on $|{\mathcal F}(v_a, r_a^\e)|_{L^1}$ to hold.

\begin{notation}\phantomsection \label{notation:b} Given $\vec u \in \R^N,$ we denote $B(\vec u): \R^N \to \R^N$ the map defined by
 $$ B(\vec u) v := B(\vec u, v) + B(v, \vec u).$$
\end{notation}

\begin{defi}[Resonances and interaction coefficients]\phantomsection\label{def-reson}
Given $(i,j) \in \{1, \dots, J\}^2,$ with $J$ as in Assumption {\rm \ref{ass:spectral}}, we define the set of resonant frequencies associated with $(i,j)$ by
$$
{\bf \mathcal R}_{ij}:=\big \{\xi\in \R^d, \quad \o = \l_i(\xi+k) - \l_{j}(\xi) \big\}.
$$
The families of matrices
$$
\Pi_{i}(\xi+k)B(\vec e_1)\Pi_{j}(\xi) \in \C^{N \times N} \quad \mbox{and} \quad
\Pi_{j}(\xi)B(\vec e_{-1})\Pi_{i}(\xi+k) \in \C^{N \times N}
,$$
indexed by $\xi \in \R^d,$ are called the interaction coefficients associated with $(i,j).$

 The scalar function $\xi \to \l_i(\xi + k) - \l_j(\xi) - \o$ is called the resonant phase associated with $(i,j).$
\end{defi}

We often say {\it the $(i,j)$ resonance}, and {\it the $(i,j)$ interaction coefficients}.

We note that auto-resonances, i.e. resonances associated with $(i,i),$ for $1 \leq i \leq J,$ are taken into account in this definition. We also note that ${\mathcal R}_{ij} \neq {\mathcal R}_{ji}$ in general.

\medskip

We introduce the property of {\it transparency}:

\begin{defi}[Transparency] \phantomsection\label{def:transp} An interaction
coefficient $\Pi_{i}(\cdot+k)B(\vec e_{1})\Pi_{j}$ is said to be {\it transparent} if the associated resonant phase can be factored out, that is if for some $C > 0,$ there holds for all $\xi \in \R^d$ the bound
\be \label{factor1}
|\Pi_{i}(\xi+k)B(\vec e_{1})\Pi_{j}(\xi)|\leq C
|\l_i(\xi+k)-\o-\l_{j}(\xi)|.
\ee
Similarly, an interaction coefficient $\Pi_{j}B(\vec e_{-1})\Pi_{i}(\cdot+k)$ is said to be transparent if for some $C > 0,$ there holds for all $\xi$
\be \label{factor2}
|\Pi_{j}(\xi)B(\vec e_{-1})\Pi_{i}(\xi+k)| \leq C |\l_i(\xi+k)-\o-\l_{j}(\xi)|.\ee
If both interaction coefficients associated with resonance $(i,j)$ are transparent, then the resonance is said to be transparent.\end{defi}

We can now state our main, and provisory, assumption:
\begin{assump} \phantomsection \label{ass:transp} We suppose
\begin{itemize}
 \item[{\rm (i)}] (boundedness) the set ${\mathcal R}_{12}$ is nonempty and bounded,

 \smallskip

 \item[{\rm (ii)}] (partial transparency) for all $(i,j) \neq (1,2),$ the $(i,j)$ resonance is transparent;

 \smallskip

 \item[{\rm (iii)}] (rank-one coefficients) for all $\xi$ in an open set containing ${\mathcal R}_{12},$ the ranks of the $(1,2)$ interaction coefficients are at most 1.
   \end{itemize}
\end{assump}

Assumption \ref{ass:transp}(i) is discussed in Appendix \ref{app:structure-resonant-set}; there we show in particular that Assumption \ref{ass:transp}(i) satisfied as soon as the eigendecomposition \eqref{sp dec} is smooth at infinity, and $\l_1$ and $\l_2$ non asymptotic at infinity, that is $|\l_j(\xi)| = c_j |\xi| + o(|\xi|)$ with $c_1 \neq c_2.$

Assumption \ref{ass:transp}(ii) is here only to simplify the exposition; we will see in Section \ref{sec:extension} that our results fully extend to the case of several non-transparent resonances, under a mild partial transparency condition.

 Assumption \ref{ass:transp}(iii) is satisfied as soon as (but not only if) $\l_1$ and $\l_2$ are simple eigenvalues.

\section{Main result} \label{sec:main-result}

We denote $\G$ the trace of the product of the $(1,2)$ interaction coefficients:
 \be \label{def:G0} \dsp{\G(\xi) := {\rm
tr}\, \Pi_{1}(\xi+k)B(\vec e_{1})\Pi_{2}(\xi)B(\vec e_{-1})\Pi_{1}(\xi+k).}
\ee
In the context of Assumption \ref{ass:transp}, the stability index is
\begin{equation} \label{def:trace}
 {\bf \Gamma} := \max\Big( \, \max_{\xi\in {\mathcal R}_{12}} \Re e \, \G(\xi), \, \max_{\xi \in {\mathcal R}_{12}} |\Im m \, \G(\xi)|\,\Big),
 \end{equation}
 so that
 \begin{itemize}
 \item
if ${\bf \G} > 0,$ then for some $\xi \in {\mathcal R}_{12},$ there holds $\G(\xi) \notin (-\infty,0].$

\smallskip

\item If ${\bf \G} < 0,$ then for all $\xi \in {\mathcal R}_{12},$ $\G(\xi) \in (-\infty,0).$
\end{itemize}

In the unstable case, corresponding to ${\bf \G} > 0,$ the limiting observation time $T_0$ and amplification exponent $K_0$ are
 \begin{equation} \label{def:T0-K0}
  T_0 := \max\Big( \frac{K}{|B|_0 |\hat a|_{L^1}}, \,\frac{K - d/2}{\g |a|_{L^\infty}}\Big), \qquad  K_0 := \min\Big(K \Big( 1 - \frac{\g |a|_{L^\infty}}{|B|_0 |\hat a |_{L^1}} \Big), \, d/2\Big),
  \end{equation}
 where
\be \label{def:B0}
|B|_0 := \max_{\xi \in {\mathcal R}_{12}}  \max\Big(\,|\Pi_{1}(\xi+k)B(\vec e_{1}) \Pi_{2}(\xi)|, \, |\Pi_2(\xi) B(\vec e_{-1})\Pi_{1}(\xi+k)| \, \Big),
\ee
and
 \be \label{def:g}
 \g := \Big| \max_{\xi \in {\mathcal R}_{12}} \Re e \, \big( \G(\xi)^{1/2}\big) \Big|.
 \ee
In the definition of $|B|_0$ above, we denote $|Z|$ the norm \label{matrix:norm} of a matrix $Z \in \C^{N \times N},$ deriving from the $L^\infty$ norm in $\C^N:$ $|Z| := \max_{|u| = 1} |Z u|,$ where $|u| = |u_1,\dots,u_N| := \max_i |u_i|.$ These norms in finite dimensions are used throughout the paper.

\begin{theo}\phantomsection \label{theorem1}
Under Assumptions {\rm \ref{ass:spectral}} (regularity of the spectral decomposition), {\rm \ref{ass-u-a}} (existence of WKB solutions) and {\rm \ref{ass:transp}} (resonances and transparency), the sign of index ${\bf \G}$ determines stability of the WKB approximate solution $u_a$ with respect to initial perturbations, as follows:

\medskip

$\bullet$ If ${\bf {\Gamma}} > 0,$ then for some $\phi(\e,\cdot) \in C^\infty_c$ such that $\sup_{0 < \e < 1} \big(\|\phi(\e,\cdot)\|_{\e,s} + |\phi(\e,\cdot)|_{L^\infty}\big) < \infty$ for all $s,$ the solution $u$ to \eqref{0} issued from the initial datum \be\label{ini org1}u(0)=u_a(0)+\e^{K}\phi_\e\ee satisfies:
 $$ \mbox{for any $K > 0,$ \,\, if $K_a +1/2 \geq K,$ \,\, for any $K' > K_0,$ \,\, for some $T < T_0,$ \,\, some $\e_0 > 0,$}$$
for all $\e \in (0,\e_0),$ $u \in C^0([0, T \sqrt \e |\ln \e|],H^s(\R^d))$ for $d/2 < s \leq s_a,$
 and
   \be\label{est:th3}\sup_{0<\e<\e_0}\sup_{0\leq t\leq T \sqrt \e |\ln \e|} \e^{- K'} | (u - u_a)(t)|_{L^2(B(x_0,\rho))}  =\infty,\ee
for some $x_0 \in \R^d,$
some $\rho > 0.$%

\medskip

$\bullet$ If ${\bf {\Gamma}} < 0,$ then for any $\phi(\e,\cdot) \in H^s$ such that $\sup_{0 < \e < 1} \| \phi(\e,\cdot) \|_{\e,s} < \infty$ for all $s \leq s_a,$ the solution $u$ to \eqref{0} issued from the initial datum \eqref{ini org1} satisfies:
$$ \mbox{for any $K \geq (d+1)/2,$ \,\, if $K_a \geq (d + 1)/2,$ \,\,for some $\e_0 > 0,$ \,\, some $C(T_a) > 0,$}$$
for all $\e \in (0,\e_0),$ $u \in C^0([0,T_a],H^s(\R^d)),$ for $d/2 < s \leq s_a,$ and
\be \label{est:stab}
\sup_{t \in [0, T_a]} \|(u-u_a)(t)\|_{\e,s} \leq C(T_a) \| (u - u_a)(0)\|_{\e,s}.
\ee
\end{theo}

In the above Theorem, we use the semiclassical Sobolev norm $\|\cdot\|_{\e,s}$ defined in \eqref{weighted:norm} on page \pageref{weighted:norm}. The time $T_a$ is any existence time for the WKB solution $u_a,$ such that the bounds stated in Assumption \ref{ass-u-a} hold, the index $s_a$ is the Sobolev index of regularity of $u_a,$ and $T_0$ is the limiting observation time defined in \eqref{def:T0-K0}.

Theorem \ref{theorem1} is proved in Sections \ref{p-o-th1} and \ref{sec:stability1}.

\section{Comments} \label{sec:comments}

{\it On our assumptions:}

\medskip

$\bullet$ Assumption \ref{ass:spectral} asserts regularity of the spectral decomposition, and bounds at infinity. We show in Appendix \ref{app:regdec} that these conditions at infinity follow from smoothness of the spectral decomposition of an associated ``short-wave" operator.

\medskip

$\bullet$ In Assumption \ref{ass-u-a}, the polarization condition \eqref{pol-con} and the bound \eqref{est:va-ra} on the correctors and remainder are standard. What is not clear, however, is that system \eqref{0} admits WKB solutions at all. Indeed, as noted by Joly, M\'etivier and Rauch \cite{JMR-TMB}, and briefly discussed on page \pageref{compa} above, in the context of supercritical geometric optics, for WKB solutions to exist the large source term must satisfy compatibility conditions, which are similar to, and weaker than, transparency in the sense of Definition \ref{def:transp}. In the context of \eqref{0}, these conditions are given in Appendix \ref{app:onwkb}. Also, in Proposition \ref{rem:weaktransp}, we give sufficient conditions for Assumption \ref{ass:several-res} to imply Assumption \ref{ass-u-a}.

\medskip

$\bullet$ Point (i) of Assumption \ref{ass:transp} is typically easy to check, and discussed in Appendix \ref{app:structure-resonant-set}. Point (ii) is too strong (indeed in examples there is typically more than one non-transparent resonance); this is remedied in Assumption \ref{ass:several-res}. In theory, the verification of a transparency condition is a simple matter; see our computations in Section \ref{sec:KG} in the case of coupled Klein-Gordon equations. In practice, the computations are sometimes involved, see in particular \cite{T3} in the case of Euler-Maxwell. Point (iii) simplifies the linear algebra in Appendix \ref{app:symb-bound}. As mentioned in Section \ref{sec:open}, it would be interesting to handle two-dimensional eigenspaces, especially in view of the extension of our results to the Euler-Maxwell system in three space dimensions, for which the longitudinal modes are two-dimensional.

\medskip

{\it On the nature of the instability and parameters $T_0$ and $K_0:$}

\smallskip

$\bullet$ The smaller the amplification exponent $K_0$ defined in \eqref{def:T0-K0}, the stronger the amplification. There holds $K_0 = K - T_0 \g |a|_{L^\infty}:$ that is, the limitation on the amplification exponent is the existence time $T_0.$ An existence proof up to time $K/(\g |a|_{L^\infty})$ would allow for $K_0 = 0,$ corresponding to an $O(1)$ deviation; this observation is exploited in Theorem \ref{th:5} below.

\smallskip

$\bullet$ For the existence part in Theorem \ref{theorem1}, two approaches are combined in the proof. This explains why $T_0$ in \eqref{def:T0-K0} is the maximum of two quantities.

We first use the constant-coefficient nature of the hyperbolic operator in \eqref{0}, and the semilinear nature of the source term in the right-hand side of \eqref{0}. Such equations are amenable to estimates in ${\mathcal F}L^1,$ the Banach algebra of functions with Fourier transforms in $L^1.$ The key is that the ${\mathcal F}L^1$ norm controls the $L^\infty$ norm. Combined a priori estimates in ${\mathcal F}L^1$ and $H^s$ yield the existence time $K/(|B|_0 |\hat a|_{L^1}),$ where $|B|_0$ is defined in \eqref{def:B0}. Details are given in Section \ref{sec:FL1}. Here we note that $\g \leq |B|_0.$ Indeed, if the trace $\G(\xi)$ is positive, then $\G(\xi)$ is the only non-zero eigenvalue of the product of the interaction coefficients. In particular, the modulus $|\G(\xi)|$ is bounded from above by the norm of the product of the interaction coefficients: $|\G(\xi)| \leq |B|_0^2,$ which implies $\g \leq |B|_0.$  We could hope for the better existence time $K/(\g |\hat a|_{L^1}).$ The issue here is that we are unable to use the precise Duhamel representation of Section \ref{est-upper-V0}. Indeed, this representation introduces pseudo-differential operators, the action of which cannot be easily estimated in ${\mathcal F}L^1.$
Remark \ref{rem:FL1-bis} expands on this point.

In a second approach, we perform Sobolev estimates. A control of the $L^\infty$ norm is then given by Sobolev embedding. An issue here is that in semiclassical norms this embedding has a large $\sim \e^{-d/2}$ norm. Hence, via the Duhamel representation of Section \ref{est-upper-V0}, in which we use Theorem \ref{th:duh}, the existence time $(K - d/2)/(\g |a|_{L^\infty}).$

 \smallskip

$\bullet$ In accordance with the above two comments, the amplification exponent $K_0$ is the minimum of two positive quantities. The first is equal to a fraction of $K,$ which goes to 0 as $|a|_{L^\infty}/|\hat a|_{L^1} \to 1$ and $\g \to |B|_0.$

 In general, however, and in particular as $K$ becomes larger, for a given system and a given initial amplitude, the amplification factor will be equal to $d/2.$ That is, even though the deviation from $u_a$ is small ($\sim \e^{d/2}$), the amplification in Theorem \ref{theorem1}, on top of being localized in space and asymptotically instantaneous in time, is {\it absolute}, meaning that the perturbation grows from $O(\e^K)$ to $O(\e^{K'}),$ with $K/K' \to \infty$ as $K \to \infty.$ In other words, the flow at $t = 0$ is not H\"older continuous: the ratio $\dsp{|u - u_a|_{L^2(B(x_0,\rho))}/|(u - u_a)(0)|^\a_{\e,s}}$ is unbounded in the limit $\e \to 0,$ with a H\"older exponent $\a = K'/K$ which tends to $0$ as $K \to \infty.$ Still, the instability is weaker than a Lyapunov instability, which would correspond to $K_0 = 0.$

\medskip

{\it On the initial perturbation in the unstable case:}

\smallskip

In the unstable case, we pick the initial perturbation that will lead to a maximal amplification. In the more general context of Theorem \ref{th:5}, these initial perturbations are precisely described in \eqref{datum} below. Essentially, $u(0) - u_a(0)$ has an $\e^K$ prefactor, oscillates at frequency $(\xi_0 + k)/\e,$ where $\xi_0$ is a distinguished resonant frequency, is spatially localized around a point at which $|a|$ is maximum, and is pointing in an eigendirection of the product of the interaction coefficients.

\medskip

{\it On the stability result and the class of initial perturbations:}

\smallskip

A relative weakness of our stability result is that initial perturbations are small $\sim \e^K,$ with $K \geq (1+ d)/2,$ where $1/2$ accounts for the large prefactor in $(1/\sqrt \e) B(u,u),$ while $d/2$ accounts for the Sobolev embedding \eqref{embed:intro}; this point was briefly discussed in Section \ref{sec:open}.

A strong point, however, is that stability is meant here with respect to a large class of initial perturbations, as mentioned in Section \ref{sec:init:pert} above. We allow indeed initial perturbations of the form $\varphi(x/\e),$ where $\varphi$ is smooth, as opposed to perturbations in the form of {\it profiles,} that is, with a dependence in $x/\e$ that is identical to the one in $u_a(0).$ In particular, the perturbations that we use in the unstable case are allowed in the stable case.

We finally note that, by Sobolev embedding, the stability estimate \eqref{est:stab} implies the pointwise estimate
 $|(u - u_a)(t,x) | \leq C(T_a) \e^{K - d/2}.$

\section{Extensions} \label{sec:extension}

 We give here four results that complement Theorem \ref{theorem1}.

\subsection{Several non-transparent resonances} \label{sec:several0}

We announced that the role of Assumption \ref{ass:transp} was only to simplify the exposition. Here is a more general, and more satisfactory, version of Assumption \ref{ass:transp}, in which we do not assume that the set of non-transparent resonances is reduced to a singleton:

\begin{assump}\phantomsection \label{ass:several-res} We suppose
\begin{itemize}
\item[{\rm (i)}] (boundedness) The resonant set $\dsp{{\mathfrak R} = \bigcup_{i,j} {\mathcal R}_{ij}}$ is bounded.

\smallskip

\item[{\rm (ii)}] (partial transparency) For some subset ${\mathfrak R}_0 \subset {\mathfrak R}:$ given $(i,j) \in {\mathfrak R} \setminus {\mathfrak R}_0,$ the $(i,j)$ resonance is transparent; given $(i,j) \in {\mathfrak R}_0,$ the $(i,j)$ interaction coefficients are transparent on a neighborhood of
\be \label{new:1605} \dsp{{\mathcal R}_{ij} \bigcap \Big( \big( {\mathcal R}_{i'i} - k \big) \bigcup \big( {\mathcal R}_{jj'} + k \big) \Big), \quad \mbox{for all $i',j'$ with $(i',i) \in {\mathfrak R}_0$ and $(j,j') \in {\mathfrak R}_0,$}}
\ee
and on a neighborhood of
  \be \label{new:1605:2} {\mathcal R}_{ij} \bigcap \Big( {\mathcal R}_{ii'} \bigcup {\mathcal R}_{j'j}\Big), \quad \mbox{for all $i' \neq j,$ all $j' \neq i,$ with $(i,i') \in {\mathfrak R}_0$ and $(j',j) \in {\mathfrak R}_0.$}
 \ee
 Besides, ${\mathfrak R}_0$ does not contain auto-resonances: for all $1 \leq i \leq J,$ $(i,i) \notin {\mathfrak R}_0.$

\smallskip

 \item[{\rm (iii)}] (rank-one coefficients) For all $(i,j)\in {\mathfrak R}_0$, for all $\xi$ in an open set containing ${\mathcal R}_{ij},$ the ranks of the $(i,j)$ interaction coefficients are at most 1; except for $(i,j) \in {\mathfrak R}_0$ such that one interaction coefficient is identically equal to zero, in which case we make no assumption on the rank of the other coefficient.
\end{itemize}
\end{assump}

In condition (ii), transparency of an interaction coefficient over a frequency set means factorization of the phase, as in \eqref{factor1} or \eqref{factor2}, for $\xi$ restricted to the frequency set in question.

We note the inclusions
$${\mathcal R}_{ij} \bigcap {\mathcal R}_{ii'} \subset \{ \l_j = \l_{i'}\}, \qquad {\mathcal R}_{ij} \bigcap {\mathcal R}_{j'j} \subset \{ \l_i(\cdot + k) = \l_{j'}(\cdot + k)\}.$$
Thus condition (ii) in Assumption \ref{ass:several-res} means that, while we allow for an arbitrarily large subset ${\mathfrak R}_0$ of non-transparent resonances, we cannot allow for some $(i,j)$ interaction coefficient to be non-transparent in all of the resonant set ${\mathcal R}_{ij}:$ we need to assume transparency at these exceptional frequencies in ${\mathcal R}_{ij}$ which correspond to a translate of another non-transparent resonance involving $\l_i$ or $\l_j$ (condition \eqref{new:1605}) or to a coalescing point in the spectrum involving $\l_i$ or $\l_j$ and another branch associated with a non-transparent resonance (condition \eqref{new:1605:2}).

\bigskip

 Given a resonance $(i,j) \in {\mathfrak R}_0,$ we let
 $$ \G_{ij}(\xi) := {\rm
tr}\, \Pi_{i}(\xi+k)B(\vec e_{1})\Pi_{j}(\xi)B(\vec e_{-1})\Pi_{i}(\xi+k).$$
The stability index is
\begin{equation} \label{def:index}
 {\bf \Gamma} := \max_{(i,j) \in {\mathfrak R}_0} \max\Big( \, \max_{\xi\in {\mathcal R}_{ij}} \Re e \, \G_{ij}(\xi), \, \max_{\xi \in {\mathcal R}_{ij}} |\Im m \, \G_{ij}(\xi)|\,\Big).
 \end{equation}
 The observation time $T_0$ and amplification rate $K_0$ are defined as in \eqref{def:T0-K0}, with $|B|_0$ defined by
 \be \label{def:B0several}
 |B|_0 := \max_{(i,j) \in {\mathfrak R}_0} \max_{\xi \in {\mathcal R}_{ij}} \max\Big(\,   | \Pi_{i}(\xi+k)B(\vec e_{1})\Pi_{j}(\xi)|\, \, , \, \, |\Pi_j(\xi) B(\vec e_{-1})\Pi_{i}(\xi+k)|\,\Big),
\ee
 and $\g$ defined by
 \be \label{def:several-gamma} \g := \max_{(i,j) \in {\mathfrak R}_0} \g_{ij}, \qquad \g_{ij} :=  \Big| \max_{\xi \in {\mathcal R}_{ij}} \Re e \, \big( \G_{ij}(\xi)^{1/2}\big) \Big|.\ee

\begin{theo}\phantomsection\label{th-two} The conclusions of Theorem {\rm \ref{theorem1}} still hold when {\rm Assumption \ref{ass:transp}} is replaced by {\rm Assumption
\ref{ass:several-res}}, with stability index ${\bf \Gamma}$ defined in
\eqref{def:index}.
\end{theo}

Theorem \ref{th-two} is proved in Section \ref{sec:proof-th-extension}.

\subsection{All non-transparent resonances are amplified}

Next we modify the partial transparency condition of Assumption \ref{ass:several-res} into a separation condition. The result (Theorem \ref{th-three} below) states that any non-transparent resonance that is separated from other resonances gives rise to an amplification.

Recall that the notation ${\mathfrak R},$ introduced in Assumption \ref{ass:several-res}, denotes the set of resonant indices.

\begin{assump}\phantomsection \label{ass:12-ampli} We suppose
\begin{itemize}
\item[{\rm (i)}] (boundedness) The resonant set ${\mathfrak R}$ is bounded.

\smallskip

\item[{\rm (ii)}] (separation) For any resonant pair $(i,j)$ that is distinct from $(1,2),$ there holds
\be \label{sep} \big( {\mathcal R}_{12} + qk \big) \bigcap {\mathcal R}_{ij} = \emptyset, \quad 0 \leq |q| \leq 1.\ee

\smallskip

\item[{\rm (iii)}] (rank-one coefficients) For all $\xi$ in an open set containing ${\mathcal R}_{12},$ the ranks of the $(1,2)$ interaction coefficients are at most 1.
\end{itemize}
\end{assump}

The separation condition \eqref{sep} asserts non-intersection of resonant sets, a much stronger property than the partial transparency condition of Assumption \ref{ass:several-res}, but only relative to the $(1,2)$ resonance. With regard to other resonances, we make no assumption besides boundedness of resonant sets and separation from ${\mathcal R}_{12}$ as prescribed by \eqref{sep}.

In the context of Assumption \ref{ass:12-ampli}, the limiting observation time is
$$
  T'_0 := \min\Big( \, \max\Big(\, \frac{K-1/2}{|B| |\hat a|_{L^1}}\, , \,\, \frac{K - (d+1)/2}{|B| |a|_{L^\infty}}\, \Big) \, \, , \,\, \frac{1}{2 ( |B| - \g_{12}) |a|_{L^\infty}}\, \Big),
$$
 and the amplification exponent is
 \be \label{def:K0prime}
 K'_0 := K - T'_0 \g_{12} |a|_{L^\infty},
 \ee
with $\g_{12}$ defined in \eqref{def:several-gamma}.

\begin{theo}\phantomsection \label{th-three} Under Assumptions {\rm \ref{ass:spectral}}, {\rm \ref{ass-u-a}} and {\rm \ref{ass:12-ampli}}, if $\g_{12} > 0,$ then the WKB solution $u_a$ is unstable, in the sense of Theorem {\rm \ref{theorem1}}, where $T_0$ and $K_0$ are replaced with $T'_0$ and $K'_0$ defined above.
\end{theo}

Theorem \ref{th-three} is proved in Section \ref{sec:proof-three}.

The limiting time $T'_0$ is the minimum of two quantities. The first looks very much like (but is smaller than) the limiting time $T_0$ in Theorems \ref{theorem1} and \ref{th-two}. The main difference is that we have $|B|$ here in the denominator in $T'_0,$ where we had $|B|_0,$ possibly much smaller than $|B|,$ in the denominator in $T_0$ \eqref{def:T0-K0}.
The second term in the definition of $T'_0$ is independent of $K$ and, for $K$ large enough, smaller than the first. The point is that the instability here is only {\it relative} to the initial size of the perturbation: there holds for $K$ large enough (depending on $B$ and $a$):
\be \label{def:tildeT0} \sup_{0 < \e < \e_0} \sup_{ 0\leq t \leq \tilde T_0 \sqrt \e |\ln \e| } \frac{ | u(t) - u_a(t)|_{L^2(B(x_0,\rho))}}{\| u(0) - u_a(0)\|_{\e,s}} = \infty, \qquad \tilde T_0 := \frac{1}{2 (|B| - \g_{12}) |a|_{L^\infty}},\ee
 and the way that the above diverges to $\infty$ is quantified by $K'_0.$ In other words, the flow is not Lipschitz continuous from $\| \cdot \|_{\e,s}$ to $\| \cdot \|_{L^2(B(x_0,\rho))}$ at $t = 0.$ 

The important point in Theorem \ref{th-three} is that we do not assume maximality of the growth rate $\g_{12}.$ That is, in order to record an Hadamard instability, it is not necessary to initially activate unstable frequencies with the larger rate of growth.

An interesting feature of Assumption \ref{ass:12-ampli} is that, given boundedness of ${\mathfrak R},$ the separation condition (ii) is local in frequency, and bears only on the eigenvalues of the hyperbolic operator. The assumption of boundedness of ${\mathfrak R}$ is not local, but typically easily verified by asymptotic expansions of the eigenvalues at infinity, as discussed in Appendix \ref{app:structure-resonant-set}.

\subsection{Improved spatial localization} \label{sec:thbetterloc}

 Next we remark that we can improve on the radius of the instability ball in the statement of Theorem \ref{theorem1}, at the price of a smaller amplification rate. That is, introducing the observation time
 $$ T''_0 := \max\Big( \frac{K - 1/2}{|B|_0 |\hat a|_{L^1}}, \,\frac{K - (d+1)/2}{\g |a|_{L^\infty}}\Big),$$
 and the amplification exponent
 \be \label{def:K''0} K''_0 := K + \frac{\b d}{2} - T''_0 \g |a|_{L^\infty},\ee
 we have the following result, with stability index ${\bf \G}$ defined in \eqref{def:index}.

 \begin{theo}\phantomsection \label{th:4} Under Assumptions {\rm \ref{ass:spectral}}, {\rm \ref{ass-u-a}} and {\rm \ref{ass:several-res}}, in the unstable case ${\bf \G} > 0,$ the conclusion of Theorem {\rm \ref{theorem1}} still holds with $\rho = \e^\b,$ for any $\b < 1/d,$ with $\e_0 = \e_0(\b),$ if $T_0$ and $K_0$ are replaced with $T''_0$ and $K''_0$ as above.
 \end{theo}

Theorem \ref{th:4} is proved in Section \ref{sec:pftheobetterloc}. Note that the observation time $T''_0$ is strictly smaller than $T_0$ \eqref{def:T0-K0}, hence the amplification exponent $K''_0$ is strictly greater than $K_0.$ That is, the deviation estimate is better localized, but we are able to follow the solution only on a shorter time interval, meaning a smaller magnitude of the deviation from $u_a.$

\subsection{A greater deviation estimate}

As noted in Section \ref{sec:comments}, the limitation on the amplification exponent $K_0$ in Theorems \ref{theorem1} and \ref{th-two} is the upper bound $T_0$ on the existence time. In our final result, we {\it assume} a better existence time, and from there deduce a greater deviation estimate.

In the context of Assumption \ref{ass:several-res}, the maximum $\g$ of the coefficients $\g_{ij},$ for $(i,j)$ ranging over the set ${\mathfrak R}_0$ of non-transparent resonances is attained at $(i_0,j_0).$  We consider the same datum as in Theorems \ref{theorem1} and \ref{th-two}, that is
\begin{equation} \label{datum}
 u(0,x)  := u_a(0,x) + \e^{K} e^{i x \cdot(\xi_0 + k)/\e} \varphi_{i_0j_0}(x) \vec e_{i_0j_0},
 \ee
 where
 \begin{itemize}
 \item $\xi_0$ is such that $\g =\big| \max_{\xi\in {\mathcal R}_{i_0j_0}} \Re e \, \big( \G_{i_0j_0}(\xi)^{1/2} \big) \big|  \quad \mbox{is attained at $\xi_0$.}$

\smallskip

 \item $x_0$ is such that $|a|_{L^\infty}$ is attained at $x_0.$

\smallskip

 \item $\varphi_{i_0j_0} \in C^\infty_c(\R^d)$ is a spatial truncation around $x_0$ (precisely defined in Section \ref{several:loc}).

\smallskip

 \item $\vec e_{i_0j_0}$ generates the range of matrix $\Pi_{i_0}(\xi_0 + k) B(\vec e_1) \Pi_{j_0}(\xi_0) B(\vec e_{-1}) \Pi_{i_0}(\xi_0 + k).$
 \end{itemize}
 Let $\dsp{T_\infty := \frac{K}{\g |a|_{L^\infty}}.}$

\begin{theo}\phantomsection \label{th:5} Under Assumptions {\rm \ref{ass:spectral}}, {\rm \ref{ass-u-a}} and {\rm \ref{ass:several-res}}, in the unstable case ${\bf \G} > 0,$ for any $K > 0,$ if $K_a +1/2 \geq K,$ for $d/2 < s \leq s_a:$ 
\begin{itemize}
\item either for some $T < T_\infty,$ for any $\e$ small enough, the initial-value problem \eqref{0}-\eqref{datum} does not have a solution $u \in C^0([0, T \sqrt \e |\ln \e|], H^s(\R^d)),$
\item or for some $T < T_\infty,$ for any $\e_0 >0,$ the solution $u$ to \eqref{0}-\eqref{datum} satisfies
 $$\sup_{0 < \e < \e_0} \, \sup_{0 \leq t \leq T \sqrt \e |\ln \e|} \, |u(t,\cdot)|_{L^\infty} = \infty,$$
\item or for any $K' > 0,$ for some $T < T_\infty,$ there holds the deviation estimate
   \be\label{est:th5}\sup_{0<\e<\e_0}\sup_{0\leq t\leq T \sqrt \e |\ln \e|} \e^{- K'} | (u - u_a)(t)|_{L^2(B(x_0,\e^\b))}  =\infty,\ee
for some $x_0 \in \R^d,$
some $\b > 0,$ some $\e_0 > 0.$
 \end{itemize}
 \end{theo}

The proof (Section \ref{sec:end-insta2}) shows that $\b \to 0$ as $K' \to 0:$ the localization becomes less precise as the amplification becomes larger.

Theorem \ref{th:5} states that, with $(u - u_a)(0)$ given in \eqref{datum}, in particular compactly supported and $O(\e^K)$ in $L^\infty$ and $\| \cdot \|_{\e,s}$ norms, where $K$ and $s$ are arbitrarily large:
\begin{itemize}
\item either the solution to \eqref{0}-\eqref{datum} is not defined in time $[0, T \sqrt \e |\ln \e|],$ meaning a catastrophic collapse of the existence time around $u_a$ (if we consider $u_a$ as a solution, with an existence time $T_a > 0,$ independent of $\e),$
\item or the solution is defined in time $[0, T \sqrt \e |\ln \e|]$ but is unbounded, meaning in particular that $|u - u_a|_{L^\infty}$ is unbounded, in the limit $\e \to 0,$
\item or the deviation $|u - u_a|$ goes from $O(\e^K),$ as measured in $\| \cdot \|_{\e,s}$ norm, to $O(\e^{K'}),$ as measured in an $L^2(B(x_0,\e^\b))$ norm, over a time interval of length $O(\sqrt \e |\ln \e|),$ with $K'$ arbitrarily small.
\end{itemize}

In conclusion, under the assumptions of Theorem \ref{th:5}, in the case ${\bf \G} > 0$ the WKB solution can certainly be deemed unstable, although there is some imprecision as to the terms of the instability.

\chapter{Main proof} 

%

\section{Proof of Theorem \lowercase{\ref{theorem1}}: instability}\label{p-o-th1}

Under the assumptions of Theorem \ref{theorem1}, we suppose ${\bf \G} > 0,$ and proceed to prove instability of $u_a.$

\subsection{Overview of the instability proof} Section \ref{sec:proj} contains the first important change of variable, in which resonances appear explicitly as crossing points for the eigenvalues of the propagator. In Section \ref{normal-form1}, we perform a normal form reduction; this essentially reduces the linear source term $B(u_a)$ to the pair of interaction coefficients associated with resonance $(1,2).$  Then in Section \ref{s-f loc}, we localize the analysis around a distinguished point $(x_0,\xi_0),$ with $\xi_0 \in {\mathcal R}_{12},$ in the cotangent space. In Section \ref{est-upper-V0}, we use the Duhamel representation formula of Appendix \ref{app:duh} in order to describe semi-explicitly the component of the solution associated with resonance $(1,2).$ Section \ref{sec:lower} is devoted to the derivation of lower bounds for the action of the solution operator on the datum. In Section \ref{est-upper-V0-a}, we give existence results and upper bounds in time $O(\sqrt \e |\ln \e|).$ These are based on the representation formula, and also on ${\mathcal F}L^1$ and Sobolev bounds. The comparison of lower bounds with upper bounds in Section \ref{sec:end-insta}
concludes the proof.

\subsection{Preparation}

By symmetry of the hyperbolic operator, for $\e > 0$ the solution $u$ to \eqref{0} issued from \eqref{ini org1} is defined over a short time interval $[0,T(\e)],$ for some $T(\e) > 0.$ It has a high Sobolev regularity: $u \in C^0([0,T(\e)],H^{s_a}(\R^d)),$ where $s_a$ is the Sobolev regularity index introduced in Assumption \ref{ass-u-a}.
The perturbative unknown $\dot u,$ defined by
\begin{equation} \label{def:dot-u}
 u =: u_a + \dot u,
 \end{equation}
  solves
\begin{equation} \label{eq:dot-u} \d_t \dot u + \frac{1}{\e} A_0 \dot u + \sum_{1 \leq j \leq d} A_j \d_{x_j} \dot u = \frac{1}{\sqrt \e} B(u_a) \dot u + \frac{1}{\sqrt \e} B(\dot u, \dot u) - \e^{K_a} r_a^\e,
\end{equation}
with the datum $\dot u(\e,0,x) = \e^{K} \phi(\e,x).$ In \eqref{eq:dot-u}, the term $r_a^\e$ is the WKB remainder introduced in Assumption \ref{ass-u-a}, and we used Notation \ref{notation:b} for $B(\cdot).$

The goal is to choose $\phi$ so that $\dot u$ grows exponentially in time.

\medskip

It is understood in this proof that Sobolev indices $s$ are strictly smaller than $s_a - d/2,$ where $s_a$ is the Sobolev index of regularity of $u_a.$ We need $s > d/2$ in order to prove short-time existence (since the system is semilinear), but the deviation estimate \eqref{est:th3} is expressed in a localized $L^2$ norm.

\medskip

We frequently use the notation $a \lesssim b$ to indicate that an inequality $a \leq C b$ holds true, with a constant $C > 0$ depending only on fixed parameters, such as dimensions, in particular not on $a,b,$ nor on $\e.$\label{notation:lesssim}

\subsubsection{Projection and frequency shift} \label{sec:proj}

We decompose $\dot u$ according to the eigenmodes of the hyperbolic operator and shift the component associated with $\Pi_1$ as we define $U = (U_1, \dots, U_J) \in \R^{NJ}$ by
\be\label{de-U} U_1 := e^{- i \theta} \op_\e(\Pi_1) \dot u, \quad U_j := \op_\e(\Pi_j) \dot u, \,\, 2 \leq j \leq J,
\ee
with notation
\be \label{def:theta}
 \theta := (k \cdot x - \o t)/\e.
\ee
The projectors $\Pi_j(\xi)$ are eigenprojectors of $A(i \xi) + A_0$ (Assumption \ref{ass:spectral}), and $\op_\e(\Pi_j)$ are the associated Fourier multipliers \eqref{quantiz}. The perturbation unknown $\dot u$ can be reconstructed from $U$ via
 \begin{equation} \label{dotu-U}
  \dot u = e^{i \theta} U_1 + \sum_{2 \leq j \leq J} U_j.
 \end{equation}
From \eqref{eq:dot-u}-\eqref{de-U}, we find that $U = (U_1, U_2 \,| \, U_3, \dots, U_J) \in \R^{2N} \times \R^{(J-2)N}$ solves
\begin{equation} \label{sys U}
 \d_t U + \frac{1}{\e} \op_\e( i \mathcal{A}) U  = \frac{1}{\sqrt\e} \op_\e({\mathcal B}) U + F.
\end{equation}
The symbol of the propagator is the diagonal matrix
\be \label{def:calA} \mathcal{A} := {\rm diag}\, \big( \l_{1,+1} - \o, \l_2 \, | \, \l_3, \dots, \l_J \big),\ee
with the notation $\l_{1,+1}(\xi) := \l_1(\xi + k),$ where $k$ is the spatial frequency of the WKB datum. More generally, we will often use the notation
\be \label{def:shift} \s_{+p}(x,\xi) := \s(x, \xi + p k).\ee
In the symbol ${\mathcal A},$ the frequency shift is caused by the fast spatial oscillation in the definition of $U_1:$ there holds indeed the identity
 $$\op_\e(\s) (e^{i p \theta} v) = e^{i p \theta} \op_\e(\s_{+p}) v, \qquad \mbox{ for all $\s, p, v.$}$$
The symbol of the singular source term is
$$ {\mathcal B} := \left(\begin{array}{c|c} {\mathcal B}_{[1,2]} & {\mathcal B}_{[1,2,J]} \\ \hline {\mathcal B}_{[J,1,2]} & {\mathcal B}_{[J,J]} \end{array}\right),$$
where the top left block is
$$ {\mathcal B}_{[1,2]} := \sum_{p =\pm1} \left(\begin{array}{cc} e^{ip
\theta} \Pi_{1,+(p+1)} B_p \Pi_{1,+1}  & e^{i(p-1) \theta} \Pi_{1,+p} B_p \Pi_2 \\ e^{i(p+1) \theta} \Pi_{2,+(p+1)} B_p \Pi_{1,+1} & e^{ip\th}\Pi_{2,+p} B_p \Pi_2 \end{array}\right) \in \R^{2N \times 2N},$$
using notation $\Pi_{j,+q}(\xi) := \Pi_j(\xi + q k),$ for $q \in \Z,$ in accordance with \eqref{def:shift}, and
\be \label{def:Bp}
 B_p := B(u_{0,p}), \qquad p \in \{-1,1\},
\ee
where $u_{0,\pm 1}$ are the leading amplitudes in the WKB solution, introduced in \eqref{def-ua}. The other blocks in the source are
$$\begin{aligned} {\mathcal B}_{[1,2,J]} & := \displaystyle
\sum_{p=\pm1}  \left(\begin{array}{ccc} e^{i(p-1) \theta}\Pi_{1,+p}B_p \Pi_{3} & \dots & e^{i(p-1) \theta}\Pi_{1,+p}B_p\Pi_{J}\\ e^{ip\th}\Pi_{2,+p}B_p\Pi_3 & \dots & e^{ip\th}\Pi_{2,+p}B_p\Pi_J\end{array}\right) \in \R^{2N \times (J-2)N}, \\ {\mathcal B}_{[J,1,2]}  & :=  {\displaystyle
\sum_{p=\pm1}} \left(\begin{array}{cc}  e^{i(p+1)\th}\Pi_{3,+(p+1)}B_p\Pi_{1,+1} &
e^{ip\th}\Pi_{3,+p}B_p\Pi_2 \\ \vdots & \vdots \\e^{i(p+1)\th}\Pi_{J,+(p+1)}B_p\Pi_{1,+1} &
e^{ip\th}\Pi_{J,+p}B_p\Pi_2  \end{array}\right) \in \R^{(J-2) N \times 2N},
\end{aligned}$$
and
$$ {\mathcal B}_{[J,J]}  := {\displaystyle  \sum_{p=\pm1}} \Big( e^{ip\th}\Pi_{i,+p}B_p\Pi_j\Big)_{3 \leq i,j, \leq J} \in \R^{(J-2)N \times (J-2)N}.$$
In \eqref{sys U}, the remainder $F$ is the sum of
\begin{itemize}
 \item the quadratic term
  $$ \e^{ - 1/2} \Big( e^{- i \theta} \op_\e(\Pi_1) B(\dot u, \dot u), \op_\e(\Pi_2) B(\dot u, \dot u), \dots, \op_\e(\Pi_J) B(\dot u, \dot u)\Big),$$
 \item the projected WKB remainder
  $$ \e^{K_a} \Big(  e^{- i \theta} \op_\e(\Pi_1) r_a^\e, \op_\e(\Pi_2) r_a^\e, \dots, \op_\e(\Pi_J) r_a^\e\Big),$$
 \item the contribution of the higher-order WKB terms $v_a^\e:$
  $$ \Big( e^{- i \theta} \op_\e(\Pi_1) B(v_a^\e)\dot u, \op_\e(\Pi_2) B(v_a^\e) \dot u, \dots, \op_\e(\Pi_J) B(v_a^\e )\dot u\Big),$$
 \item and remainder terms arising from compositions of pseudo-differential operators; these terms have the form
  $$ \e^{-1/2} e^{i q_1 \theta} \Big( \op_\e(\Pi_{i,+q_2}) B_{q_3} \op_\e(\Pi_{j,+q_4}) \dot u - \op_\e\Big(\Pi_{i,+q_2}  B_{q_3} \Pi_{j,+q_4}\Big) \dot u \Big),$$
  where $q_j \in \Z$ and $B_{q_3}$ is defined in \eqref{def:Bp}.
    \end{itemize}
From this description of $F,$ we deduce the following bound:

\begin{lemm}\phantomsection \label{lem:bd-F} There holds for $s \geq 0$ the bound
   $$\| F \|_{\e,s} \lesssim (1 + \e^{- 1/2} |\dot u|_{L^\infty}) \|\dot u\|_{\e,s} + \e^{K_a},$$
   and the bound
   $$ |\hat F|_{L^1} \lesssim (1 + \e^{-1/2} |{\mathcal F} \dot u|_{L^1}) |{\mathcal F} \dot u|_{L^1} + \e^{K_a}.$$
 \end{lemm}

In Lemma \ref{lem:bd-F} we are using the semi-classical Sobolev norms $\| \cdot \|_{\e,s}$ introduced in \eqref{def:es}.

\begin{proof} By \eqref{action:fourier-mult} and \eqref{product}, the quadratic terms satisfy for $s \geq 0:$
$$
\e^{ - 1/2} \big\| \op_\e(\Pi_j)  B(\dot u, \dot u) \big\|_{\e,s} \lesssim \e^{ - 1/2} |\dot u|_{L^\infty} \|\dot u\|_{\e,s}.
$$
 By \eqref{action:fourier-mult} and the product law \eqref{product2},
 $$\big\| \op_\e(\Pi_j)  B(v_a^\e) \dot u \big\|_{\e,s} \lesssim |v_a^\e|_{L^\infty} \|\dot u\|_{\e,s} + |(\e \d_x^s) v_a^\e|_{L^\infty} \| \dot u \|_{L^2},$$
 and with \eqref{est:va-ra}, this gives
 $$
 \| \op_\e(\Pi_j)  B(v_a^\e) \dot u \|_{\e,s} \lesssim \|\dot u\|_{\e,s}.
 $$
By the commutator estimate \eqref{est:fourier-mult}, given $v \in H^s,$ if $s_a$ is large enough then
  $$ \big\| \op_\e(\Pi_{i,+q_2}) B_{q_3} v - \op_\e\big(\Pi_{i,+q_2}  B_{q_3}\big) v \|_{\e,s} \lesssim \e | u_{0,\pm1} |_{H^{s_a}} \| v \|_{\e,s},$$
  We apply this bound to $v = \op_\e(\Pi_{j,+q_4}) \dot u,$ satisfying $\| v \|_{\e,s} \leq \| \dot u \|_{\e,s}$ and find
  $$ \big\| e^{i q_1 \theta} \e^{-1/2} \Big( \op_\e(\Pi_{i,+q_2}) B_{q_3} \op_\e(\Pi_{j,+q_4}) \dot u -  \op_\e\big(\Pi_{i,+q_2}  B_{q_3} \Pi_{j,+q_4}\big) \dot u \Big) \big\|_{\e,s} \lesssim \e^{1/2} \| u  \|_{\e,s}.$$
The bound in ${\mathcal F}L^1$ is found similarly, using $|B(\dot u, \dot u)|_{{\mathcal F}L^1} \lesssim |{\mathcal F} \dot u|_{L^1}^2,$ and \eqref{est:fouriermultFL1}.
\end{proof}

 The point of the change of variable \eqref{de-U} is that ${\mathcal R}_{12} = \{ \o = \l_{1,+1} - \l_2\}$ is now included in the locus of coalescing eigenvalues of symbol ${\mathcal A}.$

 We will see in the next paragraph that, under Assumption \ref{ass:transp}, in the source term ${\mathcal B}$ only the top left block ${\mathcal B}_{[1,2]}$ matters.

\subsubsection{Normal form reduction}\label{normal-form1}

The resonant set ${\mathcal R}_{12}$ introduced in Definition \ref{def-reson} is bounded by Assumption \ref{ass:transp}, hence compact by continuity of the eigenvalues. For $h > 0,$ to be chosen small enough below\footnote{This will be done in Section \ref{sec:end-insta}, at the end of this proof; see Remark \ref{rem:parameters} on page \pageref{rem:parameters}.}, we consider the neighborhood ${\mathcal R}_{12}^h$ of ${\mathcal R}_{12},$ corresponding to resonant phases bounded by $h$:
\begin{equation} \label{def:Rh}
 {\mathcal R}_{12}^h := \big\{ \xi \in \R^d, \quad |\l_1(\xi + k) - \l_2(\xi) - \o| \leq h \big\},
 \end{equation}
 We let $\chi_0$ be a smooth cut-off function in frequency space, such that $0 \leq \chi_0 \leq 1,$ $\chi_0 \equiv 1$ on a neigborhood of ${\mathcal R}_{12}^{h},$ and $\chi_0 \equiv 0$ further away from the $(1,2)$ resonant locus, for instance on $\R^d \setminus {\mathcal R}_{12}^{2h}.$

The top left block in the symbol ${\mathcal B}$ of the singular source term decomposes as
 $$  {\mathcal B}_{[1,2]}  = {\mathcal B}^r + {\mathcal B}^{nr},$$
 with the notation
 \be \label{def:Br} {\mathcal B}^r := \left(\begin{array}{cc} 0 & \Pi_{1,+1} B_1 \Pi_2 \\ \Pi_2 B_{-1} \Pi_{1,+1} & 0 \end{array}\right).\ee
The following proposition will imply that the operator with symbol
 $$ {\mathcal D} := \left(\begin{array}{cc} (1 - \chi_0) {\mathcal B}^r + {\mathcal B}^{nr} & {\mathcal B}_{[1,2,J]} \\ {\mathcal B}_{[J,1,2]} & {\mathcal B}_{[J,J]} \end{array}\right),$$
 can be eliminated from the evolution equation \eqref{sys U} in $U,$ up to negligible {\it uniform remainders}, defined as follows:

 \label{uniform remainder}
 \begin{defi} \phantomsection\label{def:ur} By uniform remainder, we mean any family $R_0 = R_0(\e,t)$ of linear bounded operators $H^s \to H^s,$ for $t \in [0,T_a],$  where $T_a$ is an existence time for $u_a,$ with semi-classical Sobolev norms that are bounded in $\e$ and $t:$ for some $C > 0,$ for all $v \in H^s,$ all $\e > 0,$ all $t \in [0,T_a],$ there holds
$ \| R_0 v \|_{\e,s} \leq C \| v \|_{\e,s}.$
\end{defi}

\begin{prop}\phantomsection\label{normal-form-prop}Under {\rm Assumption
\ref{ass:transp}}, there exists $Q \in S^0,$ with $\d_t Q \in S^0,$ such that \be\label{homo
eq2}\e [\d_t,\op_\e(Q)]+ [\op_\e(i \mathcal{A}),\op_\e(Q)] = \op_\e({\mathcal D}) +\e
R_0,\ee
where $R_0$ is a uniform remainder.
\end{prop}

\begin{proof} The source ${\mathcal D}$ contains oscillations in $e^{i \ell \theta},$ with $|\ell| \leq 2.$ We denote ${\mathcal D} = \sum_{|\ell| \leq 2} e^{i \ell \theta} {\mathcal D}_\ell.$ Accordingly, we look for $Q$ in the form
 \be \label{def:Q0} Q(\e,t,x,\xi) = \sum_{|\ell| \leq 2} e^{i \ell \theta} Q_\ell(t,x,\xi), \qquad Q_\ell(t,x,\xi) \in \R^{JN \times JN},\ee
 where the symbols $Q_\ell$ are tensor products:
 \begin{equation} \label{def:Q} Q_\ell(t,x,\xi) = \k_\ell(t,x) \tilde Q_\ell(\xi), \quad \mbox{$\k_\ell$ scalar, $\k_\ell \in C^1([0,T_a], H^{s_a}),$ $\tilde Q_\ell \in S^0.$}\end{equation}
  Such symbols $Q$ satisfy $Q \in S^0,$ $\d_t Q \in S^0,$ and, by \eqref{action:fourier-mult} and \eqref{product2}, the associated operators $\op_\e(Q)$ are uniform remainders.

\medskip

In the coordinatization \eqref{sys U}, the variable $U$ belongs to $\C^{N J},$ so that we are looking for the Fourier coefficients $Q_\ell$ of $Q$ in the form of $\C^{N J \times NJ}$ matrices, which depend on $(\e,t,x,\xi).$ We will denote $Z_{(i,j)} \in \C^{N \times N}$ the $(i,j)$ block of a matrix $Z \in \C^{NJ \times NJ}.$ In particular, we will use notation $(Q_{\ell})_{(i,j)} \in \C^{N \times N}$ to denote block $(i,j)$ of $Q_{\ell}.$

\medskip

With $Q$ in the form \eqref{def:Q0}, there holds
 $$ \e [\d_t,\op_\e(Q)] = \sum_{|\ell| \leq 2} e^{i \ell \theta} \Big( - i \ell \o Q_\ell + \e \op_\e(\d_t Q_\ell) \Big).$$
The symbol ${\mathcal A}$ being diagonal, there holds
 $$ \begin{aligned} {} [\op_\e({\mathcal A}), \op_\e(Q) ]_{(i,j)} & = \op_\e(\mu_i) \op_\e(Q_{(i,j)} -  \op_\e(Q_{(i,j)}) \op_\e(\mu_j) \\ & = \sum_{|\ell| \leq 2} e^{i \ell \theta} \Big( \op_\e(\mu_{i,+\ell}) \op_\e((Q_\ell)_{i,j}) - \op_\e((Q_\ell)_{i,j}) \op_\e(\mu_j) \Big), \end{aligned}$$
where the $\mu_j$ are the diagonal entries of ${\mathcal A},$ so that
\be \label{def:mu}
 \mu_1 := \l_{1,+1} - \o, \qquad \mu_j := \l_j, \quad \mbox{for $j \geq 2.$}
\ee
By \eqref{def:Q},
 $$ \op_\e(\mu_{i,+\ell}) \op_\e((Q_\ell)_{i,j}) = \op_\e(\mu_{i,+\ell}) \big( \k_\ell \op_\e\big( \, (\tilde Q_\ell)_{(i,j)}\, \big) \big),$$
 and with the commutator estimate \eqref{est:fourier-mult} and the assumed regularity of $\k_\ell,$
 $$ \op_\e(\mu_{i,+\ell}) \big( \k_\ell \op_\e(\tilde Q_\ell))_{(i,j)} \big) = \op_\e(\mu_{i,+\ell} (Q_\ell)_{i,j}) + \e R_0,$$
 where $R_0$ is a uniform remainder. This implies
$$
  [\op_\e({\mathcal A}), \op_\e(Q) ]_{(i,j)} = \sum_{|q| \leq 2} e^{i q \theta} \op_\e\Big( \big( \mu_{i,+q} - \mu_j \big) (Q_q)_{(i,j)} \Big) + \e R_0,
$$
 where $R_0$ is a uniform remainder.

From the above, we deduce that in order to solve \eqref{homo eq2}, it is sufficient to solve
 \begin{equation} \label{eq:homo} i \big(- \ell \o + \mu_{i,+\ell} - \mu_j\big) (Q_\ell)_{(i,j)} = ({\mathcal D}_\ell)_{(i,j)}, \qquad |\ell| \leq 2, \quad 1 \leq i,j \leq NJ.
 \end{equation}
 We consider equation \eqref{eq:homo} for all possible values of $\ell,$ using the definition of ${\mathcal B}$ in Section \ref{sec:proj}:

\smallskip

$\bullet$ For $\ell = 0,$ equation \eqref{eq:homo} reduces to the system
$$ \begin{aligned}
  i \big( \l_{1,+1} - \o - \l_2 \big) (Q_0)_{(1,2)} & = (1 - \chi_0) \Pi_{1,+1} B_1 \Pi_2 \quad \mbox{\rm (top left block)},\\
  i \big( \l_2 - \l_{1,+1} + \o \big) (Q_0)_{(2,1)} & = (1 - \chi_0) \Pi_{2} B_{-1} \Pi_{1,+1} \quad \mbox{\rm (top left block)}, \\
  i \big( \l_{1,+1} - \o - \l_j \big) (Q_0)_{(1,j)} & = \Pi_{1,+1} B_1 \Pi_j, \qquad 3 \leq j \leq J \quad \mbox{\rm (top right block)},\\
 i \big( \l_j - \l_{1,+1} + \o \big) (Q_0)_{(j,1)} & = \Pi_j B_{-1} \Pi_{1,+1}, \qquad 3 \leq j \leq J \quad \mbox{\rm (bottom left block)}.
  \end{aligned}$$
The first two equations involve the $(1,2)$ resonance relation. On the support of $1 - \chi_0,$ the phase $\l_{1,+1} - \o - \l_2$ is bounded away from zero. Thus we can divide the right-hand sides by the phase and thereby define $(Q_0)_{(1,2)}$ and $(Q_0)_{(2,1)}$ as an element of $S^0.$ The last two equations involve the $(1,j)$ resonance relation. The phase $\l_{1,+1} - \o - \l_j$ might vanish for some $\xi \in \R^d,$ but the transparency assumption (Assumption \ref{ass:transp}(ii)) ensures that this phase factors out in the right-hand sides, so that we can solve for $(Q_0)_{(1,j)}$ and $(Q_0)_{(j,1)}$ in $S^0.$

\smallskip

$\bullet$ For $|\ell| = 1,$ equation \eqref{eq:homo} reduces to
$$ \begin{aligned}
  i \big( \l_{1,+(\ell+1)} - \ell \o - \l_{1,+1} \big) (Q_\ell)_{(1,1)} & = \Pi_{1,+(\ell+1)} B_q \Pi_{1,+1}\quad \mbox{\rm (top left block)}, \\
  i \big( \l_{i,+\ell} - \ell \o - \l_j\big) (Q_\ell)_{(i,j)} &  = \Pi_{i,+\ell} B_\ell \Pi_j, \qquad 2 \leq i,j \leq J \quad \mbox{\rm (all blocks)} \end{aligned}$$
The first equation involves the $(1,1)$ resonance relation, and the second involves the $(i,j)$ resonance relation. We use again Assumption \ref{ass:transp} to solve for the corresponding coefficient of $Q$ in $S^0.$

\smallskip

$\bullet$ For $|\ell| = 2,$ equation \eqref{eq:homo} reduces to
$$ \begin{aligned}
 i ( \l_{1,-1} + \o - \l_j ) (Q_{-1})_{(1,j)} & = \Pi_{1,-1} B_{-1} \Pi_j, \quad 2 \leq j \leq J,\\
 i ( \l_{j,+2} - \o - \l_{1,+1}) (Q_1)_{(j,1)} & = \Pi_{j,+2} B_1 \Pi_{1,+1},  \quad 2 \leq j \leq J.\end{aligned}$$
In both equation we find the $(j,1)$ resonant phase as a prefactor in the left-hand side and an interaction coefficient associated with $(j,1)$ in the right-hand side; both phase and interaction coefficient are translated by $\pm \o.$ For instance, the first equation can be written
 $$
 - i \big( \l_j(\xi + \o) - \l_1(\xi) - \o \big) (Q_{-1}(t,x,\xi + k))_{(1,j)} = \Pi_{1}(\xi) B(u_{0,-1}(t,x)) \Pi_j(\xi + k),$$
 which we can solve for $(Q_{-1})_{(1,j)} \in S^0,$ by Assumption \ref{ass:transp}.
\end{proof}

\begin{rema}\phantomsection \label{rem:Q} With $Q$ defined in the above Proposition, there holds $Q_{(i,j)} \in S^{-1},$ unless $\l_i$ and $\l_j$ are bounded, or $i = j.$ In both these cases, $Q_{(i,j)} \in S^0.$ Indeed, $\l_i(\xi + k) - \l_i(\xi)$ is typically bounded in $\xi,$ as shown in Lemma {\rm \ref{rem:sw}}.
\end{rema}

By estimates \eqref{action:fourier-mult} and \eqref{product2}, with the symbol $Q$ given in Proposition \ref{normal-form-prop} is associated an operator $\op_\e(Q),$ which is a uniform remainder, in particular satisfies $\| \op_\e(Q) v \|_{\e,s} \lesssim \| v \|_{\e,s},$ for all $v \in H^s.$ For $\e$ small enough, $\Id + \sqrt \e \op_\e(Q)$ is invertible for all $t.$ We consider the change of variable
 \begin{equation} \label{def:check-u}
  \check U(t) := \Big(\Id + \sqrt \e \op_\e\big(Q(\sqrt \e t)\big)\Big)^{-1}U(\sqrt \e t),
  \end{equation}
corresponding to a normal form reduction and a rescaling in time.

 \begin{coro}\phantomsection \label{cor:normal-form} The equation in $\check U$ is
 \begin{equation} \label{eq:check-u}
 \d_t \check U + \frac{1}{\sqrt \e} \op_\e(i {\mathcal A}) \check U = \op_\e(\check {\mathcal B}) \check U + \sqrt \e \check F, \qquad \check {\mathcal B} := \left(\begin{array}{cc} \chi_0(\xi) {\mathcal B}^r(\sqrt \e t,x,\xi) & 0 \\ 0 & 0 \end{array}\right)
 \end{equation}
 where ${\mathcal B}^r$ is defined in \eqref{def:Br}, and $\check F$ satisfies the same bounds as $F$ in Lemma {\rm \ref{lem:bd-F}.}
\end{coro}

\begin{proof} By definition of $\check U$ and \eqref{sys U}, there holds
 $$ \begin{aligned} \d_t \check U &  = (\Id + {\bf Q}) ^{-1} \Big( -\frac{1}{\sqrt \e} \op_\e(i {\mathcal A}) + \op_\e({\mathcal B}(\sqrt \e t)) \Big) (\Id + {\bf Q}) \check U
 \\ & \quad
 - (\Id + {\bf Q})^{-1} \Big( \e \op_\e\big( (\d_t Q)(\sqrt \e t)\big) \check U + \sqrt \e F\Big),
 \end{aligned}$$
where ${\bf Q}$ is short for $\sqrt \e \op_\e\big(Q(\sqrt \e t)\big).$
There holds
$$ \op_\e(i {\mathcal A}) (\Id + {\bf Q}) = (\Id + {\bf Q}) \op_\e(i {\mathcal A}) + \sqrt \e \big[\op_\e(i {\mathcal A}), \op_\e(Q(\sqrt \e t))\big],$$
so that
$$ (\Id + {\bf Q})^{-1}  \op_\e(i {\mathcal A}) (\Id + {\bf Q}) = \op_\e(i {\mathcal A}) + \sqrt \e (\Id + {\bf Q})^{-1}  \big[\op_\e(i {\mathcal A}), \op_\e(Q(\sqrt \e t))\big].$$
Besides,
$$ (\Id + {\bf Q}) ^{-1} \op_\e({\mathcal B}(\sqrt \e t)) (\Id + {\bf Q}) = (\Id + {\bf Q}) ^{-1}  \op_\e({\mathcal B}(\sqrt \e t)) + \sqrt \e R_0,$$
where $R_0$ denotes a uniform remainder, in the sense of Definition \ref{def:ur}.
The equation in $\check U$ thus appears as
$$ \begin{aligned} \d_t \check U & + \frac{i}{\sqrt \e} \op_\e({\mathcal A}) \check U  \\ & = (\Id + {\bf Q})^{-1} \Big( \op_\e({\mathcal B}(\sqrt \e t)) - [\op_\e(i {\mathcal A}), \op_\e(Q(\sqrt \e t))] - \e \op_\e(\d_t Q( \sqrt \e t))  \Big) \check U + \sqrt \e \check F, \end{aligned}$$
with the source
 $$\check F :=  R_0 \check U - (\Id + {\bf Q})^{-1} F.$$
 By Proposition \ref{normal-form-prop}, the leading term in the right-hand side in the equation in $\check U$ reduces to $(\Id + {\bf Q})^{-1} \op_\e(\check {\mathcal B}).$ Now, expanding the inverse of $\Id + {\bf Q}$ in Neumann series and using the fact that $\op_\e(Q)$ is a uniform remainder, we see that $(\Id + {\bf Q})^{-1} = \Id + \sqrt \e R_0,$ so that
  $$ (\Id + {\bf Q})^{-1} \op_\e(\check {\mathcal B}) = \op_\e(\check {\mathcal B}) + \sqrt \e R_0.$$
 Thus  the above equation in $\check U$ reduces to \eqref{eq:check-u}.  Since $\op_\e(Q)$ is a uniform remainder, there holds $\| \check U\|_{\e,s} \lesssim \| U \|_{\e,s} \lesssim \| \dot u \|_{\e,s}$ and $\|(\Id + {\bf Q})^{-1} F\|_{\e,s} \lesssim \| F \|_{\e,s},$ so that $\| \check F\|_{\e,s} \lesssim \| F \|_{\e,s}.$
\end{proof}

 \begin{rema}\phantomsection \label{rem:normal-form-S-bound} Note that we do not really need $Q$ bounded, only $\sqrt \e Q$ small. In this sense we could approach the resonance much closer. We will do exactly so in Appendix {\rm \ref{app:symb-bound}}, specifically in the proof of Lemma {\rm \ref{lem:away}}, where we derive bounds for the symbolic flow.
 \end{rema}

\subsubsection{Space-frequency localization}\label{s-f loc}

By the assumed polarization condition (see equation \eqref{pol-con} in Assumption \ref{ass-u-a}), there holds
$$
{\rm
tr}\,\big(\Pi_{1}(\xi+k)B(u_{01}(0,x))\Pi_{2}(\xi)B(u_{0,-1}(0,x))\Pi_{1}(\xi+k)\big) = |a(x)|^2 \G(\xi),
$$
where $a$ is the leading amplitude in the initial datum \eqref{generic-data}, and $\G$ is introduced in \eqref{def:G0}.

By continuity and decay of $a$ at spatial infinity, there exists $x_0 \in \R^d$ such that
 \begin{equation} \label{def:x0}
  0 < |a(x_0)|  =\sup_{x\in \R^d}|a(x)|.
  \end{equation}
By compactness of ${\mathcal R}_{12}$ and positivity of the stability index ${\bf \Gamma}$ (defined in \eqref{def:trace}), the function $\Re e\, \G(\xi)^{1/2}$ is not identically zero on ${\mathcal R}_{12}.$ Then, for some $\xi_0 \in {\mathcal R}_{12},$
\begin{equation} \label{def:xi0}
0 < \g = \Big| \max_{\xi\in {\mathcal R}_{12}} \Re e \, \big( \G(\xi)^{1/2} \big) \Big|  \quad \mbox{is attained at $\xi_0$.}
\ee

\begin{notation}\phantomsection \label{notation:cut-off} Given two cut-offs $\theta_1, \theta_2 \in C^\infty_c(\R^d),$ with $0 \leq \theta_1, \theta_2 \leq 1,$ we denote
$$
 \theta_1 \prec \theta_2
$$
 to indicate that $\theta_2$ is an extension of $\theta_1,$ in the sense that $(1 - \theta_2) \theta_1 \equiv 0.$ In other words: $\theta_2 \equiv 1$ on the support of $\theta_1.$
 \end{notation}

 \label{intro cut-offs} We denote $\varphi_0, \varphi, \varphi_1$ spatial cut-offs, and $\chi_0, \chi, \chi_1$ frequency cut-offs, with $\varphi_j \in C^\infty_c(\R^d_x),$ $\chi_j \in C^\infty_c(\R^d_\xi),$ such that $0 \leq \varphi_j \leq 1,$ $0 \leq \chi_j \leq 1,$ $\varphi_j \equiv 1$ on a neighborhood of $x_0,$ $\chi_j \equiv 1$ on the neighborhood ${\mathcal R}_{12}^h$ of the resonant set ${\mathcal R}_{12},$ and
   $$ \varphi_0 \prec \varphi \prec \varphi_1, \quad \chi_0 \prec \chi \prec \chi_1.$$
We will further (see in particular Propositions \ref{lem:Sob-bd} and \ref{prop:ex-Sob} and Remark \ref{rem:parameters}) choose the support of $\chi_1$ to be {\it small enough}, and the support of $\varphi_0$ to be {\it large enough}. Corresponding small parameters are $h > 0$ for the frequency truncations (as in the first paragraph of Section \ref{normal-form1}) and $\delta_{\varphi_0} > 0$ for the spatial truncations (see the proof of Proposition \ref{lem:Sob-bd} below).

We let
 \be\label{V-loc0}
V :=\op_\e\big(\chi\big)\big(\varphi \check U\big),
\ee
and
\be \label{W-loc}
 W = (W_1, W_2) := \Big( \op_\e\big( \chi\big)\big((1-\varphi)\check U\big), \, \big(1-\op_\e(\chi)\big)\check U \Big),
\ee
so that
\begin{equation} \label{re:checkU}
 \check U = V + W_1 + W_2.
 \end{equation}

\begin{lemm}\phantomsection \label{lem:sys-V} The system in $(V,W)$ is
\begin{equation}\label{V11}
\left\{ \begin{aligned}
         \d_t V + \frac{1}{\sqrt\e} \ope(M) V & = \sqrt \e F_V, \\
       \d_t W + \frac{1}{\sqrt\e} \op_\e\big(i {\bf A}) W & = \op_\e(D) W + \sqrt \e F_W,
       \end{aligned}\right.
       \ee
with symbols
 $$ \label{bfA} M := i \chi_1 {\mathcal A} - \sqrt \e \varphi_1 \check {\mathcal B}, \qquad {\bf A} := \left(\begin{array}{cc} {\mathcal A} & 0 \\ 0 & {\mathcal A} \end{array}\right), \qquad D := \left(\begin{array}{cc} (1-\varphi_0) \chi_1 \check {\mathcal B} & 0 \\ 0 & 0\end{array}\right),$$
 and source terms $F_V, F_W$ satisfying the same bound as $F$ in Lemma {\rm \ref{lem:bd-F}.}
\end{lemm}

\begin{proof}
 There holds
$$ \d_t V = - \frac{i}{\sqrt \e} \op_\e(\chi) \big( \varphi \op_\e({\mathcal A}) \check U \big) + \op_\e(\chi) \big( \varphi \op_\e(\check {\mathcal B}) \check U \big) + \sqrt \e \op_\e(\chi) \big( \varphi \check F\big).$$
Using the identity $\op_\e(\chi) \equiv \op_\e(\chi_1) \op_\e(\chi),$ we compute
$$ \begin{aligned} \op_\e(\chi) \big( \varphi \op_\e({\mathcal A}) \check U \big) &  = \op_\e(\chi) [\varphi, \op_\e({\mathcal A})] \check U + [\op_\e(\chi), \op_\e({\mathcal A})] ( \varphi \check U) + \op_\e( \chi_1 {\mathcal A}) V.
\end{aligned}$$
Similarly, using the identity $\varphi_1 \varphi \equiv \varphi,$
$$ \begin{aligned} \op_\e(\chi)  \big( \varphi \op_\e(\check {\mathcal B}) \check U \big)
& =  \op_\e(\chi) \big( \varphi_1 [\varphi, \op_\e(\check {\mathcal B})] \check U) + [\op_\e(\chi), \varphi_1] \op_\e(\check {\mathcal B}) \big( \varphi \check U) \\ & \quad +  \varphi_1 [\op_\e(\chi), \op_\e(\check {\mathcal B})] (\varphi \check U) + \varphi_1 \op_\e(\check {\mathcal B}) V.\end{aligned}$$
Commutators being $O(\e)$ (in the sense of Proposition \ref{prop:composition}), the leading order term in the above right-hand side is the fourth term $\varphi_1 \op_\e(\check {\mathcal B}).$

We now use the tensor product structure of every entry of $\check {\mathcal B}$ in order to express $\ope(\varphi_1 \check {\mathcal B})$ as a {\it para}-differential operator. Going back to the definition of $\check {\mathcal B}$ in \eqref{eq:check-u} and ${\mathcal B}^r$ in \eqref{def:Br}, we denote $${\mathcal B}^r_{12} = g(t,x) \Pi_{1,+1} B(\vec e_1) \Pi_2$$ the top right entry of ${\mathcal B}^r,$ and similarly $\check {\mathcal B}_{12}$ the corresponding entry in $\check {\mathcal B}:$
 $$ \varphi_1 \op_\e(\check {\mathcal B}_{12}) = \varphi_1 (x) g(\sqrt\e t, x) \op_\e(\chi_0 \Pi_{1,+1} B(\vec e_1) \Pi_2).$$
By Remark \ref{para:tensor},
 $$ \begin{aligned} \varphi_1 (x) g(\sqrt\e t, x) & \op_\e(\chi_0 \Pi_{1,+1} B(\vec e_1) \Pi_2) - \ope\Big(\varphi_1 g(\sqrt \e t) \chi_0 \Pi_{1,+1} B(\vec e_1) \Pi_2)\Big) \\ & = \Big( \varphi_1(x) g(\sqrt \e t, x) - \ope(\varphi_1 g(\sqrt \e t, x)\Big) \op_\e(\chi_0 \Pi_{1,+1} B(\vec e_1) \Pi_2),
 \end{aligned}$$
 hence, by Proposition \ref{prop:remainder},
 $$ \big\| \Big( \varphi_1 \op_\e(\check {\mathcal B}_{12}) - \ope(\varphi_1 \check {\mathcal B}_{12})\Big) V \big\|_{\e,s} \lesssim \e \| \varphi_1 g(\sqrt \e t) \|_{\e,s} | \tilde V |_{L^\infty},$$
 with $\tilde V :=  \op_\e(\chi_0 \Pi_{1,+1} B(\vec e_1) \Pi_2) V.$ By \eqref{est:bernstein}, $|\tilde V|_{L^\infty} \lesssim |V|_{L^2},$ since $\chi_0$ is smooth and compactly supported.

The same is true of course for the other entry of $\check {\mathcal B},$ and we arrive at
 \be \label{est:remainder:in:proof} \big\| \Big( \varphi_1 \op_\e(\check {\mathcal B}) - \ope(\varphi_1 \check {\mathcal B})\Big) V \big\|_{\e,s} \lesssim \e \| V |_{L^2}.
 \ee

 Gathering the above results, and using $\op_\e(\chi_1 {\mathcal A}) \equiv \ope(\chi_1 {\mathcal A})$ (Remark \ref{para:tensor}), we obtain \eqref{V11}(i), with the source term
$$ \begin{aligned} F_V & := \sqrt \e \op_\e(\chi) \big( \varphi \check F\big) - \frac{i}{\sqrt \e} \Big(\op_\e(\chi) [\varphi, \op_\e({\mathcal A})] \check U +[\op_\e(\chi), \op_\e({\mathcal A})] \big( \varphi \check U) \Big) \\ & + \op_\e(\chi) \big( \varphi_1 [\varphi, \op_\e(\check {\mathcal B})] \check U) + [\op_\e(\chi), \varphi_1] \op_\e(\check {\mathcal B}) \big( \varphi \check U) +  \varphi_1 [\op_\e(\chi), \op_\e(\check {\mathcal B})] (\varphi \check U) \\ & + \e^{-1/2} \Big( \varphi_1 \op_\e(\check {\mathcal B}) - \ope(\varphi_1 \check {\mathcal B})\Big) V.
 \end{aligned}$$
The fact that $F_V$ satisfies the same bound as $\check F$ and $F$ follows from \eqref{est:remainder:in:proof} and the elementary results of Appendix \ref{sec:foumult}.

The equation in $W_1$ is derived in the same way, the only difference being the use of $(1 - \varphi_0)(1 - \varphi) \equiv (1 - \varphi)$ in place of $\varphi_1 \varphi \equiv \varphi.$

Finally, the equation in $W_2$ involves the source term $\op_\e(1 - \chi) \op_\e(\check {\mathcal B}) \check U.$ The symbol $(1 - \chi) \check {\mathcal B}$ vanishes identically, by $(1 - \chi) \chi_0 = 0$ and definition of $\check {\mathcal B}$ in \eqref{eq:check-u}. Hence, by estimate \eqref{est:fourier-mult},
 $\op_\e(1 - \chi) \op_\e(\check {\mathcal B}) = \e R_0,$
 where $R_0$ is a uniform remainder (in the sense of Definition \ref{def:ur}  page \pageref{uniform remainder}).
\end{proof}
System \eqref{V11} is the {\it prepared system}, in which
\begin{itemize}
 \item the symbol $M$ is the key term; it involves the diagonal hyperbolic operator ${\mathcal A}$ in a neighborhood of the $(1,2)$ resonance, and the interaction coefficients associated with $(1,2)$ via ${\mathcal B}^r,$
 \item the source term in the right-hand side of the equation in $W_1$ will be made small, by choice of a spatial cut-off $\varphi_0$ with a {\it large} support, exploiting decay at infinity of the leading profile of the WKB solution;
 \item the equation in $W_2$ is non-singular, a consequence of the normal form reduction of Section \ref{normal-form1}.
 \end{itemize}

\subsection{Duhamel representation} \label{est-upper-V0}

In this Section we use Theorem \ref{th:duh} from Appendix \ref{app:duh} and write an integral representation formula for the variable $V$ introduced in \eqref{V-loc0}. From this representation we derive an upper bound for $\|V\|_{\e,s}.$

The symbol $M$ of the propagator in the equation \eqref{V11}(i) in $V$ is
\begin{equation} \label{M-spelled-out}
 M(\e,t,x,\xi) = i \chi_1 {\mathcal A} - \sqrt \e \varphi_1 \check {\mathcal B} = \bp i \chi_1 \mu_1 &-\sqrt\e \, \tilde b_{12} &0&\cdots&0\\-\sqrt\e \, \tilde b_{21}&i \chi_1 \mu_2&0&\cdots&0\\0&0&i \chi_1 \l_3&\cdots&0\\\vdots & \vdots & \vdots & \ddots & \vdots \\0&0&0&\cdots&i \chi_1 \l_J\ep
\ee
 where
 \begin{itemize}
 \item the cut-offs functions $\varphi_1,$ $\chi_1$ were introduced just below Notation \ref{notation:cut-off} on page \pageref{intro cut-offs},

\smallskip

 \item the shifted eigenvalues $\mu_1 = \l_1(\cdot + k) - \o$ and $\mu_2 = \l_2$ were introduced in \eqref{def:mu},

\smallskip

 \item the $N \times N$ extra-diagonal blocks are
  \be \label{def:b12-21}
  \tilde b_{12}  := \chi_0(\xi) \varphi_1(x) g(\sqrt \e t, x) b^+_{12}(\xi), \quad \tilde b_{21}  = \chi_0(\xi_0) \varphi_1(x) g(\sqrt \e t ,x)^* b^-_{21}(\xi),
  \ee
  where
  \be \label{def:int-coeff}
 b_{12}^+(\xi) := \Pi_1(\xi + k) B(\vec e_1) \Pi_2(\xi), \qquad b_{21}^-(\xi) := \Pi_2(\xi) B(\vec e_{-1}) \Pi_1(\xi +k),
 \ee
 are the interaction coefficients associated with resonance $(1,2),$ in the sense of Definition \ref{def:transp}.
\end{itemize}

 By regularity of the eigenprojectors (Assumption \ref{ass:spectral}) and the approximate solution (Assumption \ref{ass-u-a}), Assumption \ref{ass:B} is satisfied, where $x_\star$ is the maximum of all $|x|$ with $x$ in the support of $\varphi_1.$

The symbolic flow $S_0$ of $M$ is defined as the solution to the initial-value problem
\begin{equation}\label{S} \d_t S_0+\frac{1}{\sqrt\e}M S_0=0, \qquad S_0(\t,\t)=\Id.
                          \end{equation}

\begin{prop}\phantomsection\label{est-flow-sym} For all $T> 0,$ all $0 \leq \t \leq t \leq T |\ln \e|,$ $\a \in \N^d,$ there holds
$$
|\d_x^\a S_0(\t,t)| \lesssim |\ln \e|^* \exp\big((t-  \t) \g^+ \big),
$$
where
\begin{equation} \label{def:g+}
 \g^+  : = |a|_{L^\infty} \, \big| \max_{\xi\in {\mathcal R}_{12}^h} \Re e \, \big(\G(\xi)^{1/2}\big)\big|,
 \end{equation}
 and $|\ln \e|^*$ denotes $|\ln \e|^{N^*}$ for some large constant $N^* > 0$ depending on all parameters, but not on $\e$ nor on $\t,t.$
\end{prop}

Note that $\g^+$ is $1$-homogeneous in $a,$ while $\g,$ defined in \eqref{def:g}, does not depend on $a.$

\begin{proof} The proof is postponed to Appendix \ref{app:symb-bound}. It uses elementary linear algebra, rendered non-trivial by the fact that the resonant frequencies are asymptotically close to crossing points in the spectrum of $M,$ and a non-stationary phase argument analogous (and complementary) to the normal form reduction of Section \ref{normal-form1}.
\end{proof}

Proposition \ref{est-flow-sym} verifies that the flow of $M$ defined in \eqref{M-spelled-out} satisfies Assumption \ref{ass:BS}.

Thus, Theorem \ref{th:duh} from Appendix \ref{app:duh} page \pageref{th:duh} applies, and the unique solution to \eqref{V11}(i) satisfies the representation
\begin{equation} \label{rep:V}
 V = \ope(S(0;t)) V(0) + \sqrt \e \int_0^t \ope(S(t';t)) \tilde F_V (t')\,dt',
\end{equation}
where $S(t';t) := \sum_{0\leq q \leq q_0} S_q,$ the leading term $S_0$ being the symbolic flow \eqref{S}, and the correctors $S_q,$ for $1 \leq q \leq q_0,$ being defined in \eqref{resolventk}. The order $q_0$ of the expansion is a function of $\Gamma$ and $T_0,$ as seen on equation \eqref{def:zeta} page \pageref{def:zeta}. The source term $\tilde F_V$ can be expressed in terms of $F_V$ and the datum $V(0),$ as in \eqref{buff0.01}. The bound \eqref{bd-R121} implies
 \be \label{bd:tildeFV} \| \tilde F_V \|_{\e,s} \leq \| F_V \|_{\e,s} + \| V(0) \|_{\e,s}.\ee
 Proposition \ref{est-flow-sym} and Lemma \ref{lem:bd-actionS} imply that $\ope(S(t';t))$ satisfies the bound
 \be \label{action:S} \| \ope(S(t';t)) v \|_{\e,s} \lesssim |\ln \e|^* e^{( t - t') \g^+} \| v \|_{\e,s}, \qquad v \in H^s.\ee
 From there, we deduce the bound, for $s \geq 0:$
 $$  \| V(t) \|_{\e,s} \lesssim e^{t \g^+} |\ln \e|^*  \| V(0) \|_{\e,s} + \e^{1/2} |\ln \e|^* \int_0^t e^{(t - t') \g^+} \| F_V(t') \|_{\e,s} \, dt'.$$
 According to Lemmas \ref{lem:bd-F} and \ref{lem:sys-V}, there holds
 \be \label{bd:FV} \| F_V\|_{\e,s} \lesssim \big(1 + \e^{-1/2} |\dot u (\e^{1/2} t) |_{L^\infty}\big) \|\dot u (\sqrt \e t)\|_{\e,s} + \e^{K_a}.
 \ee
Going up the chain of changes of variables \eqref{V-loc0}-\eqref{W-loc}, \eqref{def:check-u}, \eqref{dotu-U}, we see that
\be \label{u-vw} \|\dot u(\sqrt \e t) \|_{\e,s} \lesssim \| (V, W)(t) \|_{\e,s}.
\ee
 Since by assumption $K \leq K_a + 1/2,$ we conclude that
\begin{equation} \label{up:0}
 \begin{aligned} \| V(t) & \|_{\e,s}  \lesssim \e^K |\ln \e|^* e^{t \g^+}
 \\ & \qquad +
   \e^{1/2} |\ln \e|^* \int_0^t e^{(t - t') \g^+} (1 + \e^{-1/2} |\dot u(\e^{1/2} t') |_{L^\infty}\big) \| (V, W)(t') \|_{\e,s}  \, dt'.
 \end{aligned}
  \ee

\subsection{Lower bound} \label{sec:lower}

We now choose the datum
\begin{equation} \label{datum-dot-u}
 \dot u(0,x)  := \e^{K} e^{i x \cdot(\xi_0 + k)/\e} \varphi_{0}(x) \vec e_0,
 \ee
 where $\varphi_{0}$ is the spatial truncation introduced just below Notation \ref{notation:cut-off} on page \pageref{notation:cut-off}, and the fixed vector $\vec e_0$ satisfies
 \be \label{def:e0} \vec e_0 =\Pi_1(\xi_0 + k) B(\vec e_1) \Pi_2(\xi_0) B(\vec e_{-1}) \Pi_1(\xi_0 + k) \vec e_0, \quad |\vec e_0| = 1.\ee
The matrix $\Pi_1(\xi_0 + k) B(\vec e_1) \Pi_2(\xi_0) B(\vec e_{-1}) \Pi_1(\xi_0 + k)$ has rank one by assumption (rank at most one by Assumption \ref{ass:transp}(iii), and at least one by $\G \neq 0$), so that by \eqref{def:e0} the vector $\vec e_0$ is defined as the unitary generator of its image.
\begin{lemm}\phantomsection \label{lem:datum} With the choice \eqref{datum-dot-u}, the datum for $V$ is
$$ V(0) = \e^{K}  e^{i x \cdot \xi_0/\e} V_0 + \e^{K+1/2} \tilde V^\e_0, \quad V_0 := \varphi_{0}(x) \big( \vec e_0, 0, \dots, 0 \big), \quad \sup_{0 < \e < \e_0} \| \tilde V^\e_0\|_{\e,s} < \infty.$$
\end{lemm}

\begin{proof} We denote in this proof $f_0^\e$ any family of $H^s$ maps such that
 $$
  \sup_{0 < \e < \e_0} \| f_0^\e\|_{\e,s} < \infty.
 $$
 With this notation, given a Fourier multiplier $P \in S^0$ and $f \in C^\infty_c,$ there holds
 \be \label{P-f}
  \op_\e(P) f = P(0) f + \e f_0^\e.
 \ee
 With the choice \eqref{datum-dot-u}, there holds
 $$ \begin{aligned} U_1(0) = \e^K e^{- i k \cdot x/\e} \op_\e(\Pi_1) \big( e^{i x \cdot (\xi_0 + k)/\e} \varphi_{0} \vec e_0 \big) = \e^K e^{i x \cdot \xi_0/\e} \op_\e(\Pi_{1,+(\xi_0 + k)}) \varphi_{0} \vec e_0.\end{aligned}$$
 This implies, by \eqref{P-f},
 $$ U_1(0) = \e^{K} e^{i x \cdot \xi_0/\e} \varphi_{0}(x)  \Pi_1(\xi_0 + k) \vec e_0 + \e f_0^\e = \e^{K} e^{i x \cdot \xi_0/\e} \varphi_{0}(x) \vec e_0 + \e f_0^\e,$$
 the second equality by $\Pi_1(\xi_0 + k) \vec e_0 = \vec e_0.$
 Next we compute, for $j \geq 2,$ using \eqref{P-f} again,
  $$ \begin{aligned} \op_\e(\Pi_j) \big( e^{i x \cdot (\xi_0 + k)/\e} \varphi_{0} \vec e_0 \big)
  = e^{i x \cdot (\xi_0 + k)/\e} \varphi_{0}(x) \Pi_{j}(\xi_0 + k) \vec e_0 + \e f_0^\e.\end{aligned}$$
 This gives $U_j(0) = \e f_0^\e,$ since $\Pi_j(\xi_0 + k) \vec e_0 = 0$ for $j \geq 2.$ From there, we obtain
  $$ \check U_1(0) = U_1(0) - \sqrt \e R_0 U(0) =  \e^{K} e^{i x \cdot \xi_0/\e} \varphi_{0}(x) \vec e_0 + \sqrt \e f_0^\e,$$ and
  $$ V_1(0) = \e^K \op_\e(\chi) \big( \varphi e^{i x \cdot \xi_0/\e} \varphi_{0}\big) \vec e_0  + \e^{K + 1/2} f_0^\e = \e^K \op_\e(\chi_{+\xi_0}) \varphi_{0} \vec e_0,$$
  since $\varphi \varphi_{0} \equiv \varphi_{0},$ and then with \eqref{P-f},
  $$ V_1(0) = \e^K e^{i x \cdot \xi_0/\e} \chi(\xi_0) \varphi_{0}(x) \vec e_0 + \e^{K + 1/2} f_0^\e.$$
 Since $\chi(\xi_0) = 1,$ we obtained the first component of $V(0).$ We conclude with
 $$ V_j(0) = \op_\e(\chi) \big( \varphi (\e f_0^\e - \sqrt \e R_0 U(0))\big) = \sqrt \e f_0^\e,
 \quad j\geq 2.$$
\end{proof}

\begin{lemm}\phantomsection\label{lem-S0}
 For the datum $V(0)$ described in the above Lemma, there holds for small enough $\rho > 0,$ for some $C(\rho) > 0:$
 \begin{equation} \label{low0}
\big|\ope(S(0;t)) V(0) \big|_{L^2(B(x_0,\rho))} \geq  C(\rho) \e^{K} \Big(e^{t \g^-} - \e^{1/2} |\ln \e|^{*} e^{t \g^+}\Big),
\end{equation}
where $\dsp{\g^- := \g \min_{|x - x_0| \leq \rho} |a(x)|,}$
with $\g$ as in \eqref{def:g}.
 \end{lemm}

We recall that notation $|\ln \e|^*,$ introduced in the statement of Proposition \ref{est-flow-sym}, denotes $|\ln \e|^{N^*},$ for some $N^* > 0$ independent of $\e, t.$

Note that the lower rate of growth $\g^-$ in \eqref{low0} is $1$-homogeneous in $a,$ just like $\g^+$ \eqref{def:g+} and unlike coefficient $\g$ \eqref{def:g}.

\begin{proof}[Proof of Lemma \ref{lem-S0}] By Lemma \ref{lem:datum}, the datum $V(0)$ decomposes as a leading term and a remainder $\e^{K + 1/2} \tilde V_0^\e.$  By \eqref{action:S}, there holds
 $$ \| \ope(S(0;t)) (\e^{K + 1/2} \tilde V_0^\e) \|_{L^2} \lesssim \e^{K + 1/2} |\ln \e|^* e^{t \g^+}.$$
 We turn to the action of $\ope(S)$ on the leading term $e^{i x \cdot \xi_0/\e} V_0$ in $V(0).$
 By Remark \ref{rem:para},
 $$ \ope(S(0;t)) \big( e^{i x \cdot\xi_0/\e}  V_0\big) = e^{i x \cdot \xi_0/\e} \int e^{i x \cdot \xi} {\bf S}(0;t,x,\xi_0 + \e \xi)  \, \hat V_0(\xi)  \, d\xi,
 $$
 where
 $ {\bf S}(0;t,x,\xi) := \displaystyle{\left({\mathcal F}^{-1} \psi \star \tilde S(0;t) \right)\left(\frac{x}{\e}, \xi\right)},$ with $\tilde S(0;t,x,\xi) = S(0;t,\e x,\xi).$
 This gives
$$
  \ope(S(0;t)) \big( e^{i x \cdot\xi_0/\e}  V_0\big)  =  e^{i x \cdot\xi_0/\e}  S(0;t,x,\xi_0) V_0(x) + \e \tilde V_0,
$$
where the remainder $\tilde V_0$ is the sum $\tilde V_0 = \tilde V_{01} + \tilde V_{02}:$
 $$ \begin{aligned} \tilde V_{01} & :=
  \sum_{|\a| = 1} \int e^{i x \cdot \xi} \left( \int_0^1 \left(\d^\a_\xi {\bf S}\right)(0;t,x,\xi_0 + \e \t \xi) \, d\t \right) \widehat{ \d_x^\a V_0}(\xi) \, d\xi, \\ \tilde V_{02} & :=  \sum_{|\a| = 1} \int_{\R^d} {\mathcal F}^{-1}(\d_\eta^\a \psi)(y, \xi_0) \left(\int_0^1 \d^\a_x S(0;t,x - \e \t y, \xi_0) \, d\t \right) \, dy \, V_0(x).
 \end{aligned}$$

There holds
$$ | \tilde V_{01} |_{L^2(B(x_0,\rho))} \leq \rho^{d/2} \sup_{\xi \in \R^d} |\d_\xi {\bf S}(\cdot,\xi)|_{L^\infty(B(x_0,\rho))} \| V_0 \|_{H^{1 + d/2^+}}.$$
By Remark \ref{rem:cut-off} and Lemma \ref{lem:bd-S},
 $$ |\d_\xi {\bf S}|_{L^\infty} \lesssim |S|_{L^\infty} + |\d_\xi S|_{L^\infty} \lesssim \e^{-1/2} |\ln \e|^* e^{t \g^+}.$$
By Remark \ref{rem:cut-off} and Proposition \ref{est-flow-sym},
 $$ |\tilde V_{02} |_{L^2(B(x_0,\rho))} \lesssim \sup_{x \in \R^d} |\d_x S(0;t,\cdot,\xi_0)| |V_0|_{L^2(B(x_0,\rho))} \lesssim |\ln \e|^* e^{t \g^+}.$$
It remains to bound from below the function $S(0;t,x,\xi_0) V_0$ on $B(x_0,\rho).$ Note that here $\xi$ is frozen at $\xi_0,$ so that the regularity issues of Appendix \ref{app:symb-bound} do not come into play.

 The symbolic flow $S$ is defined just above Lemma \ref{lem:duh-remainder} on page \pageref{lem:duh-remainder} as
  $$ S(0;t) = S_0 + \e^{1/2}\big(S_1 + \dots + \e^{q_0 - 1/2} S_{q_0}\big).$$
  Lemma \ref{lem:bd-S} implies the uniform bound
   $$ \e^{1/2} \big| S_1 + \dots + \e^{q_0 - 1/2} S_{q_0}\big| \lesssim \e^{1/2} |\ln \e|^* e^{t \g^+}.$$
 According to Section \ref{sec:app pert}, and especially \eqref{rep:S0} on page \pageref{rep:S0}, the leading term $S_0$ decomposes as
 $$ S_0(0;t,x,\xi_0) = \exp\big(t M(0,x,\xi_0)/\sqrt \e\big) + \e^{1/2} \Sigma, \qquad \mbox{ where $|\Sigma| \lesssim e^{t \g^+}.$}$$

We are left with the matrix exponential $\exp\big(t M(0,x,\xi_0)/\sqrt \e\big),$ where $M$ is given explicitly in \eqref{M-spelled-out}. At $\xi = \xi_0,$ there holds $\mu_1 = \mu_2,$ so that
 $$ \exp\big(t M(0,x,\xi_0)/\sqrt \e\big) = \mbox{diag}\,\Big(e^{i t \l_2(\xi_0)/\sqrt \e} \exp \big( t \tilde M(x)\big), \, \big(e^{i t \l_j(\xi_0)/\sqrt\e}\big)_{3 \leq j \leq J}\Big).$$
The matrix $\tilde M$ is
 $$\dsp{\tilde M(x) := \left(\begin{array}{cc} 0 & \tilde b_{12}(0,x,\xi_0) \\ \tilde b_{21}(0,x,\xi_0) & 0 \end{array}\right),}$$
 where $\tilde b_{12}$ and $\tilde b_{21}$ are defined in \eqref{def:b12-21}.
It has rank two, by Assumption \ref{ass:transp}(iii), and spectrum
 $$ \Big\{ 0, \, \, \pm \mbox{tr}\,(\tilde b_{12}(0,x,\xi_0) \tilde b_{21}(0,x,\xi_0)^{1/2}\Big\}  = \big\{ 0, \, \pm |a(x)| \big( \g + i \a\big)\big\},$$
 where $\a := \Im m \, \big(\G(\xi_0)^{1/2}\big).$
(For a detailed computation, see the paragraphs just above the statement of Lemma \ref{lem:S-resonance} in Appendix \ref{app:symb-bound} on page \pageref{eig-m}.) By definition of $x_0$ and $\xi_0$ in Section \ref{s-f loc} on page \pageref{def:x0}, there holds $\g \neq 0,$ and $|a(x)| \neq 0$ locally around $x_0.$ As a consequence, locally around $x_0$ the matrix $\tilde M$ has a smooth spectral decomposition
$$\dsp{
\tilde M(x) = a(x) (\g + i \a)\big( P_{+}(x) - P_{-}(x)\big),}$$
with rank-one projectors.
The ranges of the eigenprojectors (eigenspaces of $\tilde M$) are
$$\mbox{Ran}\, P_\pm(x) = \Big\{ c \Big(\vec e_0, \,\, \pm \frac{\tilde b_{21}(0,x,\xi_0)}{|a(x)|\g} \vec e_0\Big), \quad c \in \C\Big\}.$$
 In particular,
 \be \label{p:pm} P_\pm(x) \big(\vec e_0, 0\big) \equiv \Big(1 + \big(\tilde b_{21}(0,x,\xi_0) \g^{-1} |a(x)|^{-1}\big)^2\Big)^{-1/2} (\vec e_0, 0).
 \ee
 This gives for $x$ close to $x_0:$
  $$ \big| \exp\big(t \tilde M(x)\big)  V_0(x) \big| \geq C(x) e^{t a(x) \g} - \big| P_-(x) (\vec e_0,0) \big|,$$
 where $C(x) \neq 0$ in a neighborhood of $x_0,$ and the result follows by
 \be \label{for:beta}\left(\int_{B(x_0,\rho)} e^{2 t a(x) \g} \, dx\right)^{1/2} \geq C(\rho) e^{t \g^-},\ee
 with $C(\rho) = O(\rho^d).$
\end{proof}

\begin{rema}\phantomsection \label{rem:for-endgame} In \eqref{p:pm} we see that $P_\pm V_0 = (\star,0).$ For $\ope(S(0;t)) V,$ which, at a given $(t,x),$ is a vector in $\C^{NJ}:$
 $$ \ope(S(0;t) V(0) = \Big( (\ope(S(0;t) V(0))_1, \, \ope(S(0;t) V(0))_2, \dots\Big) \in \C^{N \times N \times \dots},$$
 this implies that the leading term is $(\ope(S(0;t) V(0))_1.$ This observation will be useful at the end of Section {\rm \ref{sec:end-insta}}.
\end{rema}

\subsection{Existence over logarithmic times and upper bound}\label{est-upper-V0-a}

We denote
\be \label{datum:VW}
 (V(0), W(0)), \quad \mbox{with $V(0)$ as described by Lemma \ref{lem:datum},  and $\| W(0) \|_{\e,s} = O(\e^K)$},
\ee
the datum derived from $\dot u(0)$ \eqref{datum-dot-u} in the coordinates $(V,W)$ of the prepared system \eqref{V11}.

We prove here existence and uniqueness of a solution $(V,W)$ to \eqref{V11} issued from \eqref{datum:VW}, over the interval $[0, T_0 \sqrt \e |\ln \e|),$ where the limiting time $T_0$ is defined in \eqref{def:T0-K0}.

The difficulty is in the treatment of the $L^\infty$ norm.

In a first part (Section \ref{sec:FL1}), we perform estimates in ${\mathcal F} L^1$ norm on the prepared system \eqref{V11} and combine these with Sobolev estimates on $W$ from \eqref{V11} and the upper bound \eqref{up:0} for $V$ that we derived from the Duhamel representation \eqref{rep:V}; here we use the bound  $|u|_{L^\infty} \leq |\hat u|_{L^1}.$

In a second part (Section \ref{sec:Sob}), we only use Sobolev estimates on $W$ from \eqref{V11} and the upper bound \eqref{up:0}; there we use the Sobolev embedding $|u|_{L^\infty} \lesssim \e^{-d/2} \| u \|_{\e,s},$ for $s > d/2.$

\subsubsection{In ${\mathcal F}L^1$ and $H^s$} \label{sec:FL1}

An observation time $T_1$ is given, such that
\be \label{def:delta} T_1 < \frac{K}{|B|_0 |\hat a|_{L^1}}, \ee
where notation $|B|_0$ is introduced in \eqref{def:B0}.

\begin{lemm}\phantomsection \label{lem:Linfty-bd} If $\varphi_0 \equiv 1$ on a large enough ball around $x_0,$ and if $\e$ is small enough, then the initial value problem \eqref{V11}-\eqref{datum:VW} is well-posed in ${\mathcal F} L^1$ over the interval $[0, T_1 |\ln \e|],$ and there holds the bound
 \be \label{ap-infty}
  \sup_{0 \leq t \leq T_1 |\ln \e|} |{\mathcal F} (V,W)(t)|_{L^1} \leq \e^{\eta_1},
 \ee
for some $\eta_1 = \eta_1(\e,T_1) > 0,$ with $\eta_1 \to 0$ as $T_1 \to K/(|B|_0 |\hat a|_{L^1})$ and $\e \to 0.$
\end{lemm}

Above, ${\mathcal F}L^1$ is the Banach algebra of maps $u$ with Fourier transform $\hat u$ in $L^1.$

\begin{proof} While a para-differential formulation of \eqref{V11} was useful for the Duhamel representation of Section \ref{est-upper-V0}, we return here to a purely {\it pseudo}-differential formulation of \eqref{V11}. This simply means changing $\ope(M)$ into $\op_\e(M)$ in the left-hand side, an operation that takes one term out of the source $F_V$ (namely, the term in the third line of the definition of $F_V$ in the proof of Lemma \ref{lem:sys-V}).

Being symmetric hyperbolic and semilinear, the initial-value problem \eqref{eq:dot-u}-\eqref{datum-dot-u} is locally well-posed in ${\mathcal F}L^1.$ Since the prepared system \eqref{V11} derives from \eqref{eq:dot-u} via Fourier multipliers, and since Fourier multipliers operate in ${\mathcal F}L^1$ (as evidenced by \eqref{action:fouriermultFL1}), the initial-value problem \eqref{V11}-\eqref{datum:VW} is also locally well-posed in time.

From \eqref{datum-dot-u}, we infer, via \eqref{action:fouriermultFL1}, that there holds $| (V,W)(0)|_{{\mathcal F}L^1} \lesssim \e^K,$ on top of the bounds given in Lemma \ref{lem:datum} and \eqref{datum:VW}.

The local-in-time existence theory (based on the Cauchy-Lipschitz theorem) gives a notion of maximal existence time, which we denote $T_\star(\e).$ The map $t \to |{\mathcal F} (V,W)(t)|_{L^1}$ is continuous over $[0, T_\star(\e)).$
Consider the set
 $$ J := \Big\{ t \in (0, T_\star(\e)) \cap (0, T_1 \sqrt \e |\ln \e|], \,\,\,\, \forall\, t' \in (0,t),\,\,\,\, |{\mathcal F} (V,W)(t')|_{L^1} \leq \e^{\eta_1} \Big\},$$
 where $0 < \eta_1 < K$ will be appropriately chosen below, depending on $T_1.$

We are going to prove that, for $\e$ small enough, $J$ is non-empty, open and closed in $(0, T_1 \sqrt \e |\ln \e|].$ This will prove well-posedness over $[0, T_1 \sqrt \e |\ln \e|],$ by connectedness.

The fact that $J$ is not empty is a direct consequence of $|{\mathcal F} (V,W)(0)|_{L^1} = O(\e^K) \ll \e^{\eta_1}$ and continuity of $t \to |{\mathcal F} (V,W)(t)|_{L^1}.$ The fact that $J$ is closed follows immediately from its definition.

Now given $t \in J,$ there certainly holds $(t - \zeta, t] \subset
J$ for some $\zeta > 0.$ Therefore we only have to prove that $[t, t
+ \zeta) \cap (0, T_1 \sqrt \e |\ln \e|] \subset J$ for some $\zeta
> 0.$

After applying the Fourier transform to both equations in \eqref{V11} and factorizing the oscillations, we find, for a given $t \in J:$
$$ \hat V(t) = e^{- i t (\chi_1 {\mathcal A})(\e \xi)/\sqrt \e} \hat V(0) + \int_0^t e^{-i( t- t') (\chi_1 {\mathcal A})(\e \d_x)/\sqrt \e} \Big( {\mathcal F}\big( \op_\e(\varphi_1 \check {\mathcal B}) V\big) + \sqrt \e \hat F_V\Big) \, dt'.$$
and
$$ \hat W(t) = e^{ - i t {\bf A}(\e \xi)/\e} \hat W(0) + \int_0^t e^{- i (t - t') {\bf A}(\e \xi)/\e} \Big( {\mathcal F} \big( \op_\e(D) W \big) + \hat F_W\Big) \, dt'.$$
The symbols ${\mathcal A}$ and ${\bf A}$ are diagonal and real, so that $|e^{ it {\mathcal A}(\xi)}| \leq 1,$ $|e^{it {\bf A}(\xi)}| \leq 1.$
Besides, by Young's convolution inequality, and recalling that the norm in use in $\C^N$ is the sup norm,
$$
\Big| \op_\e(\varphi_1 \check {\mathcal B}) V\Big|_{{\mathcal F}L^1} \leq  |B|_0 | \hat \varphi_1|_{L^1} |\hat u_0(\sqrt \e t)|_{L^1} |\hat V|_{L^1},
$$
where $|B|_0$ is defined in \eqref{def:B0}, and similarly
$$
\big| \op_\e(D) W  \big|_{{\mathcal F}L^1} \leq |(1 - \varphi_0)
u_0(\sqrt \e t)|_{{\mathcal F} L^1} |B| |\hat W|_{L^1}.
$$
There holds over $[0, T_1 |\ln \e|]:$
$$ |\hat u_0(\sqrt \e t)|_{L^1} \leq |\hat a|_{L^1} + C_0 \sqrt \e |\ln\e|, \qquad |(1 - \varphi_0) u_0(\sqrt \e t)|_{{\mathcal F} L^1} \leq |(1 - \varphi_0) a|_{{\mathcal F} L^1} + C_0 \sqrt \e |\ln \e|,$$
where $C_0 > 0$ is independent of $\e,t,$ and depends on $\d_t u_0,$ which according to Assumption \ref{ass-u-a} belongs to $C^0 H^{s_a},$ hence to ${\mathcal F} L^1.$ As $\varphi_0 \to 1$ (the function identically equal to 1), there holds $|\hat \varphi_1|_{L^1} \to 1$ and $|(1 - \varphi_0) a|_{{\mathcal F}L^1} \to 0.$ In particular, \label{def:delta:phi0} for any $\delta_{\varphi_0} > 0,$ we can choose $\varphi_0, \varphi_1$ such that
$$ |\hat \varphi_1|_{L^1} \leq 1 + \delta_{\varphi_0}, \qquad |(1 - \varphi_0) a|_{{\mathcal F}L^1}  \leq \delta_{\varphi_0}.$$
 Thus we obtain
$$ \begin{aligned} |\hat V(t)|_{L^1} & \lesssim \e^K + |B|_0 (1 + \delta_{\varphi_0})(|\hat a|_{L^1} + C_0 \sqrt \e |\ln \e|)   \int_0^t |\hat V(t')|_{L^1} \, dt' + \e^{1/2} \int_0^t |\hat F_V|_{L^1} \, dt', \\
 |\hat W(t)|_{L^1} & \lesssim \e^K + |B| (\delta_{\varphi_0} + C_0 \sqrt \e |\ln \e|) \int_0^t |\hat W(t')|_{L^1} \, dt' + \e^{1/2} \int_0^t |\hat F_W|_{L^1} \, dt'. \end{aligned} $$
With the ${\mathcal F}L^1$ bound for $(F_V,F_W)$ derived from Lemma \ref{lem:bd-F} and $|\dot u|_{{\mathcal F} L^1} \lesssim |V,W|_{{\mathcal F}L^1},$ a consequence of \eqref{action:fouriermultFL1}, this yields, using $t \in J$ and $K_a + 1/2 \geq K:$
$$  |(\hat V, \hat W)(t)|_{L^1} \lesssim \e^K |\ln \e|^* + \Big(|B|_0 (1 + \delta_{\varphi_0})(|\hat a|_{L^1} + C_0 \sqrt \e |\ln \e|) + \e^{\eta_1}\Big) \int_0^t |(\hat V, \hat W)(t')|_{L^1} \, dt',$$
for $\delta_{\varphi_0}$ small enough (depending on $B$ and $|\hat a|_{L^1}$).
   We now let
  \be \label{def:eta1}
   2 \eta_1 :=  K - |B|_0 |\hat a|_{L^1} T_1.
  \ee
Then, for $\e$ small enough, and $\delta_{\varphi_0}$ small enough, depending in particular on $T_1,$ there holds
$$ \frac{3}{2} \eta_1 < K -  |B|_0 |\hat a|_{L^1} T_1 - \Big( \e^{\eta_1} + C_0 \sqrt \e |\ln \e| |B|_0 (1 + \delta_{\varphi_0}) + \delta_{\varphi_0} |B|_0 |\hat a|_{L^1} \Big) T_1.$$
   With the above bound in $|(\hat V, \hat W)|_{L^1}$ and Gronwall's lemma,  this implies, for $\e$ small enough,
  \be \label{infty-bd}
   | {\mathcal F} (V,W)(t)|_{L^1} \leq \e^{(3/2) \eta_1}, \quad t \in J.
   \ee
 Then, by continuity of $t \to |{\mathcal F} (V,W)(t)|_{L^1},$ we obtain that $|{\mathcal F}(V,W)(t + \zeta)|_{L^1} \leq \e^{\eta_1}$ if $\zeta$ and $\e$ are small enough. This concludes the verification that $J$ is open in $(0, T_1 \sqrt \e |\ln \e|).$ The bound \eqref{infty-bd} is then valid over $[0, T_1 \sqrt \e |\ln \e|],$ and this is \eqref{ap-infty}.
  \end{proof}

\begin{prop}\phantomsection
\label{lem:Sob-bd} If $\varphi_0 \equiv 1$ on a large enough ball around $x_0,$ and if $\e$ is small enough, the solution $(V,W)$ to system \eqref{V11} issued from \eqref{datum:VW} is defined over $[0, T_1 |\ln \e|],$ and there holds the bound
 \be \label{ap-Sob}
  \| (V,W)(t) \|_{\e,s} \lesssim \e^K |\ln \e|^* e^{t \g^+}.
  \ee
 \end{prop}

The observation time $T_1$ is introduced at the beginning of this Section, and the amplification rate $\g^+$ is defined in \eqref{def:g+}. The spatial cut-off $\varphi_0$ is introduced just below Notation \ref{notation:cut-off} and intervenes in the equation \eqref{V11}(ii) in $W.$

 \begin{proof} We compute $2 \Re e \, (\Lambda^s \d_t W, \Lambda^s W)_{L^2},$ where $\Lambda^s$ is the Fourier multiplier defined by $\Lambda^s := \op_\e\big(1 + |\cdot|^2)^{s/2}\big),$ and $W$ solves \eqref{V11}(ii). By symmetry, the contribution of ${\bf A}$ is zero, so that
  \be \label{up:w1} \d_t\big( \| W \|_{\e,s}^2\big) \leq 2 \big( | \Lambda^s \op_\e(D) W |_{L^2} + \e^{1/2} \| F_W\|_{\e,s}\big) \| W \|_{\e,s}.\ee
  By \eqref{est:fourier-mult},
  \be \label{up:w2} | \op_\e(D) \Lambda^s W|_{L^2} \lesssim |(1 - \varphi_0) u_0|_{L^\infty} \| W \|_{\e,s} \leq \delta_{\varphi_0} \| W \|_{\e,s},\ee
 where $\delta_{\varphi_0} > 0$ can be made arbitrarily small by letting $\varphi_{0} \equiv 1$ on a very large ball around $x_0,$ since $u_0$ decays at spatial infinity. The commutator is estimated by \eqref{est:lambda-s} (here, we are using again the fact that every entry of $D$ is a tensor product $D_1(x) D_2(\xi)$):
 \be \label{up:w3} \big| [\Lambda^s, \op_\e(D)] W \big|_{L^2} \lesssim \e  \| u_a \|_{H^{s_a}} \| W \|_{\e,s-1}.\ee
By Lemma \ref{lem:Linfty-bd}, there holds for $t \leq T_1 |\ln \e|:$
  \be \label{infty-bd-bis}
   | \dot u(\sqrt \e t)|_{L^\infty} \leq |{\mathcal F} \dot u(\sqrt \e t)|_{L^1} \lesssim |{\mathcal F}(V,W)(t)|_{L^1} \leq \e^{\eta_1},
   \ee
 where $\eta_1$ is defined in \eqref{def:eta1}. A bound for $F_W$ is given in Lemmas \ref{lem:bd-F} and \ref{lem:sys-V}. With \eqref{infty-bd-bis}, this bound is
  $$ \| F_W\|_{\e,s} \lesssim (1 + \e^{-1/2 + \eta_1}) \| V,W\|_{\e,s} + \e^{K_a}.$$
 We obtained, for $\e$ small enough,
 \be \label{bd:W}
  \begin{aligned} \| W(t) \|_{\e,s}^2 \lesssim \e^{2K} & + \delta_{\varphi_0} \int_0^t \| W(t')\|_{\e,s}^2 \,dt' +  \e^{\eta_1} \int_0^t \| (V, W)(t')\|_{\e,s}^2 \, dt' \\ & + \e^{K_a + 1/2} \int_0^t \| W(t') \|_{\e,s} \, dt'.
  \end{aligned} \ee
 Going back to the upper bound \eqref{up:0} for $V$ and exploiting \eqref{infty-bd-bis}, we see that there also holds
  \be \label{bd:V}
 \| V(t) \|_{\e,s} \lesssim  \e^K |\ln \e|^* e^{t \g^+} + \e^{\eta_1} |\ln \e|^*\int_0^t e^{(t - t') \g^+}  \| (V, W)(t') \|_{\e,s}  \, dt'.
  \ee
 From \eqref{bd:W} and \eqref{bd:V}, and $K_a + 1/2 \geq K,$ we find that $y(t) := \max_{t' \in [0,t]} \| V(t')\|_{\e,s} + \| W(t')\|_{\e,s}$ satisfies the bound
$$
  y(t) \lesssim \e^K \e^{t \g^+} |\ln \e|^* + \int_0^t \delta_{\varphi_0} y(t') \, dt' + \e^{\eta_1} |\ln \e|^* \int_0^t e^{(t - t') \g^+} y(t')\,dt'.
$$
 By application of the Gronwall Lemma of Appendix \ref{app:d}, this gives \eqref{ap-Sob}, under extra conditions on the small constant $\delta_{\varphi_0},$ implying conditions on the support of $\varphi_0,$ which are $\eta_1 - \delta_{\varphi_0} T_1 > 0$ and $\g |a|_{L^\infty} > \delta_{\varphi_0}.$

   The fact that the a priori bound \eqref{ap-Sob} translates into a bound from below for the existence time follows from a classical continuation argument, similar to the one detailed in the proof of Lemma \ref{lem:Linfty-bd}.
 \end{proof}

\begin{rema}\phantomsection \label{rem:FL1-bis} It would be tempting to use ${\mathcal F}L^1$ bounds in conjunction with the Duhamel representation \eqref{rep:V} of $V,$ instead of ${\mathcal F}L^1$ bounds for the equation {\rm \eqref{V11}(i)} in $V,$ in the hope of obtaining a better estimate on the existence time, one that would involve $\g |\hat a|_{L^1}$ instead of $|B|_0 |\hat a|_{L^1}.$ This would require ${\mathcal F}L^1 \to {\mathcal F}L^1$ estimates on $\op_\e(S),$ which do not seem to be available.
\end{rema}

\begin{rema}\phantomsection \label{rem:bd-W} Going back to the proof of Proposition {\rm \ref{lem:Sob-bd}} and using in \eqref{bd:W} the bound \eqref{ap-Sob}, we see that $W$ enjoys the better upper bound
 $$
 \| W(t)\|_{\e,s} \lesssim \e^{K - T_1 \delta_{\varphi_0}/2} + \e^{K + (\eta_1 - T_1 \delta_{\varphi_0})/2} |\ln \e|^* e^{t \g^+}, \qquad t < T_1 |\ln \e|.
 $$
This will be useful in Section {\rm \ref{sec:end-insta}}.
\end{rema}

\subsubsection{In $H^s$} \label{sec:Sob}

We revisit here the estimates of the proof of Proposition \ref{lem:Sob-bd}, and give a slightly different existence result. We now consider an observation time $T_2$ such that
 \be \label{def:T2}
  T_2 < \frac{K - d/2}{\g |a|_{L^\infty}}.
 \ee
Recall that $h > 0$ intervenes in the upper rate $\g^+$ defined in \eqref{def:g+}; it plays the role of a security distance from the resonance. By continuity, there holds $\g^+ \to |a|_{L^\infty} \g$ as $h \to 0.$
\begin{prop}\phantomsection
\label{prop:ex-Sob} If $\varphi_0 \equiv 1$ on a large enough ball around $x_0,$ and if $\e$ is small enough, the solution $(V,W)$ to \eqref{V11} issued from \eqref{datum:VW} is defined over $[0, T_2 |\ln \e|],$ and there holds the bound
 \be \label{ap-Sob2}
  \| (V,W)(t) \|_{\e,s} \lesssim \e^K |\ln \e|^* e^{t \g^+}.
  \ee
 \end{prop}

\begin{proof} We go back to the proof of Proposition \ref{lem:Sob-bd}. Instead of appealing to Lemma \ref{lem:Linfty-bd} to gain control of the $L^\infty$ norm of $\dot u,$ we use the Sobolev embedding
\be \label{embed}
 |v|_{L^\infty} \leq C_{s,d} \e^{-d/2} \| v \|_{\e,s}, \qquad C_{s,d} > 0, \quad v \in H^s, \,\, s> d/2,
\ee
 and control of $\| \dot u\|_{\e,s}$ by $\| V,W\|_{\e,s},$ as in \eqref{u-vw}, as follows:

 Local-in-time well-posedness in $H^s$ is granted by symmetric hyperbolicity, the semilinear nature of the nonlinearity, and $s > d/2.$ Let $T_*(\e)$ be the maximal existence time\footnote{The fact that notation $T_*(\e)$ was already used, with a different meaning, in the proof of Lemma \ref{lem:Linfty-bd} should not be a factor of confusion, since this use of $T_*(\e)$ is confined to the present proof. Same for $J$ below.} in $H^s.$  Consider the set
  $$ J := \Big\{ t \in [0, T_*(\e)) \cap (0, T_2 \sqrt \e |\ln \e|], \,\,\,\, \forall\, t' \in (0,t),\,\,\,\, |\dot u(t')|_{L^\infty} \leq \e^{\eta_2} \Big\},$$
  where $\eta_2 > 0$ will be chosen appropriately below, depending on $T_2.$

 We now prove that, given $t \in J,$ for some $\zeta > 0$ there holds $[t, t+\zeta) \subset J.$ Just like in the proof of Lemma \ref{lem:Linfty-bd}, this will imply $T_*(\e) > T_2 \sqrt \e |\ln \e|$ by a connectedness argument.

 From \eqref{embed} and Lemmas \ref{lem:bd-F} and \ref{lem:sys-V}, for $t \in J$ we deduce for the source term $F_W$ the bound
  $$ \| F_W \|_{\e,s} \lesssim \big(1 + \e^{-1/2 + \eta_2}\big) \|V,W\|_{\e,s} + \e^{K_a}.$$
Combined with estimates \eqref{up:w1}-\eqref{up:w2}-\eqref{up:w3} in $W,$ this gives the bound, for $t \in J:$
 $$ \begin{aligned} \| W(t) \|_{\e,s}^2 \lesssim \e^{2K} & +   \delta_{\varphi_0} \int_0^t \| W(t')\|_{\e,s}^2 \,dt' +  \e^{\eta_2} \int_0^t \| (V, W)(t')\|_{\e,s}^2 \, dt' \\ & + \e^{K_a + 1/2} \int_0^t \| W(t') \|_{\e,s} \, dt'. \end{aligned}$$
By \eqref{up:0}, for $t \in J:$
 $$ \| V(t)\|_{\e,s} \lesssim \e^K |\ln \e|^* e^{t \g^+} + C \e^{\eta_2} |\ln \e|^* \int_0^t e^{(t - t') \g^+} \| (V, W)(t') \|_{\e,s}  \, dt'.$$
By application of Lemma \ref{lem:g} we deduce from the above bounds the inequality
 \be \label{up:up} \| V, W(t)\|_{\e,s} \lesssim \e^K |\ln \e|^* e^{t \g^+}, \qquad t \in J,
 \ee
for $\delta_{\varphi_0}$ small enough (depending on $\eta_2,$ $T_2$ and $\g$).
 By \eqref{embed} and \eqref{u-vw}, this implies
$$
|\dot u(t)|_{L^\infty} \lesssim \e^{K - d/2} |\ln \e|^* e^{t\g^+}, \quad t \in J.
$$
We let
 $$ 2 \eta_2 := K - d/2 - T_2 \g |a|_{L^\infty}.$$
 Then, if $h$ is small enough, there holds
 $(3/2)\eta_2 < K - d/2 - T_2 \g^+.$
 This implies, for $\e$ small enough, the upper bound $|\dot u(t)|_{L^\infty} \leq \e^{(3/2) \eta_2},$ and we conclude as in the proof of Proposition \ref{lem:Sob-bd}: the bound \eqref{up:up}, which we now know to be valid over $[0, T_2 \sqrt \e |\ln \e|],$ is \eqref{ap-Sob2}.
\end{proof}

\subsection{Endgame: proof of the deviation estimate \eqref{est:th3}} \label{sec:end-insta}

Let $T < T_0$ be given, where the limiting observation time $T_0$ is defined in \eqref{def:T0-K0}. By Propositions \ref{lem:Sob-bd} and \ref{prop:ex-Sob}, the solution $(V,W)$ to the prepared system \eqref{V11} is defined over $[0, T |\ln \e|].$

\medskip

Consider first the case
\be \label{T0cond} T_0 = \frac{K}{|B|_0 |\hat a|_{L^1}}, \quad \mbox{so that} \quad K_0 = K \Big( 1 - \frac{\g |a|_{L^\infty}}{|B|_0 |\hat a |_{L^1}}\Big),\ee
and the bounds of Section \ref{sec:FL1} apply.

\medskip

From the Duhamel representation \eqref{rep:V} and the bound \eqref{action:S} for the action of $\op_\e(S),$ we find the lower bound
 \be \label{low:low}
  \begin{aligned} \big| V\big(T |\ln \e|\big)\big|_{L^2(B(x_0,\rho))} & \geq \big| \op_\e(S(0; T|\ln \e|)) V(0) \big|_{L^2(B(x_0,\rho))}  \\ & \qquad - C \sqrt \e |\ln \e|^* \int_0^{T |\ln \e|} e^{(T |\ln \e| - t')\g^+} | \tilde F_V(t')|_{L^2} dt'.%
  \end{aligned}
  \ee
 From the upper bounds \eqref{bd:tildeFV} and \eqref{bd:FV} for $F_V,$ and \eqref{u-vw} and Proposition \ref{lem:Sob-bd}, we deduce
 $$ \sqrt \e |\tilde F_V(t)|_{L^2} \lesssim \e^{K + 1/2} + \e^{K_a + 1/2} + (\e^{1/2} + |\dot u(\sqrt \e t)|_{L^\infty}) \e^K e^{t \g^+}.$$
 By Lemma \ref{lem:Linfty-bd}, there holds $|\dot u(\sqrt \e t)|_{L^\infty} \leq \e^{\eta_1}$ for $t < T,$ under condition \eqref{T0cond}. Together with $K_a + 1/2 \geq K,$ this implies that the above upper bound in $\tilde F_V$ takes the form
 \be \label{up:tildeF} \sqrt \e |\tilde F_V(t)|_{L^2} \lesssim \e^K + \e^{K + \eta_1} e^{t \g^+}.\ee
We now use in \eqref{low:low} the lower bound for $\op_\e(S(0;t)) V(0)$ given in Lemma \ref{lem-S0} and \eqref{up:tildeF}. This shows that $| V(T |\ln \e|)|_{L^2(B(x_0,\rho))}$ is bounded from below by
 $$
 C(\rho) \e^{K - T \g^-} - C |\ln \e|^{*} \e^{K + 1/2 - T \g^+} - C T |\ln \e|^* \e^{K + \eta_1 - T\g^+}.
 $$
Up to a multiplicative constant, we can rewrite this lower bound
 $$ \e^{K - T \g^-} \Big( 1 - C |\ln \e|^* \e^{\eta_1 - T(\g^+ - \g^-)}\Big).$$
  The smaller the exponent $K - T \g^-,$ the better the above lower bound. Under \eqref{T0cond}, there certainly holds
  $$ K_0 \leq K - T \g^-,$$
 since $\g^- \leq \g |a|_{L^\infty}.$ However, for any $K' > K_0,$ by choosing $\rho$ small enough in Lemma \ref{lem-S0}, and by choosing $T_0 - T$ small enough, we can achieve
 \be \label{k-k} K_0 <  K - T \g^- < K'.\ee
 Besides, given $T < T_0,$ we can choose $h$ in \eqref{def:Rh} and $\rho$ small enough, possibly even smaller than above, so that, for $\e$ small enough, the minimal amplification rate $\g^-$ defined just below \eqref{low0} and the maximal amplification rate $\g^+$ defined in \eqref{def:g+} are close enough so that
 \be \label{cond:hrho1} \g^+ - \g^- < \frac{\eta_1}{T},\ee
 where $\eta_1$ is defined in \eqref{def:eta1}.
 This implies
 \be \label{low-fin} \big| V\big(T |\ln \e|\big)\big|_{L^2(B(x_0,\rho))} \geq \frac{1}{2} C(\rho) \e^{K - T \g^-},\ee
 for $T_0 -T,$ $h,$ $\rho,$ and $\e$ small enough.
Now with \eqref{re:checkU},
 $$ \big| \check U\big(T |\ln \e|\big)\big|_{L^2(B(x_0,\rho))} \geq \big| V\big(T |\ln \e|\big)\big|_{L^2(B(x_0,\rho))} - | W(T |\ln \e|) |_{L^2(\R^d)}.$$
 By Remark \ref{rem:bd-W} page \pageref{rem:bd-W}, there holds
$$
\| W(T |\ln \e|) |_{L^2} \lesssim \e^{K + (\eta_1 - T_1 \delta_{\varphi_0})/2 - T \g_+}.
$$
If we now make sure, by choice of $h,$ $\rho,$ that
\be \label{cond:hrho2} \g^+ - \g^- < \frac{\eta_1 - \delta_{\varphi_0} T_1}{2 T_1},\ee
then the exponent $K - T \g^-$ in the lower bound for $V(T|\ln \e|)$ is strictly smaller than the exponent $K + (\eta_1 - \delta_{\varphi_0} T_1)/2 - T \g^+$ in the upper bound for $W(T |\ln \e|),$ and the above shows that $| \check U(T |\ln \e|)|_{L^2(B(x_0,\rho))}$ enjoys the same lower bound \eqref{low-fin} as $V.$

  Still going up the chain of changes of variables, we arrive by \eqref{def:check-u} at
 $$ |U|_{L^2(B(x_0,\rho))} \geq |\check U|_{L^2(B(x_0,\rho))} - \e^{1/2} |R_0 \check U|_{L^2(\R^d)}.$$
Since $|R_0 \check U|_{L^2} \lesssim |V,W|_{L^2},$ it suffices to use Proposition \ref{lem:Sob-bd} again. This time there is no need to further shrink our parameters, and we obtain that $U$ satisfies the same lower bound as $V.$

 Finally, by Remark \ref{rem:for-endgame}, the leading term in $V(T|\ln \e|)$ is the first component $V_1(T |\ln \e|) \in \C^N,$ in the sense that all other components are smaller by a factor $\e^{1/2},$ so that $V_1(T |\ln \e|)$ enjoys the lower bound \eqref{low-fin}. The same is true for $\check U_1(T |\ln \e|),$ and finally for $U_1,$ which shows that
 \be \label{low:fin-fin} |\dot u(\sqrt \e t)|_{L^2(B(x_0,\rho))} \geq C e^{K - T \g^-}.\ee
The lower bound \eqref{low:fin-fin} implies the deviation estimate \eqref{est:th3}, and concludes the proof of the instability statement in Theorem \ref{theorem1}, in the case \eqref{T0cond}.

 In the case
 $$ \frac{K}{|B|_0 |\hat a|_{L^1}} < \frac{K - d/2}{\g |a|_{L^\infty}}, \quad \mbox{so that} \quad K_0 = d/2,$$
 we use the bounds of Section \ref{sec:Sob} instead of the bounds of Section \ref{sec:FL1}, and arrive at \eqref{est:th3} in exactly the same fashion as above.

\begin{rema}\phantomsection \label{rem:parameters} The choice of parameters is made in the following order: under \eqref{T0cond}, given $K' > K_0,$ we choose $T_1$ so that \eqref{k-k} holds for all $\rho < \rho_0(T_1).$ Associated with this $T_1,$ we have $\eta_1$ defined in \eqref{def:eta1}. Depending on $\eta_1,$ $T_1$ and $\g,$ we choose $\varphi_0$ so that the conditions on $\delta_{\varphi_0}$ that are formulated in the proofs of Propositions {\rm \ref{lem:Sob-bd}} and {\rm \ref{prop:ex-Sob}} hold. Then, we choose $h$ and $\rho < \rho_0(T_1),$ so that \eqref{cond:hrho1} and \eqref{cond:hrho2} hold. From there, the final deviation estimate \eqref{est:th3} holds if $\e$ is small enough, depending on all the other parameters.
\end{rema}

\section{Proof of Theorem \lowercase{\ref{theorem1}:} stability}\label{sec:stability1}

 We assume ${\bf \G} < 0$ and  define a perturbation unknown by
 \be \label{dotu-stab}u =: u_a + \e^{\k} \dot u, \quad \mbox{with $(1 + d)/2 \leq \k \leq \min(K, K_a).$}
 \ee
 In a first step, we follow closely the analysis of Section \ref{p-o-th1}. The unknown $\dot u$ satisfies
 $$ \d_t \dot u + \frac{1}{\e} A_0 \dot u + \sum_{1 \leq j \leq d} A_j \d_{x_j} \dot u = \frac{1}{\sqrt \e} B(u_a) \dot u + \e^{\k - 1/2} B(\dot u, \dot u) - \e^{K_a - \k} r_a^\e.$$
 By the change of variables \eqref{de-U}, we arrive at \eqref{sys U}, and verify as in the proof of Lemma \ref{lem:bd-F} that the source term $F$ satisfies the bound
 $$
 \| F \|_{\e,s} \lesssim \big(1 + \e^{\k - 1/2} |\dot u|_{L^\infty}\big) \| \dot u \|_{\e,s} + \e^{K_a - \k}.
 $$
 By the Sobolev embedding $|\dot u|_{L^\infty} \leq C \e^{-d/2} \| \dot u \|_{\e,s},$ and $(1 + d)/2 \leq \k,$ this yields
 \be \label{est:F-stab}
  \| F \|_{\e,s} \lesssim (1 + \| \dot u\|_{\e,s}) \| \dot u \|_{\e,s} + \e^{K_a - \k}.
 \ee
 Then we perform a normal form reduction as in Section \ref{normal-form1}. By Assumption \ref{ass:transp}, Proposition \ref{normal-form-prop} page \pageref{normal-form-prop} holds true. This gives a symbol $Q,$ by which we define
$$
  \check U(t) := \big(\Id + \sqrt \e \op_\e(Q(t))\big)^{-1}U(t),
$$
 corresponding to \eqref{def:check-u} without the rescaling in time. Indeed, we prove here stability in time $O(1),$ whereas the instability analysis of Section \ref{p-o-th1} takes place in short time $O(\sqrt \e |\ln \e|).$

  As in Corollary \ref{cor:normal-form}, we find that the equation in $\check U$ is
$$
 \d_t \check U + \frac{i}{\e} \op_\e({\mathcal A}) \check U = \frac{1}{\sqrt \e} \op_\e(\check {\mathcal B}) \check U + \check F, \qquad \check {\mathcal B} := \left(\begin{array}{cc} \chi_0 {\mathcal B}^r & 0 \\ 0 & 0 \end{array}\right),
$$
 where ${\mathcal A}$ is defined in \eqref{def:calA} page \pageref{def:calA}, ${\mathcal B}^r$ in \eqref{def:Br} page \pageref{def:Br}, and $\check F$ satisfies bound \eqref{est:F-stab}. The compactly supported frequency cut-off $\chi_0,$ introduced in Section \ref{normal-form1} page \pageref{normal-form1}, is identically equal to one in a neighborhood of the resonant set ${\mathcal R}_{12}.$

\subsection{Symmetrizer}

By Assumption \ref{ass-u-a}, there holds $u_{0,1} = g(t,x) \vec e_{1},$ $u_{0,-1} = g(t,x)^* \vec e_{-1},$ with
$g \in \C$ and constant vectors $\vec e_{\pm 1}.$ In particular, the symbol ${\mathcal B}^r,$ defined in \eqref{def:Br} page \pageref{def:Br}, appears as
$$
{\mathcal B}^r = \left(\begin{array}{cc} 0 & g(t,x) {b}^+_{12}(\xi) \\ g(t,x)^* {b}_{21}^-(\xi) & 0 \end{array}\right),
$$
where $b_{12}^+$ and $b_{21}^-$ \eqref{def:int-coeff} are the interaction coefficient associated with resonance $(1,2).$
Since ${\bf \G} < 0,$ going back to the definition of ${\bf \G}$ in \eqref{def:G0}-\eqref{def:trace}, we see that $\G(\xi) < 0$ on the whole resonant set ${\mathcal R}_{12}.$ In particular, if the support of $\chi_1$ is small enough (that is, contains ${\mathcal R}_{12}$ and not much more), there holds
 \be \label{trace-stab}
 \G(\xi) = \mbox{tr}\, {b}^+_{12} {b}^-_{21}(\xi) < 0, \qquad \mbox{for all $\xi$ in the support of $\chi_1.$}
 \ee

\begin{lemm}\phantomsection\label{symmetrizable} Let $C_{12}, C_{21}: \xi \in \Omega \to C_{12}(\xi), C_{21}(\xi) \in \C^{N \times N}$ be smooth families of matrices defined in a bounded open set $\Omega \subset \R^d,$ and such that
\be \label{cond:sym}
 {\rm rank}\, C_{12} \equiv {\rm rank}\,C_{21} \equiv 1 \quad \mbox{and} \quad {\rm tr}\, C_{12} C_{21} \neq 0, \qquad \xi \in \bar \Omega.
 \ee
 Then, there exists smooth scalar maps $c_{12},$ $c_{21}$ and a block-diagonal, smooth family of change of basis $P = \bp P_{11} &0\\0&P_{22},\ep$ such that ${\rm tr}\, C_{12} C_{21} = c_{12} c_{21},$ and, given $\nu_{12}, \nu_{21} \in \C:$
\begin{equation} \label{eq:P}
 \left(\begin{array}{cc} 0 & \nu_{12} C_{12} \\ \nu_{21} C_{21} & 0 \end{array}\right) =P^{-1}\left(\begin{matrix} 0 & \tilde C_{12} \\ \tilde C_{21} & 0
\end{matrix}\right)P, \qquad \tilde C_{ij} = \left(\begin{array}{cc} \nu_{ij} c_{ij} & 0 \\ 0 & 0_{\C^{(N-1)\times(N-1)}} \end{array}\right),
\ee
with
\be \label{bd:P}
\sup_{\xi \in \Omega} |\d_\xi^\a P(\xi)| + |\d_\xi^\a P^{-1}(\xi)| < \infty, \qquad \a \in \N^d.
\ee
\end{lemm}

\begin{proof} In a first step, we work with fixed $\xi \in \Omega.$ Given $x \in \C^N,$ we denote $x^\sharp = (x,0) \in \C^{2N}$ and $x_\sharp = (0,x) \in \C^{2N}.$

By ${\rm rank}\,C_{12} C_{21}\leq {\rm rank}\,C_{21}=1$ and ${\rm tr}\,C_{12} C_{21}\neq0$,
there holds ${\rm rank}\, C_{12} C_{21}=1,$ and for some $\l_{12} \neq 0$ and some $e \in \C^N,$ there holds $C_{12} C_{21} e = \l_{12} e.$ The vector $e$ generates the range of $C_{12}.$

By symmetry, the same holds for $C_{21} C_{12}:$ for some $f \in \C^N,$ there holds $C_{21} C_{12} f = \l_{21} f,$ with $\l_{12} \neq 0.$ The vector $f$ generates the range of $C_{21}.$

Besides, ${\rm rank}\,C_{12} C_{21}= {\rm rank}\,C_{21}$ implies $\dim \ker C_{12} C_{21} =\dim
\ker C_{21}=N-1,$ hence equality of the kernels: $\ker C_{21} = \ker
C_{12} C_{21}$. Denoting $\{a_1,\cdots,a_{N-1}\}$ a basis of $\ker C_{21}$, we find that
$\{e,a_1,\cdots,a_{N-1}\}$ is a basis of $\C^N,$ since $e \notin \ker C_{21}.$  Then
$\{e^\sharp,a_1^\sharp,\cdots,a_{N-1}^\sharp\}$ is a basis of $\C^N \times \{ 0 \}.$

Similarly, denoting $\{b_1,\cdots,b_{N-1}\}$ a basis of $\ker C_{12} = \ker C_{21} C_{12},$ since $f \notin \ker C_{12},$ the family $\{ f, b_1, \dots, b_{N-1}\}$ is a basis of $\C^N,$ and $\{f_\sharp,b_{1,\sharp},\cdots,b_{n-1,\sharp}\}$ is a basis of
$\{ 0 \} \times \C^N.$

Consider now $C_{21} e$ and $C_{12} f \in \C^N.$ The vector $C_{21} e$ belongs to the range of $C_{21},$ hence it is colinear to $f.$ There cannot be $C_{21} e = 0,$ since then $C_{12} C_{21} e = 0,$ which does not hold. Hence $C_{21} e=  c_{21} f,$ for some $c_{21} \neq 0.$ Similarly, $C_{12} f = c_{12} e,$ with $c_{12} \neq 0.$

In particular, $C_{12} C_{21} e = c_{12} c_{21} e,$ so that $\mbox{tr}\,C_{12} C_{21} = c_{12} c_{21}.$

Then, given the matrix $\un C := \left(\begin{array}{cc} 0 & \nu_{12} C_{12} \\ \nu_{21} C_{21} & 0 \end{array}\right),$ there holds $\un C a_i^\sharp = \nu_{21} (0, C_{21} a_i) = 0,$ and similarly, $\un C b_{i\sharp} = \nu_{12} (C_{12} b_i,0) = 0.$ Besides, $\un C e^\sharp = \nu_{21} (0, C_{21} e) = \nu_{21} c_{21} f_\sharp,$ and $\un C f_\sharp = \nu_{12} (C_{12} f, 0) = \nu_{12} c_{12} e^\sharp.$

The above implies that for the matrix $P$ defined by columns as
\be \label{def:P} P = \mbox{col}\,\big(e^\sharp, a_1^\sharp,\cdots,a_{n-1}^\sharp,f_\sharp,b_{1\sharp},\cdots,b_{n-1\sharp}\big)\ee
there holds \eqref{eq:P}.

The trace of $C_{12} C_{21}$ is bounded away from 0 on the compact $\bar \Omega.$ It is also equal to $\l_{12}.$ This means that the image and kernel of $C_{12} C_{21}$ are strictly separated over $\bar \Omega,$ implying smoothness of the projection onto the kernel and parallel to the image. This, in turn, implies existence of a smooth basis of the kernel (see Kato's treatise \cite{K}, Section II.4.2). Since $\ker C_{12} C_{21} = \ker C_{21},$ this means that we can choose the $a_i$ to vary smoothly over $\Omega.$ Similarly, we can choose the $b_i$ to vary smoothly over $\Omega.$

We can also choose $e$ and $f$ to vary smoothly in $\xi,$ since these are eigenvectors associated to simple eigenvalues. Then $P$ is smooth, and everywhere invertible, with determinant bounded away from 0 on $\bar \Omega.$ This gives regularity of $P^{-1}$ by the comatrix formula, which translates into estimate \eqref{bd:P}. Finally, the $c_{ij}$ are smooth by consideration of \eqref{eq:P}.
\end{proof}

 Assumption \ref{ass:transp}(iii) ensures that the rank condition \eqref{cond:sym} is satisfied by $C_{12} = b_{12}^+,$ $C_{21} = b_{21}^-.$ Thus we apply Lemma \ref{symmetrizable} to these matrices, with $\nu_{12} = g(t,x),$ $\nu_{21} = g(t,x)^*,$  and $\mbox{supp}\, \chi_{1} = \bar \Omega.$ This gives a change of basis $P.$ We coordinatize
$$
 \check U =: (\check U_{12}, W_2^{\rm (s)}) \in \C^{2 N} \times \C^{(J-2)N},
$$
 and let
$$
V^{\rm (s)} := \op_\e(\chi_{1} P^{-1}) \check U_{12}, \quad W_1^{\rm (s)} := \op_\e(1 - \chi_{1}) \check U_{12}, \quad W^{\rm (s)} = (W_1^{\rm (s)},W_2^{\rm (s)}).
$$
so that
$$
 \check U_{12} = \op_\e(P) V^{\rm (s)} + W_1^{\rm (s)}.
$$
Here the exponent (s) indicates that these local unknowns are used only in this stability proof, and distinguishes these from unknowns $V, W_1, W_2$ in the instability proof. This is relatively heavy notation, but we will not carry it very far.

 \begin{lemm}\phantomsection \label{lem:V-stab} There holds
 \be\label{eq:V-stab} \left\{ \begin{aligned} \d_t V^{\rm (s)} + \frac{1}{\e}
\op_\e({\bf M}_{12}) V^{\rm (s)}  & = F_V, \\
 \d_t W^{\rm (s)} + \frac{1}{\e} \op_\e(i {\bf A}) W^{\rm (s)}  & = F_W,\end{aligned}\right.\ee
 where $${\bf M}_{12} := \left(\begin{array}{cc} i (\l_{1,+1} - \o) & -\sqrt \e g \chi_0 {\bf b}_{12}(\xi) \\ -\sqrt \e g^* \chi_0{\bf b}_{21}(\xi) & i \l_2 \end{array}\right), \qquad {\bf b}_{ij}:= \left(\begin{array}{cc} m_{ij} & 0 \\ 0 & 0_{\C^{(N-1) \times (N-1)}} \end{array}\right),$$
 with $m_{ij}(\xi)$ such that
  $$\G(\xi) = {\rm
tr}\, \Pi_{1}(\xi+k)B(\vec e_{1})\Pi_{2}(\xi)B(\vec e_{-1})\Pi_{1}(\xi+k) = m_{12}(\xi) m_{21}(\xi).$$
The diagonal Fourier multiplier ${\bf A}_0$ is
 $$ {\bf A}_0 := \left(\begin{array}{cc} {\mathcal A} & 0 \\ 0 & {\mathcal A}_{3,J} \end{array}\right), \quad {\mathcal A}_{3,J} := {\rm diag}(\l_3, \dots, \l_J).$$
The sources $F_V$ and $F_W$ satisfy bound \eqref{est:F-stab}.
\end{lemm}

\begin{proof} The change of basis $P$ being block-diagonal \eqref{def:P}, there holds
 $$ P \left(\begin{array}{cc} \l_{1,+1} - \o & 0 \\ 0 & \l_2 \end{array}\right) = \left(\begin{array}{cc} \l_{1,+1} - \o & 0 \\ 0 & \l_2 \end{array}\right) P.$$
 Besides, by estimate \eqref{est:fourier-mult}:
 $$ \op_\e(\chi_{1} P) \op_\e(\chi_0 {\mathcal B}^r) = \op_\e(\chi_{0}  P {\mathcal B}^r) + \e R_0,$$
  where $R_0$ is a uniform remainder in the sense of Definition \ref{def:ur}, so that
 $$ \op_\e(\chi_{1} P) \op_\e(\chi_0 {\mathcal B}^r) \check U_{12} = \op_\e(\chi_{0} P {\mathcal B}^r P^{-1}) V^{\rm (s)} + \op_\e(\chi_{0} P {\mathcal B}^r) W_1^{\rm (s)} + \e R_0(V^{\rm (s)} + W_1^{\rm (s)}).$$
 But then
 $$\op_\e(\chi_{0} P {\mathcal B}^r) W_1^{\rm (s)} = \op_\e(\chi_0 P {\mathcal B}^r) \op_\e(1 - \chi_1) \check U_{12} = \e R_0 \check U_{12},$$
 since $\chi_0(1 - \chi_1) \equiv 0.$ By Lemma \ref{symmetrizable}, there holds for $\xi$ in the support of $\chi_1$ the identity
  $$ P {\mathcal B}^r P^{-1} = \left(\begin{array}{cc} 0 & g {\bf b}_{12} \\ g^* {\bf b}_{21} & 0 \end{array}\right).$$
The above verifies the form of the equation in $V^{\rm (s)}.$ For the equation in $W_1^{\rm (s)},$ we use
 $$ \op_\e(1 - \chi_1) \op_\e(\chi_0 {\mathcal B}^r) = \op_\e\big( (1 - \chi_1) \chi_0 {\mathcal B}^r) + \e R_0 = \e R_0,$$
 by \eqref{est:fourier-mult} and $(1 - \chi_1) \chi_0 \equiv 0.$
 \end{proof}

 The symmetrizer is defined as the Fourier multiplier
$$
{\bf S}(\xi):=\bp \Id_N &0\\0& - m_{12}^*(\xi) m_{21}(\xi)^{-1}\Id_N
\ep, \quad \mbox{for $\xi\in\supp\,\chi_{1},$}
$$
where $m_{12}^*$ denotes complex conjugate of $m_{12}.$ We then choose any extension to all of $\R^d_\xi$ so that ${\bf S}$ is smooth, real diagonal, with diagonal entries bounded and bounded away from zero. We note that since $m_{12} m_{21} \in \R,$ there holds $m_{12}^*/m_{21} \in \R,$ so that, in particular, ${\bf S}^* \equiv {\bf S}.$

The coefficients $m_{ij}$ are bounded on the support of $\chi_1.$ Besides, by \eqref{trace-stab} and the fact that $\G = m_{12} m_{21},$ they are bounded away from zero on the support of $\chi_{1}.$ This implies the bounds
\be\label{sym:low} |u|_{L^2}\lesssim
 | \op_\e({\bf S}^{1/2}) u|_{L^2}, \qquad u \in L^2,
 \ee
 and
 \be \label{sym:up} \quad \|\op_\e({\bf S}) u\|_{\e,s}\lesssim
 \|u\|_{\e,s}, \quad u \in H^s.\ee
The fact that ${\bf S}$ is a symetrizer is expressed in the following lemma:
\begin{lemm}\phantomsection \label{lem:symmetry}
There holds for all $u \in H^s$ the bound
 \be \label{symmetry}
\Re e \, \Big( \op_\e({\bf S})  \op_\e({\bf M}_{12}) u + \Big( \op_\e({\bf S}) \op_\e({\bf M}_{12}) \Big)^*\, u, \, \, u \Big)_{L^2} \lesssim \e \| u \|_{L^2}^2.
\ee
\end{lemm}

Recall that if $z \in \C,$ then $z^*$ denotes complex conjugate, if $z \in \C^{N \times N},$ then $z^*$ denotes complex transpose, and if $z$ is linear bounded $L^2 \to L^2,$ then $z^*$ denotes the adjoint operator. We use the latter in \eqref{symmetry}, and all three in the forthcoming proof.

\begin{proof} We compute
 $$ \op_\e({\bf S}) \op_\e\left(\begin{array}{cc} i ( \l_{1,+1} - \o) & 0 \\ 0 &  \l_2\end{array}\right) = \op_\e \left(\begin{array}{cc} i(\l_{1,+1} - \o) & 0 \\ 0 & i \l_2\end{array}\right) \op_\e({\bf S}),$$
and
$$  \op_\e\left(\begin{array}{cc} i ( \l_{1,+1} - \o) & 0 \\ 0 &  \l_2\end{array}\right)^* = - \op_\e\left(\begin{array}{cc} i ( \l_{1,+1} - \o) & 0 \\ 0 &  \l_2\end{array}\right), \qquad \op_\e({\bf S})^* = \op_\e({\bf S}),$$
so that the diagonal entries of ${\bf M}_{12}$ contribute nothing to \eqref{symmetry}. Next we compute
$$\op_\e(- m_{12}^* m_{21}^{-1}) \op_\e(\chi_0 g^* {\bf b}_{21}) = g^* \op_\e(\chi_0 {\bf b}_{12}^*) + \e R_0,$$
via \eqref{est:fourier-mult} and definition of ${\bf b}_{ij}.$ This implies
  \be \label{s1} \op_\e({\bf S}) \op_\e\left(\begin{array}{cc} 0 & g \chi_0 {\bf b}_{12} \\ g^* \chi_0 {\bf b}_{21} & 0 \end{array}\right) = \op_\e\left(\begin{array}{cc} 0 & g \chi_0 {\bf b}_{12} \\ - g^* \chi_0 {\bf b}_{12}^* & 0 \end{array}\right) + \e R_0.\ee
Besides, using \eqref{est:fourier-mult} once more,
$$ \op_\e\left(\begin{array}{cc}  0 & g \chi_0 {\bf b}_{12} \\ g^* \chi_0 {\bf b}_{21} & 0 \end{array}\right)^* = \op_\e\left(\begin{array}{cc}  0 & g \chi_0 {\bf b}_{21}^* \\ g^* \chi_0 {\bf b}_{12}^* & 0 \end{array}\right) + \e R_0,$$
so that
\be \label{s2} \op_\e\left(\begin{array}{cc}  0 & g \chi_0 {\bf b}_{12} \\ g^* \chi_0 {\bf b}_{21} & 0 \end{array}\right)^* \op_\e({\bf S})^* = \op_\e\left(\begin{array}{cc} 0 & - g \chi_0 {\bf b}_{12} \\ g^* \chi_0 {\bf b}_{12}^* & 0 \end{array}\right) + \e R_0.\ee
With \eqref{s1} and \eqref{s2}, the contribution of the extra-diagonal entries of ${\bf M}_{12}$ to the left-hand side of \eqref{symmetry} has the form $\e \Re e \, (R_0 u, u)_{L^2}.$ This concludes the proof.
\end{proof}

\subsection{Uniform bounds} \label{sec:endsta}

We use the Fourier multiplier $\Lambda^s = \op_\e\big((1 + |\cdot|^2)^{s/2}\big),$ and compute
 $$
  \d_t \Big( \big| \op_\e({\bf S}^{1/2}) \Lambda^s {V^{\rm (s)}} \big|^2_{L^2}\Big) = 2 \Re e\, \big( \op_\e({\bf S}) \Lambda^s \d_t {V^{\rm (s)}}, \Lambda^s {V^{\rm (s)}}\big)_{L^2}.
  $$
 Following \eqref{eq:V-stab}, the above right-hand side decomposes into two terms. The first is
 $$ \begin{aligned} - \frac{2}{\e} \Re e & \, \Big(\op_\e({\bf S}) \op_\e({\bf M}_{12}) \Lambda^s V^{\rm (s)}, \, \Lambda^s V^{\rm (s)}\Big)_{L^2} \\ & = - \frac{2}{\e} \Big( \op_\e({\bf S}) \op_\e({\bf M}_{12}) \Lambda^s V^{\rm (s)} + \Big( \op_\e({\bf S}) \op_\e({\bf M}_{12}) \Big)^* \Lambda^s V^{\rm (s)},\, \Lambda^s V^{\rm (s)}\Big)_{L^2}\,,\end{aligned}$$
and, by Lemma \ref{lem:symmetry}, is controlled by $\| V^{\rm (s)} \|_{\e,s}^2.$
The second term is
 $$ \begin{aligned} \Big| 2 \Re e \, \big(\op_\e({\bf S} \Lambda^s F_V, \, \Lambda^s V^{\rm (s)}\big)_{L^2}\Big| & \lesssim \| F_V\|_{\e,s} \| V^{\rm (s)} \|_{\e,s} \\ & \lesssim \Big( (1 + \| \dot u \|_{\e,s}) \| \dot u\|_{\e,s} + \e^{K_a - \k}\Big) \| V^{\rm (s)}\|_{\e,s}, \end{aligned}$$
 the first inequality by \eqref{sym:up} and the second by estimate \eqref{est:F-stab}. Gathering the above estimates and using $\| \dot u\|_{\e,s} \lesssim \| V^{\rm (s)}, W^{\rm (s)}\|_{\e,s}$ and $K_a - \k \geq 0,$ we obtain
 \be \label{est1} \d_t \big( \big| \op_\e({\bf S}^{1/2}) \Lambda^s {V^{\rm (s)}} \big|^2_{L^2}\Big) \lesssim  \| V^{\rm (s)} \|_{\e,s} + \| V^{\rm (s)}, W^{\rm (s)} \|_{\e,s}^2.
 \ee
Besides, from \eqref{eq:V-stab} we deduce
 \be \label{est2} \d_t (\| W^{\rm (s)} \|_{\e,s}^2) \lesssim \| F_W \|_{\e,s} \| W^{\rm (s)} \|_{\e,s} \lesssim \| W^{\rm (s)} \|_{\e,s} + \| V^{\rm (s)}, W^{\rm (s)} \|_{\e,s}^2.\ee
From \eqref{est1}-\eqref{est2} and the lower bound \eqref{sym:low}, we deduce the stability estimate
$$ \| V^{\rm (s)}, W^{\rm (s)} \|_{\e,s} \leq \e^{K - \k} C(T_a), \qquad 0 \leq t \leq T_a.$$
The stability estimate \eqref{est:stab} then follows from $\| \dot u\|_{\e,s} \lesssim \| V^{\rm (s)}, W^{\rm (s)}\|_{\e,s},$ and definition of $\dot u$ in \eqref{dotu-stab}.

\chapter{Other proofs}


\section{Proof of Theorem \lowercase{\ref{th-two}}} \label{sec:proof-th-extension}

The differences with the main proof (Sections \ref{p-o-th1} and \ref{sec:stability1}) are essentially notational.

\medskip

An issue that we face right away as we consider more than one non-transparent resonance is the definition of the frequency-shifted and projected variable $U.$ Consider indeed the case of a set ${\mathfrak R}_0$ of non-transparent resonances equal to
${\mathfrak R}_0 = \{ (1,2), (2,3), (3,1) \}.$ It is easy to see that shifted and projected variables, following \eqref{de-U}, are not appropriate.
That is, if we let $U_1 = e^{- i \theta} \op_\e(\Pi_1) \dot u$ and $U_2 = \op_\e(\Pi_2) \dot u$ to account for the $(1,2)$ resonance, and then $U_3 = e^{i \theta} \op_\e(\Pi_3) \dot u$ to account for the $(2,3)$ resonance, then the frequency shifts in $U_1$ and $U_3$ are not suitable for the $(3,1)$ resonance.

We overcome this issue by localizing the definitions of the projected variables (Section \ref{sec:several-coord}). This is relatively straightforward, but notations are heavy.

 Then, by taking advantage of the partial transparency hypothesis (Assumption \ref{ass:several-res}(ii)), all couplings are eliminated, except for those describing non-transparent resonances. This is done in Section \ref{sec:several-normal-form}.

 No further difficulty arises, and the estimates, both  in the stable and unstable case, are similar to the estimates in Sections \ref{p-o-th1} and \ref{sec:stability1}.

\subsection{Coordinatization} \label{sec:several-coord}

The perturbation variable $\dot u$ is defined by $u =: u_a + \e^{\k} \dot u,$ with $\k = 0$ in the unstable case, as in \eqref{def:dot-u}, and $(1 + d)/2 \leq \k \leq \min(K, K_a)$ in the stable case, as in \eqref{dotu-stab}.

\medskip

 \label{encore cut-offs} Associated with $(i,j) \in {\mathfrak R}_0,$ we define frequency cut-offs $\chi_{ij},$ $\chi_{ij}^\sharp,$ such that $\chi_{ij} \prec \chi_{ij}^\sharp$ (in the sense of Notation \ref{notation:cut-off} page \pageref{notation:cut-off}), $\chi_{ij} \equiv 1$ on ${\mathcal R}_{ij},$ and the support of $\chi^\sharp_{ij}$ is a small neighborhood of ${\mathcal R}_{ij}.$ All truncations are compactly supported (owing to Assumption \ref{ass:several-res}(i)) and take values in $[0,1].$ The normal form reduction of Section \ref{sec:several-normal-form} will require the supports of $\chi_{ij}^\sharp$ to be not much larger than ${\mathcal R}_{ij}.$

  Let $i$ such that $(i,j) \in {\mathfrak R}_0$ for some $j.$ We let $I^+_i = \{ j \in [1,J],\, (i,j) \in {\mathfrak R}_0\}.$ A local variable associated with resonance $(i,j)$ is defined by
\be \label{u-ij+} \begin{aligned} u_{ij}^+ := \op_\e(\chi_{ij}) e^{-i\theta} \op_\e(\Pi_i) \dot u, \qquad j \in I^+_i.
\end{aligned}\ee
We let  $I^-_i = \{ j' \in [1,J], \, (j',i) \in {\mathfrak R}_0\},$ so that if $i$ is such that $(j',i) \in {\mathfrak R}_0$ for some $j',$ then $I^-_i$ is not empty. A local variable is defined by
\be \label{u-ij-} u_{ij'}^- := \op_\e(\chi_{j'i}) \op_\e(\Pi_i) \dot u, \qquad j' \in I_i^-.
\ee
We now let
\be \label{vi}
 v_i := \op_\e\Big( 1 - \sum_{j \in I^+_i} (\chi_{ij})_{-1} - \sum_{j' \in I^-_i} \chi_{j'i}\Big) \op_\e(\Pi_i) \dot u,
 \ee
where $(\chi_{ij})_{-1} := \chi_{ij}(\cdot - k),$ in accordance with \eqref{def:shift}, so that
\be \label{u-i} \op_\e(\Pi_i) \dot u = v_i + \sum_{j \in I^+_i} e^{i \theta} u^+_{ij} + \sum_{j' \in I^-_{i}} u^-_{ij'},\ee
and of course $\dot u$ is then reconstructed by summation of \eqref{u-i} over $i \in [1,J].$ By convention, sums over empty sets are equal to zero.

\begin{rema}\phantomsection \label{rem:coord} If $(1,2) \in {\mathcal R}_0,$ then the variables that describe the $(1,2)$ resonance are $u_{12}^+$ and $u_{21}^-.$ Indeed, $2 \in I_1^+$ and $1 \in I_2^-.$ Thus $u_{12}^+$ and $u_{21}^-$ will play here the role played by variables $U_1$ and $U_2$ in Section {\rm \ref{sec:proj}}.
\end{rema}

We now introduce a more compact notation. Variables are indexed by resonant indices $(i,j),$ the position $+$ or $-$ in the resonance (meaning first or second term: $i$ or $j$ in $(i,j$)), and also the nature of the variable: ``inner" variables are denoted by the letter $u$ and ``outer'' variables are denoted $v.$

We introduce the set of indices
 \be \label{def:mA} \mA := \big\{ (i,\out), \,\, i \in [1,J] \big\} \, \bigcup \, \big\{ (i,j,p, \inn), \,\, (i,j) \in {\mathfrak R}_0, \,\, p \in \{+,-\} \big\}.\ee
 The associated local variables are $({\bf u}_\a)_{\a \in \mA}:$
  $${\bf u}_\a := \left\{\begin{aligned} u_{ij}^\pm, & \quad \a = (i,j,\pm,\inn) \\ v_{i}, & \quad \a = (i,\out). \end{aligned}\right.$$
   We denote ${\bm \chi}_\a$ the truncation in ${\bf u}_\a;$ this is the truncation that appears in \eqref{u-ij+} or \eqref{u-ij-} or \eqref{vi}, explicitly:
   $$ {\bm \chi}_{\a} := \left\{\begin{aligned} \chi_{ij}, & \quad \a = (i,j,+,\inn), \\ \chi_{ji}, & \quad  \a = (i,j,-,\inn), \\ 1 - \sum_{j \in I^+_i} (\chi_{ij})_{-1} - \sum_{j' \in I^-_i} \chi_{j'i}, & \quad \a = (i,\out).
   \end{aligned}\right.$$
    With the notation
    $$\delta_{\a} := \left\{\begin{aligned} 1, & \quad \a = (i,j,+,\inn), \,\,\mbox{for some $(i,j)$}, \\ 0, & \quad \mbox{otherwise},\end{aligned}\right.$$
  it appears from \eqref{u-ij+}-\eqref{u-ij-}-\eqref{vi} that there holds
\be \label{u-a}
 {\bf u}_\a = \op_\e\big({\bm \chi}_\a^\sharp \Pi_{i,+\delta_{\a}}\big) {\bf u}_\a, \qquad {\bm \chi}_{\a} \prec {\bm \chi}_\a^\sharp.
 \ee

The derivation of the coupled system satisfied by the ${\bf u}_\a$ is essentially identical to the computations of Section \ref{sec:proj}. We find
\be \label{u-a-sys}
 \d_t {\bf u}_\a + \frac{1}{\e} \op_\e(i {\bm \mu}_\a) {\bf u}_\a = \frac{1}{\sqrt \e } \sum_{p = \pm 1} \sum_{\b \in \mA} \op_\e({\bf B}_{p\a\b}) {\bf u}_\b + {\bf F}_\a,
\ee
where
\begin{itemize}
\item the symbols of the propagators are \label{def:bm:mu}
$$ {\bm \mu}_{\a} := \left\{\begin{aligned} \l_{i,+1} - \o, & \quad \a = (i,j,+,\inn), \\ \l_i, & \quad \a \in \big\{ (i,j,-,\inn), \, (i, \out)\big\};\end{aligned}\right.$$
\item using \eqref{u-i} and \eqref{u-a}, we find that the symbols of the coupling terms are
\be \label{b-abp}
 {\bf B}_{p\a\b}
  : = e^{i (p - \delta_{\a} + \delta_{\b}) \theta}   {B}_{p\a\b}, \qquad   {B}_{p\a\b}  : = \chi_{\a, +(p - \delta_{\a} + \delta_{\b})} \chi_\b^\sharp \Pi_{i, +(p + \delta_{\b})} B_p \Pi_{i', +\delta_{\b}},
\ee
for $p \in \{-1,1\},$ $(\a,\b) \in \mA \times \mA,$ where $B_p$ is defined in \eqref{def:Bp} page \pageref{def:Bp};
\item as in Lemma \ref{lem:bd-F} page \pageref{lem:bd-F}, the source term ${\bf F}_\a$ satisfies
\be \label{bd:F-alpha}\| {\bf F}_\a\|_{\e,s} \lesssim (1 + \e^{\k - 1/2} |\dot u|_{L^\infty}) \|\dot u\|_{\e,s} + \e^{K_a - \k}.\ee
\end{itemize}

\subsection{Normal form reduction} \label{sec:several-normal-form}

The goal is to derive from \eqref{u-a-sys} a reduced system in which non-resonant and non-transparent coupling terms do not appear. This is done by a change of variables that is very much similar to the one conducted in Section \ref{normal-form1}, and now described in some detail. The difference with Section \ref{normal-form1} is that in the current context of several non-transparent resonances, a partial transparency assumption bearing on non-transparent resonances (Assumption \ref{ass:several-res}(ii)) has to be introduced in order for the normal form reduction to go through as before.

\medskip

Recall that $R_0$ denotes uniform remainders in the sense of Definition \ref{def:ur}. Here in a slight misuse of notation we also denote $R_0$ any {\it symbol} such that the associated pseudo-differential operators are uniform remainders in the sense of Definition \ref{def:ur}.

\begin{prop}\phantomsection \label{prop:several-res-new} Under Assumption {\rm \ref{ass:several-res}(ii)}, given $(\a,\b)\in \mA \times \mA,$  $p \in \{-1,1\},$ equation
\be \label{homo-new}
 i \big( - \ell \o + {\bm \mu}_{\a,+\ell} - {\bm \mu}_\b \big) Q_{\ell\a\b} = B_{p\a\b} + \e R_0, \qquad \ell = p - \delta_{\a} +  \delta_{\b}
\ee
has a solution $Q_{\ell\a\b} \in S^0$ except if {\rm (a)} $p = 1,$ $\a = (i,j,1,\inn),$ $\b = (j,i,-1,\inn)$ with $(i,j) \in {\mathfrak R}_0,$ or {\rm (b)} $p = -1,$ $\a = (i,j,-1,\inn),$ $\b = (j,i,1,\inn),$ with $(j,i) \in {\mathfrak R}_0.$
\end{prop}

\begin{proof} Following \eqref{def:int-coeff}, we denote $b_{ij}^+$ and $b_{ji}^-$ the interaction coefficients associated with a given resonance $(i,j):$
 \be \label{def:int-coeff:ij} {b}_{ij}^+ = \Pi_{i+1} B(\vec e_1) \Pi_j, \qquad {b}_{ji}^- = \Pi_j B(\vec e_{-1}) \Pi_{i,+1}.
 \ee
We let $\a = (i,j,q,\inn)$ or $\a = (i,\out),$ and $\b =
(i',j',q',\inn)$ or $\b =(i',\out).$ On the complement of the
support of $B_{p\a\b}$ a trivial solution to \eqref{homo-new} is
$Q_{\ell\a\b} = 0,$ so that it is sufficient to consider
\eqref{homo-new} for frequencies such that $\chi_{\a, +(p -
\delta_{\a} + \delta_{\b})} \neq 0$ and $\chi_\b^\sharp \neq 0.$ We
go over most cases in detail. In all cases Assumption
\ref{ass:several-res}(ii) is invoked. We call ``phase" the scalar $-
\ell \o + {\bm \mu}_{\a,+\ell} - {\bm \mu}\b$ in factor of the
unknown $Q_{\ell\a\b},$ and ``source" the right-hand side $B_{p \a
\b}$ in the homological equation \eqref{homo-new}.

\medskip

$\bullet$ {\it ``\inn-\inn'' coupling terms:} here $\a = (i,j,q,\inn),$ $\b = (i',j',q',\inn).$  There are four subcases.

\medskip

 $\,\,\bullet \bullet$ If $q = q' = 1,$ then the source terms and phases for $p = 1$ and $p = -1$ are
   $$ \begin{aligned} B_{1\a\b} & = g(t,x) \chi_{ij,+1} \chi_{i'j'}^\sharp {b}^+_{ii',+1}, \quad \mbox{with phase $(\l_{i,+1} - \o - \l_{i'})_{+1},$ and }
   \\
    B_{-1\a\b} & = g(t,x)^* \chi_{ij,-1} \chi_{i'j'}^\sharp {b}_{ii'}^-, \quad \mbox{with phase $-(\l_{i',+1} - \o - \l_i),$}
    \end{aligned}$$
    respectively. We recall notation from \eqref{pol-con}: $u_{0,1} = g(t,x) \vec e_1,$ $u_{0,-1} = g(t,x)^* \vec e_{-1}.$

     If $(i,i') \notin {\mathfrak R}_0,$ meaning that $(i,i')$ is a transparent resonance (or is not a resonant pair), then by definition the phase $(\l_{i,+1} - \o - \l_{i'})_{+1}$ factorizes in the interaction coefficient $b_{ii',+1}^+.$ This implies that the symbol
 \be \label{res:homo}  g(t,x) \big(\l_{i,+1} - \o - \l_{i'}\big)_{+1}^{-1} \,  \chi_{ij,+1} \chi_{i'j'}^\sharp {b}^+_{ii',+1}
 \ee
 is bounded, belongs to $S^0,$ and provides a solution to \eqref{homo-new}.

 Otherwise $(i,i') \in {\mathfrak R}_0.$ By definition of the truncations associated with non-transparent resonances, the phase $\l_{i,+1} - \o - \l_{i'}$ is bounded away from zero over the support of $1 - \chi_{ii'}.$ Thus for $p =1$ it suffices to solve \eqref{homo-new} for the source terms
 $$ \begin{aligned} \chi_{ii',+1}   B_{1\a\b} & = g \chi_{ii',+1} \chi_{ij,+1} \chi_{i'j'}^\sharp {b}^+_{ii',+1}, & \quad \mbox{with phase $(\l_{i,+1} - \o - \l_{i'})_{+1}.$}
 \end{aligned}$$
We now invoke the transparency condition \eqref{new:1605} from Assumption \ref{ass:several-res}. Indeed, by \eqref{new:1605}, the interaction coefficient $b_{ii',+1}^+$ is transparent on the support of $\chi_{ii',+1} \chi_{i'j'}^\sharp,$ if the supports of these cut-offs are tightly cut around the corresponding resonance sets. Again, this implies that \eqref{res:homo} is bounded, belongs to $S^0$ and solves \eqref{homo-new}.

 In the case $p = -1,$ the same argument applies.
 \medskip

  $\,\,\bullet\bullet$ If $q = 1,$ $q' = -1,$ then the source terms are
  $$ B_{1\a\b} = g \chi_{ij} \chi_{j'i'}^\sharp {b}_{ii'}^+, \qquad B_{-1\a\b} = g^* \chi_{ij,-2} \chi_{j'i'}^\sharp {b}_{ii',-1}^-.$$
Again, if $(i,i') \notin {\mathfrak R}_0$ (for $p = -1,$ if $(i',i) \notin {\mathfrak R}_0$), the corresponding equation \eqref{homo-new} is solved by dividing the source by the phase. Thus it suffices to solve \eqref{homo-new} for
 $$ \begin{aligned} \chi_{ii'}B_{1\a\b} & = g \chi_{ii'} \chi_{ij} \chi_{j'i'}^\sharp {b}_{ii'}^+, & \quad \mbox{with phase $\l_{i,+1} - \o - \l_{i'};$} \\ \chi_{i'i,-1} B_{-1\a\b} & = g \chi_{i'i,-1} \chi_{ij,-2} \chi_{j'i'}^\sharp {b}_{ii',-1}^-, & \quad \mbox{with phase $-(\l_{i',+1} - \o - \l_{i})_{-1},$}\end{aligned}$$
 assuming $(i,i') \in {\mathfrak R}_0$ in the case $p = 1$ and $(i',i) \in {\mathfrak R}_0$ in the case $p = -1.$

 For $B_{1\a\b}$ ($p = 1$), we use \eqref{new:1605:2} if $j' \neq i$ or $i' \neq j.$ Indeed, given $\xi \in {\mathcal R}_{ii'},$ if in addition $\xi$ belongs the support of $\chi_{ij},$ then $\xi$ belongs to a neighborhood of ${\mathcal R}_{ij}.$ Since frequencies that belong to ${\mathcal R}_{ii'} \cap {\mathcal R}_{ij}$ necessarily belong to $\{ \l_{i'} = \l_j\},$ we can indeed use the partial transparency condition \eqref{new:1605:2}.

  The remaining case for $p = 1$ is $(i',j') = (j,i).$ Then, there holds $\a = (i,j,1,\inn)$ and $\b = (j,i,-1,\inn).$ But then $(i,i') \notin {\mathfrak R}_0,$ otherwise we would be in excluded case (a).

  For $B_{-1\a\b},$ we use \eqref{new:1605} again.

\medskip

 {\it In the remainder of this proof, upon consideration of an interaction coefficient, we will always assume that the relevant resonance, that is $(i,i')$ if $p = 1$ and $(i',i)$ if $p = -1,$ is non-transparent. Otherwise the reduction is trivial, meaning that we do not need to appeal to the partial transparency conditions \eqref{new:1605} and \eqref{new:1605:2} and may simply solve \eqref{homo-new} by dividing the source by the phase, without having to consider the form of the truncation functions.}

 \medskip
 $\,\,\bullet\bullet$ If $q = -1,$ $q' = 1,$ then the source terms are
 $$ B_{1\a\b} = g \chi_{ji,+2} \chi_{i'j'}^\sharp {b}_{ii',+1}^+, \quad B_{-1\a\b} = g^* \chi_{ji} \chi_{i'j'}^\sharp {b}_{ii'}^-.$$
It suffices to solve for
 $$ \begin{aligned} \chi_{ii',+1}  B_{1\a\b} & = g \chi_{ii',+1} \chi_{ji,+2} \chi_{i'j'}^\sharp {b}_{ii',+1}^+, \quad & \mbox{with phase $(\l_{i,+1} - \o - \l_{i'})_{+1};$} \\ \chi_{i'i} B_{-1\a\b} & = g^* \chi_{i'i} \chi_{ji} \chi_{i'j'}^\sharp {b}_{ii'}^-, & \quad \mbox{with phase $-(\l_{i',+1} - \o - \l_{i}),$} \end{aligned}$$
 For the first source term ($p =1$), we use \eqref{new:1605}. For the other source term ($p = -1$) we use \eqref{new:1605:2} unless $(i',j') = (j,i),$ in which case $\a = (i,j,-1,\inn)$ and $\b = (j,i,1,\inn).$ Then, $(j,i) \notin {\mathfrak R}_0,$ otherwise we would be in excluded case (b). Hence the interaction coefficient $b^-_{i'i} = b^-_{ji}$ is transparent on the support of $\chi_{i'i}.$

 \medskip

 $\,\,\bullet\bullet$ If $q = q' = - 1,$ then the source terms are
 $$ B_{1\a\b} = g \chi_{ji,+1} \chi_{j'i'}^\sharp {b}_{ii'}^+, \quad B_{-1\a\b} = g^* \chi_{ji} \chi_{j'i'}^\sharp {b}_{ii',-1}^-.$$
It suffices to solve for
 $$ \begin{aligned} \chi_{ii'}  B_{1\a\b} & = g \chi_{ii'} \chi_{ji,+1} \chi_{j'i'}^\sharp {b}_{ii'}^+, & \quad \mbox{with phase $\l_{i,+1} - \o -\l_{i'};$} \\  \chi_{i'i,-1} B_{-1\a\b} & = g^* \chi_{i'i,-1} \chi_{ji,-1} \chi_{j'i'}^\sharp {b}_{ii',-1}^-, & \quad \mbox{with phase $-(\l_{i',+1} - \o - \l_{i})_{-1},$}\end{aligned}$$
For both terms \eqref{new:1605} applies.

\medskip

$\bullet$ {\it ``in-out'' coupling terms:} we consider first $u_{ij}^\pm/v_{i'}$ coupling terms, that is $\a = (i,j,q,\inn),$ $\b = (i',\out).$ 
\medskip

 $\,\,\bullet\bullet$ If $q = 1$ and $p = 1,$ the source and phase are
  $$ g \chi_{ij} \Big( 1 - \sum_{j' \in I^+_{i'}} (\chi_{i'j'})_{-1} - \sum_{j' \in I^-_{i'}} \chi_{j'i'} \Big)^\sharp \, {b}_{ii'}^+ \quad \mbox{and} \quad \l_{i,+1} - p \o -  \l_{i'}.$$
Since we may assume $(i,i') \in {\mathfrak R}_0,$ there holds $i \in I_{i'}^-,$ and the phase is bounded away from zero on $(1 - \chi_{ii'})^\sharp.$ Therefore it suffices to solve for the source terms
 $$ g \chi_{ii'} \chi_{ij} \chi_{i'j',-1}^\sharp {b}_{ii'}^+, \,\, j' \in I^+_{i'}, \quad \mbox{and} \quad g \chi_{ii'} \chi_{ij} \chi_{j'i'}^\sharp {b}_{ii'}^+, \,\, j' \in I^-_{i'}, \, j'\neq i.$$
Conditions \eqref{new:1605} and \eqref{new:1605:2} apply to the first and second terms above, respectively.

 \medskip

 $\,\,\bullet\bullet$ If $q = 1$ and $p = -1,$ the source and phase are
 $$ g^* \chi_{ij,-2} \Big( 1 - \sum_{j' \in I^+_{i'}} \chi_{i'j',-1} - \sum_{j' \in I^-_{i'}} \chi_{j'i'} \Big)^\sharp \, {b}_{ii',-1}^- \quad \mbox{and} \quad - (\l_{i',+1} - \o -  \l_{i})_{-1}.$$
We apply the same reasoning as in the previous case: we may assume $i \in I_{i'}^+,$ so that the phase factorizes in the source term over the support of $1 - \chi_{i'i},$ and as a consequence it suffices to solve for the source terms
 $$ g^* \chi_{i'i,-1} \chi_{ij,-2} \chi_{i'j',-1}^\sharp {b}_{ii',-1}^-, \,\, j' \in I^+_{i'}, \, j' \neq i,  \quad \mbox{and} \quad  g^* \chi_{i'i,-1} \chi_{ij,-2} \chi_{j'i'}^\sharp{b}_{ii',-1}^-, \,\, j' \in I^-_{i'}.$$
  Conditions \eqref{new:1605:2} and \eqref{new:1605} apply to the first and second terms above, respectively.

 \medskip

 $\,\,\bullet\bullet$ If $q = - 1$ and $p = 1,$ the source and phase are
  $$ g \chi_{ji,+1} \Big( 1 - \sum_{j' \in I^+_{i'}} (\chi_{i'j'})_{-1} - \sum_{j' \in I^-_{i'}} \chi_{j'i'}\Big)^\sharp \,  {b}_{ii'}^+ \quad \mbox{and} \quad \l_{i,+1} - \o - \l_{i'}.$$
Here $i \in I_{i'}^-,$ so that in the above source term, we can multiply by $\chi_{ii'}$ and neglect $(1 - \chi_{ii'})^\sharp.$ The remaining source terms are
 $$ g \chi_{ii'} \chi_{ji,+1} \chi_{i'j',-1}^\sharp {b}_{ii'}^+, \,\, j' \in I^+_{i'}, \quad \mbox{and} \quad g \chi_{ii'} \chi_{ji,+1} \chi_{j'i'}^\sharp {b}_{ii'}^+, \,\, j' \in I^-_{i'}, \, j' \neq i.$$
Conditions \eqref{new:1605} and \eqref{new:1605:2} apply to the first and second terms above, respectively.
\medskip

 $\,\,\bullet\bullet$ If $q = - 1$ and $p = - 1,$ the source and phase are
  $$  g^* \chi_{ji,-1} \Big( 1 - \sum_{j' \in I^+_{i'}} (\chi_{i'j'})_{-1} - \sum_{j' \in I^-_{i'}} \chi_{j'i'}\Big)^\sharp \,  {b}_{ii',-1}^- \quad \mbox{and} \quad -(\l_{i',+1} - \o - \l_i)_{-1}.$$
By the same arguments as above, it suffices to handle the source terms
 $$ g^* \chi_{i'i,-1} \chi_{ji,-1} \chi_{i'j',-1}^\sharp b^-_{ii',-1}, \,\, j' \in I^+_{i'}\, j' \neq i, \quad \mbox{and} \quad g^* \chi_{i'i,-1} \chi_{ji,-1} \chi_{j'i'}^\sharp b_{ii',-1}^-, \,\,j' \in I^-_{i'}. $$
Conditions \eqref{new:1605:2} and \eqref{new:1605} apply to the first and second terms above, respectively.

\medskip

Next we turn to $v_i/u_{i'j'}^\pm$ coupling terms, corresponding to $\a = (i,\out),$ $\b = (i',j',q',\inn):$ 
\medskip

 $\,\,\bullet\bullet$ If $q' = 1$ and $p = 1,$ the source and phase are
   $$ g \Big(1 - \sum_{j\in I^+_i} \chi_{ij,+1} - \sum_{j \in I^-_i} \chi_{ji,+2}\Big) \chi_{i'j'}^\sharp b_{ii',+1}^+ \quad \mbox{and} \quad (\l_{i,+1} - \o - \l_{i'})_{+1}.$$
  By the same arguments as above, it suffices to handle the source terms
  $$ g \chi_{ii',+1} \chi_{ij,+1} \chi_{i'j'}^\sharp b_{ii',+1}^+, \,\, j \in I^+_i, \, j \neq i', \quad \mbox{and} \quad g \chi_{ii',+1} \chi_{ji,+2} \chi_{i'j'}^\sharp b_{ii',+1}^+, \quad j \in I^-_i.$$
  Conditions \eqref{new:1605:2} and \eqref{new:1605} apply to the first and second terms above, respectively.

 \medskip

 $\,\,\bullet\bullet$  The other three terms associated with $v_i/u_{i'j'}^\pm$ couplings are entirely similar.

\medskip

$\bullet$ {\it ``out-out'' coupling terms} and phases are, in the case $p = 1:$
 $$ g \big( 1 - \sum_{j \in I^+_i} \chi_{ij} - \sum_{j \in I^-_i} \chi_{ji,+1}\big)  \big( 1 - \sum_{j' \in I^+_{i'}} \chi_{i'j,-1} - \sum_{j' \in I^-_{i'}} \chi_{j'i'}\big) b_{ii'}^+, \quad \l_{i,+1} - \o - \l_{i'},$$
 and in the case $p = -1:$
 $$g^* \big( 1 - \sum_{j \in I^+_i} \chi_{ij,-2} - \sum_{j \in I^-_i} \chi_{ji,-1}\big) \big( 1 - \sum_{j' \in I^+_{i'}} \chi_{ij,-1} - \sum_{j' \in I^-_{i'}} \chi_{j'i'}\big)^\sharp b_{ii',-1}^-,\quad -(\l_{i',+1} - \o - \l_{i})_{-1}.$$
Employing the same arguments as in the ``in-out" case, we are reduced to considering the source terms, for $p = 1:$
 $$ g \chi_{ii'} \big( \chi_{ij_1} + \chi_{j_2i,+1}\big) \big( \chi_{i'j'_1,-1} + \chi_{j'_2i'}\big) b_{ii'}^+, \quad j_1 \in I^+_i \setminus \{i'\}, \,\, j_2 \in I^-_i, \,\, j'_1 \in I^+_{i'}, \,\, j'_2 \in I^-_{i'} \setminus \{i\}.$$
To the terms involving $j_1$ and $j'_2,$ condition \eqref{new:1605:2} applies. To the terms involving $j_2$ and $j'_1,$ condition \eqref{new:1605} applies. The case $p = -1$ is handled in the same way.
\end{proof}

\begin{rema}\phantomsection \label{rem:ontransp} At first sight it might like look condition \eqref{new:1605} is too strong for our purposes, since it involves intersections of only two resonant sets, while the interaction coefficients $B_{p\a\b}$ involve three frequency cut-offs. A look at \eqref{res:homo} shows however that we cannot do with less than \eqref{new:1605}. Indeed, given $(i,i') \in {\mathfrak R}_0,$ for the normal form reduction to go through the symbol in \eqref{res:homo} has to be bounded for all values of $j,j'$ (such that $(i,j)$ and $(i',j') \in {\mathfrak R}_0$), including $j = i'.$
\end{rema}

System \eqref{u-a-sys} has size $N \times |\mA|,$ where $|\mA|$ is the cardinal of $\mA$ defined in \eqref{def:mA}. We let
$$ {\bf Q}_{(\a,\b)} = \sum_{p \in \pm 1} e^{i (p + \delta_\b - \delta_\a) \theta} Q_{\ell\a\b}, \qquad \ell = p + \delta_\b - \delta_\a,$$
where $Q_{\ell\a\b}$ is given by Proposition \ref{prop:several-res-new} for relevant indices, meaning all $p, \a,\b$ at the exclusion of cases (a) and (b), and $Q_{\ell\a\b} := 0$ otherwise. We then form a large matrix ${\bf Q} \in \C^{N |\mA| \times N |\mA|}$ by assembling the $N \times N$ blocks ${\bf Q}_{(\a,\b)},$ and, similarly to \eqref{def:check-u}, let
\begin{equation} \label{reduced-several} \check {\bf U} := \Big(\Id + \sqrt \e \op_\e\big({\bf Q}(\sqrt \e t)\big)\Big)^{-1} {\bf U}(\sqrt \e t), \quad {\bf U} := ({\bf u}_\a)_{\a \in \mA}, \quad \check {\bf U} =: (\check {\bf u}_\a)_{\a \in \mA},
\end{equation}
It can then be checked, exactly as in the proof of Corollary \ref{cor:normal-form}, that $\check {\bf U}$ solves the reduced system
\be \label{several:reduced} \d_t \check {\bf u}_\a + \frac{1}{\sqrt \e} \op_\e(i {\bm \mu}_\a) \check {\bf u}_\a = \op_\e(\check {\bf B}_{\a\b}) \check {\bf u}_\b + \sqrt \e \check {\bf F}_\a,\ee
where the only remaining source terms correspond to constructive interactions between non-transparent resonances, that is
 \be \label{def:checkbfB} \check {\bf B}_{\a\b} := \left\{\begin{aligned}
  {\bf B}_{1\a\b}, & \quad \mbox{if $\a = (i,j,1,\inn),$ $\b = (j,i,-1,\inn),$ with $(i,j) \in {\mathfrak R}_0,$}\\ {\bf B}_{-1\a\b}, & \quad \mbox{if $\a = (i,j,-1,\inn),$ $\b = (j,i,1,\inn),$ with $(j,i) \in {\mathfrak R}_0,$} \\  0, & \quad \mbox{otherwise.}
   \end{aligned}\right.\ee
According to \eqref{b-abp}, if $\a = (i,j,1,\inn)$ and $\b = (j,i,-1,\inn),$ then ${\bf B}_{1\a\b}$ simplifies into
  $$ {\bf B}_{1\a\b} = \chi_{ij}^\sharp \Pi_{i,+1} B_1 \Pi_{j},$$
  and if $\a = (i,j,-1,\inn)$ and $\b = (j,i,1,\inn),$ then ${\bf B}_{-1\a\b}$ simplifies into
   $$ {\bf B}_{-1\a\b} = \chi_{ji}^\sharp \Pi_i B_{-1} \Pi_{j,+1}. $$
In particular, in \eqref{several:reduced} there are no fast oscillations in the source. System \eqref{several:reduced} is the reduced system, analogous to system \eqref{eq:check-u} in the case of one non-transparent resonance.

\subsection{Space-frequency localization} \label{several:loc}

We now isolate the family, indexed by ${\mathfrak R}_0,$ of $2N \times 2N$ subsystems in \eqref{several:reduced} that correspond to non-zero coupling terms $\check {\bf B}_{\a\b},$ just like the subsystem corresponding to resonance $(1,2)$ was naturally isolated from the rest of \eqref{eq:check-u}.

\medskip

We let $\check U_{ij} := \big( \check u_{ij}^+, \,  \check u_{ji}^-\big),$ and
 \be\label{loc:ij}
 V_{ij} :=\op_\e(\chi_{ij}) \big( \varphi_{ij} \check U_{ij} \big), \quad W_{ij1}  := \op_\e\big( \chi_{ij}\big)\big((1-\varphi_{ij})\check U_{ij}\big), \quad W_{ij2}  := \big(1-\op_\e(\chi_{ij})\big)\check U_{ij},
 \ee
so that, just like in \eqref{re:checkU},
$ \check U_{ij} = V_{ij} + W_{ij1} + W_{ij2}.$ \label{encore x cut}The spatial cut-off $\varphi_{ij}$ is identically equal to one in a large neighborhood of $x_0;$ it is associated with $\varphi_{ij}^\flat,$ $\varphi_{ij}^\sharp$ such that $\varphi_{ij}^\flat \prec \varphi_{ij} \prec \varphi_{ij}^\sharp.$

We then verify exactly as in Lemma \ref{lem:sys-V} that, with this further coordinatization, the reduced system \eqref{several:reduced} takes the form of the prepared system

\be \label{several:prepared}
 \left\{ \begin{aligned} \d_t V_{ij}  + \frac{1}{\sqrt\e} \ope(M_{ij}) V_{ij} & = \sqrt \e F_{V_{ij}},  \quad (i,j) \in {\mathfrak R}_0, \\  \d_t {\bf W} + \frac{1}{\sqrt \e} \op_\e(i {\bf A}) {\bf W} &  = \op_\e(D) {\bf W} + \sqrt \e {F}_{\bf W},\end{aligned}\right.
\ee
where
\begin{itemize}
 \item the interaction matrices $M_{ij}$ are
 \be \label{def:Mij} M_{ij}(\e,t,x,\xi) := \chi_{ij}^\sharp \left(\begin{array}{cc} i  (\l_{i,+1} - \o) & - \sqrt \e  \varphi_{ij}^\sharp g \Pi_{i,+1} B_1 \Pi_j \\ - \sqrt \e  \varphi_{ij}^\sharp  g^* \Pi_j B_{-1} \Pi_{i,+1} &  i\l_j \end{array}\right).\ee
 \item the variable ${\bf W}$ is the collection of all ${\bf u}_\a$ for which the source terms $\check {\bf B}_{\a\b}$ are all equal to zero\footnote{Corresponding, by definition of $\check {\bf B}_{\a\b}$ in \eqref{def:checkbfB}, to all $\a$ such that for all $\b,$ there holds $(\a,\b) \neq ((i,j,1,\inn), (j,i,-1,\inn)),$ for any $(i,j) \in {\mathfrak R}_0,$ and $(\a,\b) \neq  ((i,j,-1,\inn), (j,i,1,\inn)),$ for any $(j,i) \in {\mathfrak R}_0.$},
   and all $(W_{ij1}, W_{ij2}),$ for $(i,j) \in {\mathfrak R}_0,$
  \item the Fourier multiplier ${\bf A}$ is diagonal and purely imaginary,
 \item the source term $D$ depends linearly on $(1 - \varphi_{ij}^\flat) u_{0,\pm 1}.$ In particular it can be made arbitrarily small, by choosing the support of $\varphi_{ij}^\flat$ large enough, since $u_{0,\pm 1}$ are decaying at infinity,
 \item the source terms $F_{V_{ij}}$ and $F_{\bf W}$ are bounded as was $F_\a$ in \eqref{bd:F-alpha}.
 \end{itemize}

In \eqref{several:prepared}, we see $2N \times 2N$ interaction systems, indexed by $(i,j) \in {\mathfrak R}_0,$ that are weakly coupled with a large system in ${\bf W}.$ The system in ${\bf W}$ is symmetric hyperbolic with very small linear and nonlinear source terms. The coupling terms between the subsystem indexed by ${\mathfrak R}_0$ and the system in ${\bf W}$ are the source terms $F$ in the right-hand sides. We call these coupling terms weak because of the $\e^{1/2}$ prefactors. System \eqref{several:prepared} is the prepared system, analogous to system \eqref{V11}.

\subsection{Conclusion} \label{several:estimates}

 From \eqref{several:prepared}, the estimates are as in the case of one non-transparent resonance.

 Assumption \ref{ass:several-res}(iii) allows for large-rank interaction coefficients for those resonances $(i,j) \in {\mathfrak R}_0$ for which one interaction coefficient is identically zero. In such cases, the bounds for the symbolic flow are trivial. Indeed, equation \eqref{S} page \pageref{S} reduces to
  $$ \d_t S + \frac{1}{\sqrt \e} \left(\begin{array}{cc} i \chi_i \mu_i & -\sqrt \e \tilde b_{ij} \\ 0 & i \chi_1 \mu_j \end{array}\right) S = 0,$$
  a triangular system of ordinary differential equations. The corresponding bounds are
  \be \label{bd:g=0} |\d_x^\a S(\t;t)| \lesssim 1 + (t - \t)^r,\ee
  where $r$ is the rank of $b_{ij}^+.$

 For the other pairs $(i,j) \in {\mathfrak R}_0,$ we have an amplification coefficient $\g_{ij}$ defined in \eqref{def:several-gamma}. With $\g_{ij}$ are associated an upper rate of growth $\g_{ij}^+,$ defined as in \eqref{def:g+}, and, unless $\g_{ij} = 0,$ a lower rate of growth $\g_{ij}^-,$ defined as in the statement of Lemma \ref{lem-S0}. Both depend on how tightly we cut around the resonance. This is quantified by parameters $h > 0$ (in frequency space) and $\rho > 0$ (in physical space).

  In the unstable case ${\bf \G} > 0,$ we isolate $(i_0,j_0)$ such that $\g = \g_{i_0j_0} = \max_{(i,j) \in {\mathfrak R}_0} \g_{ij} > 0.$  The other components have slower (maybe not strictly slower) rates of growth.

  Given $(i,j) \in {\mathfrak R}_0,$ if $\g_{ij} > 0$ we use the Duhamel representation formula as in Section \ref{est-upper-V0}, and this gives an upper bound for $V_{ij}$ that is analogous to \eqref{up:0}.

  Given $(i,j) \in {\mathfrak R}_0,$ if $\g_{ij} < 0,$ then we use a symmetrizer to estimate $V_{ij}$ as in Section \ref{sec:stability1}, and if $\g_{ij} = 0,$ we simply use bound \eqref{bd:g=0}.

  This gives upper bounds as in Propositions \ref{lem:Sob-bd} and \ref{prop:ex-Sob}:
  \be \label{up:up2} \| (V_{ij})_{(i,j) \in {\mathfrak R}_0}, {\bf W})(t') \|_{\e,s} \lesssim \e^{K - \k} |\ln \e|^* e^{t \g^+},\ee
where $\k$ intervenes in the definition of $\dot u$ in the first line of Section \ref{sec:several-coord}. The upper rate of growth $\g^+$ is defined from $\g$ by \eqref{def:g+}. The above estimate is valid for $T < T_0,$ where $T_0$ is defined in \eqref{def:T0-K0}.

  By an appropriate choice of the initial datum (namely, \eqref{datum}), we derive a lower bound for the distinguished variable $V_{i_0j_0}$ as in Section \ref{sec:lower}, and conclude as in Section \ref{sec:end-insta}.

  In the stable case ${\bf \G} < 0,$ all resonances $(i,j) \in {\mathfrak R}_0$ are symmetrizable. We define a symmetrizer by blocks, and proceed as in Section \ref{sec:stability1}.

\section{Proof of Theorem \lowercase{\ref{th-three}}} \label{sec:proof-three}

\subsection{Preparation} \label{sec:12-prep}

The perturbation variable is defined by
\be \label{def:dot-u3} u =: u_a + \e^{1/2} \dot u,
\ee
The $\e^{1/2}$ prefactor will provide the necessary cushion as we follow the growth of a component of the solution that might not have maximal growth. There holds $\| \dot u(0)\|_{\e,s} = O(\e^{K - 1/2}).$

 We use the coordinatization of Section \ref{sec:several-coord}, leading to system \eqref{u-a-sys} in the variable
 ${\bf U}:= ({\bf u}_\a)_{\a \in \mA}.$ The source ${\bf F}_\a$ satisfies \eqref{bd:F-alpha} with $\k = 1/2;$ explicitly
 \be \label{bd:F-alpha2}| {\bf F}_\a |_{L^2} \lesssim (1 + |\dot u|_{L^\infty}) |\dot u|_{L^2} + \e^{K_a - 1/2}.\ee

 In the present context, the non-transparent resonance $(1,2)$ plays a distinguished role. The associated variables are $(u_{12}^+, u_{21}^-) \in \C^{2N},$ as defined in \eqref{u-ij+}-\eqref{u-ij-}. We now eliminate from the system in ${\bf U}$ all the coupling terms ${\bf B}_{p\a\b}$ which involve $u_{12}^+$ or $u_{21}^-,$ at the exception of the crucial interaction coefficients $b_{12}^+$ and $b_{21}^-:$

 \begin{prop}\phantomsection \label{prop:all-res-are-amplified} Under Assumption {\rm \ref{ass:12-ampli}(ii)}, equation
\be \label{homo-newbis}
 i \big( - \ell \o + {\bm \mu}_{\a,+\ell} - {\bm \mu}_\b \big) Q_{\ell\a\b} = B_{p\a\b} + \e R_0, \qquad \ell = p - \delta_{\a} +  \delta_{\b}
\ee
has a solution $Q_{\ell\a\b} \in S^0$ for $\a = (1,2,+,\inn)$ and all $(p,\b)$ unless $B_{p\a\b} = \chi_{12}^\sharp b_{12}^+,$ and also for $\a = (2,1,-,\inn)$ and all $(p,\b)$ unless $B_{p\a\b} = \chi_{12}^\sharp b_{21}^-.$ In \eqref{homo-newbis}, $R_0$ is such that $\op_\e(R_0)$ is a uniform remainder in the sense of Definition {\rm \ref{def:ur}}.
\end{prop}

 \begin{proof} We revisit the proof of Proposition \ref{prop:several-res-new}. We are going to use the separation condition \eqref{sep}.

 In a first step, $\a = (1,2,+,\inn)$ and $\b = (i',j',q',\inn)$ are given. There are four corresponding coupling terms. For $p = 1$ and $q =1,$ the source has prefactor $\chi_{ij,+1} \chi_{i'j'} \chi_{ii',+1},$ which vanishes identically by the separation condition \eqref{sep}. The same holds, with a different combination of cut-offs, for $p =1$ and $q = -1,$ and $p = -1$ and $q = -1.$ The remaining term is $\chi_{12}^\sharp b_{12}^+.$

 Next $\b = (i',\out)$ is given, with the same $\a.$ Here the reduction is non trivial: we use the fact that any phase is bounded away from its corresponding resonant set to reduce the analysis to frequency sets which are intersections of supports of cut-offs functions, as in the proof of Proposition \ref{prop:several-res-new}. After this is done, we find in the case $p = 1$ products $\chi_{ii'} \chi_{ij} (\chi_{i'j',-1} + \chi_{j''i'}),$ with $j'' \neq i,$ in factor of the source. By separation, these products vanish identically. The case $p = -1$ is similar.

 The case $\a = (2,1,-,\inn)$ is treated in the same way, by examination of the corresponding cases in the proof of Proposition \ref{prop:several-res-new}: all ``in-in" coupling terms are trivial, save for $b_{21}^-,$ and all ``in-out" coupling terms are trivial, except on frequency sets over which the phase is bounded away from zero.
 \end{proof}

With Proposition \ref{prop:all-res-are-amplified}, we define a change of variable ${\bf Q}$ as we did in Section \ref{sec:several-normal-form}, and from there define $\check {\bf U}$ by \eqref{reduced-several}. The variable $\check U_{12} = (\check u_{12}^+, \check u_{21}^-)$ describing resonance $(1,2)$ satisfies
 \be \label{eq:12-ampli} \d_t \check U_{12} + \frac{1}{\sqrt \e} \left(\begin{array}{cc} i \mu_1 & - \chi_{12}^\sharp g(\sqrt \e t, x) b_{12}^+ \\ - \chi_{12}^\sharp g(\sqrt \e t, x)^* b_{21}^- & i \mu_2 \end{array}\right) \check U_{12} = \sqrt \e \check F_{12},\ee
 while the other variables in $\check {\bf U},$ which we denote $({\bf u}_{\a'})_{\a' \in \mA'},$ satisfy
  \be \label{eq:other-ampli} \d_t {\bf u}_{\a'} + \frac{1}{\sqrt \e} \op_\e(i {\bm \mu}_{\a'}) \check {\bf u}_{\a'} = \sum_{p = \pm 1} \sum_{\b' \in \mA'} \op_\e({\bf B}_{p\a'\b'}) \check {\bf u}_{\b'} + \sqrt \e \check {\bf F}_{\a'}.
  \ee
 where $\mA'$ is the set of indices alien to the $(1,2)$ resonance: $\mA' := \mA \setminus \big\{ (1,2,+,\inn), (2,1,-,\inn)\big\}.$ Equations \eqref{eq:12-ampli} and \eqref{eq:other-ampli} are coupled only via the source terms $F.$

 Next we introduce unknowns $V_{12}, W_{12,1}, W_{12,2}$ that are local to the $(1,2)$ resonance, as in \eqref{loc:ij}, and arrive at the prepared system:
\be \label{several:prepared2}
 \left\{ \begin{aligned} \d_t V_{12}  + \frac{1}{\sqrt\e} \ope(M_{12}) V_{12} & = \sqrt \e F_{V_{12}}, \\  \d_t {\bf W} + \frac{1}{\sqrt \e} \op_\e(i {\bf A}) {\bf W}  &  = \op_\e(D) {\bf W}  + \sqrt \e {F}_{{\bf W}},\end{aligned}\right.
\ee
analogous to \eqref{several:prepared},
 where
\begin{itemize}
 \item ${\bf W} = \big( ({\bf u}_{\a'})_{\a' \in \mA'}, W_{12,1}, W_{12,2}),$ so that ${\bf U}$ can be reconstructed from $(V_{12}, {\bf W})$ (as in \eqref{def:check-u} and \eqref{re:checkU}), and there holds $\| V_{12} \|_{\e,s} + \| {\bf W} \|_{\e,s} \lesssim \| {\bf U}\|_{\e,s};$
 \item the interaction matrix $M_{12}$ is defined in \eqref{def:Mij}, with $(i,j) = (1,2);$
  \item the Fourier multiplier ${\bf A}$ is diagonal and purely imaginary;
 \item the important difference with \eqref{several:prepared} is in the source $D,$ which here is
  $$ D = \left(\begin{array}{cc} D_{12} & 0 \\ 0 & \sum_{p = \pm 1} ({\bf B}_{p\a'\b'})_{\a',\b' \in \mA'}\end{array}\right), \quad D_{12} = \left(\begin{array}{cc} (1 - \varphi_{12}^\flat) \check {\mathcal B}_{12} & 0 \\ 0 & 0  \end{array}\right),$$
  where $\check {\mathcal B}_{12}$ is defined as $\check {\mathcal B}$ in \eqref{eq:check-u}. The point is that $D$ is {\it not} small: there holds\footnote{Multiplicative constants do matter here, since upper bounds translate into growth rates. In \eqref{est:D}, $|\cdot|$ denotes the sup norms in $\C^N$ and $\C^{N \times N},$ in accordance with notation set up on page \pageref{matrix:norm}.}
 \be \label{est:D} |D(t,x,\xi)| \leq |u_0(\sqrt \e t,x)| |B|.
 \ee
 \item the source terms $F_{V_{12}}$ and $F_{\bf W}$ satisfy bound \eqref{bd:F-alpha2}.
 \end{itemize}

\subsection{Upper bounds} \label{sec:12-estimates}

The estimates for \eqref{several:prepared2} differ from the analogous ones conducted in Section \ref{est-upper-V0-a} only in the estimate \eqref{est:D} for $D.$

\begin{prop}\phantomsection \label{prop:existence3} Given
\be \label{tildeT0} T < \max\Big(\, \frac{K-1/2}{|B| |\hat a|_{L^1}}\, , \, \frac{K - (d+1)/2}{|B| |a|_{L^\infty}}\, \Big),\ee
if $\e$ is small enough, then the solution to \eqref{several:prepared2} issued from $\| (V_{12}, {\bf W})(0) \|_{\e,s} = O(\e^K),$ with $s > d/2,$ is
defined over $[0, T |\ln \e|],$
with the estimate
\be \label{toprove:propex3} \| (V_{12}, {\bf W})(t) \|_{\e,s} \lesssim \e^{K - 1/2} |\ln \e|^* e^{t |B| |a|_{L^\infty}}.\ee
\end{prop}

\begin{proof} We follow Section \ref{est-upper-V0-a}.

{\it First step: ${\mathcal F}L^1$ estimates.} We follow the proof of Lemma \ref{lem:Linfty-bd}. Modulo differences in notation, the only change is in the bound for $D.$ Here we have, using \eqref{est:D}:
$$
| \op_\e(D) {\bf W} |_{{\mathcal F}L^1} \leq |B| |\hat u_0(\sqrt \e t)|_{L^1} |\hat {\bf W}|_{L^1}.
$$
This leads to
 $$  |(\hat V_{12}, \hat {\bf W})(t)|_{L^1} \lesssim \e^{K-1/2} |\ln \e|^* + \Big(|B| (1 + \delta_{\varphi_0})(|\hat a|_{L^1} + O(\e^{\eta_1})\Big) \int_0^t |(\hat V_{12}, \hat {\bf W})(t')|_{L^1} \, dt',$$
having assumed the bound $| \hat V_{12}, \hat {\bf W}|_{L^1} \leq \e^{\eta_1}$ up to time $t.$ Above $\delta_{\varphi_0} > 0$ can be made arbitrarily small. By the same argument as in the proof of Lemma \ref{lem:Linfty-bd}, this shows that the above bound propagates: there holds $\sup_{0 \leq t \leq T_1 |\ln \e|} |(V_{12}, {\bf W})(t)|_{{\mathcal F}L^1} \leq \e^{\eta_1},$ for any $T_1 < (K - 1/2)/(|B| |\hat a|_{L^1}),$ for some $\eta_1 = \eta_1(T_1).$

\medskip

{\it Second step: Sobolev estimates based on ${\mathcal F}L^1$ estimates.} We follow the proof of Proposition \ref{lem:Sob-bd}. In the present context, the only significant difference is that instead of \eqref{up:w2}, we have here
 $$ |\ope(D) \Lambda^s {\bf W}|_{L^2} \leq |B| |u_0(\sqrt \e t)|_{L^\infty} \| {\bf W} \|_{\e,s}.$$
With the above first step, the source satisfies
\be \label{bd:FtildeU} \|F_{{\bf W}}\|_{\e,s} \lesssim \| V_{12}, {\bf W} \|_{\e,s}  + \e^{K_a - 1/2}.\ee
This implies the upper bound
 $$ \| {\bf W} \|_{\e,s}^2 \lesssim \e^{2(K - 1/2)} + \big(|B| |a|_{L^\infty} + \e^{1/2} |\ln \e|\big) \int_0^t \| V_{12}, {\bf W} \|_{\e,s}^2 + \e^{K_a} \int_0^t \| {\bf W}(t')\|_{\e,s} \, dt'.$$
The estimate for $V_{12}$ is identical to the estimate in $V$ \eqref{bd:V}. Just like in the proof of Proposition \ref{lem:Sob-bd}, this implies, via Lemma \ref{lem:g}, the bound \eqref{toprove:propex3} for $\dsp{t \leq \frac{K - 1/2}{|B| |\hat a|_{L^1}}.}$
Here we are using Lemma \ref{lem:g} with $\delta_0 = |B| |a|_{L^\infty},$ $\g_0 = \g_{12},$ and $\max(\delta_0,\g_0) = |B| |a|_{L^\infty}.$

\medskip

{\it Third step: Sobolev estimates based on the $L^\infty \hookrightarrow H^s$ embedding.} In the proof of Proposition \ref{prop:ex-Sob}, we replace $\e^{\eta_2}$ with $C(T_2) > 0$ in the definition of $J.$ Indeed, in Proposition \ref{prop:ex-Sob}, the role of $\eta_2 > 0$ was to guarantee smallness of the nonlinear terms. Here so long as $\dot u$ is bounded, the nonlinear terms are small, since \eqref{def:dot-u3} implied the better estimate \eqref{bd:FtildeU}. This gives, exactly as in the proof of Proposition \ref{prop:ex-Sob}, the bound \eqref{toprove:propex3} for
 $\dsp{t \leq \frac{K - (d+1)/2}{|B| |a|_{L^\infty}}.}$
\end{proof}

\subsection{Conclusion} \label{sec:conclu-th3}

The lower bound is exactly as in the main proof (Section \ref{sec:lower}). Starting from \eqref{low:low} with $V_{12}$ in place of $V,$ this gives
\be \label{low:3} \begin{aligned} |V_{12}(T |\ln \e|)|_{L^2(B(x_0,\rho))} & \geq C(\rho) \e^{K - T \g_{12}^-} - C |\ln \e|^{*} \e^{K + 1/2 - T \g_{12}^+} \\ & - C \sqrt \e |\ln \e|^* \int_0^{T |\ln \e|} e^{(T |\ln \e| - t')\g_{12}^+} | F_{V_{12}}(t')|_{L^2} dt'.%
 \end{aligned}
 \ee
 The source $F_{V_{12}}$ satisfies estimate \eqref{bd:FtildeU}. With Proposition \ref{prop:existence3} and its proof, this gives
 $$ | F_{V_{12}}(t) |_{L^2} \leq \e^{K-1/2} |\ln \e|^* \e^{t |B| |a|_{L^\infty}} + \e^{K_a - 1/2},$$
 so that $|V_{12}(T |\ln \e|)|_{L^2(B(x_0,\rho))}$ is bounded from below by
 $$ \e^{K - 1/2 - T \g_{12}^-} - C |\ln \e|^{*} \e^{K - T \g_{12}^+} - C |\ln \e|^* \e^{K - T |B| |a|_{L^\infty}},$$
 up to a multiplicative constant. We now impose for the final observation time the upper bound
\be \label{cond:T3} \frac{1}{2} - T (|B|  - \g_{12}) |a|_{L^\infty} > 0,
\ee
corresponding to $T < \tilde T_0,$ with notation introduced in \eqref{def:tildeT0}.
For $T$ satisfying \eqref{tildeT0} and \eqref{cond:T3}, for $h, \rho, c$ and $\e$ small enough, there holds the lower bound
$$ |V_{12}(T |\ln \e|)|_{L^2(B(x_0,\rho))} \geq C \e^{K - 1/2 - T \g_{12}^-}.$$
This gives an amplification exponent $K'_0$ that is bounded from below by $K - T \g_{12} |a|_{L^\infty},$
corresponding to \eqref{def:K0prime}. The end of the proof, going up the chain of changes of variables from $V_{12}$ to $\dot u,$ is identical to Section \ref{sec:end-insta}.

\section{Proof of Theorem \lowercase{\ref{th:4}}} \label{sec:pftheobetterloc}

 We posit here $u =: u_a + \e^{1/2} \dot u,$ as in \eqref{def:dot-u3}, and then follow closely the proof of Theorem \ref{th-two}. The $\e^{1/2}$ factor in the definition of $\dot u$ modifies the estimates for the source terms $F,$ as in the proof of Theorem \ref{th-three}. This gives an existence time
  $$ T''_0 = \max\Big(\, \frac{K-1/2}{|B| |\hat a|_{L^1}}\, , \,\, \frac{K - (d+1)/2}{\g |a|_{L^\infty}}\, \Big),$$
  with $\g$ defined in \eqref{def:several-gamma}.
 Indeed, the numerators are as in \eqref{tildeT0}, since the ansatz is as in \eqref{def:dot-u3}, and the denominators are as in $T_0$ \eqref{def:T0-K0}, since the assumptions are the same as in Theorem \ref{th-two}. On $[0, T \sqrt \e |\ln \e|],$ the upper bounds are identical to \eqref{up:up2}, except for the size $O(\e^{K-1/2})$ of the initial datum:
  \be \label{up:up4} \| (V_{ij})_{(i,j) \in {\mathfrak R}_0}, {\bf W})(t') \|_{\e,s} \lesssim \e^{K-1/2} |\ln \e|^* e^{t \g^+},\ee
  with $\g = \g_{i_0j_0}$ \eqref{def:several-gamma}, where $V_{ij}$ and ${\bf W}$ are defined exactly as in Section \ref{several:loc}. The estimate for the source term $F_{V_{i_0 j_0}}$ is identical to \eqref{bd:FtildeU}:
  \be \label{bd:F4} |F_{V_{i_0j_0}}|_{L^2} \lesssim \| (V_{ij})_{(i,j) \in {\mathfrak R}_0}, {\bf W} \|_{L^2} + \e^{K_a}.\ee

 Now the difference is in the lower bound, which we consider on a small ball $B(x_0,\e^\b).$ We follow the proof of Lemma \ref{lem-S0}, until we arrive at \eqref{for:beta}. In the present context, we replace \eqref{for:beta} with
  $$
   \left(\int_{B(x_0,\e^\b)} e^{2 t a(x) \g} \, dx\right)^{1/2} \geq C \e^{\b d/2} e^{t \g |a|_{L^\infty}},
   $$
   for $\e$ small enough. Since by assumption $\b d/2 < 1/2,$ this term is indeed the leading term in the lower bound for the action of the solution operator on the initial datum, and there holds
  \be \label{low:4} \big|\ope(S(0;t)) V_{i_0j_0}(0) \big|_{L^2(B(x_0,\e^\b))} \geq  C \e^{K} \Big(\e^{\b d/2} e^{t \g^-} - \e^{1/2} |\ln \e|^{*} e^{t \g^+}\Big).
  \ee
 From \eqref{up:up4}, \eqref{bd:F4} and \eqref{low:4}, via an integral representation of $V_{i_0j_0}$ (Section \ref{est-upper-V0}), we deduce as in \eqref{low:low} and \eqref{low:3} the lower bound
 $$  |V_{i_0j_0}(T |\ln \e|)|_{L^2(B(x_0,\e^\b))} \geq C \e^{K + \b d/2 - 1/2 - T \g|a|_{L^\infty}} - C |\ln \e|^{*} \e^{K - T \g^+}.
 $$
Since $\b < 1/d,$ given $T < T''_0,$ as soon as $h$ is small enough, the exponent $K + \b d/2 - 1/2 - T \g |a|_{L^\infty}$ is strictly smaller than the exponent $K - T \g^+,$ yielding a deviation estimate.
The lower bound for the amplification exponent is
$$ K + \frac{\b d}{2} - T''_0 \g |a|_{L^\infty},$$
corresponding to $K''_0$ given in \eqref{def:K''0}.

\section{Proof of Theorem \lowercase{\ref{th:5}}} \label{sec:end-insta2}

 We assume that given $T < T_\infty,$ for $\e$ small enough, the solution to \eqref{0}-\eqref{datum} is defined over $[0, T \sqrt \e |\ln \e|],$ belongs to $C^0([0, T \sqrt \e |\ln \e|], H^s(\R^d)),$ and is uniformly bounded in $(\e,t,x).$ From there, we proceed to prove the deviation estimate \eqref{est:th5}.

Let $K' > 0$ be given. The smaller $K',$ the better the amplification in \eqref{est:th5}. In particular we may assume $K'$ to be smaller than $1/2.$ We posit
$$
 u =: u_a + \e^{K'} \dot u,
$$
 and, following the proof of Theorem \ref{th-two}, we arrive at a system that is identical to \eqref{several:prepared}.  The upper bound for the source terms is here
  $$ \| F_{V_{ij}}, F_W \|_{\e,s} \leq  (1 + \e^{K' - 1/2} |\dot u|_{L^\infty}) \|V_{ij}, W\|_{\e,s} + \e^{K_a- K'},$$
   for $t \in [0, T \sqrt \e |\ln \e|],$ for all $T < T_\infty.$

 Next we follow the proof of Proposition \ref{prop:ex-Sob} in order to derive an upper bound. The control of $L^\infty$ is here a priori given, so that we can forgo the Sobolev embedding. In particular, the above estimate for the source term simplifies into
  $$ \| F_{V_{ij}}, F_W \|_{\e,s} \leq  (1 + \e^{- 1/2 +K'}C(T)) |\dot u|_{L^2} + \e^{K_a- K'},$$
  for some $C(T) > 0.$ This gives
 $$ \| (V_{ij})_{(i,j) \in {\mathfrak R}_0}, {\bf W})(t') \|_{\e,s} \lesssim \e^{K - K'} |\ln \e|^* e^{t \g^+},\qquad \mbox{$t \in [0, T \sqrt \e |\ln \e|],$ for all $T < T_\infty,$}$$
 where $\g = \g_{i_0j_0}$ \eqref{def:several-gamma}.
 For the lower bound, we consider small balls with radius $\e^\b,$ as in Section \ref{sec:pftheobetterloc}. We arrive at
 $$ \begin{aligned} |V_{i_0j_0}(T |\ln \e|)|_{L^2(B(x_0,\e^\b))} & \geq C\e^{K - K' + \b d/2 - T \g |a|_{L^\infty}} - C |\ln \e|^{*} \e^{K - K' + 1/2 - T \g^+} \\ & - C \sqrt \e |\ln \e|^* \int_0^{T |\ln \e|} e^{(T |\ln \e| - t')\g^+} | F_{V_{i_0j_0}}(t')|_{L^2} dt'.
 \end{aligned}
 $$
 With the above upper bound for $F_{V_{12}},$ this gives
 $$ \begin{aligned} |V_{i_0j_0}(T |\ln \e|)|_{L^2(B(x_0,\e^\b))} & \geq C \e^{K - K' + \b d/2 - T \g |a|_{L^\infty}} - C |\ln \e|^{*} \e^{K - K' + 1/2 - T \g^+} \\ & \quad - C |\ln \e|^* \e^{K - T \g^+}.\end{aligned}$$
Since $K' < 1/2,$ among the last two terms in the above upper bound, the biggest is the second one. For $T_\infty - T$ small enough, and $\rho,$ $h$ and $\b$ small enough, there holds
 $$ 0 < K - K' + \b d/2 - T \g |a|_{L^\infty} < K - T \g^+,$$
 and we conclude as in the other cases.

\chapter{Examples} \label{chap:ex}

\section{Raman and Brillouin instabilities} \label{sec:BK}

 We consider here systems \be \label{0rb}
  \d_t u + A(\d_x) u = \frac{1}{\sqrt \e} B(u,u), \qquad A(\d_x) = \sum_{1 \leq j \leq d} A_j \d_{x_j},
   \ee
 where $A$ satisfies Assumption \ref{ass:spectral}, and $B$ is bilinear $\C^N \times \C^N \to \C^N.$ We analyze the stability of special solutions of the form
   \be \label{ua}
   u_a(t,x) = {\bf a}(\eta \cdot x - \tau t) \in \C^N,
  \ee
with
\be \label{ua2}
 \tau \in {\rm sp}\,A(\eta), \quad {\bf a}(y) \equiv ({\bf a}(y), \vec e\,) \vec e, \quad \vec e \in {\rm Ker}\,A(\eta) - \tau \Id, \quad          B(\vec e,\vec e\,) = 0,
 \ee
 where $(\cdot,\cdot)$ denotes the Hermitian scalar product in $\C^N.$

 The systems \eqref{0rb} and solutions \eqref{ua} that we are considering here are not exactly typical of our main framework \eqref{0}-\eqref{generic-data}. We chose however to present our application to the Raman and Brillouin instabilities first, since these are connected to the Euler-Maxwell equations, which later on will be relevant for Klein-Gordon systems (Sections \ref{sec:KG} and \ref{sec:KG2}).

 The untypical features of \eqref{0rb}-\eqref{ua2} are as follows:
  \begin{itemize}

  \item the hyperbolic operator in \eqref{0rb} is non-dispersive, in the sense that $A_0 = 0.$ As a consequence, the symbol $A(\xi)$ is $1$-homogeneous in $\xi.$

  \smallskip

  \item The datum for \eqref{ua} is not highly-oscillating, and the reference solution is complex-valued. The reason is that the solution $u$ to \eqref{0rb} describes a vector of complex {\it envelopes} of highly-oscillating fields. That is, systems \eqref{0rb} and solutions \eqref{ua} can be thought of as resulting from a WKB approximation, and the stability analysis is performed on the limiting system.

  \end{itemize}

 The solutions \eqref{ua}-\eqref{ua2} satisfy Assumption \ref{ass-u-a}, with $v_a = 0,$ $r_a^\e = 0,$ if ${\bf a}$ has a large Sobolev regularity. Here $(\o,k) = (0,0):$ the reference solution is not highly oscillating. In particular, in this context, resonances (Definition \ref{def-reson} page \pageref{def-reson}) are crossing points on the variety:
  $$ {\mathcal R}_{ij} = \{ \xi \in \R^d, \, \l_i(\xi) = \l_j(\xi) \},$$
  where $\l_i$ and $\l_j$ are eigenvalues of $A.$ By homogeneity of $A,$ the eigenvalues are $1$-homogeneous in $\xi.$ In particular, if $\xi \in \S^{d-1}$ is a resonant frequency, then the whole line defined by $\xi$ is resonant. As a consequence, the resonant set is bounded (Assumption \ref{ass:several-res}(i)) only if eigenvalues cross only at $\xi = 0.$

\bigskip

A variant of \eqref{0rb} is given by systems
\be \label{0rb2}
    \d_t u + \frac{1}{\e} A(\d_x) u = \frac{1}{\e} B(u,u),
\ee
with the same assumptions on $A$ and $B.$ We posit the ansatz
\be \label{ansatzb}
 u(\e,t,x) = v\big(\e, \, \e^{-1/2} t, \, \e^{1/2} x \big).
 \ee
 Then, $u$ solves \eqref{0rb2} if and only if $v$ solves \eqref{0rb}. Special solutions to \eqref{0rb} of the form \eqref{ua} correspond to special solutions of \eqref{0rb2} of the form
 \be \label{ua:0rb2}
  u_a(\e,t,x) = {\bf a}\Big( \e^{1/2} \big( \eta \cdot x - \frac{\tau t}{\e}\big)\Big).
 \ee
Note that an instability in short time $O(\sqrt \e |\ln \e|)$ for \eqref{0rb} translates into an instability in shorter time $O(\e |\ln \e|)$ for \eqref{0rb2}, expressing the fact that \eqref{0rb2} is more singular than \eqref{0rb}.

\subsection{Three-wave interaction systems} \label{sec:threewave}

We consider here specifically
                          \be\label{raman}\left\{ \begin{aligned} \d_t u_1+ c_1 \d_x u_1 & = \frac{b_1}{\sqrt\e} \bar u_2 u_3 ,\\
\d_t u_2 + c_2 \d_x u_2 & =\frac{b_2}{\sqrt \e} \bar u_1 u_3,\\ \d_t u_3+ c_3 \d_x u_3 & = \frac{b_3}{\sqrt \e} u_1 u_2,
                          \end{aligned} \right.\ee
where $t\in\R_+,$ $x\in\R,$ $u_i\in\C,$ with velocities $c_i \in \R$ and coefficients $b_i \in \R.$
  The reference solution \eqref{ua}-\eqref{ua2} is explicitly
\be \label{ua:raman} u_a(t,x) = \big( a(x - c_1 t), 0, 0 \big) \in \R^3,
\ee
corresponding to $\vec e = (1,0,0) \in \C^3,$ $(\tau,\eta) = (c_1,1).$

\begin{theo}\phantomsection \label{th:raman} The assumptions of Theorem {\rm \ref{theorem1}} are satisfied by \eqref{raman}-\eqref{ua:raman}, with stability index ${\rm sgn} \,{\bf \G} = {\rm sgn} \,b_2 b_3.$ In the case $b_2 b_3 > 0,$ this implies instability of arbitrarily small initial perturbations of $(a(x),0,0)$ in time $O(\sqrt \e |\ln \e|),$ in the sense of Theorems {\rm \ref{theorem1}}, {\rm \ref{th:4}} and {\rm \ref{th:5}}.
\end{theo}

\begin{proof} Assumption \ref{ass:spectral} is trivially satisfied, since the hyperbolic operator is in diagonal form. Assumption \ref{ass-u-a} is obviously satisfied by the exact solution \eqref{ua:raman}. Eigenvalues cross only at $\xi  = 0.$ As noted above, this implies that the only resonance is $\xi = 0,$ and Assumption \ref{ass:transp}(i) is satisfied. The linearized source $B(u_a)$ is block-diagonal
 $$ B(u_a) = \left(\begin{array}{ccc} 0 & 0 & 0 \\ 0 & 0 & b_2 \bar a \\ 0 & b_3 a & 0 \end{array}\right),$$
 and $(2,3)$ is the only non-transparent resonance, so that Assumption \ref{ass:transp}(ii) is satisfied. The trace of the product of the (rank-one, satisfying Assumption \ref{ass:transp}(iii)) interaction coefficients is
  $ \Gamma_{23} = b_2 b_3 |a|^2,$
  implying the result.
\end{proof}

As an example of systems \eqref{0rb2}, we consider
\be\label{bk}\left\{ \begin{aligned}  \d_t u_1+ \frac{c_1}{\e} \d_x u_1 & = \frac{b_1}{\e}  \bar u_2 u_3,\\
 \d_t u_2 + \frac{c_2}{\e} \d_x u_2 & =  \frac{b_2}{\e} \bar u_1 u_3,\\ \d_t u_3+ c_3 \d_x u_3 & = b_3 u_1 u_2,
                          \end{aligned} \right.\ee
corresponding to scaling \eqref{0rb2}, where, as in \eqref{raman}, $t\in\R_+,$ $x\in\R,$ $c_i \in \R,$ $b_i \in \R,$ $u_i\in\C,$ with reference solution
\be \label{ua:brillouin} u_a(\e,t,x) = \sqrt \e \Big( a\big(\e^{1/2}\big( x - c_1 t/\e\big)\big), 0, 0 \Big) \in \C^3,
\ee
corresponding to \eqref{ua:0rb2} with $\vec e = (1,0,0) \in \C^3,$ $(\t,\eta) = (c_1,1).$

\begin{theo}\phantomsection \label{th:brillouin} If $b_2 b_3 > 0,$ then the solution \eqref{ua:brillouin} of \eqref{bk} is unstable under small initial perturbations of the form $\e^{K+1/2} \phi(\e,\e^{1/2} x),$ with $\sup_{0 < \e < 1} \| \phi(\e,\cdot)\|_{\e,s} < \infty.$ The amplification occurs in time $O(\e |\ln \e|),$ in $L^2(\R^d)$ norm.
\end{theo}

 Naturally, the above ``amplification" is meant in the sense of Theorems \ref{theorem1} and \ref{th:5}. There are three notable  differences with Theorem \ref{th:raman}: the form of the admissible perturbations and the time intervals and radii of the balls over which the instability may be recorded. From our general analysis, we can also deduce a stability result in the case $b_2 b_3 < 0;$ its time range, however, is only $O(\sqrt \e).$

\begin{proof} Given the datum
 $$ \e^{1/2} \big( a(\e^{1/2} x),0,0 \big) + \e^{K + 1/2} \phi(\e, \e^{1/2} x),$$
 we posit the ansatz, similar to \eqref{ansatzb}:
 $$ \left(\begin{array}{c} u_1 \\ u_2 \\ u_3 \end{array}\right) = \left(\begin{array}{ccc} \e^{1/2} v_1 \\ \e^{1/2} v_2 \\ v_3 \end{array}\right)(\e^{-1/2} t, \e^{1/2} x).$$
 Then $(v_1,v_2,v_3)$ solves \eqref{raman}, and Theorem \ref{th:raman} applies. The instability occurs in time $O(\sqrt \e |\ln \e|)$ for the corresponding system in the scaling \eqref{0rb}, hence in time $O(\e |\ln \e|)$ for \eqref{bk}. It occurs in small balls $B(x_0,\rho)$ or $B(x_0,\e^\b)$ for $(v_1,v_2,v_3),$ hence in large balls $B(x_0,\e^{-1/2} \rho)$ in the scaling \eqref{bk}. \end{proof}

 We can find more than one change of variables leading from \eqref{bk} to \eqref{raman}. In particular, we can prove instability of reference solutions of amplitude $O(1)$ to \eqref{bk}, assuming slow variations in space, as follows. Consider the family of solutions to \eqref{bk}:
 \be \label{ua:brillouin2}
  u_a(\e,t,x) =  \Big( a(\e x - c_1 t), 0, 0\Big)
 \ee

\begin{theo}\phantomsection \label{th:brillouin2} The stability of solution \eqref{ua:brillouin2} to \eqref{bk} under perturbations of the form $\e^K \phi(\e, \e x),$ with $\sup_{0 < \e < 1} \| \phi(\e,\cdot)\|_{\e,s} < \infty,$ and $\phi = (\phi_1, \phi_2, \e^{1/2} \phi_3),$ is determined by the sign of $b_2 b_3.$ Instability occurs in time $O(\sqrt \e |\ln \e|)$ and is measured in norm $L^2(\R^d).$
\end{theo}

 (In)stability is meant in the sense of Theorems \ref{theorem1} and \ref{th:5}.

\begin{proof} Given the datum
 $$ \Big(a(\e x) + \phi_1(\e, \e x), \, \phi_2(\e, \e x), \, \e^{1/2} \phi_3(\e, \e x)\,\Big),$$
 we posit the ansatz
 $$ \left(\begin{array}{c} u_1 \\ u_2 \\ u_3 \end{array}\right) = \left(\begin{array}{ccc}  v_1 \\  v_2 \\ \e^{1/2} v_3 \end{array}\right)(t, \e x).$$
 Then $(v_1,v_2,v_3)$ solves \eqref{raman}, and Theorem \ref{th:raman} applies. The instability occurs in small balls for $(v_1,v_2,v_3),$ hence in large balls in the scaling \eqref{bk}.
 \end{proof}

\begin{rema}\phantomsection We note that for homogeneous ($x$-independent) fields $u_i,$ the condition $b_2 b_3 > 0$ implies spectral instability of the equilibrium $(a,0,0).$ In Section 9.2 of \cite{Rau}, Rauch observes that for \eqref{raman} or \eqref{bk} with $\e = 1,$ in the case $b_1 < 0,$ $b_2 > 0,$ $b_3 > 0,$ the quantity $a_1 |u_1|_{L^2}^2 + a_2 |u_2|_{L^2}^2 + a_3 |u_3|_{L^2}^2,$ with $a_1 b_1 + a_2 b_2 + a_3 b_3 = 0,$ is conserved, implying Lyapunov {\emph stability} of the trivial solution, in contrast with instability of the progressing wave \eqref{ua:raman} as given by Theorem {\rm \ref{th:raman}}. Section 9.2 of \cite{Rau} also contains a blow-up result in the case $b_i > 0.$
\end{rema}

\subsection{Derivation of three-wave interaction systems from Euler-Maxwell} \label{sec:der3EM}

The Euler-Maxwell equations describe laser-plasma interactions \cite{Chen,DB,SS}.  In the non-dimensional form introduced in \cite{T2}, they are
 $$\mbox{(EM}) \left\{ \begin{aligned}  \d_t B + \nabla \times E  & =  0,  \nonumber \\
       \d_t E -  \nabla \times  B  & =  \frac{1}{{\bm \epsilon}} ( (1 + n_e + f( n_e)) v_e  - \frac{1}{{\bm \epsilon}} \frac{\theta_i}{\theta_e} ( 1  + n_i + f( n_i)) v_i,\nonumber \\
   \d_t v_e + \theta_e (v_e \cdot \nabla) v_e  & = - \theta_e  \nabla n_e -  \frac{1}{{\bm \epsilon}} ( E + \theta_e v_e \times B),  \nonumber \\
  \d_t n_e + \theta_e \nabla \cdot v_e   +  \theta_e ( v_e \cdot \nabla) n_e  & =  0,  \nonumber \\
\d_t v_i + \theta_i (v_i \cdot \nabla) v_i  & =  - \alpha^2 \theta_i \nabla n_i +  \frac{1}{{\bm \epsilon}}\frac{\theta_i}{\theta_e} ( E + \theta_i v_i \times B ),\nonumber \\
  \d_t n_i +  \theta_i \nabla \cdot  v_i   +  \theta_i (v_i \cdot \nabla) n_i  & =  0, \nonumber
\end{aligned} \right. $$
where $(B,E) \in \R^6$ is the electromagnetic field, $(v_e, v_i) \in \R^6$ are the electronic and ionic velocities, and $(n_e,n_i) \in \R^2$ are the electronic and ionic fluctuations of density from a constant background. The function $f$ is $f(x) = e^x - 1 - x.$ The small parameter with respect to which the WKB will be performed is
$$
 {\bm \epsilon} = \frac{1}{\o_{pe} t_0},
$$
 where $\o_{pe}$ is the electronic plasma frequency $\dsp{\o_{pe} = \sqrt{\frac{4 \pi e^2 n_0}{m_e}}}$ and $t_0$ is the duration of the laser pulse. The other parameters are
 $$
\theta_e = \frac{1}{c}  \sqrt{\frac{\gamma_e T_e}{m_e}}, \quad \theta_i = \frac{1}{c} \sqrt{\frac{\gamma_i T_i}{m_i}}, \quad \a = \frac{T_i}{T_e} \, ,
$$
with $-e$ the charge of the electrons, $+e$ the charge of the ions, $m_e$ and $m_i$ the masses, $n_0$ the background density, $c$ the speed of light, $\g_e$ and $\g_i$ the specific heat ratios, $T_e$ and $T_i$ the temperatures.
For plasmas created by lasers, the parameters $\theta_e, \theta_i$ and $\a$ are typically small, and ${\bm \epsilon}$ is even smaller. Since the mass of the ions is much larger than the mass of the electrons, there holds $\theta_i \ll \theta_e.$

In the right-hand side of the Amp\`ere equation in (EM), we find the current density, and in the right-hand sides of the equations of conservation of momentum, we find the pressure terms in $\nabla n$ and the Lorentz forces in $v \times B.$

We denote $x \in \R$ the direction of propagation of the laser pulse, and coordinatize
$$(x,y) = (x,y_1,y_2) \in \R^3, \qquad M = ( M_x, M_y) = (M_x, M_{y_1}, M_{y_2}) \in \C^3.$$
 The hyperbolic operator in (EM) splits into {\it transverse} and {\it longitudinal} components. We denote
 $$ u_\perp = (B_{y_1},B_{y_2}, E_{y_1}, E_{y_2}, v_{ey_1}, v_{ey_2}), \quad u_\parallel = (B_x, E_x, v_{ex}, n_e, v_{ix}, n_i).$$
 Then, (EM) takes the form
$$
\left\{\begin{aligned}
  \big(\d_t + A_\perp + \frac{1}{{\bm \epsilon}} A_{\perp 0}\big) u_\perp & = F_\perp(u), \\ \big(\d_t + A_\parallel + \frac{1}{{\bm \epsilon}} A_{\parallel0}\big) u_\parallel & = F_\parallel(u),
  \end{aligned}\right.
  $$
  where $F_\perp, F_\parallel$ contain all nonlinear terms: convection, current density and Lorentz force. The transverse and longitudinal hyperbolic operators have the form \eqref{0} with $A_0 \neq 0.$ The {\it transverse} eigenvalues $i \o$ are the eigenvalues of $A_\perp(i \vec k) + A_{\perp 0},$ where $$\vec k = (k, 0,0) \in \R^3;$$ they satisfy
  \be \label{eig:t}
  (- i \o)^4 \Big( \o^2 - k^2 - 1 - \frac{\theta_i^2}{\theta_e^2} \Big)^2 = 0.
  \ee
  The {\it longitudinal} eigenvalues $ i\o$ are the eigenvalues of $A_\parallel(i \vec k) + A_{\parallel 0};$ they satisfy
  \be \label{eig:l}
  (- i \o)^2 \Big(  (\o^2 - \a^2 k^2 \theta_i^2) ( \o^2 - 1 - k^2 \theta_e^2 ) - (\o^2 - k^2 \theta_e^2) \frac{\theta^2_i}{\theta^2_e} \Big) = 0.
  \ee
The above longitudinal dispersion relation takes the alternate form
\be \label{eig:lbis}
 1 = \frac{1}{\o^2 - k^2 \theta_e^2} + \frac{\theta_i^2}{\theta_e^2} \frac{1}{\o^2 - k^2 \a^2 \theta_i^2},
\ee
which we will find useful in Section \ref{sec:Oe}.

We denote (t) the branches of non-trivial solutions of \eqref{eig:t}. We denote $(\ell)$ and (s) the branches of non-trivial solutions to \eqref{eig:l}. The (t) branches correspond to electromagnetic waves with a transverse polarization, typically light sources. The waves corresponding to the $(\ell)$ branches are called {\it electronic plasma waves}. They encode part of the response of the plasma to an incident light source. The (s) branches comprise acoustic waves.

 These satisfy the following expansions in the limit $\theta_i \to 0:$
\be \label{dl:em}
   \o_{\mbox{\footnotesize{$(\ell)$}}}^2 = 1 + k^2 \theta_e^2 + O(\theta_i)^2, \quad \o_{\mbox{\footnotesize{(s)}}}^2 = k^2 \theta_i^2 \Big( \a^2 + \frac{1}{1 + k^2 \theta_e^2} \Big) + O(\theta_i^4),\ee
 locally uniformly in $k.$ For large $k,$
 \be \label{asymptotes:em}   \o_{\mbox{\footnotesize{$(\ell)$}}} =  \pm \theta_e k  +  O\Big(\frac{1}{k}\Big), \qquad  \o_{\mbox{\footnotesize{(s)}}} = \pm \a \theta_i k  +  O\Big(\frac{1}{k}\Big).\ee
The characteristic variety for the Euler-Maxwell equations, meaning the collection of all branches of solutions $\o(k)$ to \eqref{eig:t} and \eqref{eig:l} is pictured on Figure \ref{figem}.

\begin{figure}
\begin{center}\includegraphics[scale=.4]{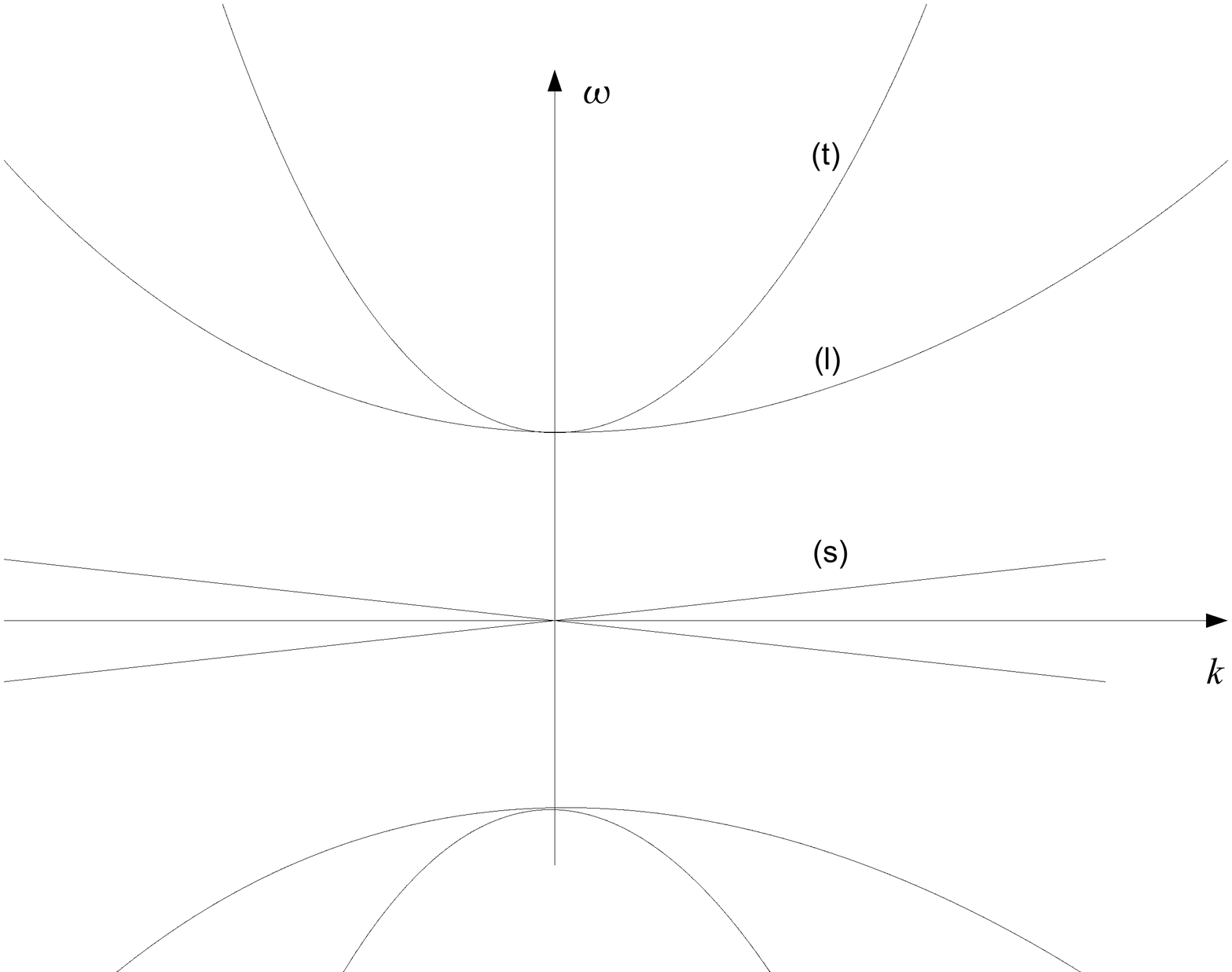}
\caption{The characteristic variety for the Euler-Maxwell equations.}
\label{figem}\end{center}
\end{figure}

\begin{rema}\phantomsection \label{rem:pausader} The proof by Guo, Ionescu and Pausader \cite{GIP} of existence of small, global solutions to the Euler-Maxwell equations is based on an interesting reformulation of {\rm (EM)} as three coupled dispersive equations (system (3.9) in \cite{GIP}) with dispersion relations given by two Klein-Gordon modes and an acoustic mode, in agreement with Figure {\rm \ref{figem}}.
\end{rema}

\subsubsection{Polarization and compatibility conditions} \label{sec:polacompa}

Given $\b = (\o,k),$ we denote ${\bf \Pi}(\b)$ the orthogonal projector onto the kernel of the total operator
$$ \left(\begin{array}{cc} - i \o + A_\perp(i \vec k) + A_{\perp 0} & 0 \\ 0 & - i \o + A_\parallel(i \vec k) + A_{\parallel 0} \end{array}\right).$$

We denote $\b_\perp$ a transverse phase, and $\b_\parallel$ a longitudinal phase.
The transverse polarization condition ${\bf \Pi}(\b_\perp) u  = u$ is explicitly (with $\b_\perp = (\o,k)$):
 \be \label{pola:t} \begin{aligned}
 u_\parallel = 0, \quad (B_{y_1},B_{y_2}) = k \o^{-1} (E_{y_2}, -E_{y_1}), \quad v_{e y_j} = \frac{1}{ip\o} E_{y_j}, \quad v_{i y_j} = - \frac{1}{i p \o} \frac{\theta_i}{\theta_e} E_{y_j}.
 \end{aligned} \ee
The longitudinal polarization condition ${\bf \Pi}(\b_\parallel) u =u$ is explicitly (with $\b_\parallel = (\o,k)$):
\be \label{pola:l} \begin{aligned}
 u_\perp & = 0, \qquad B_x  = 0, \\ v_{ex} & = - \Big(-i \o + i \o^{-1} k^2 \theta_e^2\Big)^{-1} E_x, \qquad v_{iz} = \Big( - i  \o + i \o^{-1} \a^2 k^2 \theta_i^2\Big)^{-1} \frac{\theta_i}{\theta_e} E_x, \\
 n_e & = - \frac{k\theta_e}{\o} \Big(-i \o + i \o^{-1} k^2 \theta_e^2\Big)^{-1} E_x, \\ n_i  & = \frac{\a^2 k \theta_i}{\o} \frac{\theta_i}{\theta_e}  \Big( - i \o + i \o^{-1} \a^2 k^2 \theta_i^2\Big)^{-1} E_x.
 \end{aligned}
 \ee
The transverse compatibility condition ${\bf \Pi}(\b_\perp) u = 0$ is explicitly (with $\b_\perp = (\o,k)$):
 \be \label{compa:t}
  \frac{k}{\o} B_{y_2} = E_{y_1} - \frac{1}{i\o} \Big( v_e - \frac{\theta_i}{\theta_e} v_i\Big)_{y_1}, \quad - \frac{k}{\o} B_{y_1} = E_{y_2} - \frac{1}{i\o} \Big( v_e - \frac{\theta_i}{\theta_e} v_i\Big)_{y_2}.
 \ee
 The longitudinal compatibility condition ${\bf \Pi}(\b_\parallel) u = 0$ is explicitly (with $\b_\parallel = (\o,k)$):
 \be \label{compa:l}
 \begin{aligned}
 \frac{i\o}{\o^2 - k^2 \theta_e^2}\Big( \frac{k\theta_e}{\o} n_e + v_{ex}\Big) - \frac{\theta_i}{\theta_e} \frac{i\o}{\o^2 - k^2 \a^2 \theta_i^2} \Big( \frac{k \a^2 \theta_i}{\o} n_i + v_{ix} \Big) + E_x = 0.
  \end{aligned}\ee

\subsubsection{Initial data and ansatz} \label{sec:init:ansatz}

  We consider initial data with two oscillating phases:
 \be \label{datum:two} u({\bm \epsilon},0,x,y) = \Re e \, \Big( a_1(x,y) e^{i k_1 x/{\bm \epsilon}} + a_2(x,y) e^{i k_2 x/{\bm \epsilon}} \Big),
 \ee
 where the amplitudes $a_1$ and $a_2$ satisfy polarization conditions
 $$
 {\bf \Pi}(\b_1) a_1 = a_1, \quad {\bf \Pi}(\b_2) a_2 = a_2,
 $$
  where $\b_1$ and $\b_2$ are transverse phases. In accordance with \eqref{datum:two}, we consider the two-phase ansatz
 $$ u({\bm \epsilon},t,x,y) = {\bf u}\left({\bm \epsilon}, t, x, y, \frac{k_1 x - \o_1 t}{{\bm \epsilon}}, \frac{k_2 x - \o_2 t}{{\bm \epsilon}}\right),$$
 where ${\bf u}$ is $2\pi$-periodic in both fast variables:
 $$ {\bf u}({\bm \epsilon}, t,x,y,\theta_1,\theta_2) = \sum_{\begin{smallmatrix} p_1 \in \Z \\ p_2 \in \Z \end{smallmatrix}} e^{i p_1 \theta_1 + i p_2 \theta_2} u_{p_1p_2}({\bm \epsilon},t,x,y).$$
 Each amplitude $u_{p_1p_2}$ is decomposed into powers of $\e:$
 \be \label{ansatz:fourier} {\bf u}({\bm \epsilon}, t,x,y,\theta_1,\theta_2) = {\bm \epsilon} \sum_{j \geq 0} \sum_{\begin{smallmatrix} p_1 \in \Z \\ p_2 \in \Z \end{smallmatrix}} e^{i p_1 \theta_1 + i p_2 \theta_2} {\bm \epsilon}^j u_{j,p_1p_2}(t,x,y).\ee
We inject \eqref{ansatz:fourier} into the equation in ${\bf u}$ and sort out powers of ${\bm \epsilon}.$ We will see that a formal solution  that is consistent at order 2 has leading terms which solve the three-wave interaction systems.

\begin{rema}\phantomsection \label{rem:KG EM} It was shown in \cite{T2} that for the Euler-Maxwell equations, the weakly nonlinear regime of geometric optics leads to {\rm linear} transport equations. The weakly nonlinear regime is precisely the one we consider with the ansatz \eqref{ansatz:fourier}. (By constrast, the scaling for the Euler-Maxwell equations that would lead to \eqref{0} is $u = O(\sqrt {\bm \epsilon}).$) However, we derive in Section {\rm \ref{sec:Oe}} below {\rm nonlinear} transport equations, namely the three-wave interaction systems of Section {\rm \ref{sec:threewave}}. The discrepancy with \cite{T2} simply comes from the fact that we are looking at two-phase expansions here, while the result of \cite{T2} holds for single-phase expansions. The Euler-Maxwell equations are {\it less transparent} when two phases are considered, since there are more potential couplings. For instance, the fundamental phase corresponding to Fourier modes $(1,0)$ is produced by the bilinear interaction $(1,0) = (0,0) + (1,0),$ but also by $(1,0) = (1,1) + (0,-1).$ As seen in Section {\rm \ref{sec:Oe}} below, transparency implies $u_{0,(0,0)}= 0,$ so that the interaction $(1,0) = (0,0) + (1,0)$ is not constructive, but the interaction $(1,0) = (1,1) + (0,-1)$ is constructive, as seen on \eqref{evol:t}-\eqref{evol:t2}.
\end{rema}

\subsubsection{Choice of phases} \label{sec:phases}

 The phases $\b_1$ and $\b_2$ are chosen to be transverse, and such that
 \be \label{res:em}\b = \b_1 + \b_2
 \ee is a longitudinal characteristic phase. We let
 $$
  \b_1 = (\o_1, k_1), \quad \b_2 = (\o_2, k_2), \quad \b = (\o, k).
 $$
 In particular, in the following the notation $\b$ does not denote a generic characteristic phase, but {\it the} longitudinal characteristic phase resulting from $\b_1 + \b_2.$

 By symmetry, $-\b_1, -\b_2$ and $-\b$ also are characteristic phases. We assume that $p \b,$ $p \b_1$ and $p \b_2,$ for $p \notin \{-1,0,1\},$ are not characteristic, corresponding to the typical situation for a given $(k_1, k_2).$

 If $\b$ is an electronic plasma wave, then the resonance \eqref{res:em} is associated with the phenomenon known as Raman instability (``scattering of light from optical phonons", \cite{Boyd}, paragraph 8.1). This case is examined in Section \ref{sec:raman}. The corresponding ${\rm (t)}{\rm (t)}(\ell)$ resonance is pictured on Figure \ref{ttl}.

 \begin{figure}\begin{center}
\includegraphics[scale=.4]{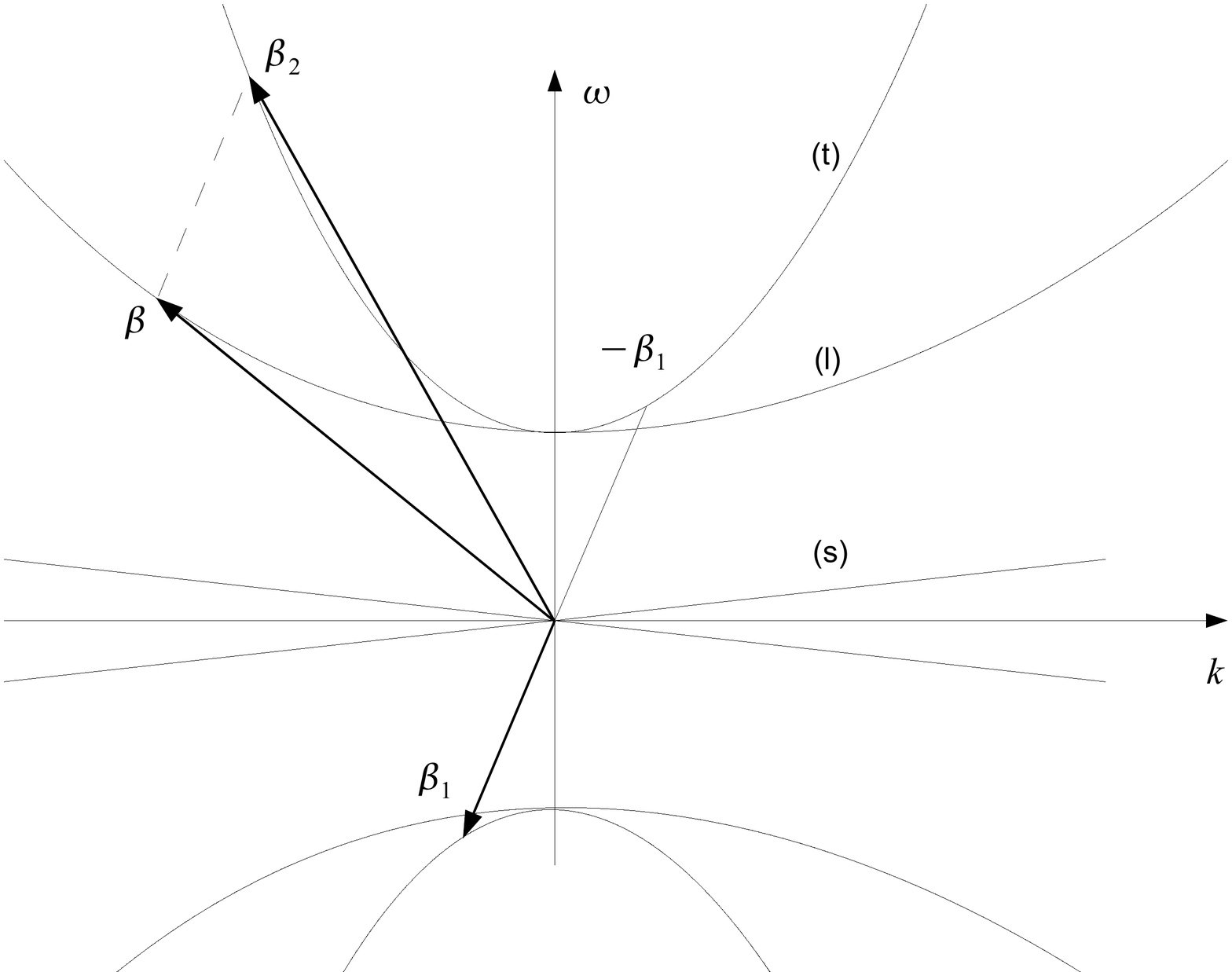}
\caption{Resonance associated with the Raman instability.}
\label{ttl}\end{center}
\end{figure}

 If $\b$ is an acoustic wave, then in \eqref{res:em} we are looking at the Brillouin instability (``scattering of light from acoustic phonons", \cite{Boyd}, paragraph 8.1). This case is examined in Section \ref{sec:brillouin}. The corresponding ${\rm (t)}{\rm (t)}{\rm (s)}$ resonance is pictured on Figure \ref{tts}.

 \begin{figure}\begin{center}
\includegraphics[scale=.4]{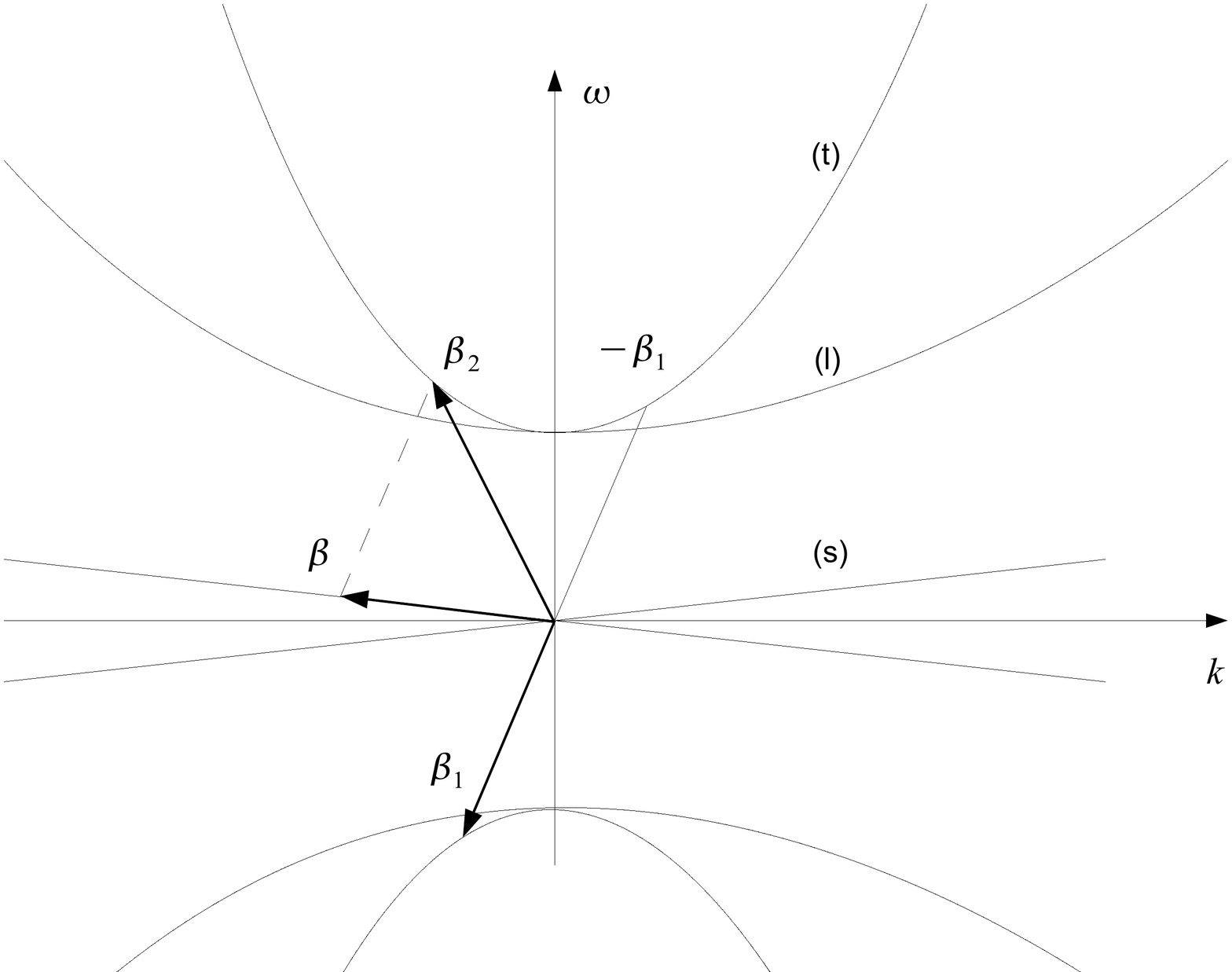}
\caption{Resonance associated with the Brillouin instability.}
\label{tts}\end{center}
\end{figure}

 The notation $u_{j,p_1,p_2},$ introduced in \eqref{ansatz:fourier}, denotes the $(p_1,p_2)$ Fourier mode of the profile ${\bf u}_j,$ which is the $O({\bm \epsilon}^j)$ term in the expansion of ${\bf u}$ into powers of ${\bm \epsilon}.$ In the following,   we often use the slightly different notation
 \be \label{notation:phases}
  u_{j,p\b_1} := u_{j,p,0}, \qquad u_{j,p\b_2} := u_{j,0,p}, \qquad u_{j,p\b} := u_{j,p,p}, \qquad p \in \{-1,0,1\}.
  \ee

\subsubsection{WKB equations: $O(1)$} \label{sec:O1}

The equations for the terms $O(1)$ are
 $$ (-i \o + A_\perp(i k) + A_{\perp0}) {\bf u}_{0\perp} = 0, \quad  (-i \o + A_\parallel(i k) +A_{\parallel0}) {\bf u}_{0\parallel} = 0,$$
 corresponding to the dispersion relations for the phases $\b, \b_1, \b_2$ and the polarization conditions for the amplitudes:
 \be \label{pola:wkb} {\bf \Pi}(p \b_j) u_{0, p\b_j}  = u_{0,p\b_j}, \quad {\bf \Pi}(p \b) u_{0,p \b} = u_{0,p \b}, \quad j \in \{1,2\}, \,\, p \in \{-1,0,1\},\ee
explicitly given by \eqref{pola:t} and \eqref{pola:l} for $p \in \{-1,1\}.$ In accordance with Section \ref{sec:phases}, these are the only non-trivial harmonics. In \eqref{pola:wkb} we used notation \eqref{notation:phases}.

We note that \eqref{pola:wkb} for $p = 0$ implies $E_{0,0} = 0:$ the mean mode of the leading amplitude in the electric field vanishes identically.

 \subsubsection{WKB equations: $O({\bm \epsilon})$} \label{sec:Oe}

The equations for the terms $O({\bm \epsilon})$ are
$$
  (-i \o + A_\perp(i k)+ A_{\perp0} ) {\bf u}_{1\perp} = {\bf F}_{1\perp}, \quad  (-i \o + A_\parallel(i k) +A_{\parallel0}) {\bf u}_{0\parallel} = {\bf F}_{1\parallel}.
$$
  We project onto the kernels. The compatibility conditions
  \be \label{compa:for:t}
   {\bf \Pi}(p \b_1) {F}_{1\perp,p \b_1} =0 ,\qquad {\bf \Pi}(p \b_2) {\bf F}_{1\perp, p \b_2} = 0,
   \ee
  give the evolution equations in $u_{0,p\b_1}$ and $u_{0,p\b_2},$
  and the compatibility condition
  \be \label{compa:for:l}
   {\bf \Pi}(p \b) F_{1\parallel,p\b} = 0\ee
  gives the evolution equations in $u_{0,p\b}.$ There are two types of terms in these evolution equations: transport operators at the group velocities, and bilinear coupling terms. We consider in succession the mean mode $(0,0),$ the transverse modes $p \b_1$ and $p \b_2,$ and the longitudinal modes $p \b,$ with $p = \pm 1.$

\medskip

{\it Mean mode.} The Lorentz force terms are transparent: there holds, given ${\bf u}_0$ and phases as described in Sections \ref{sec:init:ansatz} to \ref{sec:O1}:
 $$ \big({\bf v}_{e0} \times {\bf B}_0\big)_{(0,0)} = 0, \qquad \big({\bf v}_{i0} \times {\bf B}_0\big)_{(0,0)} = 0.$$
From there, and $E_{0,0} = 0$ (Section \ref{sec:O1}), we find by direct computation on the Euler-Maxwell equations that the mean mode is constant. Given the form of the datum, we infer $u_{0,(0,0)} = 0:$ the mean mode of the leading amplitude vanishes identically.

\medskip

{\it Transverse modes.} The nonlinear terms satisfy a form of transparency. This was first observed in \cite{T2} (see Proposition 2.1 from that reference). Given ${\bf u}_0$ and phases as described in Sections \ref{sec:init:ansatz} to \ref{sec:O1}:
 $$
 \begin{aligned}
  \Big( {\bf v}_{e0z} i k {\bf v}_{e0y} + ({\bf v}_{e0} \times {\bf B}_0)_{y}\Big)_{p \b_{j}} = 0, \qquad j \in \{1,2\},\,\, p \in \{-1,1\}, \\
  \Big( {\bf v}_{i0z} i k {\bf v}_{i0y} - \frac{\theta_i}{\theta_e} ({\bf v}_{i0} \times {\bf B}_0)_{y}\Big)_{p \b_j} = 0, \qquad j \in \{1,2\}, \,\, p \in \{-1,1\}.
  \end{aligned}
 $$
 This implies that only the current density contributes to the evolution equations \eqref{compa:for:t} for the transverse amplitudes $u_{0,p\b_j}.$
With \eqref{compa:t}, and the specific form of ${\bf F}_{1\perp}$ and ${\bf F}_{1\parallel}$ as given by the Euler-Maxwell equations, we find that these equations are
 $$
  \d_t E_{0y,p\b_j} + \frac{k_j}{\o_j} \d_x E_{0y,p\b_j} = \Big( {\bf n}_{e0} {\bf v}_{e0y} - \frac{\theta_i}{\theta_e} {\bf n}_{i0} {\bf v}_{i0y} \big)_{p \b_j}\,\,, \qquad p \in \{-1,1\}.
 $$
 By the form of the transverse dispersion relation \eqref{eig:t}, the ratio $k_j/\o_j$ is the group velocity\footnote{We further comment on the form of the transport equations in geometric optics on page \pageref{ref for transp}.}:
  $$ \frac{k_j}{\o_j} = \frac{d\o_{\rm (t)}}{dk}|_{k = k_j}.$$
   The coupling terms are made explicit, in terms of the electrical amplitudes, by use of the polarization conditions \eqref{pola:t} and \eqref{pola:l}. We obtain
  \be \label{evol:t}
  \d_t E_{0y,\b_1} +  \frac{k_1}{\o_1} \d_z E_{0y,\b_1} = \frac{1}{\o_2} \Big( \frac{k \theta_e}{\o^2 - k^2 \theta_e^2} - \frac{\theta_i^3}{\theta_e^3} \frac{\a^2 k \theta_i}{\o^2 - \a^2 k^2 \theta_i^2}\Big) E_{0z, \b} \bar E_{0y,\b_1}.
  \ee
and
  \be \label{evol:t2}
  \d_t  E_{0y,\b_2}  +  \frac{k_2}{\o_2} \d_x  E_{0y,\b_2} = \frac{1}{\o_1} \Big( \frac{k \theta_e}{\o^2 - k^2 \theta_e^2} - \frac{\theta_i^3}{\theta_e^3} \frac{\a^2 k \theta_i}{\o^2 - \a^2 k^2 \theta_i^2}\Big) E_{0z, \b} \bar E_{0y,\b_2} .
  \ee
  By symmetry, $E_{0y,-\b_j} \equiv \bar E_{0y,\b_j}.$

\medskip

{\it Longitudinal modes.} Next we turn to the equation \eqref{compa:for:l} in $u_{0,p \b},$ for $|p| = 1.$ The convective terms are transparent:
$$
 \left( v_{e0x} \cdot ik \left(\begin{array}{c} {\bf n}_{e0} \\ {\bf v}_{e0x} \end{array}\right)\right)_{p\b} = 0, \qquad  \left( v_{i0x} \cdot ik \left(\begin{array}{c} {\bf n}_{i0} \\ {\bf v}_{i0x} \end{array}\right)\right)_{p\b} = 0.
$$
The current density terms also are transparent:
$$
   \Big( {\bf n}_{e0} {\bf v}_{e0x} - \frac{\theta_i}{\theta_e} {\bf n}_{i0} {\bf v}_{i0x} \Big)_{p \b} = 0.
$$
Thus the only nonlinear term in the longitudinal equation comes from the Lorentz force. Just like for the transverse equations, in order to spell out equation \eqref{compa:for:l}, we use the compatibility condition \eqref{compa:l} together with the explicit expression of ${\bf F}_{1\parallel}$ as read on the Euler-Maxwell equations. We find
\be \label{evol:l} \begin{aligned}
 \Big(1 + \frac{\o^2 + k^2 \theta_e^2}{{\bf e}^2} + \frac{\theta_i^2}{\theta_e^2} \frac{\o^2 + \a^2 k^2 \theta_i^2}{{\bf i}^2}\Big) \d_t E_{0x,\b} & + 2k \Big( \frac{\theta_e^2\o}{{\bf e}^2} + \frac{\theta_i^2}{\theta_e^2} \frac{\a^2 \theta_i^2 \o}{{\bf i}^2}\Big) \d_x E_{0x,\b} \\ & = \mbox{Lorentz force term,}
  \end{aligned}
\ee
with notation
 $$
 {\bf e} := \o^2 - k^2 \theta_e^2, \qquad {\bf i} := \o^2 - k^2 \a^2 \theta_i^2.
 $$
 From \eqref{eig:lbis}, we compute by term-by-term differentiation
 $$ 0 = 2 \o(k) \o'(k) \Big(\frac{1}{{\bf e}^2} + \frac{\theta_i^2}{\theta_e^2} \frac{1}{{\bf i}^2}\Big) - 2 k \Big( \frac{\theta_e^2}{{\bf e}^2} + \frac{\theta_i^2}{\theta_e^2}\frac{\a^2 \theta_i^2}{{\bf i}^2}\Big), \qquad \o' = \frac{d\o_{\rm (\ell)}}{dk} \,\,\, \mbox{or} \,\,\, \o' = \frac{d\o_{\rm (s)}}{dk}.$$
Together with \eqref{eig:lbis}, this shows that, for $p  =1,$ \eqref{evol:l} is a transport equation at group velocity:
$$  \d_t E_{0x,\b} + \o'(k) \d_x E_{0x,\b} = \Big(1 + \frac{\o^2 + k^2 \theta_e^2}{{\bf e}^2} + \frac{\theta_i^2}{\theta_e^2} \frac{\o^2 + \a^2 k^2 \theta_i^2}{{\bf i}^2}\Big) ^{-1} \times \big(\mbox{Lorentz force term}\big).$$
 The $\b$-harmonics of the electronic Lorentz force term is
 $$ ({\bf v}_{e0} \times {\bf B}_0\big)_{\b} = v_{e0,\b_1} \times B_{0,\b_2} + v_{e0,\b_2} \times B_{0,\b_1} = \left(\begin{array}{cc} v_{e0y_1,\b_1} \times B_{0y_2,\b_2} - v_{e0y_2,\b_1} B_{0y_1,\b_2} \\ 0 \\ 0 \end{array}\right).$$
 With the polarization \eqref{pola:t}, the $x$ component of the $\b$-harmonics of the electronic Lorentz force appears as
 $$ \frac{k}{i\o_1\o_2} \big( E_{0y_1,\b_1} E_{0y_2,\b_2} + E_{0y_1,\b_1} E_{0y_2,\b_2}\big).$$
Taking into account the ionic component of the Lorentz force, and the compatibility condition \eqref{compa:l}, the longitudinal transport equation finally takes the form
\be \label{transport:l} \begin{aligned}
  \d_t E_{0x,\b} & + \o'(k) \d_x E_{0x,\b} \\ & = \frac{k\o}{\o_1\o_2} \Big( \frac{\theta_e}{\o^2 - k^2 \theta_e^2} + \frac{\theta_i^2}{\theta_e^2} \frac{\theta_i}{\o^2 - k^2 \a^2 \theta_i^2}\Big) \Big( E_{0y_1,\b_1} E_{0y_1,\b_2} + E_{0y_2,\b_1} E_{0y_2,\b_2}\Big). \end{aligned}
\ee
In the case of one spatial transverse dimension $y \in \R,$ or $E_{0y_2} \equiv 0,$ systems \eqref{evol:t}-\eqref{evol:t2}-\eqref{transport:l} fall into the category of three-wave interaction systems as described in Section \ref{sec:threewave}.

\begin{rema}\phantomsection What about {\rm stability} of the WKB expansion that was sketched here ? Correctors can be constructed in a classical way, implying consistency of the WKB approximation in the sense of \eqref{0a}. Short-time stability of small initial perturbations then follows by Sobolev estimates for the singular equations satisfied by the profiles, since the regime is weakly nonlinear.
\end{rema}

 \begin{rema}\phantomsection In paragraph 2.2 of Boyd's treatise \cite{Boyd}, three-wave interactions systems such as \eqref{evol:t}-\eqref{evol:t2}-\eqref{transport:l} are derived from Maxwell's equations, under the assumption of an ad hoc expansion for the nonlinear polarization encoding the nonlinear response of the medium. Another derivation, this time from the three-level Maxwell-Bloch equations, is given in Section 12 of \cite{JMR-TMB}, by means of a similar two-phase expansion in a weakly nonlinear regime. In \cite{BBCNT}, Schr\"odinger-Bloch systems were derived in the high-frequency limit from three-level Maxwell-Bloch systems, and formal arguments were given to further derive three-wave interaction systems from Schr\"odinger-Bloch systems. In \cite{SchW}, three-wave interaction systems are rigorously derived from the gravity-capillary water-wave system. Another derivation from Euler-Maxwell is given in \cite{CC}.
\end{rema}

\subsection{Raman} \label{sec:raman} The Raman instability corresponds to a growth of the electronic plasma waves. Here $\b$ belongs to the ${\rm (\ell)}$ branch of the variety. With \eqref{dl:em}, if the spatial transverse dimension is equal to one, or $E_{0y_2} \equiv 0,$ the three-wave interaction system \eqref{evol:t}-\eqref{evol:t2}-\eqref{transport:l} is
\be \label{raman1}
 \left\{\begin{aligned}\d_t E_{\b_1} + \frac{k_1}{\o_1} \d_x E_{\b_1} & = \left(\frac{k \theta_e}{\o_2} + O(\theta_i^2)\right)\bar E_{\b_2} E_\b, \\
 \d_t E_{\b_2} + \frac{k_2}{\o_2} \d_x E_{\b_2} & = \left(\frac{k \theta_e}{\o_1} + O(\theta_i^2)\right) \bar E_{\b_1} E_\b, \\
 \d_t E_{\b} + \left(\frac{k \theta_e^2}{\o} + O(\theta_i^2)\right)\d_x E_{\b} & = \left(\frac{k \o \theta_e}{\o_1 \o_2} + O(\theta_i^2)\right) E_{\b_1} E_{\b_2},
 \end{aligned}\right.
 \ee
locally uniformly in $k,$ where $E_{\b_j}$ stands for $E_{0y_1,\b_j},$ and similarly $E_{\b} = E_{0x,\b}.$

We consider in \eqref{raman1} the parameters $k$ and $\o, \o_1, \o_2$ to be fixed, $\theta_e$ to be small, and $\theta_i$ to be even smaller than $\theta_e.$ For the solution $U = (E_{\b_1}, E_{\b_2}, E_\b)$ to \eqref{raman1}, we posit the ansatz
$$ U(\theta_e, t,x) = V(\theta_e^{3/2} t, \theta_e^{3/2} x).$$
Then, $V$ solves the $3$-wave interaction system \eqref{raman}, with $\e = \theta_e,$ and $\dsp{\mbox{sgn}\,b_2 b_3 = \mbox{sgn}\, \frac{\o}{\o_2} > 0.}$ Theorem \ref{th:raman} asserts instability of the reference solution $V_a(t,x) = (a(x - c_1 t),0,0),$ $\dsp{c_1 = \frac{k_1}{\o_1}},$ under initial perturbations of the form $\e^K \phi(\e,x),$ in time $O(\sqrt \e |\ln \e|),$ in small balls $B(x_0,\rho).$

In the scaling of \eqref{raman1}, this translates as instability, in the sense of Theorem \ref{theorem1}, of the reference solution
 $$ U_a(t,x) = \Big( E_{\b_1}\Big(\theta_e^{3/2} \big( x - \frac{k_1}{\o_1} t\big)\Big), 0, 0 \Big),$$
under initial perturbations of the form $\e^K \phi(\theta_e, \theta_e^{3/2} x),$ with $\sup_{0 < \theta_e < 1} \| \phi(\theta_e, \cdot)\|_{\theta_e,s} < \infty,$ in long time $O(\theta_e^{-1} |\ln \theta_e|),$ in norm $L^2(\R^d).$

That is, the coupling in \eqref{raman1} is {\it weak}, implying that instability are recorded only in {\it long time,} for which our analysis applies only if initial perturbations are {\it slowly varying} in $x.$

\smallskip

 As shown in particular in Section \ref{sec:end-insta}, the amplification is maximal for the components of the solution associated with the unstable resonance. Here, this means in particular that small perturbations of the initially null electromagnetic plasma field $E_{\b}$ are amplified, corresponding to the Raman instability.

\begin{rema}\phantomsection The 2d model of Colin and Colin \cite{CC} is a refinement of system \eqref{raman1} for the description of Raman scattering.
\end{rema}

\subsection{Brillouin} \label{sec:brillouin} The Brillouin instability corresponds to a growth of the acoustic waves. Here $\b$ belongs to the (s) branch of the variety. With \eqref{dl:em}, if the spatial transverse dimension is equal to one, or $E_{0y_2} \equiv 0,$ the three-wave interaction system \eqref{evol:t}-\eqref{evol:t2}-\eqref{transport:l} is
 \be \label{brillouin1}
 \left\{\begin{aligned}\d_t E_{\b_1} + \frac{k_1}{\o_1} \d_x E_{\b_1} & = \left( \frac{-1}{\o_2 k \theta_e} + O(\theta_i^2)\right) \bar E_{\b_2} E_\b, \\
 \d_t E_{\b_2} + \frac{k_2}{\o_2} \d_x E_{\b_2} & = \left( \frac{-1}{\o_1 k \theta_e}  + O(\theta_i^2)\right) \bar E_{\b_1} E_\b, \\
 \d_t E_{\b} + \left(\theta_i {\bm \a} + O(\theta_i^2)\right) \d_x E_{\b} & = \left( \frac{- 1}{\o_1 \o_2} \frac{\theta_i}{\theta_e} \bm{\a_{}}  + O(\theta_i^2)\right) E_{\b_1} E_{\b_2},
 \end{aligned}\right.
 \ee
 where $\dsp{\bm{\a_{}} := \left(\a^2 + \frac{1}{1 + k^2 \theta_e^2}\right)^{1/2}},$ and $E_{\b_j}$ stands for $E_{0y_1,\b_j},$ similarly $E_{\b} = E_{0z,\b}.$

We consider in \eqref{brillouin1} the parameters $k$ and $\o, \o_1, \o_2$ to be fixed, $\theta_e$ to be small, and $\theta_i$ to be even smaller than $\theta_e;$ for instance
 \be \label{e:bri} \theta_e = \e^{1/2}, \qquad \theta_i = \e.\ee
For the solution $U = (E_{\b_1}, E_{\b_2}, E_\b)$ to \eqref{raman1}, we posit the ansatz
$$ U(\e, t,x) = V(\e^{1/2} t, \e^{-1/2} x).$$
Then, $V$ solves the $3$-wave interaction system \eqref{bk}, with $\e$ defined in \eqref{e:bri}, and
$$ \mbox{sgn}\, b_2 b_3 = \mbox{sgn}\, \frac{\theta_i \bm{\a_{}}}{k \o_2} = \mbox{sgn}\, \Big( \frac{\o}{\o_2} \frac{\bm{\a_{}}}{\a} + O(\theta_i) \Big) > 0.$$
Theorem \ref{th:brillouin} applies. In the scales of \eqref{brillouin1}, it asserts instability of the reference solution
  $$ \sqrt \e \Big( E_{\b_1}( x - \frac{k_1}{\o_1} t),0,0\Big)$$
 under initial perturbations of the form $\e^K \phi(\e, x),$ in time $O(\sqrt \e |\ln \e|),$ in $L^2(B(x_0,\rho))$ or $L^2(B(x_0,\e^\b))$ norms.

 Theorem \ref{th:brillouin2} also applies. It asserts instability of the reference solution
 $$ \Big( E_{\b_1}\Big(\sqrt \e (x - \frac{k_1}{\o_1} t)\Big), 0,0\Big)$$
under initial perturbations of the form $\e^K \phi(\e, \e^{1/2} x),$ with $\phi_3 = O(\e^{1/2}),$ in time $O(|\ln \e|),$ as measured in $L^2(\R^d).$

\smallskip

The proof of Section \ref{p-o-th1} shows that small perturbations of the initially null acoustic field $E_{\b}$ are amplified, corresponding to the Brillouin instability.

\section{Coupled Klein-Gordon systems with equal masses} \label{sec:KG}

Our second class of examples comprises coupled Klein-Gordon systems in $\R^d,$ with equal masses and different velocities. Our motivation here is the Euler-Maxwell system describing laser-plasma interactions, which, when linearized around zero, precisely gives two such Klein-Gordon systems and an acoustic system, as we saw in Section \ref{sec:der3EM} above. This form of the linearized Euler-Maxwell system induces us to think that high-frequency instabilities in the full Euler-Maxwell could be captured by the model systems that we now describe.

We denote
 $$
 \d_t + A_1(\d_x) + \frac{1}{\e} L_0, \qquad \d_t + A_1(\theta_0 \d_x) + \frac{1}{\e} L_0
 $$
 the Klein-Gordon operators, with $0 < \theta_0 < 1,$ implying different velocities, and
 \begin{equation}\label{A-L0-kg} A_1(\d_x)=
\begin{pmatrix}
0 & -\d_x  & 0\\
\d_x\cdot   & 0 & 0\\
0 & 0 & 0
\end{pmatrix},~~
L_0=
\begin{pmatrix}
0 & 0  & 0\\
0& 0&\o_0 \\
0&  -\o_0& 0
\end{pmatrix},
\ee where $\o_0 > 0,$ and $x\in\R^d$. The coupled systems have the form, for $U = (u,v)\in\R^{2(d+2)}$ with
 $u=(u_1,u_2,u_3)$ where $u_1 \in \R^d,$ $u_2 \in \R,$ $u_3 \in \R$,
 and $v=(v_1,v_2,v_3)$ where $v_1 \in \R^d,$ $v_2 \in \R,$ $v_3 \in
 \R$:
\be \label{kg1}
 \left\{ \begin{aligned} \big( \d_t + A_1(\d_x) + L_0\big) u & = \frac{1}{\sqrt\e} B^1(U,U), \\  \big( \d_t + A_1(\theta_0 \d_x) + L_0\big) v & = \frac{1}{\sqrt\e} B^2(U,U),
 \end{aligned}\right.
\ee where $B^1$ and $B^2$ are bilinear $\R^{2(d+2)} \times \R^{2(d+2)} \to
\R^{2(d+2)}.$ The eigenvalues $\l(\xi)$ of matrix $A_1(\xi) + L_0/i$ are
\be \label{fast} \l_1(\xi) = \sqrt{\o_0^2 + |\xi|^2} = -
\l_4(\xi),\ee and a multiplicity-$d$ null branch. Similarly, the
eigenvalues of matrix $A_1(\theta_0 \xi) + L_0/i$ are \be
\label{slow} \l_2(\xi) = \sqrt{\o_0^2 + \theta_0^2 |\xi|^2} = -
\l_3(\xi),\ee
 and a multiplicity-$d$ null branch. We denote $\l_5 \equiv 0$ the null branch for the whole system, with total multiplicity $2d.$ The eigenvalues are depicted on Figure \ref{fig1}. There are two fast Klein-Gordon branches, corresponding to \eqref{fast}, and two slow Klein-Gordon branches (slow since $\theta_0 < 1$), corresponding to \eqref{slow}.

\smallskip

  Thus we see that in approximating Euler-Maxwell by a system of the form \eqref{kg1}, disregarding the specific form of the right-hand side of \eqref{kg1}, we are simply approximating the speed of sound by zero and neglecting convective terms.

\begin{figure}\begin{center}
\includegraphics[scale=.4]{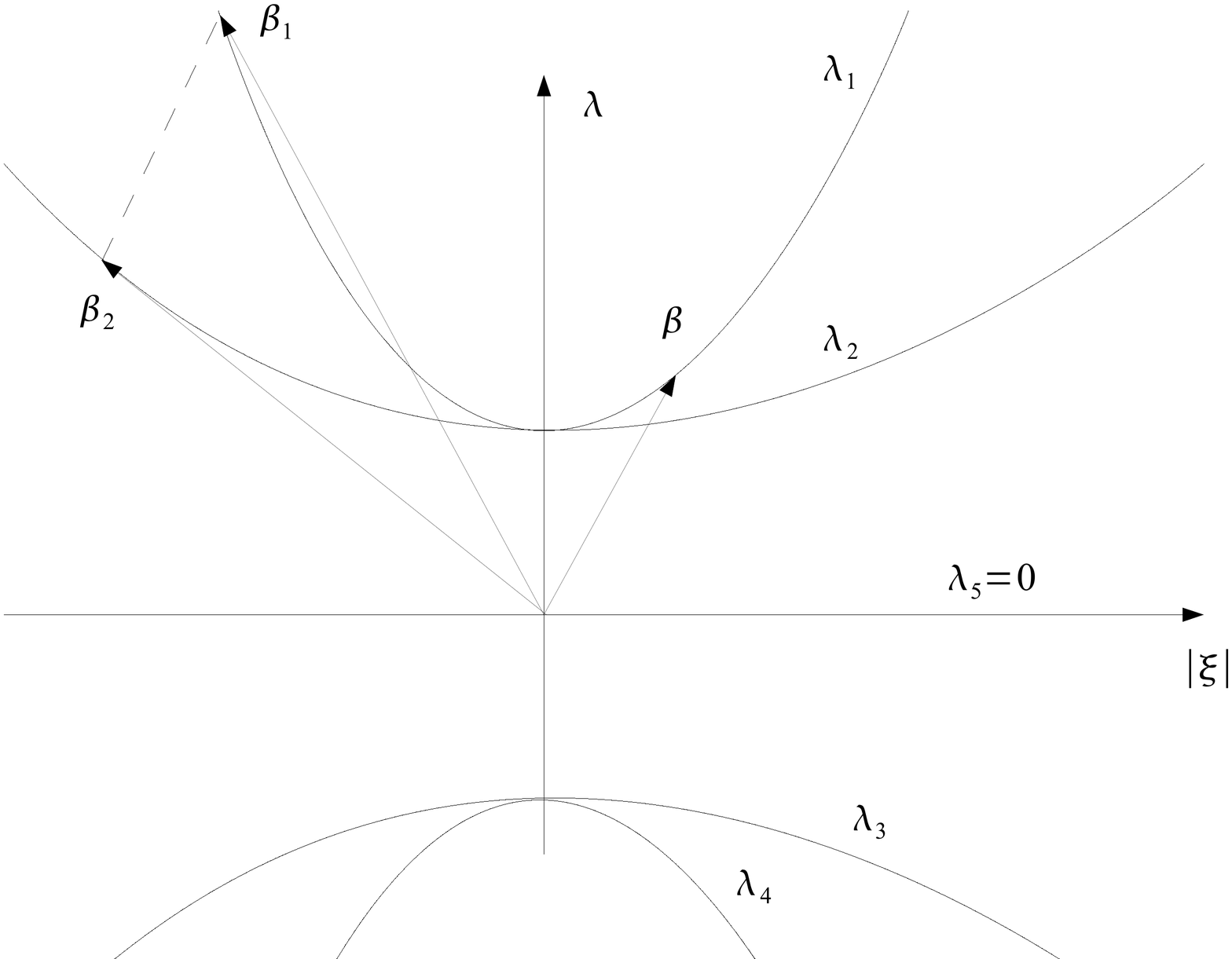}
\caption{Coupled Klein-Gordon with equal masses.}
\label{fig1}\end{center}
\end{figure}

\smallskip

 We show that for these Klein-Gordon operators, some bilinear coupling terms $B^1, B^2$ allow for non-trivial WKB solutions that are unstable, which we believe to be the situation for Euler-Maxwell (in the case of highly-oscillating data; non-oscillating data are known to generate stable WKB solutions \cite{T3}). In Euler-Maxwell, it is the current density in the Amp\`ere equation and the Maxwell-Lorentz force in the equation of conservation of momentum that couple the Klein-Gordon and acoustic systems, as we saw in Section \ref{sec:der3EM}.

For definiteness we will perform computations on the following explicit expressions for these coupling terms:
 coordinatizing $$U = (u_1, u_2, u_3, v_1, v_2, v_3), \qquad U' = (u'_1, u'_2, u'_3, v'_1, v'_2, v'_3),$$ we let
  \be\label{def-F-G}
\begin{aligned} B^1(U,U') & = \frac{1}{2} \left(\begin{array}{ccc} 0 \\  u_3 v'_3+ v_3 u'_3+  v_3 v'_3 \\  0 \end{array}\right), \\
B^2(U,U') & = \frac{1}{2} \left(\begin{array}{ccc} 0 \\  - u_2 u'_2 + v_2v'_3+v_3v'_2 \\  0 \end{array}\right).\end{aligned} \ee

\medskip

We denote
$$
A(\d_x):=\bp A_1(\d_x) &0 \\0&
A_1(\th_0 \d_x)\ep,\quad A_0:=\bp L_0&0
\\0& L_0\ep,\quad B(\cdot,\cdot):=\bp B^1(\cdot,\cdot)
\\B^2(\cdot,\cdot)\ep.
$$
 Then equation \eqref{kg1} takes the form \eqref{0}.

For system \eqref{kg1}, we check that Assumption \ref{ass:spectral} is satisfied (a simple property of the Klein-Gordon operators), that for most initial wavenumbers Assumption \ref{ass-u-a} is satisfied as well (the result of WKB computations performed in Section \ref{sec:WKB}), and finally, in Section \ref{sec:indexKG}, that Assumption \ref{ass:several-res} is satisfied in one space dimension $d = 1,$ with ${\bf \G} > 0.$

Thus the conclusions of Theorems \ref{th-two}, \ref{th:4} and \ref{th:5} apply, proving instability of the WKB solutions described below, if $d = 1.$

\subsection{Verification of Assumption {\rm \ref{ass:spectral}}: smooth spectral decomposition}\label{sec:spKG}

The spectral decomposition is
$$
A_1(\xi)+L_0/i=\l_1 (\xi) P_1(\xi) -\l_1(\xi) P_3(\xi) +0\cdot P_5(\xi).
$$
The eigenvalues are separated, implying regularity of both eigenvalues and eigenprojectors (see for instance Theorem 1.8 in Chapter 2 of \cite{K}).

It remains to prove bounds \eqref{bd:spectral}. By Lemma \ref{lem:regspec} in Appendix \ref{app:regdec}, these follow from a regularity result at infinity. The associated symbol at infinity is
 $$ A_\infty(\bar \omega,x) = A_1(\bar \o) - i x L_0 = \left(\begin{array}{ccc} 0 & \bar \o & 0 \\ \bar \o \cdot & 0 & - i x \o_0 \\ 0 & i x \o_0 & 0 \end{array}\right), \qquad (\bar \o,x) \in \S^{d-1} \times \R.$$
 It has simple eigenvalues $\pm \big( \bar \o^2 + x^2 \o_0^2 \big)^{1/2},$ and a multiplicity-three null eigenvalue. Given $\bar \o \in {\mathbb S}^{d-1},$ the eigenvalues are separated in a neighborhood of $x = 0.$ Hence the spectral decomposition of $A_\infty$ is smooth at $(\bar \o,0),$ for all $\bar \o \in \S^{d-1},$ and Lemma \ref{lem:regspec} applies, implying bounds \eqref{bd:spectral} for $A(\d_x) + L_0.$ Naturally, this also applies to $A_1(\theta_0 \d_x) + L_0,$ hence to the total operator $A(\d_x) + A_0.$

\subsection{Verification of Assumption {\rm \ref{ass-u-a}}: WKB expansion} \label{sec:WKB}

We select a characteristic temporal frequency $\o \in \R$ associated with the initial wavenumber $k \in \R^d,$ such that the following conditions are satisfied:
\begin{itemize}
\item The phase $\b = (\o,k)$ belongs to the fast positive Klein-Gordon branch on the variety:
 $$ \o = \sqrt{\o_0^2 + |k|^2}.$$
\item The only harmonics of $\b$ on the fast Klein-Gordon branches are $p \in \{-1,1\}:$
 $$ p^2 \o^2 = \o_0^2 + p^2 |k|^2 \implies p \in \{-1,1\}.$$
\item No harmonics of $\b$ belongs to the slow branches on the variety:
 $$ p^2 \o^2 \neq \o_0^2 + p^2 \theta_0^2 |k|^2, \quad\mbox{for all $p \in \Z.$}$$
\item There are no auto-resonances: the equations in $\xi \in \R^d$
 $$ \l_1(\xi + k) = \pm \o + \l_1(\xi), \quad \l_2(\xi + k) = \pm \o + \l_2(\xi)$$
 have no solution.
\end{itemize}
While this may seem like a lot of requirements on the fundamental phase, a look at Figure \ref{fig1} should suffice to convince the reader that for most phases $\b$ on the fast positive Klein-Gordon branch, these assumptions are satisfied.

The zeroth harmonics $p = 0$ belongs to the variety. With the above, this implies in particular that condition \eqref{harmonics} on page \pageref{harmonics}, describing $\{-1,0,1\}$ as the set of characteristic harmonics of the fundamental phase $(\o,k),$ is satisfied.

By Proposition \ref{lem:WKB} in Appendix \ref{app:onwkb}, in order to verify Assumption \ref{ass-u-a}, it
then suffices to check that the weak transparency condition
\eqref{weak:transp} is satisfied.

In this view, borrowing notation from Appendix \ref{app:onwkb}, we denote ${\bf \Pi}(p\b)$ the orthogonal projector onto $\ker (-i p\o + A(i pk) + A_0).$ For $|p| = 1,$ these kernels are one-dimensional, generated by $\vec e_1$ and $\vec e_{-1} = (\vec e_1)^*,$ with notation
 \be \label{def:e1}
\vec e_{1}:=\frac{1}{\sqrt 2} \Big(-\frac{k}{\o} ,1,\frac{i\o_0}{\o},0_{\C^{d+2}}\Big) \in \C^{2(d+2)},
\ee
so that
$$ {\bf \Pi}(\b) U = (U, \vec e_1\,) \vec e_1, \qquad {\bf \Pi}(-\b) U = (U,\vec e_1^*) \vec e_1^*,$$
denoting \label{hermit} $(\cdot,\cdot)$ the Hermitian scalar product in $\C^{2(d+2)}.$ The orthogonal projector ${\bf \Pi}(0)$ onto the six-dimensional kernel of $A(0) + A_0/i$ is
 $${\bf \Pi}(0)U=(u_1,0,0,v_1,0,0),$$
 implying, for $B$ given by \eqref{def-F-G}, the identities
 $$ {\bf \Pi}(0) B \equiv 0, \quad B( {\bf \Pi}(0) \cdot, \cdot) \equiv 0, \quad B(\cdot, {\bf \Pi}(0) \cdot) \equiv 0,$$
 which yield \eqref{weak:transp}.

\subsection{Verification of Assumption {\rm \ref{ass:several-res}}: resonances and transparency} \label{sec:indexKG}

We verify here the conditions (i) boundedness, (ii) partial transparency, and (iii) rank-one interaction coefficients of Assumption \ref{ass:several-res}.

By the form of the characteristic variety, and the choice of $\b,$
resonant pairs are
 $${\mathfrak R}:=\big\{ (1,2), \,\, (1,5), \,\, (2,5), \,\, (3,4), \,\, (5,3), \,\, (5,4)\big\}.$$
 A resonant frequency is pictured on Figure \ref{fig1}. It corresponds to a $(1,2)$ resonance: $\b_1 = \b + \b_2,$ that is $\l_1(\xi + k) = \o + \l_2(\xi).$

\medskip

(i) Boundedness of ${\mathfrak R}.$ Here we apply Lemma \ref{lem:resonantset} page \pageref{lem:resonantset}, as we may since the assumptions of Lemma \ref{lem:regspec} have been verified in Section \ref{sec:spKG}. The asymptotic branches on the variety are obviously distinct (Figure \ref{fig1}).

\medskip

(ii) Partial transparency. From the definition of $B$ in \eqref{def-F-G} and $\vec e_1$ in \eqref{def:e1}, we see that
$$
 \begin{aligned} B(\vec e_1) U &
 =\frac{1}{\sqrt 2}\Big(0_{\C^{d+2}}, \frac{i \o_0}{\o} v_3, 0,0_{\C^d}, - u_2, 0\Big), \\
 B(\vec e_{-1}) U &
 =\frac{1}{\sqrt 2}\Big(0_{\C^{d+2}}, \frac{-i \o_0}{\o} v_3, 0,0_{\C^d}, - u_2, 0\Big). \end{aligned}
$$
 An element in the image of $\Pi_5(\xi),$ the orthogonal
projector onto the kernel of $A(\xi) + A_0/i,$ has the form
 $$ U_5(\xi) = \Big( u_1, 0, \frac{-i \xi \cdot u_1}{\o_0}, v_1, 0, \frac{- i \th_0\xi \cdot v_1}{\o_0}\Big), \qquad (u_1,v_1) \in \C^6.$$
In particular, for all $U \in \C^{2(d+2)},$ $B(\vec e_{\pm 1}) U$ belongs to the orthogonal of the range of $\Pi_5(\xi),$ so that
\be \label{5all} \Pi_5(\cdot) B(\vec e_{\pm 1}) \equiv 0.\ee
The other projectors are
$$
\Pi_j(\xi) U= \frac{1}{\sqrt 2} \big(U,\O_j(\xi)\big) \,
\O_j(\xi),$$ where
$$
\O_j(\xi):=\frac{1}{\sqrt 2} \Big(-\frac{\xi}{\l_j}
,1,\frac{i\o_0}{\l_j},0_{\C^{d+2}}\Big),~j=1,4;\qquad \O_{j'}(\xi):=\Big(0_{\C^{d+2}},
-\frac{\th_0\xi}{\l_{j'}} ,1,\frac{i\o_0}{\l_{j'}}\Big),~j'=2,3.
$$
From there, we compute
$$
 \Pi_2(\xi) B(e_{\pm 1}) \Pi_5(\xi') \equiv 0, \quad \Pi_3(\xi) B(\vec e_{\pm 1}) \Pi_5(\xi') \equiv 0.
$$
and together with \eqref{5all} this implies that resonances $(2,5)$ and $(5,3)$ are transparent. Besides,
\be \label{calc:resintcoeffKG}
  \begin{aligned}
   \Big(\Omega_1(\xi'), \, B(\vec e_1) \Omega_2(\xi)\, \Big) & =\frac{-\o_0^2}{2 \o \l_2(\xi)}, \\
   \Big(\Omega_2(\xi), \, B(\vec e_{-1}) \Omega_1(\xi')\,\Big) & = \Big( \Omega_3(\xi), \, B(\vec e_1) \Omega_4(\xi')\,\Big) = \frac{-1}{2}, \\
   \Big( \Omega_1(\xi + k), \, B(\vec e_1) U_5(\xi)\,\Big) & =\frac{\th_0\xi\cdot v_1}{2\o}, \\
    \Big( \Omega_4(\xi), \, B(\vec e_{-1}) U_5(\xi + k) \, \Big) & =-\frac{\th_0(\xi+k)\cdot v_1}{2\o}.
    \end{aligned}
    \ee
 Thus
 $${\mathfrak R}_0:=\big\{(1,2),\, (1,5),\, (3,4),\, (5,4)\big\},$$
 a subset of ${\mathfrak R}$ which does not contain auto-resonances.  The associated resonant sets are
\be\label{reso-sets}\begin{split}
&{\mathcal R}_{12} =\left\{\xi:\sqrt{\o_0^2 + |\xi+k|^2}=\o+\sqrt{\o_0^2 + \th_0^2|\xi|^2}\right\},\\
&{\mathcal R}_{15} =\{\xi:|\xi+k|=|k|\},\\
&{\mathcal R}_{34} = \left\{\xi:\sqrt{\o_0^2 + |\xi|^2}=\o+\sqrt{\o_0^2 + \th_0^2|\xi+k|^2}\right\},\\
& {\mathcal R}_{54}=\{\xi: |\xi|=|k|\},
\end{split} \ee
where we recall
$\o=\sqrt{\o_0^2 + |k|^2}.$

We now turn to the verification of condition \eqref{new:1605} in the partial transparency condition Assumption \ref{ass:several-res}(ii). The only relevant intersection here is
$$
\begin{aligned}
{\mathcal R}_{15}\bigcap \Big({\mathcal
R}_{54}+k\Big)
&=\left\{\xi:|\xi+k|=|k|\right\}\bigcap\big(\left\{\xi:
|\xi-k|=|k|\right\}\big)=\{0\}.
\end{aligned}
$$
The ratio interaction coefficient over phase is
$$
\frac{\Big( \Omega_1(\xi + k), \, B(\vec e_1)
U_5(\xi)\Big)}{\l_1(\xi+k)-\o} ={ \frac{\th_0\xi\cdot
v_1}{4\o}}\frac{\sqrt{\o_0^2 + |\xi+k|^2}+\sqrt{\o_0^2+ |k|^2}}{|\xi|^2+2\xi\cdot
k }. $$
This ratio is bounded in a neighborhood of $\xi = 0$ only in one space dimension. Hence Assumption \ref{ass:several-res}(i) holds in one space dimension only.

\medskip

We finally turn to the verification of
condition \eqref{new:1605:2}. There holds
 $${\mathcal R}_{12} \bigcap {\mathcal R}_{15} \subset \{ \l_2 = \l_5 \} = \emptyset.$$
Similarly, by \eqref{reso-sets}(iii)(-iv),
$${\mathcal R}_{34} \bigcap {\mathcal R}_{54} = \emptyset.$$

\medskip

(iii) Rank-one coefficients: the eigenvalues $\l_j,$ for $j \in \{1,2,3,4\},$ are simple eigenvalues, implying that the interaction coefficients have rank at most one. The kernel $\l_5 \equiv 0$ has multiplicity 5, but by \eqref{5all} the interaction coefficients $b_{51}^-$ and $b_{54}^+$ are identically zero. This verifies Assumption \ref{ass:several-res}(iii).

\subsection{Stability index} \label{sec:KG:index}

There holds
$$\begin{aligned}
  \Pi_1(\xi+ k) B(\vec e_1) \Pi_2(\xi) & B(\vec e_{-1})  \Pi_1(\xi + k) \\ & = \Big( \Omega_1(\xi + k), \, B(\vec e_1) \Omega_2(\xi)\,\Big) \Big(\Omega_2(\xi), \, B(\vec e_{-1}) \Omega_1(\xi + k)\,\Big) \Pi_1(\xi + k),\end{aligned}$$
so that the trace of the product of the $(1,2)$ interaction coefficients is equal to
$$ \G_{12}(\xi) = \Big( \Omega_1(\xi + k), \, B(\vec e_1) \Omega_2(\xi)\,\Big) \Big(\Omega_2(\xi), \, B(\vec e_{-1}) \Omega_1(\xi + k)\,\Big),$$ and with \eqref{calc:resintcoeffKG}(i) and \eqref{calc:resintcoeffKG}(ii) we find
$$ \G_{12}(\xi) = \frac{\o_0^2}{4 \o \l_2(\xi)}.$$
Besides,
$$\G_{34}(\xi)= \frac{\o_0^2}{4 \o \l_2(\xi+k)}, \quad \G_{15}(\xi) = \G_{54}(\xi) \equiv 0.$$
This gives ${\bf \G} > 0,$ implying instability.

\section{Coupled Klein-Gordon systems with different masses} \label{sec:KG2}

Our third class of examples is made up of coupled Klein-Gordon systems with different
velocities and  different masses. As we will see, the assumption that masses are different implies a much smaller resonant set than in our previous example.

 Borrowing notation from Section \ref{sec:KG}, in particular
\eqref{A-L0-kg}, we consider systems
$$
\left\{ \begin{aligned}
                 &\d_t u+A_1(\d_x)u+\frac{1}{\e}\a_0 L_0 u=\frac{1}{\sqrt{\e}}B^3(U,U),\\
                 &\d_t v+A_1(\th_0\d_x)v+\frac{1}{\e}L_0 v=\frac{1}{\sqrt{\e}}B^4(U,U).\\
                          \end{aligned} \right.
$$
We assume here $\a_0 > 1,$ in contrast with \eqref{kg1}. We consider the bilinear forms $B^3$ and $B^4$ defined by
\be\label{def-F-G-kg2}
\begin{split} B^3(U,U') & = \frac{1}{2} \left(\begin{array}{ccc} 0 \\  u_3 v'_3+ v_3 u'_3+  v_3 v'_3 \\  0 \end{array}\right), \\
B^4(U,U') & = \frac{-\iota}{2} \left(\begin{array}{ccc} 0 \\  u_2
u'_2 + u_2v'_2+v_2u'_2 \\  0 \end{array}\right),\quad
\iota\in\{-1,1\},
\end{split}\ee
with
$$ U=\big(u_1,u_2,u_3,v_1,v_2,v_3\big) \in \C^{2(d+2)}, \quad U'=\big(u_1',u_2',u_3',v_1',v_2',v_3'\big) \in \C^{2(d+2)}.$$
For the rest, in particular $t,x,A_1,L_0$ and $\theta_0,$ the notations are borrowed from Section \ref{sec:KG}.

\begin{figure}\begin{center}
\includegraphics[scale=.4]{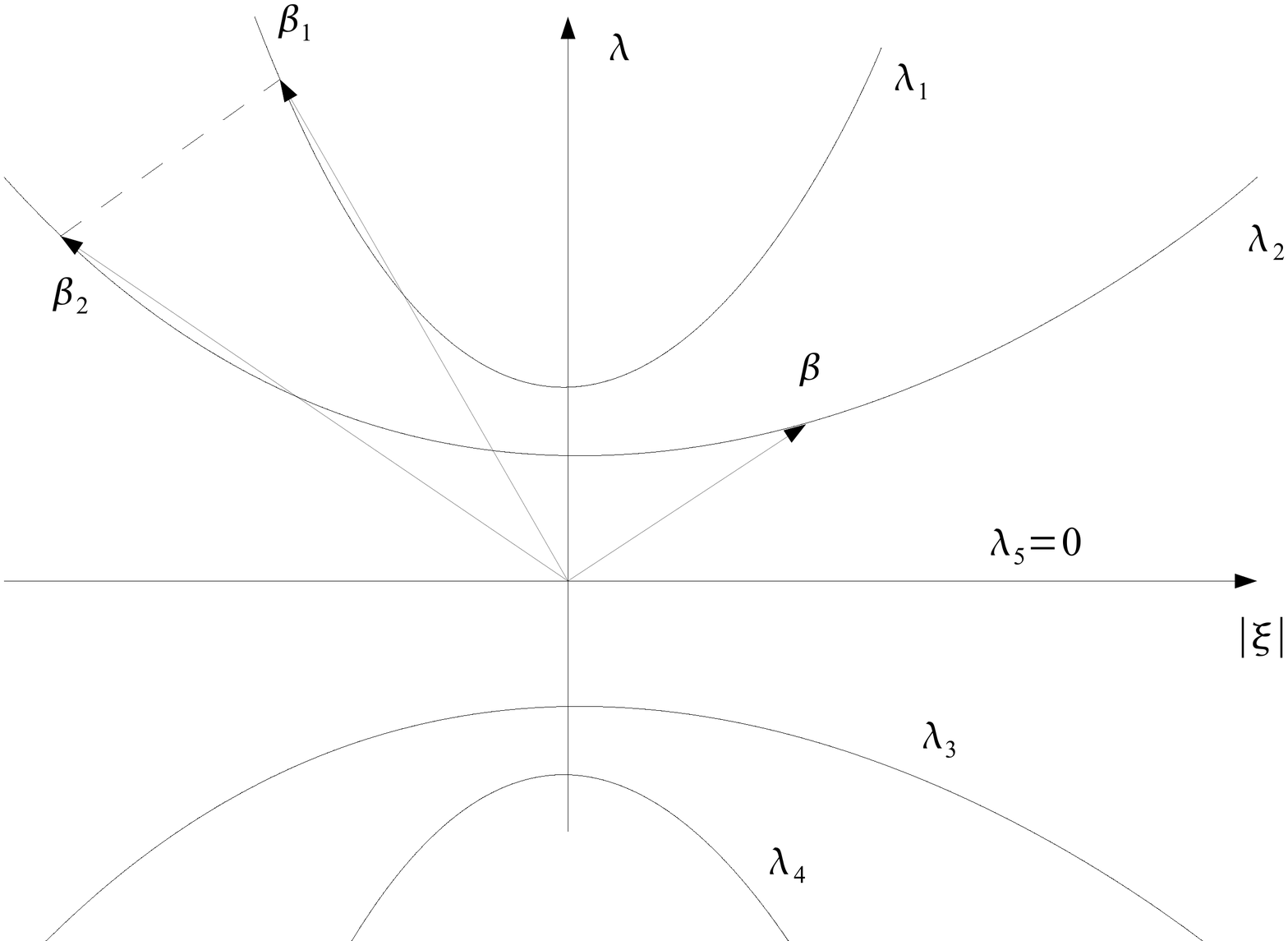}
\caption{Coupled Klein-Gordon with different masses.}
\label{fig2}\end{center}
\end{figure}

The spectral decomposition is
$$
A(\xi)+A_0/i=\sum_{j=1}^5\l_j(\xi)\Pi_j(\xi),
$$
 with
eigenvalues
$$
\l_1(\xi):=\sqrt{\a_0^2\o_0^2+ |\xi|^2}=-\l_4(\xi),\quad
\l_2(\xi):=\sqrt{\o_0^2 + \th_0^2|\xi|^2}=-\l_3(\xi),\quad
\l_5(\xi)\equiv0.
$$

The verification of Assumption \ref{ass:spectral} goes exactly as in Sections \ref{sec:spKG}. In the upcoming Sections, we verify that Assumptions \ref{ass-u-a} and \ref{ass:several-res} are satisfied, then compute $\mbox{sgn}\,{\bf \G} = \mbox{sgn}\,\iota.$ Stability ensues if $\iota = -1,$ and instability if $\iota = 1.$

\subsection{Verification of Assumption {\rm \ref{ass-u-a}}: WKB expansion} \label{sec:WKB-kg2}

We select a fundamental phase $\b = (\o,k) \in \R^{1 + d}$ as follows:
\begin{itemize}
\item The phase $\b$ belongs to the slow positive Klein-Gordon branch on the variety:
 $$\o = \sqrt{\o_0^2 + \theta_0^2 |k|^2}.$$
\item The only harmonics of $\b$ on the slow Klein-Gordon branches are $-1$ and $1.$

\smallskip

\item No harmonics of $\b$ belong to the fast branches on the variety.

\smallskip

\item There are no auto-resonances.
\end{itemize}
 For more details, and comments on these conditions, see Section \ref{sec:WKB}. In addition, we restrict the range of $k$ as we assume
 \be \label{nores}
 |k|^2 < \frac{1}{\theta_0^2}(\a_0^2 - 1)\o_0^2.
 \ee
It is easy to verify that condition \eqref{nores} implies that there are no $(1,5)$ resonances, and no $(5,4)$ resonances.

Denoting ${\bf \Pi}(p\b)$ the orthogonal projector onto $\ker (-i p \o + A(ipk) + A_0),$
we find the identities
$$
{\bf \Pi}(0)U=(u_1,0,0,v_1,0,0), \quad
{\bf \Pi}(\b) U= (U,\vec e_1)\vec e_1,\quad {\bf \Pi}(-\b) U=
(U,\vec e_{-1}\,) \vec e_{-1},
$$
for any
$U=(u_1,u_2,u_3,v_1,v_2,v_3),$
where \be\label{def:e1-kg2} \vec e_{1}:=\big(0_{\C^{d+2}},-\frac{\th_0k}{\o}
,1,\frac{i\o_0}{\o}\big),\quad \vec e_{-1}=(\vec e_1)^*. \ee Then, as in Section \eqref{sec:WKB}, there holds
$$
{\bf \Pi}(0) B\equiv0,\quad B\big({\bf
\Pi}(0),\cdot\big)\equiv0,\quad B\big(\cdot,{\bf
\Pi}(0)\big)\equiv0,
$$
and condition \eqref{weak:transp} is satisfied. By Proposition \ref{lem:WKB}, Assumption \ref{ass-u-a} is then satisfied.

\subsection{Verification of Assumption {\rm \ref{ass:several-res}}: resonances and transparency} \label{sec:indexKG-kg2}

By the form of the characteristic variety, and choice of $\b$ (in particular, condition \eqref{nores}), resonant pairs are
 $$ {\mathfrak R} = \big\{ (1,2), \quad (2,5), \quad  (3,4),\quad (5,3)\big\}.$$
 A $(1,2)$ resonant frequency is pictured on Figure \ref{fig2}.
\medskip

(i) Boundedness of ${\mathfrak R}.$  The asymptotic branches on the
variety are obviously distinct (Figure \ref{fig2}). Lemma
\ref{lem:resonantset} then implies boundedness of ${\mathfrak R}.$

\medskip

(ii) Partial transparency. From the definition of $B$ in
\eqref{def-F-G-kg2} and $\vec e_1$ in \eqref{def:e1-kg2}, we see that
$$
 \begin{aligned} B(\vec e_1) U &  =\frac{1}{\sqrt 2}\Big(0_{\C^{d}}, \frac{i \o_0}{\o} (u_3+v_3), 0,0_{\C^d}, -\iota  u_2, 0\Big), \\
 B(\vec e_{-1}) U &  {=\frac{1}{\sqrt 2}\Big(0_{\C^{d}}, \frac{-i \o_0}{\o} (u_3+v_3), 0,0_{\C^d}, -\iota u_2, 0\Big)}. \end{aligned}
$$
 An element in the image of $\Pi_5(\xi),$ the orthogonal
projector onto the kernel of $A(\xi) + A_0/i,$ has the form
 $$ U_5(\xi) = \Big( u_1, 0, \frac{-i \xi \cdot u_1}{\a_0\o_0}, v_1, 0, \frac{- i \th_0\xi \cdot v_1}{\o_0}\Big), \qquad (u_1,v_1) \in \C^{2d}.$$
In particular, for all $U \in \C^{2(d+2)},$ $B(\vec e_{\pm 1}) U$ belongs to
the orthogonal of the range of $\Pi_5(\xi),$ so that \be
\label{5all-kg2} \Pi_5(\cdot) B(\vec e_{\pm 1}) \equiv 0.\ee The other
projectors are
$$
\Pi_j(\xi) U= \frac{1}{\sqrt 2} \big(U,\O_j(\xi)\big) \,
\O_j(\xi),$$ where
$$
\O_j(\xi):=\frac{1}{\sqrt 2} \Big(-\frac{\xi}{\l_j}
,1,\frac{i\a_0\o_0}{\l_j},0_{\C^{d+2}}\Big),~j=1,4;\qquad
\O_{j'}(\xi):=\Big(0_{\C^{d+2}}, -\frac{\th_0\xi}{\l_{j'}}
,1,\frac{i\o_0}{\l_{j'}}\Big),~j'=2,3.
$$
From there, we compute \be \label{calc:25-kg2}
 \Pi_2(\xi) B(\vec e_{\pm 1}) \Pi_5(\xi') \equiv 0, \quad \Pi_3(\xi) B(\vec e_{\pm 1}) \Pi_5(\xi') \equiv 0.
\ee and together with \eqref{5all-kg2} this implies that resonances
$(2,5)$ and $(5,3)$ are transparent. Besides,
$$
 \begin{aligned} \Big(\Omega_1(\xi'), \, B(\vec e_1) \Omega_2(\xi)\, \Big) & =\frac{-\o_0^2}{2 \o \l_2(\xi)}, \qquad \Big(\Omega_2(\xi'), \, B(\vec e_{-1}) \Omega_1(\xi)\,\Big) & =\frac{-\iota}{2}.\end{aligned}
 $$
In particular, the $(1,2)$ resonance is non-transparent: $(1,2) \in
{\mathfrak R}_0.$ Similarly, $(3,4) \in {\mathfrak R}_0.$

 The partial transparency is automatically
satisfied because the sets in  \eqref{new:1605} and
\eqref{new:1605:2} are all empty, a direct consequence of the form of ${\mathfrak R}_0$ (in which no index appears more than once).
\medskip

(iii) Rank-one coefficients: the eigenvalues $\l_j,$ for $j \in \{1,2,3,4\},$ are simple eigenvalues, implying that the interaction coefficients have rank at most one.

\subsection{Stability index} \label{sec:KG:index2}

Computing as in Section \ref{sec:KG:index} and using \eqref{calc:25-kg2}, we find
$$
\G_{12}(\xi) =\iota \frac{\o_0^2}{4 \o \l_2(\xi)}, \qquad \G_{34}(\xi)=  \iota  \frac{\o_0^2}{4 \o \l_2(\xi+k)},$$
implying ${\rm sgn}\,{\bf\Gamma}={\rm sgn}\,\iota.$

\chapter{Appendix}

\section{Symbols and operators} \label{app:symbols}

Given $m \in \R,$ we denote $S^m$ the set of matrix-valued symbols $a \in C^{\bar s}(\R^d_x; C^\infty(\R^d_\xi)),$ such that for all $\a \in \N^d$ with $|\a| \leq \bar s,$ for all $\b \in \N^d,$ for some $C_{\a\b} > 0,$ for all $(x,\xi),$
 $$ |\d_x^\a \d_\xi^\b a(x,\xi)| \leq C_{\a\b} \langle \xi \rangle^{m - |\b|}, \qquad \langle \xi \rangle := (1 + |\xi|^2)^{1/2}.$$
That is, we consider symbols with a finite, but large, spatial regularity $\bar s,$ in connection with the finite Sobolev regularity $s_a$ of the approximate solution $u_a,$ postulated in Assumption \hyperref[ass-u-a]{\ref*{ass-u-a}}.

 We call $S^m$ the space of classical symbols of order $m.$ Given $a \in S^m,$ the associated family of pseudo-differential operators in semi-classical quantization is denoted $\op_\e(a)$ and formally defined by their action on functions or distributions $u$ in the variable $x:$
\begin{equation} \label{quantiz}
 \op_\e(a) u := \int e^{i x \cdot \xi} a(x,\e \xi) \hat u(\xi) \, d\xi, \qquad \e > 0.
\end{equation}

The semi-classical Sobolev norms $\| \cdot \|_{\e,s}$ are defined by
\begin{equation} \label{def:es} \| u \|_{\e,s}^2 := \int (1 + |\e \xi|^2)^s |\hat u(\xi)|^2 \, d\xi.
\end{equation}

\subsection{Estimates for Fourier multipliers} \label{sec:foumult}

 {\it  Fourier multipliers are pseudo-differential symbols (by extension, the associated operators) which do not depend on $x.$ Examples of Fourier multipliers in the text are given by the eigenprojectors $\Pi_j$ and the eigenvalues $\l_j.$ The interaction coefficients ${\mathcal B},$ $b_{ij},$ the normal form $Q,$ and the interaction matrix $M$ all depend on $(x,\xi),$ but as tensor products $M(x,\xi) = M_1(x) M_2(\xi)$ they are handled just like Fourier multipliers.}

 \medskip

 Given a Fourier multiplier $a \in S^0,$ the associated operators $\op_\e(a)$ map $H^s$ to $H^{s},$ for all $s,$ and for all $u \in H^s,$
 \begin{equation} \label{action:fourier-mult}
  \| \op(a) u\|_{\e,s} \lesssim |a|_{L^\infty} \| u \|_{\e,s}, \qquad a = a(\xi).
  \end{equation}
Also, denoting $|\cdot|_{{\mathcal F}L^1}$ the $L^1$ norm of the Fourier transform: $|u|_{{\mathcal F}L^1} := |\hat u|_{L^1},$ there holds, by Young's convolution inequality, the bound
 \be \label{action:fouriermultFL1}
  \big| a(x) \op_\e(b) u\big|_{{\mathcal F}L^1} \leq |b|_{L^\infty} |a|_{{\mathcal F}L^1} |u|_{{\mathcal F}L^1}, \qquad b = b(\xi) \in L^\infty.
 \ee
Pointwise estimates follow from Hausdorff-Young and Cauchy-Schwartz:
 \be \label{est:bernstein}
  | \op_\e(a) u |_{L^\infty} \leq C |a|_{L^2} |u|_{L^2}, \qquad a = a(\xi).
 \ee
 Given a Fourier multiplier $a \in S^1,$ and $f \in H^{s + d/2 + \eta}$ for some $\eta > 0,$ there holds for all $u \in H^s,$
 \begin{equation} \label{est:fourier-mult}
 \big\| \big[ \op_\e(a), f \big] u \big\|_{\e,s} \lesssim \e \,  |\nabla a|_{L^\infty} \| f \|_{H^{s + d/2 + \eta}} \| u \|_{\e,s}, \qquad a = a(\xi),
 \end{equation}
 and
 \be \label{est:fouriermultFL1}
 \Big| {\mathcal F} \Big( \big[ \op_\e(a), f \big] u \Big) \Big|_{L^1} \lesssim \e \,  |\nabla a|_{L^\infty} | \widehat{\d_x f}|_{L^1} |\hat u|_{L^1}, \qquad a = a(\xi).
 \ee
Similarly, for the Fourier multiplier $\Lambda^s := \op_\e\big(\langle \cdot \rangle^s\big),$ given $f \in H^{s + d/2 + \eta}$ for some $\eta > 0,$ there holds for all $u \in H^{s-1}:$
 \be \label{est:lambda-s}
  \big| \big[ \Lambda^s, f \big] u \big|_{L^2} \lesssim \e \| f \|_{H^{s + d/2 + \eta}} \| u \|_{\e,s-1}.
 \ee

\subsection{Estimates for pseudo-differential operators}

{\it Genuine pseudo-differential operators arise in the proof via $S,$ the flow of $M.$ Essentially, $S$ is the exponential $\exp(t M),$ hence cannot be written as a function of $x$ times a Fourier multiplier.}

\medskip

We first introduce para-differential symbols, which are regularized pseudo-differential symbols. Then we give an action result (Proposition \ref{prop:action}) and a composition result (Proposition \ref{prop:composition}), before giving a comparison result (Proposition \ref{prop:remainder}).

Given $\phi_0 \in C^\infty_c(\R^d),$ $0 \leq \phi_0 \leq 1,$ and real numbers $0 < A <  B < 2 A$ such that
$$
 \mbox{$\phi_0 \equiv 1$ for $|\xi| \leq A,$\, and \, $\phi_0 \equiv 0$ for $|\xi| \geq B.$}
$$
 We let
$$
 \phi_j(\xi) := \phi_0(2^{-j} \xi) - \phi_0(2^{-(j-1)} \xi), \qquad \mbox{for $j \geq 1,$}
$$
 so that for $j \geq 1,$ $\phi_j$ has support included in the annulus ${\mathcal C}_j := \{ A 2^{j-1} \leq |\xi| \leq B {2^{j}}\},$ and is constant equal to one in the annulus $\tilde {\mathcal C}_j := \{ B 2^{j-1} \leq |\xi| \leq {A 2^{j}}\}.$
 The function $\psi: \R^d \times \R^d \to \R$ defined for $N \geq 2$ by
 $$ \psi(\eta,\xi) = \sum_{k \geq 0} \phi_0(2^{-k + N} \eta) \phi_k(\xi).$$
 is called a Bony admissible cut-off \cite{Bony}. It satisfies
$$
\psi(\eta,\xi) \equiv \left\{ \begin{aligned} 1, & \quad |\eta| \leq 2^{-N} \langle \xi \rangle, \\ 0, & \quad |\eta| \geq 2^{1-N} \langle \xi\rangle.
 \end{aligned}\right.
$$

\begin{defi}\phantomsection \label{defi:para} Given a Bony admissible cut-off $\psi$ and $a \in \Gamma^m_k,$ we call para-differential symbol associated with $a$ the symbol
$$
 a^\psi(x,\xi) := {\mathcal F}^{-1} \Big( \psi(\cdot,\xi) \hat a(\cdot, \xi)\Big) (x) = \Big(\big( {\mathcal F}^{-1} \psi(\cdot,\xi) \big) \star a(\cdot,\xi) \Big)(x),
$$
where convolution takes place in the spatial variable $x,$ the smooth function ${\mathcal F}^{-1} \psi(\cdot,\xi)(x)$ being the inverse Fourier transform of $\psi$ in its first variable $\eta.$ The pseudo-differential operator
 $\op^\psi(a) = \op(a^\psi)$
 is said to be the para-differential operator associated with $a$ in classical quantization.
\end{defi}

We define the para-differential operator (precisely, the $\e$-dependent family of operators) associated with $a$ in semi-classical quantization by
\be \label{def:para} \op_\e^\psi(a) := h_\e^{-1} \op^\psi(\tilde a) h_\e, \qquad \mbox{with $\tilde a(x,\xi) := a(\e x, \xi),$}\ee
where
$$
 (h_\e u)(x) := \e^{d/2} u(\e x), \qquad \| h_\e u \|_{H^s} = \| u \|_{\e,s}.
$$

\begin{rema}\phantomsection \label{rem:para} Note that
the maps $a \to \tilde a$ and $a \to a^\psi$ do not commute:
 $$ \op_\e^\psi(a) = h_\e^{-1} \op\big( \big(\tilde a\big)^\psi\big) h_\e \neq h_\e^{-1} \op\big(\widetilde{a^\psi}\big) h_\e =  \op_\e(a^\psi),$$
so that the para-differential operator associated with $a$ in semi-classical quantization is \emph{not}
 $\op_\e(a^\psi).$ The classical symbol of $\ope(a)$ is
 $$  (x,\xi) \to {\bf a}(x,\xi) = \Big( {\mathcal F}^{-1} \psi \star \tilde a\Big)\left(\frac{x}{\e}, \e \xi\right) = \int {\mathcal F}^{-1} \psi(y, \e \xi) a(x - \e y, \e \xi) \, dy,$$
 in the sense that $\ope(a) \equiv \op_1({\bf a}).$
\end{rema}

\begin{rema}\phantomsection \label{rem:cut-off} {\rm An admissible cut-off $\psi$ satisfies the bound
 $$ \|\d_\xi^\b {\mathcal F}^{-1} (\d_\eta^\a \psi)(\cdot,\xi) \|_{L^1(\R^d)} \leq  C_\b \langle \xi \rangle^{-|\b|-|\a|}, \qquad \a, \b \in \N^d, \, C_{\a\b} > 0.$$}
   \end{rema}

\begin{rema}\phantomsection \label{para:tensor} By property $\psi(0,\xi) \equiv 1$ of the Bony admissible cut-off, pseudo- and para-differential operators agree for Fourier multipliers:
 $$ \ope(a) \equiv \op_\e(a), \qquad a = a(\xi).$$
 For tensor products $a(x,\xi) = a_1(x) a_2(\xi),$ there holds $(a_1 a_2)^\psi \equiv a_1^\psi a_2,$ so that
 $$ \ope(a) = \ope(a_1) \op_\e(a_2) \equiv \ope(a_1) \ope(a_2), \qquad a_1 = a_1(x), \,\, a_2 = a_2(\xi).$$
\end{rema}

 \begin{prop}\phantomsection \label{prop:action} Given $m \in \R,$ $a \in S^m,$ there holds for $s \in \R,$ all $u \in H^{s+m}$ the bound
 $$\|\op^\psi_\e(a) u\|_{\e,s} \lesssim {\bf M}^m_{0,d}(a) \|u\|_{\e,s+m},$$
where
\be \label{def:Mnorm}
{\bf M}^m_{k,k'}(a) = \sup_{\begin{smallmatrix} (x,\xi) \in \R^d \times \R^d \\ |\a| \leq k, |\b| \leq k' \end{smallmatrix}} \langle \xi \rangle^{-(m - |\b|)} |\d_x^\a \d_\xi^\b a(x,\xi)|.
\ee
 \end{prop}

\begin{proof} See for instance Theorem 4.3.5 of \cite{M}.
\end{proof}

 \begin{prop}\phantomsection  \label{prop:composition} For all $m_1, m_2 \in \R,$ all $r \in \N^*,$ with $r \leq \bar s,$ given $a_1 \in S^{m_1},$ $a_2 \in S^{m_2},$ there holds
 $$ \op^\psi_\e(a_1) \op^\psi_\e(a_2) = \op^\psi_\e\big( a_1 \sharp_\e a_2\big) + \e^{r} R^\psi_r(a_1,a_2),$$
 with the notation
$$
 a_1 \sharp_\e a_2 = \sum_{|\a| < r} \e^{|\a|} \frac{(-i)^{|\a|}}{\a!} \d_\xi^\a a_1 \d_x^\a a_2,
$$
 and the bound, for $d^* = 2 d + r + 1,$ for all $s \in \R,$ all $u \in H^{s + m_1 + m_2 - r},$
 $$ \big\| R^\psi_r(a_1,a_2) u \big\|_{\e,s} \lesssim  \Big( {\bf M}^{m_1}_{0,d^*}(a_1) {\bf M}^{m_2}_{r,d}(a_2) + {\bf M}^{m_1}_{r,d^*}(a_1) {\bf M}^{m_2}_{0,d}(a_2)\Big) \|u\|_{\e,s + m_1 + m_2 -r}.$$
 \end{prop}

\begin{proof} See for instance Theorem 6.1.4 of \cite{M}, or Proposition B.21 of \cite{MZ}.
\end{proof}

 The need for a different action result, one which involves a smaller number of $\xi$-derivatives, was evoked in Section \ref{sec:criterion} in the introduction.

 \begin{prop}\phantomsection \label{prop:actionH}
Given $a \in S^0,$ there holds for $s \geq 0,$ $r \in \N$ with $r \leq s,$ $u \in H^{s}$ the bound
  \be \label{bd:h} \| \ope(a) u \|_{\e,s} \lesssim \sum_{\begin{smallmatrix} 0 \leq |\a| \leq d+1 \\ 0 \leq |\b| \leq r -1\end{smallmatrix}} \e^\b \sup_{\xi \in \R^d} | \d_x^{\a + \b} a(\cdot,\xi) |_{L^1(\R^d_x)} \| u \|_{\e,s} + \e^r {\bf M}^0_{r,d}(a) \| u \|_{\e,s-r},\ee
with norms ${\bf M}^m_{k,k'}$ defined in \eqref{def:Mnorm}. \end{prop}

\begin{proof} Theorem 18.8.1 in H\"ormander's treatise \cite{Hom3} asserts the bound, for $b \in S^0$ and $u \in L^2:$
 $$ | \op_1(b) u|_{L^2} \lesssim \sum_{|\a| \leq d+1} \sup_{\xi \in \R^d} | \d_x^\a b(\cdot,\xi) |_{L^1(\R^d_x)} | u |_{L^2}.$$
If $s = 0,$ it suffices to apply the above to $b = {\bf a},$ with notation drawn from Remark \ref{rem:para}. In the case $s > 0,$ we introduce a commutator:
 $$ \| \ope(a) u \|_{\e,s} \leq | \ope(a) \Lambda^s u |_{L^2} + \big| \big[ \Lambda^s, \ope(a) \big] u \big|_{L^2}.$$
For the first term in the above upper bound, we use the result for $s = 0.$ For the commutator, we use Remark \ref{para:tensor} and Proposition \ref{prop:composition}:
 $$ \big| \big[ \Lambda^s, \ope(a) \big] u \big|_{L^2} = \big| \big[ \ope(\langle \cdot \rangle^s), \ope(a) \big] u \big|_{L^2} \leq | \ope(\langle \cdot \rangle^s \sharp_\e a) u |_{L^2} + \e^r |R^\psi_r(\langle \cdot \rangle^s,a) u|_{L^2}.$$
 There holds
  $$ \ope\big(\d_\xi^\a \langle \cdot \rangle^s \d_x^\a a\big) u \equiv \ope\Big(\langle \cdot \rangle^{-(s - |\a|)} \d_\xi^\a \langle \cdot \rangle^s \d_x^\a a\Big) \Lambda^{s - |\a|} u,$$
  so that, by \eqref{bd:h} for $s = 0:$
 $$ | \ope(\d_\xi^\a \langle \cdot \rangle^s \d_x^\a a) u |_{L^2} \lesssim \sum_{|\b| \leq d+1} \sup_{\xi \in \R^d} | \d_x^{\a + \b} a(\cdot,\xi) |_{L^1(\R^d_x)} \| u \|_{\e,s - |\a|}.$$
Besides, by Proposition \ref{prop:composition}, $|R^\psi_r(\langle \cdot \rangle^s,a) u|_{L^2} \lesssim {\bf M}^0_{r,d}(a) \| u \|_{\e,s - r}.$
\end{proof}

We finally give two para-linearization estimates:
 \begin{prop}\phantomsection \label{prop:remainder} For all $r \in \N^*,$ $s \geq r,$ given $a \in H^s,$ there holds for all $u \in L^\infty,$
  \be \label{r:para1} \big\| \big( a - \op^\psi_\e(a)\big) u \big\|_{\e,s} \lesssim \| (\e \d_x)^r a \|_{\e,s-r} |u|_{L^\infty},
  \ee
  and for all $u \in L^2,$
  \be \label{r:para2} \big\| \big( a - \op^\psi_\e(a)\big) u \big\|_{\e,s} \lesssim |(\e\d_x)^s a|_{L^\infty} |u|_{L^2}.\ee
 \end{prop}

\begin{proof} For \eqref{r:para1}, see Proposition 5.2.2 in \cite{M}; for \eqref{r:para2} see Theorem 5.2.8 in \cite{M}.
\end{proof}
\subsection{Product laws in weighted Sobolev spaces}  \label{app:products}

{\it The need for product laws arises from the semilinear nature of the equations \eqref{0}.}

\medskip

 Given $u, v \in H^s \cap L^\infty,$ with $s \geq 0,$ there holds
 \begin{equation} \label{product}
  \| u v \|_{\e,s} \leq C \big( |u |_{L^\infty} \| v \|_{\e,s} + |v |_{L^\infty} \| u \|_{\e,s}\big),
 \end{equation}
 where $C > 0$ does not depend on $u,v.$
 Estimate \eqref{product}, used in connection with the Sobolev embedding $H^s \hookrightarrow L^\infty,$ for $s > d/2,$ gives a product law in $\| \cdot \|_{\e,s}.$ A problem with \eqref{product} is that the Sobolev embedding has a large norm $\sim \e^{-d/2}$ when $H^s$ is equipped with $\| \cdot \|_{\e,s}.$

 A way around this difficulty is given by the following estimate:
 \begin{equation} \label{product2}
 \| u v \|_{\e,s} \leq C \big( |u |_{L^\infty} \| v \|_{\e,s} + |(\e \d_x)^s u|_{L^\infty} |v|_{L^2}\big),
 \end{equation}
 where $C > 0$ does not depend on $u,v.$ Estimate \eqref{product2} is a direct consequence of Proposition \ref{prop:action} and estimate \eqref{r:para2} in Proposition \ref{prop:remainder}.

\section{An integral representation formula} \label{app:duh}

 We adapt to the present context an integral representation formula introduced in \cite{T4}. Consider the initial value problem
 \begin{equation} \label{buff01}
  \begin{aligned} \d_t u + \frac{1}{\sqrt \e} \ope(M) u  = f, \qquad u(0) \in H^s,\end{aligned}
  \end{equation}
 where $f \in L^\infty([0,T_\star |\ln \e|],H^s(\R^d)),$ for some $T_\star > 0.$
We assume that $M = M(\e,t,x,\xi)$ is a family of matrix-valued, time-dependent symbols as follows:

 \begin{assump} \phantomsection \label{ass:B} The family of symbols $M(\e,t,x,\xi) \in S^0$ is constant in $x$ outside the ball $\{ |x| \leq x_\star\},$ for some $x_\star > 0$ which does not depend on $(\e,t,\xi).$ Besides, there holds the uniform bounds
$$ \langle \xi \rangle^{|\b|} \big| \d_x^\a \d_\xi^\b M(\e,t,x,\xi) \big| \leq C_{\a\b} < \infty, \qquad \a, \b \in \N^d,$$
uniformly in $\e \in (0,\e_0),$ $t \in[0, T_\star |\ln \e|],$ and $(x,\xi) \in \R^{2d}.$
\end{assump}
 Above and below, multi-indices $\a$ for $x$-derivatives are restricted to $|\a| \leq \bar s,$ where $\bar s$ is the spatial regularity index introduced in Appendix {\rm \ref{app:symbols}.}

\medskip

We define the {\it flow} $S_0$ of $\e^{-1/2} M$ as the solution to the following system of linear ordinary differential equations, $0 \leq \t \leq t \leq T_0 |\ln \e|:$
 \begin{equation} \label{resolvent0}\d_t S_0(\t;t) + \frac{1}{\sqrt\e} M S_0(\t;t) = 0, \qquad S_0(\t;\t) = \mbox{Id}.\end{equation}
 For $S_0$ we assume an exponential growth in time:

 \begin{assump}\phantomsection \label{ass:BS}
There holds for some $\g^+ > 0,$ for all $\a,$
\be
\label{summ-cor-B} |\d_x^\a S_0(\t,t)| \lesssim |\ln \e|^* \exp\big( (t - \t) \g^+ \big),\ee
 uniformly in $0 \leq \t \leq t \leq T_\star |\ln \e|,$ where $|\ln \e|^* = |\ln \e|^{N_*}$ for some $N_* > 0$ independent of $(\e,\t,t,x,\xi).$
\end{assump}

In \eqref{summ-cor-B}, we do not assume that $x$-derivatives of $S_0$ lead to losses in powers of $\e,$ in spite of the $1/\sqrt \e$ prefactor in front of $M$ in \eqref{resolvent0}. This assumption is tailored for our application to $M$ defined by \eqref{M-spelled-out}, in which $x$-dependent terms have a $\sqrt \e$ prefactor.

\medskip

We introduce correctors $\{ S_q\}_{1 \leq q \leq q_0},$ defined
  as the solutions of the triangular system of linear ordinary differential equations
 \begin{equation} \label{resolventk}
 \left\{\begin{aligned} \d_t S_q + \frac{1}{\sqrt \e} M S_q + \sum_{1 \leq |\a| \leq [(q+1)/2]} \frac{(- i)^{|\a|}}{|\a|!}  \d_\xi^\a M \d_x^\a S_{q+1-2|\a|} & = 0, \\  S_{q}(\t;\t) & = 0. \end{aligned}\right.
  \end{equation}

 \begin{lemm}\phantomsection \label{lem:bd-S} Under Assumptions {\rm \ref{ass:B}} and {\rm \ref{ass:BS}}, there holds, for all
 $q \in [0,q_0],$ all $\a, \b,$ the bounds, for $0 \leq \t \leq t \leq T_\star |\ln \e|:$
  $$ \langle \xi \rangle^{|\b|} |\d_x^\a \d_\xi^\b S_q(\t;t) | \lesssim \e^{-|\b|/2} |\ln \e|^* \exp( (t - \t) \g^+).$$
 \end{lemm}

 \begin{proof}  By \eqref{resolvent0} and \eqref{resolventk}, there holds for $q \geq 1$
 \begin{equation} \label{sq} S_q(\t;t) =  \sum_{1 \leq |\a| \leq [(q+1)/2]} \frac{(- i)^{|\a|}}{|\a|!}  \int_{\t}^t S_0(t';t)  \d_\xi^\a M(t') \d_x^\a S_{q+1-2|\a|}(\t;t') \, dt'.
\end{equation}
From there, we see immediately that the bound $|\d_x^\a S_q| \lesssim \exp((t - \t) \g^+),$ which holds true for $q = 0$ by Assumption \ref{ass:BS}, propagates from $q$ to $q+1.$ The case $|\b|  > 0$ is proved similarly by induction. The loss of half a power of $\e$ with each $\xi$-derivative comes of course from the prefactor $1/\sqrt \e$ in front of $M$ in \eqref{resolvent0}.
\end{proof}

 \begin{lemm}\phantomsection \label{lem:bd-actionS} Under Assumptions {\rm \ref{ass:B}} and {\rm \ref{ass:BS}}, there holds, for all
 $q \in [0,q_0],$ all $u \in L^2,$ the bound
   \begin{equation} \label{action-S1} \| \op_\e^\psi(S_q(\t;t)) u \|_{\e,s} \lesssim |\ln \e|^* e^{(t - \t) \g^+} \| u \|_{\e,s}, \qquad 0 \leq \t \leq t \leq T_\star |\ln \e|.
 \end{equation}
 \end{lemm}

\begin{proof} We start with $q = 0.$ Let $\varphi_\star \in C^\infty_c(\R^d_x),$ with $0 \leq \varphi_\star \leq 1,$ and such that $\varphi_\star \equiv 1$ on $\{ |x| \leq x_\star\},$ where $x_\star$ is introduced in Assumption \ref{ass:B}. Then, $(1 - \varphi_\star) M,$ and consequently $(1 - \varphi_\star) S_0$ are independent of $x.$ We can thus apply estimate \eqref{est:fourier-mult} for Fourier multipliers to bound the action of $(1 - \varphi_\star) S_0:$
 $$ \| \op_\e((1 - \varphi_\star) S_0) u\|_{\e,s} \leq \| u \|_{\e,s},$$
 since $|(1 - \varphi_\star) S_0| \leq 1.$ For $\varphi_\star S_0,$ we use Proposition \ref{prop:actionH}:
 $$ \| \ope(\varphi_\star S_0) u \|_{\e,s} \lesssim \sum_{\begin{smallmatrix} |\a| \leq d+1 \\ |\b| \leq r -1\end{smallmatrix}} \e^\b \sup_{\xi \in \R^d} | \d_x^{\a + \b} (\varphi_\star S_0)(\cdot,\xi) |_{L^1(\R^d_x)} \| u \|_{\e,s} + \e^r {\bf M}^0_{r,d^*}(\varphi_\star S_0) \| u \|_{\e,s-r}.$$
For the $L^1$ norms, we use compactness of the support of $\varphi_\star$ and Assumption \ref{ass:BS}:
 $$   | \d_x^{\a} (\varphi_\star S_0)(\cdot,\xi) |_{L^1(\R^d_x)} \lesssim \sup_{|\a'| \leq |\a|} |\d_x^{\a'} S_0|_{L^\infty} \lesssim |\ln \e|^* \exp\big( (t - \t) \g^+ \big).$$
We bound the remainder with Lemma \ref{lem:bd-S}:
 $$ \e^r {\bf M}^0_{r,d_*}(\varphi_\star S_0) \lesssim \e^{r - d_*/2} |\ln \e|^* e^{(t - \t)\g^+}.$$
 If $r$ is large enough, then $r - d_*/2 = r/2 - d - 1/2 > 0,$ and we conclude that \eqref{action-S1} holds for $q = 0.$

 From \eqref{sq}, we see that the correctors $S_q,$ for $q \geq 1,$ vanish identically outside $\{ |x| \leq x_\star\}.$ Thus it suffices to consider $\varphi_\star S_q,$ and we use Proposition \ref{prop:actionH} again.
 \end{proof}

We let
$$
S := \sum_{0 \leq q \leq q_0} \e^{q/2} S_q.
$$
 The following Lemma expresses the fact that $\ope(S)$ is an approximate solution operator:
 \begin{lemm}\phantomsection \label{lem:duh-remainder} Under Assumptions {\rm \ref{ass:B}} and {\rm \ref{ass:BS}}, there holds
  \begin{equation} \label{tildeS1-duh} \op_\e^\psi(\d_t S) = \op_\e^\psi(M) \op_\e^\psi(S) + \rho_S,\end{equation}
  where for $0 \leq \t \leq t \leq T_\star |\ln \e|,$ for all $s \in \R$ and $u \in H^s:$
  \begin{equation} \label{tilde-remainder} \| \rho_S(\t;t) u \|_{\e,s} \lesssim \e^{q_0/4 - (d+2)} |\ln \e|^* \exp\big( (t - \t) \g^+\big) \| u \|_{\e,s-1}.\end{equation}
  \end{lemm}

  \begin{proof}
  By definition of $S$ and \eqref{resolventk},
  \begin{equation} \label{c1}
   - \d_t \ope(S) = {\rm I} + {\rm II},
   \end{equation}
  with the notation
 $$ \begin{aligned}
 {\rm I} & := \sum_{0 \leq q \le q_0} \e^{(q-1)/2} \ope(M S_q), \\ {\rm II} & :=  \sum_{\begin{smallmatrix} 0 \leq q \leq q_0 \\ 1 \leq |\a| \leq [(q+1)/2] \end{smallmatrix}} \e^{q/2} \frac{(- i)^{|\a|}}{|\a|!} \ope\Big(    \d_\xi^\a M \d_x^\a S_{q+1-2|\a|}\Big),
 \end{aligned}
$$
where by convention a sum over the empty set is zero. By Proposition \ref{prop:composition},
$$
 \begin{aligned} \ope(M S_q)  = \ope(M) \ope(S_q) & - \sum_{1 \leq |\a| \leq [(q_0 - q+1)/2]} \e^{|\a|} \frac{(-i)^{|\a|}}{|\a|!}
 \ope\Big(\d_\xi^\a M \d_x^\a S_q\Big) \\ & - \e^{1 + [(q_0 - q+1)/2]} R^\psi_{1 + [(q_0 - q+1)/2]}(M,S_q),
\end{aligned}$$
so that
$$
{\rm I}  = \e^{-1/2} \ope(M) \ope(S) - \sum_{\begin{smallmatrix} 0
\leq q \leq q_0-1 \\1 \leq |\a| \leq [(q_0 - q+1)/2]
\end{smallmatrix}} \e^{(q-1)/2 + |\a|} \frac{(-i)^{|\a|}}{|\a|!}
\ope\Big( \d_\xi^\a M \d_x^\a S_q\Big) - \rho_S,
$$
with
$$ \rho_S := - \sum_{0 \leq q \leq q_0}\e^{(1 + q)/2 + [(q_0 - q+1)/2]} R^\psi_{1 + [(q_0 - q+1)/2]}(M,S_q).$$
Changing variables in the double sum, we find
$$ {\rm I} = \e^{-1/2} \ope(M) \ope(S) - \sum_{\begin{smallmatrix} 1 \leq q' \leq  q_0  \\ 1 \leq |\a| \leq [(q'+1)/2]
\end{smallmatrix}} \e^{q'/2} \frac{(-i)^{|\a|}}{|\a|!} \ope\Big(\d_\xi^\a M \d_x^\a S_{q' +1 - 2 |\a|}\Big) - \rho_S,$$
hence
\begin{equation} \label{c2} {\rm I} + {\rm II} = \e^{-1/2} \ope(M) \ope(S) -
\rho_S.
\end{equation}
Identities \eqref{c1} and \eqref{c2} prove \eqref{tildeS1-duh}. From Proposition \ref{prop:composition} and Assumption \ref{ass:B}, we deduce, for $\tilde q := 1 + [(q_0  - q +1)/2]:$
 $$ \| R^\psi_{\tilde q}(M,S_q) u \|_{\e,s} \lesssim {\bf M}^0_{\tilde q ,2d+ 1 + \tilde q}(S_q) \| u \|_{\e,s - \tilde q},$$
 and with Lemma \ref{lem:bd-S} this implies
 $$ \| R^\psi_{\tilde q}(M,S_q) u \|_{\e,s} \lesssim \e^{- (d + 2) - (q_0 - q)/4} \exp\big( (t - \t) \g^+\big) \| u \|_{\e,s - \tilde q},$$
 whence \eqref{tilde-remainder} follows.
 \end{proof}

 The following theorem gives an integral representation formula for a solution to \eqref{buff01} in terms in $\ope(S):$

  \begin{theo}\phantomsection \label{th:duh} Under Assumptions {\rm \ref{ass:B}} and {\rm \ref{ass:BS}}, the initial value problem \eqref{buff01} has a unique solution $u \in C^0([0, T_\star |\ln \e|, H^s(\R^d)),$ which satisfies the representation
  \begin{equation} \label{buff0.01} u = \op_\e^\psi(S(0;t)) u(0) + \int_0^t \op_\e^\psi(S(t';t))(\Id + \e^\zeta R_1) \big( f + \e^{\zeta} R_2(\cdot) u(0)\big)(t')\, dt',
  \end{equation}
  for some $\zeta > 0,$ where for all $v \in L^\infty H^s,$ all $w \in H^s,$ 
  \begin{equation} \label{bd-R121}
  \| (R_1 v)(t) \|_{\e,s} \lesssim |\ln \e|^* \sup_{0 \leq t' \leq T_\star |\ln \e|} \| v(t') \|_{\e,s}, \quad  \| R_2(t) w \|_{\e,s} \lesssim |\ln \e|^{*} \| w \|_{\e,s},
  \end{equation}
uniformly in $0 \leq t \leq T_\star |\ln \e|.$ 
 \end{theo}

The parameter $\zeta$ is defined in \eqref{def:zeta} below. There we see that the larger $T_\star$ and $\g^+,$ the larger $q_0$ needs to be. In other words, given an observation time $T_\star,$ the spatial regularity that we need for $M$ is a function of $|M|_{L^\infty}.$ Indeed, spatial regularity is connected with $q_0,$ since the proof is essentially a Taylor expansion at order $q_0,$ and the growth rate $\g^+$ is connected with the $L^\infty$ norm of $M:$ think of $M$ as being independent of $t;$ then $S = \exp(-t M/\sqrt \e),$ and $\g^+$ indeed appears as the sup of $M.$

\begin{proof} Let $g \in L^\infty([0,T_\star |\ln \e|,H^s).$ By Lemma \ref{lem:duh-remainder}, the map
 $$
 u := \op_\e^\psi( S(0;t)) u(0) + \int_0^t \op_\e^\psi( S(t';t)) g(t') \, dt'
 $$
   solves \eqref{buff01} if and only if there holds for all $0 \leq t \leq T_1 |\ln \e|,$
\begin{equation} \label{cond-g1}
  (\Id + r) g(t)  = f(t) - \rho_S(0;t) u(0),
   \end{equation}
 where $r$ is the linear integral operator
  $$r: \qquad v \in L^\infty H^s \to \Big(t \to \int_0^t  \rho_S(\t;t) v(\t) \, d\t\Big) \in L^\infty H^s,$$
 and $\rho_S$ is the remainder in Lemma \ref{lem:duh-remainder}. We now choose the index $q_0$ that appears in the definition of $S$ such that
 \begin{equation} \label{def:zeta} \zeta := \frac{q_0}{4} - (d+2) - T_1 \g^+ > 0.
 \end{equation}
 Then, by estimate \eqref{tilde-remainder}, for all $0 \leq t \leq T_\star |\ln \e|,$ the operator $r$ maps $L^\infty(0, T_\star |\ln \e|, H^s)$ to itself, with the bound
  \begin{equation} \begin{aligned} \label{bd-rf1} \sup_{0 \leq t \leq T(\e)} \| (r v)(t) \|_{\e,s}
   & \lesssim \e^{\zeta} \sup_{0 \leq t \leq T(\e)} \| v \|_{\e,s} \end{aligned}
  \end{equation}
  In particular, for $\e$ small enough the operator $\Id + r$ is invertible,
  with inverse $(\Id + r)^{-1}$ bounded as an operator from $L^\infty([0,T_\star |\ln \e|], H^s_\e)$ to itself, uniformly in $\e.$ As a consequence,
  we can solve \eqref{cond-g1} in $L^\infty([0,T_\star |\ln \e|], H^s_\e),$
 and obtain the representation formula \eqref{buff0.01}, with
 $$ \e^\zeta R_1 :=  (\Id + r)^{-1} - \Id, \qquad \e^\zeta R_2(t) := - \rho_S(0;t).$$
Bound \eqref{bd-R121} follows from \eqref{bd-rf1} and \eqref{tilde-remainder}. Since $M \in S^0,$ $\ope(M)$ is linear bounded $H^s \to H^s,$ hence uniqueness is a consequence of the Cauchy-Lipschitz theorem.
\end{proof}

\section{Bounds for the symbolic flow}\label{app:symb-bound}

We give here a proof of Proposition \ref{est-flow-sym} page \pageref{est-flow-sym}, which we reproduce below:

\begin{prop}\phantomsection\label{est-flow-sym-app}
For all $T > 0,$ all $0 \leq \t \leq t \leq T |\ln \e|,$ $\a \in \N^d,$ the solution $S_0$ to
\be \label{S:in app}
 \d_t S_0+\frac{1}{\sqrt\e}M S_0=0, \qquad S_0(\t,\t)=\Id,
\ee
with $M$ defined in \eqref{M-spelled-out},
 satisfies the bound
$$
 |\d_x^\a S_0(\t,t)| \lesssim |\ln \e|^* \exp(t \g^+),
 $$
where the growth rate $\g^+$ is defined in \eqref{def:g+}:
 $\dsp{\g^+ = |a|_{L^\infty} \, \big| \max_{\xi\in {\mathcal R}_{12}^h} \Re e \, \big(\G(\xi)^{1/2}\big)\big|.}$
 \end{prop}

In a first step (Section \ref{sec:proof-autonomous}), we approximate $M$ by its value $M(\e,0,x,\xi)$ at $t = 0.$ Then the solution to \eqref{S:in app} is a matrix exponential. The general case follows by a simple perturbation argument (Section \ref{sec:app pert}).

\subsection{The autonomous case} \label{sec:proof-autonomous}

The matrix $M(\e,0,x,\xi)$ is block-diagonal:
\begin{equation*}
\begin{split}M(\e,0,x,\xi)&=\bp i \mu_1 &-\sqrt\e \, b_{12} &0&\cdots&0\\-\sqrt\e \,  b_{21} &i\mu_2&0&\cdots&0\\0&0&i \l_3&\cdots&0\\\vdots & \vdots & \vdots & \ddots & \vdots \\0&0&0&\cdots&i \l_J\ep
\end{split}
\end{equation*}
denoting $\mu_1 = \l_1(\cdot + k) - \o,$ $\mu_2 = \l_2,$ as in \eqref{def:mu}, and
$$
 b_{12}(x,\xi) := \varphi_1(x) a(x) \chi_0(\xi)  b_{12}^+(\xi), \qquad b_{21}(x,\xi) =   \varphi_1(x)  a(x)^* \chi_0(\xi) b_{21}^-(\xi).
 $$
Above, $b_{12}^+$ and $b_{21}^- \in \C^{N \times N}$ are the interaction coefficients associated with resonance $(1,2),$ as in \eqref{def:int-coeff}, $a$ is the initial amplitude \eqref{generic-data}, and $\varphi_1$ and $\chi_0$ are spatial and frequency cut-offs, respectively, as defined on page \pageref{intro cut-offs}.

We prove in this Section the bound:
\be
\label{summ1} \big| \exp\big(-\e^{-1/2} t M(\e,0,x,\xi) \big) \big|  \lesssim
|\ln \e|^* \exp(t \g^+).
\ee

\medskip

By reality of $\l_j,$ and the fact that the truncated interaction coefficients vanish identically outside $\mbox{supp}\,\varphi_1 \times \mbox{supp}\, \chi_0 \subset \mbox{supp}\,\varphi_1 \times {\mathcal R}_{12}^h,$ it suffices to prove the bound
\be \label{for:summ1} |\exp(-\e^{-1/2} t M_0)| \lesssim \exp(t \g^+), \qquad (x,\xi) \in \mbox{supp}\,\varphi_1 \times {\mathcal R}_{12}^h, \quad 0 \leq t \leq T |\ln \e|,
\ee
 where $M_0$ is the top left block of $M(\e,0,x,\xi):$
 $$ M_0 := \bp
i\mu_1 &-\sqrt\e
b_{12} \\-\sqrt\e
b_{21}&i\mu_2\ep \in \C^{2N \times 2N}.$$
The top left block of $M_0$ being a multiple of the identity, we can compute the characteristic polynomial of $M_0$ by blocks:
$$ \det( x \Id - M_0) = \det\big( (x-i\mu_1)(x-i\mu_2) \Id - \e b_{21}b_{12} \big), \qquad x \in \C.$$
The characteristic polynomial of $M_0$ has a zero at $x$ if and only if $(x-i\mu_1)(x-i\mu_2)$ is an eigenvalue
of $\e b_{21}b_{12}$. The ranks of $b_{12}$ and $b_{21}$ are at most one (Assumption \ref{ass:transp}(iii)), implying that the rank of $b_{21} b_{12}$ is at most one. Then the only possible nonzero
eigenvalue for $b_{21}b_{12}$ is its trace $\tr ( b_{12}b_{21})$, and
the spectrum of $M_0$ appears as
\be\label{eig-m}
i\mu_1,~~i\mu_2,~~\mu_{\pm}:=\frac{i}{2} \big(\mu_1 + \mu_2 \big) \pm
\frac{1}{2}
  \Big( 4 \e \tr b_{12} b_{21} - (\mu_1 - \mu_2)^2
  \Big)^{1/2},
\ee
with respective algebraic multiplicities $N-1,$ $N-1,$ one and one.



\begin{lemm}[Near the resonant set] \phantomsection\label{lem:S-resonance} If
$|\mu_1(\xi) - \mu_2(\xi)| \leq \sqrt{\e} |\ln \e|^2,$
 then there holds the bound $$|\exp(-\e^{-1/2} t M_0)| \lesssim |\ln \e|^* \exp(t \g^+),$$ for $(x,\xi) \in \mbox{supp}\,\varphi_1 \times {\mathcal R}_{12}^h,$ and $0 \leq t \leq T |\ln \e|.$
\end{lemm}

\begin{proof} Substracting $i \mu_1 \Id,$ we equivalently bound the exponential of
$$
\e^{-1/2} \tilde M_0 = \bp
0 &-
b_{12} \\-
b_{21}&i \e^{-1/2} (\mu_2 - \mu_1)\ep.
$$
In the space-frequency domain under consideration, the entries of $\e^{-1/2} \tilde M_0$ are $O(|\ln \e|^2).$ The eigenvalues of $\e^{-1/2} \tilde M_0$ are $0,$ $i \e^{-1/2} (\mu_2 - \mu_1),$ $\e^{-1/2} (\mu_\pm - i \mu_1),$ as given in \eqref{eig-m}. In particular, in the relation of unitary similarity to an upper triangular matrix (Schur decomposition): $\e^{-1/2} \tilde M_0 = Q^* U Q,$ the real parts of the diagonal entries of the upper triangular matrix $U$ are equal to $0$ or $\Re e \, \e^{-1/2} \mu_\pm.$ Since the norm of $\e^{-1/2} \tilde M_0$ controls the norm of $U,$ and since entries of $U$ are controlled by the norm of $U,$ the entries of $U$ above the diagonal are $O(|\ln \e|^2).$ Let the $2N \times 2N$ diagonal matrix
$$ Q_\e = \mbox{diag}\,(1, |\ln \e|^4, |\ln \e|^8, \dots, |\ln \e|^{4(2N-1)}).$$ 
Then, the matrix
$$ \tilde U = Q_\e U Q_\e^{-1}$$
is upper triangular, with the same diagonal entries as $U,$ and with entries above the diagonal which are $O(|\ln \e|^{-2}).$ There holds moreover
$$ \big| \exp(-\e^{-1/2}  t\tilde M_0) \big| = \big| \exp (t U) \big| \lesssim |\ln \e|^*  \big| \exp (t \tilde U) \big|.$$
Above $|\cdot|$ denotes the sup norm of the entries of a matrix in $\C^{2N \times 2N}.$ Let $\| \cdot \|$ denote the canonical Hermitian norm in $\C^{2N},$ and also the associated norm in $\C^{2 N \times 2 N}.$ Let $z(t) = \exp(t \tilde U) z(0).$ Then $z$ solves $z' = \tilde U z,$ and 
 $$ \d_t \| z \|^2 =  \langle \d_t z, z \rangle + \langle z, \d_t z \rangle = 2 \langle \Re e \, \tilde U z, z \rangle, \quad 2 \Re e \, \tilde U := \tilde U + \tilde U^*.$$
 so that
 $$ \frac{1}{2} \d_t \| z \|^2 \leq \| \Re e \, \tilde U \| \| z \|^2,$$
 implying, by Gronwall's lemma, for some $C > 0,$ 
 $$ \big| \exp(t \tilde U) \big| \leq C \exp \big( t \big\| \Re e \, \tilde U \big\| \big),$$
 where $|\cdot|$ denotes the sup norm in $\C^{2 N \times 2 N}.$ Given $x \in \C^{2N}$ with $\| x \| = 1,$ there holds
 $$ \| \Re e \, \tilde U x \|^2 = \sum_{1 \leq i \leq 2N} (\Re e \, \tilde U x)_i^2 = \sum_{1 \leq i \leq 2N} \left( (\Re e \, \tilde U)_{i,i} x_i + O(|\ln \e|^{-2})\right)^2,$$
 so that
 $$ \| \Re e \, \tilde U x \|^2 \leq \max_i (\Re e \, \tilde U)_{i,i} \| x \|^2 + O(|\ln \e|^{-2}),$$
 and this implies
 $$  \| \Re e \, \tilde U \| \leq \max_i (\Re e \, \tilde U)_{i,i} + O(|\ln \e|^{-2}).$$
The diagonal entries of $\Re e \, \tilde U$ are the real parts of the diagonal entries of $U,$ that is the real parts of the eigenvalues of $\e^{-1/2} \tilde M_0.$ Here $\mu_1$ and $\mu_2$ contribute zero. Thus it suffices to bound from above $\Re e \, \e^{-1/2} \mu_\pm.$ Based on the explicit formula
 $$ \Re e \, (a + i b)^{1/2} = \frac{1}{2} \big( (a^2 + b^2)^{1/2} + a \big)^{1/2}, \qquad a,b \in \R,$$
 we observe that 
 $$ \Re e \, \e^{-1/2} \mu_\pm \leq \Re e \, \tr \big(b_{12} b_{21}\big)^{1/2}.$$
Since $\tr \,b_{12} b_{21} = (\varphi_1(x) \chi_0(\xi))^2 |a(x)|^2 \G(\xi),$ with $\G$ defined in \eqref{def:G0}, the result follows from the bound $\Re e \, (\tr \,b_{12} b_{21})^{1/2} \leq \g^+.$ 
  \end{proof}
%

\begin{lemm}[Away from the resonant set] \phantomsection\label{lem:away} If
$
|\mu_1(\xi) - \mu_2(\xi)| > \sqrt{\e} |\ln \e|^2,$
then for $(x,\xi) \in \mbox{supp}\,\varphi_1 \times {\mathcal R}_{12}^h,$ and $0 \leq t \leq T |\ln \e|,$ there holds
\be \label{bd:away}|\exp(-\e^{-1/2} t M_0)| \lesssim 1.\ee
\end{lemm}

\begin{proof} We use a non-stationary phase argument that is analogous to the normal form reduction of Section \ref{normal-form1}. Consider $\tilde S_0 := \left(\begin{array}{cc} e^{- it \mu_1/\sqrt\e} & 0 \\ 0 & e^{-it \mu_2/\sqrt \e} \end{array}\right) S_0.$ A column $(y_1,y_2) \in \R^{N} \times \R^N$ of $\tilde S_0$ solves the system
 $$ y'_1 + e^{- it (\mu_1 - \mu_2)/\sqrt \e} b_{12} y_2 = 0, \quad y'_2 + e^{-it (\mu_2 - \mu_1)/\sqrt \e} b_{21} y_1 = 0.$$
 Integrating in time and then integrating by parts, we find
 $$ \begin{aligned} y_1(t)  = y_1(0) & - \frac{\sqrt \e}{ i (\mu_1 - \mu_2)} \Big( e^{-it(\mu_1 - \mu_2)/\sqrt \e} b_{12} y_2(t) -b_{12} y_2(0) \Big) \\ & \qquad - \frac{\sqrt \e}{i (\mu_2 - \mu_1)}  \int_0^t b_{12} b_{21} y_1(t') \, dt'. \end{aligned}$$
  This implies the bound, for $z_j(t) := \max_{[0,t]} |y_j|,$
  $$ z_1(t) \leq |y_1(0)| + C |\ln \e|^{-2} z_2(t) + C t |\ln \e|^{-2} z_1(t).$$
 In a time interval $[0,T |\ln \e|],$ this gives $z_1(t) \lesssim |y_1(0)| + |\ln \e|^{-2} z_2(t).$ By symmetry, we find the same bound for $z_2.$ We conclude that the symbolic flow is uniformly bounded in $\e,t,x,\xi$ in the frequency domain under consideration.
\end{proof}

\medskip

{\it Conclusion:} Lemmas \ref{lem:S-resonance} and \ref{lem:away} imply bound \eqref{for:summ1}, which in turns implies \eqref{summ1}.

\subsection{The general case} \label{sec:app pert}

Here we use a perturbative argument to show that the bound \eqref{summ1} for $\exp(t M(\e,0,x,\xi)/\sqrt \e)$ carries over to a bound for $\d_x^\a S_0.$ This will prove Proposition \ref{est-flow-sym-app}.

It suffices to bound $|\d_x^\a S_0|$ on the compact set $(x,\xi) \in \mbox{supp}\,\varphi_1 \times {\mathcal R}_{12}^h,$ since $M$ is diagonal with purely imaginary entries outside of it.

By assumption on $g$ (Assumption \ref{ass-u-a}), there holds $M =
M(0) + \e M_1,$ where $M(0) = M(\e,0,x,\xi),$ and $| M_1 | \lesssim |\ln \e|^{*},$ uniformly in $\e,t,x,\xi$ in the domain under consideration. The
equation in $\d_x^\a S_0$ is
 \begin{equation} \label{eq:S} \d_t \d_x^\a S_0 + \e^{-1/2} M(0) \d_x^\a  S_0 = - \e^{1/2} M_1 \d_x^\a S_0 - \e^{-1/2}[\d_x^\a, M] S_0.
 \end{equation}
 In the right-hand side of \eqref{eq:S}, the first term involves the unknown $\d_x^\a S_0$ but is small. The second term is bounded in $\e$ since $[\d_x^\a, M] = O(\sqrt \e),$ and involves only lower-order derivatives $\d_x^{a'} S_0,$ with $|\a'| < |\a|.$ From \eqref{eq:S} and definition of $S_0$ \eqref{S:in app}, we deduce the implicit integral representations
\begin{equation} \label{rep:S0}
  S_0(\t;t) = {\bf S}(\t;t)  - \e^{1/2} \int_\t^t {\bf S}(t';t) M_1(t') S_0(\t;t'),
  \end{equation}
where ${\bf S}(\t;t) := \exp\big( - \e^{-1/2}(t - \t) M(\e,0,x,\xi)\big),$
 and for $|\a| > 0:$
\begin{equation} \label{rep:S} \begin{aligned}
  \d_x^\a S_0(\t;t) & = - \e^{1/2} \int_\t^t {\bf S}(t';t) M_1(t') \d_x^\a S_0(\t;t') \\ & \qquad + \e^{-1/2} \int_\t^t {\bf S}(t';t) [\d_x^\a, M(t')] S_0(\t;t'))  \, dt'.\end{aligned}
  \end{equation}
We factor out the exponential growth by consideration of $
S_0^\flat(\t;t) := \exp\big( (t - \t) \g^+ \big) S_0(\t;t),$ and
define similarly ${\bf S}^\flat(\t;t) := \exp\big( (t - \t) \g^+\big)
{\bf S}(\t;t),$ so that, by bound \ref{summ1}, there holds
$|{\bf S}^\flat(\t;t)| \leq C |\ln \e|^{*}$. This gives
for $t \leq T |\ln \e|$ the bound
  $$ |S_0^\flat(\t;t)| \leq C |\ln \e|^{*}\Big(1 + \e^{1/2}\int_\t^t |S_0^\flat(\t;t')| \,dt'\Big),$$
 from which we immediately deduce $|S_0^\flat(\t;t)| \leq  C(T_1) |\ln \e|^{*}.$ Finally, assuming inductively the bound
  $$|\d_x^\a S_0(\t;t)| \leq C_\a |\ln \e|^{*} \exp\big( (t - \t) \g^+\big),$$ for $|\a| < \a_0$ and some $C_\a > 0,$ we deduce from \eqref{rep:S} with $0 < |\a| = \a_0$ the bound
 $$ |\d_x^\a S_0^\flat| \leq \e^{1/2} C |\ln \e|^{*} \int_\t^t |\d_x^\a S_0^\flat| \, dt' + C |\ln \e|^{*},  $$
 implying $|\d_x^\a S_0^\flat| \leq C |\ln \e|^{*},$ which concludes the proof.

\section{A Gronwall Lemma} \label{app:d}

Consider the integral inequality in $y: [0,T|\ln \e|] \to \R_+:$
\be \label{g1}
 y(t) \leq e^{t \g_0} y_0 + \delta_0 \int_0^t y(t') \,dt' + \e^{\eta_0} \int_0^t e^{(t - t') \g_0} y(t') \, dt',
\ee
where $\g_0 > 0,$ $\delta_0 > 0$ and $\eta_0 > 0.$ This type of inequality is typical of situations in which we perform mixed-type estimates: the second term in the right-hand side comes from an $L^2$ ``energy" estimate (as in Section \ref{est-upper-V0-a}), while the third term comes from a semi-group estimate (as in Section \ref{est-upper-V0}).

\begin{lemm}\phantomsection \label{lem:g} If inequality \eqref{g1} holds for all $t \in [0, T |\ln \e|],$ then under the condition
\be \label{cond:g}
 \eta_0 - T \delta_0 > 0,
 \ee
there holds, for $\e$ small enough, the bound
 \be \label{bd:g}
  y(t) \leq C(y_0,\delta_0) \exp\big( t \max(\g_0,\delta_0)\big), \qquad C(y_0,\delta_0) > 0.
 \ee
\end{lemm}

The proof is elementary and based on three applications of the standard Gronwall's lemma.

\begin{proof} We let $\dsp{z := y - \delta_0 \int_0^t y(t') \,dt',}$ so that, by Gronwall's lemma,
 $$
  y(t) \leq e^{t \delta_0} \bar z(t), \qquad \bar z(t) := \max_{[0,t]} z,
 $$
 and
 $$
 \bar z(t) \leq y_0 e^{t \g_0} + \e^{\eta_0} \int_0^t e^{(t - t') \g_0} e^{t' \delta_0} \bar z(t') \, dt'.
 $$
We now let $w(t) := e^{-t \g_0} \bar z(t),$ so that
$$
 w(t) \leq y_0 + \e^{\eta_0} \int_0^t e^{t' \delta_0} w(t')\,dt',
$$
and, with Gronwall's lemma
$$
  w(t) \leq y_0 \exp\Big( \frac{\e^{\eta_0}}{\delta_0} e^{t \delta_0}\Big).
$$
Under condition \eqref{cond:g}, for $\e$ small enough and $t \leq T |\ln \e|,$ this gives $w(t) \leq 2 y_0.$ Hence
 $$ y \leq 2 y_0 e^{t \g_0} + \int_0^t \delta_0 y(t') \,dt',$$
from which we deduce \eqref{bd:g} by another application of Gronwall's lemma.
 \end{proof}

\section{On regularity of the spectral decomposition} \label{app:regdec}

Coalescing eigenvalues of smooth matrices are typically {\it not} smooth; the canonical example being $\left(\begin{array}{cc} 0 & 1 \\ x & 0 \end{array}\right).$ In some cases, however, there is an ordering of the eigenvalues so that regularity is preserved, an example being $\left(\begin{array}{cc} 0 & 1 \\ x^2 & 0 \end{array}\right).$ A symmetric example is given by $\left(\begin{array}{cc} 0 & x \\ x & 0 \end{array}\right).$ The smoothness condition in Assumption \ref{ass:spectral} should be understood with this latter example in mind. Besides this smoothness condition, Assumption \ref{ass:spectral} states that the eigenvalues $\l_j$ and eigenprojectors $\Pi_j$ of family $A_0/i + A(\xi),$ where $A(\xi) = \sum_j \xi_j A_j,$ satisfy bounds \eqref{bd:spectral} page \pageref{bd:spectral}, which we reproduce here: for all $\b \in \N^d,$ some $C_\b >0:$
 \be \label{bd:spectralapp} |\d_\xi^\b \l_j(\xi)| \leq C_\b (1 + |\xi|^2)^{(1 - |\b|)/2}, \quad |\d_\xi^\b \Pi_j(\xi)| \leq C_\b (1 + |\xi|^2)^{ - |\b|/2}.\ee
Consider the family of matrices $A_\infty(\o,x) = A(\o) - i x A_0,$ with $(\o,x) \in \S^{d-1} \times \R.$

\begin{lemm}\phantomsection \label{lem:regspec} If for all $\o \in \S^{d-1},$ the family $A_\infty$ has smooth eigenvalues and eigenprojectors in a neighborhood of $(\o,0)$ in $\S^{d-1} \times \R,$ then bounds \eqref{bd:spectralapp} hold.
\end{lemm}

\begin{proof} By the assumed smoothness of $\l_j$ and $\Pi_j,$ we only need to prove bounds \eqref{bd:spectralapp} for large $|\xi|.$ There holds
 $$ A_0/i + A(\xi) = |\xi| A_\infty(\o,x), \quad \mbox{with $\dsp{\o = \frac{\xi}{|\xi|}, \,\, x = \frac{1}{|\xi|}.}$}$$
Thus, denoting $\l^\infty_j$ and $\Pi^\infty_j$ the eigenvalues and eigenprojectors of $A_\infty,$ there holds the correspondence
\be \label{corresp} \l_j(\xi) = |\xi| \l_j^\infty(\o, x), \quad \Pi_j(\xi) = \Pi^\infty_j(\o,x).\ee
By assumption, the eigenvalues of $A_\infty$ have Taylor expansions at all orders in $x$ at $(\o,0),$ with coefficients that are smooth in $\o:$
\be \label{taylor:eig} \l^\infty_j(\o,x) = \l^\infty_{j0}(\o) + x \l^\infty_{j1}(\o) + \dots + x^m \tilde \l^\infty_{jm}(\o,x),\ee
and smooth remainders $\tilde \l^\infty_{jm}.$ Via \eqref{corresp}, these translate into \eqref{bd:spectralapp} for $\l_j.$ Similarly, the Taylor expansions of the $\Pi^\infty_j$ around at $(\o,0)$ translate into large-frequency bounds for $\d_\xi^\b \Pi_j.$ \end{proof}

In one space dimension, Rellich's theorem \cite{R1,R2} ensures that analytic family of symmetric matrices have analytic eigenvalues and eigenvectors, so that the assumption of Lemma \ref{lem:regspec} is always satisfied. In dimension greater than one, eigenvalues are Lipschitz (a consequence of the characteristic polynomial being hyperbolic; see Brohnstein \cite{Br}, or Kurdyka and Paunescu \cite{KP}), but eigenvectors may fail to be even continuous, as shown by Example 6.1 in \cite{KP}:
 $$ \left(\begin{array}{cc} x_1^2 & x_1 x_2 \\ x_1 x_2 & x_2^2 \end{array}\right),$$
 for which the eigenvectors are $(x_1,x_2)$ and $(x_2,-x_1).$
\medskip

We conclude this paragraph by noting that under a smoothness condition at infinity, we have an asymptotic expansion for the eigenvalues. This will be useful in Appendix \ref{app:structure-resonant-set}, where we discuss existence of resonances at infinity.

\begin{lemm}\phantomsection \label{rem:sw} Under the assumption of Lemma {\rm \ref{lem:regspec}}, the eigenvalues $\l_j$ have asymptotic expansions
 \be \label{odd} \l_j(\xi) = c_{j}(\o) |\xi| + O\left(\frac{1}{|\xi|}\right).\ee
\end{lemm}

\begin{proof} Indeed, since $A_j$ are real symmetric and $A_0$ skew-symmetric, the transpose matrix of $A_\infty(\o,x)$ is equal to $A_\infty(\o,-x).$ But then the determinant of a matrix is equal to the determinant of its transpose, so that $A_\infty(\o,x)$ and $A_\infty(\o,-x)$ have the same eigenvalues. That is, there is an ordering of the eigenvalues of $A_\infty(\o,\cdot)$ so that the eigenvalues are even in $x.$ In their Taylor expansions \eqref{taylor:eig} at $x = 0,$ only even powers of $x$ appear. Via the correspondence \eqref{corresp}, this means in particular only odd powers of $|\xi|^{-1}$ in the asymptotic expansion of $\l_j,$ implying \eqref{odd}.
\end{proof}

\section{On existence of WKB approximate solutions} \label{app:onwkb}

We give here sufficient conditions for Assumption \ref{ass-u-a}, stating that the family of systems \eqref{0} admits WKB approximate solutions, to hold true.

\medskip

We first remark on conditions \eqref{e+-}, reproduced here:
\be \label{pola:app} \big( A_0 + \sum_{1 \leq j \leq d} A_j \d_{x_j} \big) \big( e^{\pm i (k \cdot x - \o t)/\e} \vec e_{\pm  1} \big) = 0, \quad \vec e_{-1} = (\vec e_1)^*.
\ee
In \eqref{pola:app} we state that (i) the matrix $A_0/i + A(k) = A_0/i + \sum_j k_j A_j$ has real eigenvalue $\o,$ (ii) the matrix $A_0/i - A(k)$ has eigenvalue $-\o,$ and (iii) these eigenvalues are associated with respective eigenvectors $\vec e_1$ and $\vec e_{-1}$ that satisfy the component-by-component conjugation relation $\vec e_{-1} = (\vec e_1)^*.$

Points (ii) and (iii) are consequences of point (i), and the structure of the differential operator. Indeed, from the equality
 $(A_0/i + A(k)) \vec e_1 = \o \vec e_1 \in \C^N,$
 applying component-by-component complex conjugation we find
 $$ (-A_0/i + A(k)) (\vec e_1)^* = \o (\vec e_1)^*,$$
 which by linearity of $A(\cdot),$ translates into points (ii) and (iii).

\medskip

A characteristic phase $\b = (\o,k) \in \R^{1 + d}$ is given, satisfying \eqref{pola:app}, such that $k \neq 0.$ For some $j_0 \in [1,J],$ there holds $\o = \l_{j_0}(k).$ (The eigenvalues $\l_j$ are introduced in Assumption \ref{ass:spectral}). We assume  that
\be \label{harmonics}
 \det \big(p \o + A_0/i + A(p k)\big) = 0 \iff p \in \{-1,0,1\},
\ee
meaning in particular that higher harmonics of the fundamental phase are not characteristic.
We now show that Assumption \ref{ass-u-a} holds under the {\it weak transparency} assumption
\be \label{weak:transp}
 {\bf \Pi}( p \b) \sum_{p_1 + p_2 = p}  B\big(\,{\bf \Pi}(p_1 \b) u, \, {\bf \Pi}(p_2 \b) v\,\big) \equiv 0, \qquad p \in \{-1,0,1\},
 \ee
 for all $u, v \in \C^N,$ where ${\bf \Pi}(\o',k')$ denotes the orthogonal projector onto the kernel of the skew-hermitian matrix $\o' + A_0 + A(ik').$ If $\b,$ for $p \in \{-1,0,1\},$ is a simple point on the characteristic variety, meaning that the branch $\xi \to \l_{j_0}(\xi)$ in the spectrum of $\big(A_0/i + A(\xi)\big)_{\xi\in \R^d}$ does not change multiplicity at $k,$ then ${\bf \Pi}(\b) = \Pi_{j_0}(k).$ We do not want to exclude the case of coalescing eigenvalues, however, especially at $p \b = (0,0),$ since it is frequently met in applications (see for instance the examples given in Section 2.3 of \cite{T1}).

\begin{prop}\phantomsection \label{lem:WKB} Given $(\o,k)$ satisfying \eqref{pola:app}, conditions \eqref{harmonics}\footnote{Condition \eqref{harmonics} is introduced here only as a matter of notational simplification; without any additional difficulty we could allow for a larger set of characteristic harmonics.} and \eqref{weak:transp} imply Assumption {\rm \ref{ass-u-a}}.
\end{prop}

\begin{proof} The goal is to construct $u_a$ satisfying \eqref{0a}, in the form \eqref{def-ua}, such that the polarization conditions \eqref{e+-}-\eqref{pol-con} and the bounds \eqref{est:va-ra} hold. From \eqref{def-ua}, we see in particular that we are considering highly-oscillating data:
\be \label{donnee:ua}
 u_a(\e,0,x) = e^{-ik\cdot x/\e} g(0,x)^* \vec e_{-1} + e^{i k \cdot x/\e} g(0,x) \vec e_1 + \sqrt \e v_a(\e,0,x), \quad k \neq 0.
 \ee
We introduce notation borrowed from \cite{JMR-TMB}:
 $$ \begin{aligned} L(\b \d_\theta) := - \o \d_\theta + A_0 + A(k \d_\theta) = - \o \d_\theta + A_0 + \sum_{1 \leq j \leq d} k_j A_j \d_\theta, \,\,\,\,
   L_1(\d_t, \d_x) := \d_t + A(\d_x),
   \end{aligned}$$
  and look for an approximate solution $u_a$ in the form of a {\it profile}:
 $$
  u_a(\e,t,x) = \big[ {\bf u}_a(\e,t,x,\theta) \big]_{\theta = (k \cdot x - \o t)/\e}
 $$
  where ${\bf u}_a(\e,t,x,\theta)$ is $2\pi$-periodic in the fast variable $\theta.$
  Then, for $u_a$ to satisfy \eqref{0a} it suffices that its representation ${\bf u}_a$ satisfies
  \be \label{eq:profil}
  \frac{1}{\e} L(\b \d_\theta)  {\bf u}_a + L_1(\d_t,\d_x) {\bf u}_a = \frac{1}{\sqrt \e} B({\bf u}_a, {\bf u}_a) + \e^{K_a} {\bf r}_a^\e,  \ee
 for some remainder ${\bf r}_a^\e$ with a trace $r_a^\e$ satisfying bound \eqref{est:va-ra}. We look for a solution to \eqref{eq:profil} in the form of a WKB expansion:
  \be \label{wkb:exp}
  {\bf u}_a = {\bf u}_0 + \e^{1/2} {\bf u}_1 + \e {\bf u}_2 + \dots + \e^{K_a} {\bf u}_{2K_a},
    \ee
    where each ${\bf u}_k$ is a profile and in particular can be expanded in Fourier series in $\theta.$
 We denote $u_{k,p} = u_{k,p}(t,x)$ the $p$-th Fourier coefficient in $\theta$ of the $k$-th profile ${\bf u}_k,$ and {\it assume} that the coefficients $u_{0,-1}$ and $u_{0,1}$ of ${\bf u}_0$ satisfy the polarization conditions \eqref{e+-}-\eqref{pol-con}.

  From \eqref{eq:profil} we derive a cascade of WKB equations, which are sufficient conditions for \eqref{wkb:exp} to solve \eqref{eq:profil}. The first, comprising terms of order $O(1/\e),$ is
  \be \label{eq:0}
  L(\b \d_\theta) {\bf u}_0 = 0.
  \ee
Decomposing in Fourier series, we find that \eqref{eq:0} is equivalent to
\be \label{eq:00}
 L(i p \b) u_{0,p} = 0, \qquad p \in \Z.
 \ee
   Under condition \eqref{harmonics}, equation \eqref{eq:00} is equivalent to
  \be \label{pola:0}
  u_{0,p} = 0, \, p \notin \{-1,0,1\}, \qquad {\bf \Pi}(p \b) u_{0,p} = u_{0,p}, \quad p \in \{-1,0,1\},
  \ee
  where we recall that notation ${\bf \Pi}$ was introduced just below \eqref{weak:transp}.
 In agreement with \eqref{donnee:ua}, we let $u_{0,0} \equiv 0.$ Then, conditions ${\bf \Pi}(\pm \b) u_{0,\pm 1} = u_{0,\pm 1}$  are implied by \eqref{e+-}-\eqref{pol-con}.

 Thus \eqref{eq:00} is satisfied, and we move on to the equation at order $O(1/\sqrt \e):$
  $$ L(p \d_\theta) {\bf u}_1  = B({\bf u}_0, {\bf
  u}_0),
 $$
 or, at the level of the Fourier coefficients:
\be \label{eq:1}
L(p \b) u_{1,p}  =\sum_{p_1+p_2=p} B\big(u_{0,p_1}, u_{0,p_2}\big).
\ee
For $|p| > 2,$ the above right-hand side is identically zero, and since by \eqref{harmonics} for such $p$ the matrix $L(i p \b)$ is invertible, we solve \eqref{eq:1} by $u_{1,p} = 0.$ For $p \in \{-1,1\},$ the right-hand side of \eqref{eq:1} is identically zero as well, since $u_{0,0} \equiv 0,$ so that \eqref{eq:1} reduces to the polarization conditions
\be \label{pola:1}
 u_{1,p} = {\bf \Pi}(p \b) u_{1,p}, \qquad p \in \{-1,1\}.
\ee
For $p \in \{-2,0,2\},$ projecting with ${\bf \Pi}(p \b)$ onto the kernel of matrix $L(i p \b),$ we find
 \be \label{wt:wkb} {\bf \Pi}(p\b) \sum_{p_1+p_2=p} B\big(u_{0,p_1}, u_{0,p_2}\big) = 0.\ee
By the polarization \eqref{pola:0} and the assumed transparency \eqref{weak:transp}, identity \eqref{wt:wkb} holds. The matrix $L(i p\b)$ being skew-hermitian, there holds $\C^N = \ker L(ip\b) \oplus \mbox{ran}\, L(ip\b),$ and we can define a partial inverse $L(i p\b)^{(-1)}$ by
 $$ L(ip\b)^{(-1)}(x+y) = z, \quad x \in \ker L(ip\b), \quad y \in \mbox{ran}\, L(ip\b), \quad y = L(ip\b) z,$$
 so that
 $$ L(ip\b)^{(-1)} L(i p\b) = \Id - {\bf \Pi}(p \b).$$
If $p \in\{-1,0,1\},$ then $L(ip\b)^{(-1)} = L(ip\b)^{-1},$ the actual matrix inverse.
Multiplying \eqref{eq:1} to the left by $L(ip\b)^{(-1)},$ we then find
\be \label{1-pi}
 (1 - {\bf \Pi}(p\b)) u_{1,p} = L(ip\b)^{(-1)}  \sum_{p_1+p_2=p} B\big(u_{0,p_1}, u_{0,p_2}\big), \qquad p \in \{-2,0,2\}.
\ee
At this stage, equation \eqref{eq:1} is solved. The ${\bf \Pi}(p k) u_{1,p},$ for $|p| \leq 1,$ are still undetermined.

  The equations at order $O(1)$ are
$$
   L(\b \d_\theta) {\bf u}_2 + L_1(\d_t,\d_x) {\bf u}_0 = B({\bf u}_0, {\bf u}_1) + B({\bf u}_1, {\bf u}_0),
$$
corresponding to equations
\be \label{eq:2}
 L(ip\b) u_{2,p} + L_1(\d_t, \d_x) u_{0,p} = \sum_{p_1 + p_2 = p} B(u_{0,p_1}, u_{1,p_2}) + B(u_{1,p_1}, u_{0,p_2})
 \ee
for the Fourier coefficients.

For $|p| > 3,$ the above right-hand side is identically zero, and we solve \eqref{eq:2} by $u_{2,p} \equiv 0.$ For $|p| = 2,$ equation \eqref{eq:2} reduces to
$$
 u_{2,p} =  \sum_{|p_1 + p_2| = 2} L(ip\b)^{-1} \Big( B(u_{0,p_1}, u_{1,p_2}) + B(u_{1,p_1}, u_{0,p_2}) \Big), \qquad |p| = 2.
$$
In particular, the coefficients $u_{2,2}$ and $u_{2,-2}$ are determined as soon as ${\bf \Pi}(p \b)u_{1,1}$ and $u_{1,-1}$ are determined. For $p = 0,$ projecting \eqref{eq:2} onto the kernel of $L(i p\b),$ we find that necessarily
$$
 \sum_{p_1 + p_2 = 0} {\bf \Pi}(0) \Big(  B(u_{0,p_1}, u_{1,p_2}) + B(u_{1,p_1}, u_{0,p_2}) \Big) = 0,
$$
an identity which holds indeed by \eqref{weak:transp}, \eqref{pola:0} and \eqref{pola:1}. For $p = 0,$ the other component is
$$
 (1 - {\bf \Pi}(0)) u_{2,0} = \sum_{p_1 + p_2 = 0} L(0)^{(-1)} \Big( B(u_{0,p_1}, u_{1,p_2}) + B(u_{1,p_1}, u_{0,p_2})\Big).
$$
Finally, for $|p| = 1,$ projecting \eqref{eq:2}, we find
$$
 {\bf \Pi}(p \b) L_1(\d_t, \d_x) {\bf \Pi}(p \b) u_{0,p} =  \sum_{|p_1 + p_2| = 1} {\bf \Pi}(p \b) \Big( B(u_{0,p_1}, u_{1,p_2}) + B(u_{1,p_1}, u_{0,p_2})\Big) .
$$
Only $u_{1,\pm 2}$ and $u_{1,0}$ contribute to the above right-hand side. By transparency \eqref{weak:transp}, we see that actually only $(1 - {\bf \Pi}( \pm 2 \b)) u_{1,\pm 2}$ and $(1 - {\bf \Pi}(0)) u_{1,0}$ contribute to the above right-hand side. With \eqref{1-pi}, we obtain
 \be \label{2.3}  \begin{aligned}
 {\bf \Pi} (\b) L_1(\d_t, \d_x) & {\bf \Pi}(\b) u_{0,1}  = {\bf \Pi}(p \b) B\big(u_{0,-1}\big)\, L(2 i k)^{-1} B(u_{0,1}, u_{0,1})
 \\ & + {\bf \Pi}(p \b) B\big( u_{0,1}\big)\, L(0)^{(-1)}\big( B(u_{0,1}) \, u_{0,-1} \big),
 \end{aligned}\ee
 and a similar equation in $u_{0,-1}.$ In \eqref{2.3}, we used Notation \ref{notation:b}: $B(u) \, v = B(u,v) + B(v,u).$

 The operator $\Pi(p \b) L_1(\d_t, \d_x) \Pi(p \b)$ is a transport operator at group velocity, or a family of transport operators, depending on whether $p \b$ is a simple point on the characteristic variety or not. If $\b$ is a simple point on the variety, meaning that there is no change in multiplicity for the branch $\l_{j_0}$ at $\b,$ then the operator $\Pi(\b) L_1(\d_t, \d_x) \Pi(\b)$ is the scalar transport operator
 $$ \Pi_{j_0}(k) L_1(\d_t, \d_x) \Pi_{j_0}(k) = \big( \d_t + \big(\nabla_\xi \l_{j_0})(k) \cdot \d_x \big) \Pi_{j_0}(k).$$
 If there is a change in multiplicity for $\l_{j_0}$ at $\b,$ meaning that several eigenvalues coalesce at $\b,$ then ${\bf \Pi}(\b) L_1(\d_t, \d_x) {\bf \Pi}(\b)$ is a family of transport operators \label{ref for transp}, with velocities given by the directional derivatives of the branches that intersect at $\b$\footnote{In the case of separated eigenvalues, these facts are proved in \cite{La,DJMR}; in the case of coalescing eigenvalues, these facts are proved in \cite{Lt,T1}. The article \cite{T1} contains unified proofs for both simple and coalescing cases, and also for higher-order operators, such as Schr\"odinger, that arise in three-scale approximations.}.

In both cases, simple or coalescing eigenvalues, equation \eqref{2.3} is hyperbolic with a cubic nonlinearity, and can be solved locally in time for smooth Sobolev data.

The other component of \eqref{eq:2} for $|p| = 1$ is
$$ (1 - {\bf\Pi}(p\b)) u_{2,p} =  \sum_{|p_1 + p_2| = 1} L(ipk)^{(-1)} \Big( B(u_{0,p_1}, u_{1,p_2}) + B(u_{1,p_1}, u_{0,p_2})\Big).$$

Summing up, we see that at this stage:
\begin{itemize}
\item the leading term ${\bf u}_0 = u_{0,-1} e^{-i\theta} + u_{0,1} e^{i \theta}$ in \eqref{wkb:exp} is completely determined, with amplitudes $u_{0,\pm1}$ solving semilinear hyperbolic equations \eqref{2.3};
\item the first corrector ${\bf u}_1 = \sum_{|p| \leq 2} e^{i p \theta} u_{1,p}$ is known, except for ${\bf \Pi}(p \b) u_{1,p},$ $|p| \leq 1,$ for which we have no information so far;
\item the second corrector ${\bf u}_2 = \sum_{|p| \leq 3} e^{i p \theta} u_{2,p}$ is known, except for $u_{2,\pm 2},$ which will be determined by ${\bf \Pi}(\pm \b) u_{1,\pm 1},$ and for  ${\bf \Pi}(p \b) u_{2,p},$ $|p| \leq 1,$ for which we have no information so far.
\end{itemize}

  We can go on with the expansion up to any order $2 K_a.$ The components ${\bf \Pi}(p \b) u_{1,p},$ for $|p| \leq 1,$ are determined by the equation at order $O(\sqrt \e):$ they satisfy {\it linear} transport equations
  $$ {\bf \Pi}(p \b) L_1(\d_t, \d_x) {\bf \Pi}(p \b) u_{1,p} = \big({F}_1({\bf u}_0) {\bf u}_1\big)_p,$$
 where a typical term in the source ${F}_1$ is
 $$ {\bf \Pi}(p \b) B\Big( {\bf \Pi}(p \b) u_{1,p}, L(0)^{(-1)} B(u_{0,1}, u_{0,-1})\Big).$$
 Similarly, the components ${\bf \Pi}(p \b) u_{2,p}$ are determined by the equations at order $O(\e);$ more generally, the components ${\bf \Pi}(p \b) u_{\ell, p}$ are determined by the equations at order $O(\e^{\ell/2}),$ they satisfy linear tranport equations, with source terms which are polynomials in the Fourier coefficients of the lower-order profiles ${\bf u}_{\ell'},$ with $0 \leq \ell' < \ell.$ In particular, these equations can be solved over any interval of existence for \eqref{2.3}.

 By construction, the correctors ${\bf v}_a$ and the remainder ${\bf r}_a^\e$ are trigonometric polynomials in $\theta,$ so that bounds \eqref{est:va-ra} hold. This completes the proof of Proposition \ref{lem:WKB}. \end{proof}

Condition \eqref{weak:transp}, introduced by Joly, M\'etivier and Rauch \cite{JMR-TMB}, can be linked to transparency in the sense of Definition {\rm \ref{def:transp}}:

\begin{prop}\phantomsection \label{rem:weaktransp} Assuming \begin{itemize}
\item condition \eqref{harmonics} describing the set of harmonics of the fundamental phase; and denoting $j_0, j_1$ the indices such that $\l_{j_0}(k) = -\l_{j_1}(-k) = \o_0;$
\item assuming that $\vec e_1$ and $\vec e_{-1}$ generate the kernels of $i \o + A_0 + A(k)$ and $-i \o + A_0 + A(-i k)$ respectively;
\item denoting ${\mathcal J}$ the set of indices $j$ such that $\l_j(0) = 0,$ and assuming that the resonances $(j_0,j)$ and $(j,j_1)$ are transparent, for all $j \in {\mathcal J},$
\end{itemize}
 then weak transparency \eqref{weak:transp} holds.
 \end{prop}

\begin{proof} There holds $\Pi_{j_0}(k) = {\bf \Pi}(\b),$ $\Pi_{j_1}(-k) = {\bf \Pi}(-\b),$ and ${\bf \Pi}(0) = \sum_{j \in {\mathcal J}} \Pi_{j}(0),$ so that the weak transparency condition \eqref{weak:transp} is implied by conditions
\be \label{wt2}
\left\{\begin{aligned}
 \Pi_j(0) \Big( B(\vec e_1, \vec e_{-1}) + B(\vec e_{-1}, \vec e_1) \Big) & = 0, \\
 \Pi_{j_0}(k) \Big( B(\vec e_1, \Pi_j(0)\, \cdot\,) + B(\Pi_j(0)\, \cdot\,, \vec e_1) \Big) & = 0,\\
 \Pi_{j_1}(-k) \Big( B(\vec e_{-1}, \Pi_j(0)\, \cdot \,) + B(\Pi_j(0) \,\cdot \,, \vec e_{-1}) \Big) & = 0.
 \end{aligned}\right.
 \ee
 With Notation {\rm \ref{notation:b}}, conditions \eqref{wt2} take the form\footnote{The lack of symmetry between $\vec e_1$ and $\vec e_{-1}$ in \eqref{wt3} is only apparent: it suffices indeed to reformulate the first condition in \eqref{wt3} as $\Pi_j(0) B(\vec e_1) \Pi_{j_1}(-k) = 0,$ associated with resonance \eqref{res1}, to restore symmetry.}
 \be \label{wt3}
 \Pi_j(0) B(\vec e_{-1}) \Pi_{j_0}(k) = 0, \quad \Pi_{j_0}(k) B(\vec e_1) \Pi_j(0) = 0, \quad \Pi_{j_1}(-k) B(\vec e_{-1}) \Pi_j(0) = 0.
 \ee
Since $j \in {\mathcal J},$ the (trivial) resonance
$$
$$
 occurs at $\xi = 0.$ Similarly, the trivial resonance
 \be \label{res1} \o = \l_j(0) - \l_{j_1}(-k) = 0 -  \l_{j_1}(-k)
 \ee occurs at $\xi = - k.$ The first two conditions in \eqref{wt3} can then be seen as a (partial) transparency condition for the resonances $(j_0,j),$ with $j \in {\mathcal J}.$ By partial we mean here that under \eqref{wt3} the bounds of Definition {\rm \ref{def:transp}} hold a priori only at $\xi = 0.$ The third condition in \eqref{wt3} is a (partial, i.e., only at $\xi = -k$) transparency condition for resonances $(j,j_1),$ with $j \in {\mathcal J}.$ Since by assumption the resonances $(j_0,j)$ and $(j,j_1)$ are transparent, condition \eqref{wt3} holds, implying \eqref{weak:transp}.
 \end{proof}

 Propositions \ref{lem:WKB} and \ref{rem:weaktransp} lead to sufficient conditions for Assumption {\rm \ref{ass:several-res}} to imply Assumption {\rm \ref{ass-u-a}}, as follows:

\begin{coro} \phantomsection \label{cor:8-2} Under \eqref{pola:app}-\eqref{harmonics}, if $\vec e_{\pm 1}$ generate the respective kernels of $\pm i \o + A_0 + A(\pm ik),$ and if eigenvalues $\l_j$ such that $\l_j(0) = 0$ are not involved in non-transparent resonances, meaning that for such $j,$ for all $j',$ there holds $(j,j') \notin {\mathfrak R}_0$ and $(j',j) \notin {\mathfrak R}_0,$ then Assumption {\rm \ref{ass:several-res}} implies Assumption {\rm \ref{ass-u-a}.}
\end{coro}

\begin{proof} This is an immediate consequence of Propositions \ref{lem:WKB} and \ref{rem:weaktransp}. \end{proof}

\begin{rema}\phantomsection \label{rem:jmr} In different settings, weak transparency conditions such as \eqref{weak:transp} may not guarantee existence of WKB solutions. One possible obstruction is a lack of well-posedness of the limiting equations.

 For instance, the regime considered by Joly, M\'etivier and Rauch in \cite{JMR-TMB} is
\be \label{jmr} \d_t u + \frac{1}{\e} A_0 u + \sum_{1 \leq j \leq d} A_j \d_{x_j} u = \frac{1}{\e} B(u,u).\ee
This is a more singular regime than \eqref{0}, on which we further comment in Remark {\rm \ref{rem:last}} below. In particular, in the context of \eqref{jmr} the profile equations in $\Pi_{j_0}(p k) u_{0,p}$ are typically \emph{quasi-linear}, when \eqref{2.3} was semi-linear. For triangular source terms $B,$ Joly, M\'etivier and Rauch give sufficient conditions, in the form of transparency conditions, for these quasi-linear profiles equations to be well-posed. This is Assumption 2.2 in \cite{JMR-TMB}, and it is strictly stronger than their weak transparency assumption (Assumption 2.1 in \cite{JMR-TMB}) guaranteeing existence of a WKB cascade, and strictly weaker than the conditions that guarantee stability (Assumption 2.5 in \cite{JMR-TMB}).

Another example is given in \cite{T3}. There, the second author considered \emph{quasi-linear} Euler-Maxwell systems in the scaling
\be \label{em2} \d_t u + \frac{1}{\e} A_0 + \frac{1}{\sqrt \e} \sum_{1 \leq j \leq d} A_j \d_{x_j} u + \sum_{1 \leq j \leq d} \tilde A_j(u) \d_{x_j} u = \frac{1}{\sqrt \e} B(u,u),\ee
and proved stability of WKB solutions with leading amplitudes solving the Zakharov system describing Langmuir turbulence. In particular, for \eqref{em2} just like for \eqref{jmr}, and as opposed to \eqref{0}, the well-posed character of the limiting equations is far from trivial. For the Zakharov system, local-in-time well-posedness in smooth Sobolev spaces was first proved by Schochet and Weinstein {\rm \cite{SW}} and Ozawa and Tsutsumi {\rm \cite{OT}}.

\end{rema}

\begin{rema}\phantomsection \label{rem:ghosts} For some physical systems, the weak transparency condition \eqref{weak:transp} is actually \emph{not} satisfied. This was proved for Maxwell-Bloch by Joly, M\'etivier and Rauch \cite{JMR-TMB}, and by the second author for Euler-Maxwell \cite{T2}. WKB solutions can then sometimes be constructed for a restricted set of initial data.

Consider for instance the situation in which after spelling out one component of \eqref{weak:transp}, we arrive at condition
 \be \label{defect} n_{0,0} v_{0,1} = 0,\ee
 where $n$ and $v$ are components of the solution $u$ (with $n$ representing for instance variation of density and $v$ velocity, as in Euler-Maxwell). Here we are using notation introduced in the proof of Proposition {\rm \ref{lem:WKB}}, so that $n_{0,0}$ denotes the mean mode of the leading term in the variation of density. Equation \eqref{defect} can be solved by $n_{0,0} = 0,$ if compatible with the datum, for instance in the case of highly-oscillating initial data in velocity,
with $O(\e)$ initial variations of density.

In such situations, the limiting system often involves both $v_{0,1}$ and $n_{1,0},$ meaning a coupling between leading order terms and corrector terms: the term $\e n_{1,0},$ which vanishes in the limit $\e \to 0,$ in particular cannot be measured, has a measurable effect on the leading term $v_{0,1}.$  This is akin to the \emph{ghost effect} that was studied in depth by the Kyoto school \cite{Sone,TA} for rarefied gas dynamics.
\end{rema}

\begin{rema}\phantomsection \label{rem:last} We finally comment on the link between the specific regime in Joly, M\'etivier and Rauch's article, as described in Remark {\rm \ref{rem:jmr}} above, absence of transparency as described in Remark {\rm \ref{rem:ghosts}}, and our supercritical regime \eqref{0}.

 From \eqref{0}, imagine that the hyperbolic operator is block-diagonal, like for instance system \eqref{kg1} from Section {\rm \ref{sec:KG}.} This is the case for the Maxwell-Bloch equations. Suppose then that in the coordinate system $u = (u_1,u_2)$ in which the hyperbolic operator is block-diagonal, the source $B$ has the form
  $$ B(u,u) = \left(\begin{array}{c} B_1(u_1,u_2) \\ B_2(u_1,u_1) \end{array}\right),$$
  and assume that $B_1$ does {\rm not} satisfy the weak transparency condition \eqref{weak:transp}. Then, rescaling
  $\tilde u = (\tilde u_1, \tilde u_2) := (u_1, u_2/\sqrt \e),$
  we find \eqref{jmr}, with $(1/\e)B$ replaced by $$\dsp{\frac{1}{\e} \left(\begin{array}{c} 0 \\ B_2(\tilde u_1, \tilde u_1) \end{array}\right)  + \left(\begin{array}{c} B_1(\tilde u_1, \tilde u_2) \\ 0 \end{array}\right)}.$$ The regime is now more singular, but the source is more transparent, and has a triangular structure, as in \cite{JMR-TMB}.
\end{rema}

\section{On structure of the resonant set} \label{app:structure-resonant-set}

\begin{lemm}\phantomsection \label{lem:resonantset} Under the assumption of Lemma {\rm \ref{lem:regspec}}, the set ${\mathfrak R}$ comprising all resonant frequencies is bounded as soon as the asymptotic branches on the characteristic variety are distinct: $c_i \neq c_j$ for $i \neq j,$ with notation borrowed from Lemma {\rm \ref{rem:sw}}.
\end{lemm}

\begin{proof} If the resonant set is unbounded, then some ${\mathcal R}_{ij}$ has an accumulation point at infinity. By Lemmas \ref{lem:regspec} and \ref{rem:sw}, this gives an equality
$$ c_i|\xi + k| = \o + c_j |\xi| + O\left(\frac{1}{|\xi|}\right),$$
along a sequence $|\xi| \to\infty,$ implying $c_i = c_j.$
\end{proof}

\subsection{Euler-Maxwell} \label{sec:EM} We verify that the assumptions of Lemma \ref{lem:resonantset} are satisfied by the Euler-Maxwell equations (EM) of Section \ref{sec:der3EM}. First we check that the assumption of Lemma \ref{lem:regspec} is satisfied. Here $A_\infty(\o)$ is block-diagonal:
 $$ A_\infty(\omega) = \left(\begin{array}{cccccc} 0 & \o \times & 0 & 0 & 0 & 0 \\ - \o \times & 0 & 0 & 0 & 0 & 0 \\ 0 & 0 & 0 & \theta_e \o & 0 & 0 \\ 0 & 0 & \theta_e \omega \cdot & 0 & 0 & 0 \\ 0 & 0 & 0 & 0 & 0 & \a^2 \theta_i \o \\ 0 & 0 & 0 & 0 & \theta_i \o \cdot & 0 \end{array}\right).$$
In particular, the eigenvalues of $A_\infty(\o)$ are $\big\{\pm |\o|, \,\, \pm \theta_e |\o|, \,\, \pm \a \theta_i |\o|\big\}.$
Given $\o$ on the sphere, the eigenvalues are separated, since $\theta_i \ll \theta_e \ll 1,$ and $\a < 1.$ By Rouch\'e's theorem, smoothness is preserved under small perturbations. In particular, the eigenvalues of $A_\infty(\o) - i x A_0$ are smooth, with respect to $\o$ and $x,$ locally around $(\o,0) \in \S^{d-1} \times \R.$ This verifies the assumption of Lemma \ref{lem:regspec}.

 Next the asymptotic description of the eigenvalues of (EM) in \eqref{asymptotes:em}, or a look at Figure \ref{figem}, shows that the separation assumption at infinity is satisfied.

 Hence Lemma \ref{lem:resonantset} applies: the resonant set for Euler-Maxwell is bounded.

\subsection{Maxwell-Landau-Lifschitz} \label{sec:MLL}

We conclude this Appendix by giving an example in which the assumptions of Lemma \ref{lem:resonantset} do {\it not} appear to be satisfied.
The Maxwell-Landau-Lifschitz equations are
 \begin{equation}\label{MLL} \left\{\begin{aligned}
 \d_t E-\nabla \times H &= 0, \\
      \d_t H+\nabla \times E &= M \times H, \\  \d_t M & =-M\times H.
                          \end{aligned} \right.
\end{equation}
For the linearized equations around the family of constant solutions
$$ (E, H, M) = (0,\a M_0, M_0) \in \R^9,$$
with $\a \in \R,$ coordinatizing $M_0 = (1,0,0),$ the characteristic variety has equation
$$ \l^3\Big( \l^6-2(2+|\xi|^2)\l^4+(|\xi|^2(6+|\xi|^2)-2\xi_1^2)\l^2-|\xi|^2(2|\xi|^2-\xi_1^2)\Big) = 0.$$
The one-dimensional Maxwell-Landau-Lifshitz are transparent in a strong sense; this was shown by the first author in \cite{Lu}.

\addtocounter{figure}{5}
\begin{figure}\begin{center}
\includegraphics[scale=.4]{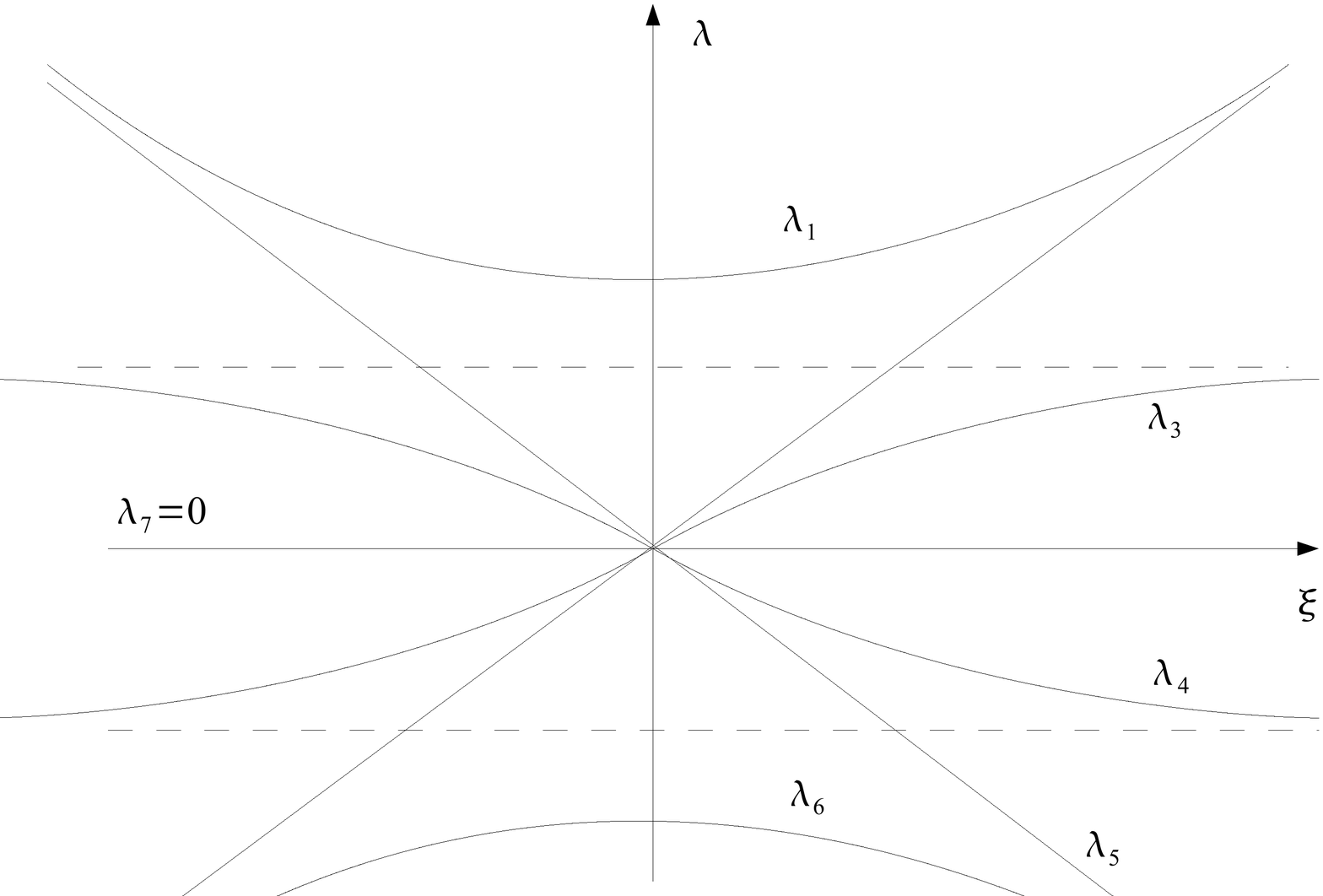}
\caption{The characteristic variety for the Maxwell-Landau-Lifschitz equations.}
\label{fig4}\end{center}
\end{figure}

In one space dimension, the variety is pictured on Figure \ref{fig4}. The asymptotic branches are not distinct, meaning that we cannot apply Lemma \ref{lem:resonantset} in order to prove boundedness of the resonant set.

In three space dimensions, numerical calculations by the first author show existence of resonances for large values of $|\xi|,$ meaning that Assumption \ref{ass:several-res}(i) is probably not satisfied by \eqref{MLL}.  Moreover, these resonances are non-transparent, suggesting instability.

\section{Notation index}

 \setlength{\columnseprule}{0.001cm}
 \begin{multicols}{2}

$\lesssim\,:$ inequality up to a constant, page \pageref{notation:lesssim}

\smallskip

$(\cdot,\cdot):$ scalar product in $\C^N,$ page \pageref{hermit}

\smallskip

$|\cdot|:$ sup norm in $\C^N, \C^{N \times N},$ page \pageref{matrix:norm}

\smallskip

$\|\cdot\|_{\e,s}:$ weighted Sobolev norm, page \pageref{weighted:norm}

\smallskip

$\prec\,:$ binary relation for cut-offs, page \pageref{notation:cut-off}

\smallskip

${\mathcal A}:$ real diagonal symbol, \eqref{def:calA} page \pageref{def:calA}

\smallskip

${\bf A}:$ real diagonal symbol, page \pageref{bfA}

\smallskip

$B(\cdot):$ linearized source term, page \pageref{notation:b}

\smallskip

$B_p:$  avatar of $B,$ \eqref{def:Bp} page \pageref{def:Bp}

\smallskip

${\mathcal B}^r:$ avatar of $B$, \eqref{def:Br} page \pageref{def:Br}

\smallskip

$\check {\mathcal B}:$ avatar of $B$, \eqref{eq:check-u} page \pageref{eq:check-u}

\smallskip

$b_{ij}^\pm:$ interaction coefficients, pages \pageref{def:int-coeff}, \pageref{def:int-coeff:ij}

\smallskip

$\g:$ maximal growth coefficient, page \pageref{def:several-gamma}

\smallskip

$\g_{ij}:$ growth coefficient, page \pageref{def:several-gamma}

\smallskip

$\g^-:$ lower growth rate, page \pageref{low0}

\smallskip

$\g^+:$ upper growth rate, page \pageref{def:g+}

\smallskip

${\bf \G}:$ stability index, pages \pageref{def:trace} and \pageref{def:index}

\smallskip

$\G:$ trace \!of \!interaction \!coefficients, page \pageref{def:G0}

\smallskip

$g:$ leading WKB amplitude, page \pageref{pol-con}

\smallskip

$\theta = (k \cdot x - \o t )/\e,$ page \pageref{def:theta}

\smallskip

$|\ln \e|^*:$ arbitrary power of $|\ln \e|,$ page \pageref{def:g+}

\smallskip

$\l_j:$ eigenvalues, page \pageref{sp dec}

\smallskip

$\mu_j:$ shifted eigenvalues \eqref{def:mu}, page \pageref{def:mu}

\smallskip

${\bm \mu}_\a:$ shifted eigenvalues, page \pageref{def:bm:mu}

\smallskip

$\xi_0:$ frequency argmax, \eqref{def:xi0} page \pageref{def:xi0}

\smallskip

$\ope:$ para-differential operator, page \pageref{def:para}

\smallskip

$\op_\e:$ pseudo-differential operator, page \pageref{quantiz}

\smallskip

$\Pi_j:$ eigenprojectors, page \pageref{sp dec}

\smallskip

$\varphi_{\dots}:$ spatial cut-offs, pages \pageref{intro cut-offs} and \pageref{encore x cut}

\smallskip

$\chi_{\dots}:$ frequency cut-offs, pages \pageref{intro cut-offs} and \pageref{encore cut-offs}

\smallskip

${\mathfrak R}:$ resonant frequencies, page \pageref{ass:several-res}

\smallskip

${\mathfrak R}_0:$ non-transparent resonant frequencies, page \pageref{ass:several-res}

\smallskip

$R_0:$ uniform remainder, page \pageref{uniform remainder}

\smallskip

${\mathcal R}_{ij}:$ $(i,j)$ resonant frequencies, page \pageref{def-reson}

\smallskip

${\mathcal R}_{12}^h:$ neighborhood of $(1,2)$ resonant frequencies, page \pageref{def:Rh}
\smallskip

$S_0(0;t):$ symbolic flow, pages \pageref{rep:V} and \pageref{resolvent0}

\smallskip

$\sigma_{+p}:$ frequency-shifted symbol $\sigma,$ page \pageref{def:shift}

\smallskip

$x_0:$ spatial argmax, \eqref{def:x0} page \pageref{def:x0}
\end{multicols}

\section{Parameter list}

Temporal parameters:
\begin{itemize}
\item $T_0,$ the final observation time, is defined in \eqref{def:T0-K0} page \pageref{def:T0-K0}.
\item $T_1$ is an observation time for mixed ${\mathcal F}L^1$-$H^s$ estimates, defined in \eqref{def:delta} page \pageref{def:delta}.
\item $T_2$ is an observation time for purely $H^s$ estimates, introduced in \eqref{def:T2} page \pageref{def:T2}.
\item $T_\star$ in Appendix \ref{app:duh} is arbitrary.
\item $T$ in Appendix \ref{app:symb-bound} is arbitrary.
\end{itemize}

\medskip
Localization parameters:
\begin{itemize}
\item $h > 0$ is a security distance from the resonance (Section \ref{normal-form1} page \pageref{normal-form1}).
\item $\rho > 0$ is the radius of an observation ball in the unstable case (Theorem \ref{theorem1} page \pageref{theorem1}).
\item $\delta_{\varphi_0} > 0$ controls the size of the support of the spatial truncation $\varphi_0.$ It is introduced in the proof of Lemma \ref{lem:Linfty-bd} page \pageref{def:delta:phi0}. There holds $\varphi_0 \to 1$ as $\delta_{\varphi_0} \to 0.$ The smaller $\delta_{\varphi_0},$ the smaller the errors associated with (spatially) non-localized terms.
\item $x_\star$ in Appendix \ref{app:duh} corresponds to the radius of the support of $\varphi_0$ in the main proof.
\end{itemize}

\medskip
Amplitude parameters:
\begin{itemize}
\item $K_a$ measures the consistency of the WKB approximation (defined in \eqref{0a} page \pageref{0a}).
\item $K$ measures the size of the initial perturbation in \eqref{generic-data}.
\item $K_0$ is the amplification exponent in the main result, introduced in \eqref{def:T0-K0} page \pageref{def:T0-K0}.
\item $K'_0$ and $K''_0,$ introduced respectively in \eqref{def:K0prime} page \pageref{def:K0prime} and \eqref{def:K''0} page \pageref{def:K''0}, are the amplification exponents in Theorems \ref{th-three} and \ref{th:4}.
\item $\eta_1$ controls the size of the ${\mathcal F}L^1$ norm of the solution for times $< T_1.$ It can be made arbitrarily small (in statement of Lemma \ref{lem:Linfty-bd} page \pageref{lem:Linfty-bd}).
\item $\eta_2$ controls the size of the $L^\infty$ norm of the solution for times $< T_2.$ It can be made arbitrarily small (in proof of Proposition \ref{prop:ex-Sob} page \pageref{prop:ex-Sob}).
\end{itemize}




\end{document}